\documentclass[11pt]{article}

\usepackage[english]{babel}
\usepackage[utf8]{inputenc}

\usepackage{stmaryrd}

\usepackage{xparse}
\usepackage{amsmath, amsfonts, amssymb, amsthm}
\usepackage{mathtools}
\usepackage{enumitem}

\usepackage{subcaption}

\usepackage{tikz}
\usetikzlibrary{decorations.markings}
\usetikzlibrary{decorations.pathreplacing}

\allowdisplaybreaks[1]
\usepackage{microtype}

\allowdisplaybreaks

\usepackage{calc}

\usepackage{pgf}

\usepackage[affil-it]{authblk}

\newcommand{\R}{\mathbb{R}}

\newcommand{\N}{\mathbb{N}}

\newcommand{\C}{\mathbb{C}}

\newcommand{\ee}{\mathrm{e}}

\DeclareDocumentCommand\dd{ o g d() }{
	\IfNoValueTF{#2}{
		\IfNoValueTF{#3}
			{\mathrm{d}\IfNoValueTF{#1}{}{^{#1}}}
			{\mathinner{\mathrm{d}\IfNoValueTF{#1}{}{^{#1}}\argopen(#3\argclose)}}
		}
		{\mathinner{\mathrm{d}\IfNoValueTF{#1}{}{^{#1}}#2} \IfNoValueTF{#3}{}{(#3)}}
	}

\newcommand{\dx}{\dd{x}}
\newcommand{\dy}{\dd{y}}
\newcommand{\dz}{\dd{z}}
\newcommand{\dr}{\dd{r}}
\newcommand{\dq}{\dd{q}}
\newcommand{\deta}{\dd{\eta}}
\newcommand{\dsig}{\dd{\sigma}}
\newcommand{\dxi}{\dd{\xi}}
\newcommand{\ds}{\dd{s}}
\newcommand{\dt}{\dd{t}}

\newcommand{\dtau}{\dd{\tau}}
\newcommand{\del}{\partial}
\newcommand{\eps}{\varepsilon}
\newcommand{\im}{\mathrm{i}}
\newcommand{\sgn}{\operatorname{sgn}}

\renewcommand{\Re}{\operatorname{Re}}
\renewcommand{\Im}{\operatorname{Im}}

\newcommand{\vcc}{\vcentcolon}

\newcommand{\cut}{\eta}
\newcommand{\coup}{\gamma}
\newcommand{\esp}{\mathcal{Y}}
\newcommand{\dmu}{\Delta\mu}

\newcommand{\LL}{\mathcal{L}}

\newcommand{\T}{{\mathcal{T}}}

\newcommand{\Y}{y}

\newcommand{\abs}[1]{\left\lvert #1 \right\rvert}

\newcommand{\snorm}[2]{\left[ #2 \right]_{\left( #1 \right)}}
\newcommand{\fnorm}[2]{\left\lVert #2 \right\rVert_{\left( #1 \right)}}
\newcommand{\lnorm}[2]{\left\llbracket #2 \right\rrbracket_{\left( #1 \right)}}

\newcommand{\pnorm}[2]{\left\lVert #2 \right\rVert_{\left( #1 \right)}}
\newcommand{\psnorm}[2]{\left[ #2 \right]_{\left( #1 \right)}}
\newcommand{\B}{\mathcal{B}}

\theoremstyle{plain}
\newtheorem{theorem}{Theorem}[section]
\newtheorem{lemma}[theorem]{Lemma}
\newtheorem{proposition}[theorem]{Proposition}

\theoremstyle{definition}

\theoremstyle{remark}
\newtheorem{remark}[theorem]{Remark}

\numberwithin{equation}{section}

\title{A revised proof of uniqueness of self-similar profiles to Smoluchowski's coagulation equation for kernels close to constant}

 \author[1]{B. Niethammer\thanks{\texttt{niethammer@iam.uni-bonn.de}}}
 \author[2]{S. Throm\thanks{\texttt{throm@ma.tum.de}}}
 \author[1]{J.~J.~L.~Vel\'azquez\thanks{\texttt{velazquez@iam.uni-bonn.de}}}
 \affil[1]{Institute of Applied Mathematics, University of Bonn, Endenicher Allee 60, 53115 Bonn, Germany}
 \affil[2]{Faculty of Mathematics, Research Unit 
``Multiscale and Stochastic Dynamics'', Technical University of Munich, 85748 Garching b.~M\"unchen, Germany}

\date{}

\begin{document}
 
\maketitle

\begin{abstract}
In this article we correct the proof of a uniqueness result for self-similar solutions to Smoluchowski's coagulation equation for kernels $K=K(x,y)$ that are homogeneous of degree zero and close to constant in the sense that
\begin{equation*}
 -\eps \leq K(x,y)-2 \leq \eps \left( \Big(\frac{x}{y}\Big)^{\alpha} + \Big(\frac{y}{x}\Big)^{\alpha}\right)
\end{equation*}
for $\alpha \in [0,\frac 1 2)$. Assuming in addition that $K$ has an analytic extension to $\C\setminus (-\infty,0]$ and prescribing the precise asymptotic behaviour of $K$ at the origin, we prove that self-similar solutions with given mass are unique if $\eps$ is sufficiently small.
\end{abstract}

\section{Introduction}

In Smoluchowski's classical mean-field model a system of coagulating particles is described via the number density $\phi(\xi,t)$ of particles of size $\xi>0$ at time $t$. 
Under the assumption that only binary coagulation needs to be taken into account and that correlations can be neglected, $\phi$ satisfies 
the following non-local integral equation
\begin{equation}\label{smolu1}
 \partial_t \phi(\xi,t) = \frac 1 2 \int_0^\xi K(\xi-\eta,\eta) \phi(\xi-\eta,t) \phi(\eta,t)\deta - \phi(\xi,t) \int_0^{\infty} K(\xi,\eta) 
 \phi(\eta,t)\deta =\vcc Q[\phi](\xi)\,.
\end{equation}
Here $K(\xi,\eta)$ denotes the so-called rate kernel, a non-negative and symmetric function, 
that describes the rate at which particles of size $\xi$ and $\eta$ coagulate. 
An important example is  Smoluchowski's classical kernel
\begin{equation}\label{smolukernel}
 K(\xi,\eta) = K_0 \Big( \xi^{1/3} + \eta^{1/3}\Big) \Big( \xi^{-1/3} + \eta^{-1/3}\Big),
\end{equation}
that describes spherical particles in three dimensions that coagulate quickly if they get close to another particle and  otherwise diffuse by Brownian motion.

It is well known, under certain regularity assumptions on $K$, that if $K$ grows at most linearly at infinity, 
that the initial value problem for \eqref{smolu1} is well posed for non-negative locally integrable data with finite mass and that
the mass is conserved for all times. A fundamental hypothesis is the one of dynamic scaling, that presumes that
a solution to \eqref{smolu1} converges to a uniquely determined self-similar solution as $t \to \infty$. This question is only well-understood for the solvable
kernels, such as the constant one (see \cite{MePe04} and the references therein). For non-solvable kernels such as \eqref{smolukernel} so far only existence results
are available \cite{FouLau05,EMR05}, while uniqueness of self-similar solutions remained unknown. For further background and references on this topic we refer to our 
previous article \cite{NV14}.

In the article \cite{NV14} two of us attempted to prove uniqueness of self-similar solutions with finite mass for kernels that are perturbations of the constant one. Unfortunately we realized later that there is a serious gap in the proof. It is the goal of the present paper to give a complete proof of uniqueness of self-similar solutions
with finite mass for kernels that are close to constant. 

More precisely,  we are interested in self-similar solutions of \eqref{smolu1} with finite first moment for kernels with 
homogeneity zero.
Such self-similar solutions are given by
\begin{equation*}
 \phi(\xi,t) = t^{-2} f(x) \qquad \mbox{ with } x = \frac{\xi}{t},
\end{equation*}
where $f$ satisfies
\begin{equation}\label{eq1}
 -xf'(x) - 2f(x) =Q[f](x)\,
\end{equation}
with
  \begin{equation}\label{eq:norm:mass}
  \int_{0}^{\infty}xf\left(x\right)\dx=M
 \end{equation}
for some given positive constant $M$. It is convenient to rewrite equation \eqref{eq1}  as 
\begin{equation}\label{eq1b}
 -\big(x^2 f(x)\big)'=x Q[f](x) = -\partial_x \int_0^x \int_{x-y}^{\infty} K(y,z) y f(z)f(y)\dz\dy
\end{equation}
and by integrating \eqref{eq1b} to reformulate \eqref{eq1} as   
\begin{equation}\label{eq:selfsim}
 x^{2}f\left(x\right)=\int_{0}^{x}\int_{x-y}^{\infty}yK\left(y,z\right)f\left(y\right)f\left(z\right)\dz\dy.
\end{equation}
We call $f$ a self-similar profile with finite mass  if $f \in L^1_{\mathrm{loc}}(\R)$,  $f \geq 0$, $\int x f(x)\,dx < \infty$  and if $f$ satisfies \eqref{eq:selfsim}
for almost all $x \in \R$. 
Notice also, that if $f$ is a solution to \eqref{eq:selfsim}, and if $K$ has homogeneity zero, then the rescaled function
\begin{equation}\label{eq:rescaling}
g(x)=af(ax)\quad \text{for } a>0 
\end{equation}
is also a solution. We can fix the parameter $a$ by fixing $M$ in \eqref{eq:norm:mass}.

Our goal in this paper is to show that solutions to \eqref{eq:selfsim} and \eqref{eq:norm:mass} are unique if the kernel $K$ is close to the constant one in the following sense.

We assume for the kernel $K\colon (0,\infty)^2 \to [0,\infty)$ that it is symmetric, 
and that \begin{equation}\label{kernel1}
 K \mbox{ is homogeneous of degree zero, that is } K(\lambda x,\lambda y) = K(x,y) \; \mbox{ for all } x,y,\lambda>0\,.
\end{equation}
Furthermore we assume that there exists $\eps>0$ and $\alpha \in [0,\frac 1 2)$ such that 
\begin{equation*}
K(x,y)=2+\eps W(x,y)
\end{equation*}
with a symmetric function $W$ that satisfies
\begin{equation}\label{kernel3}
0\leq  W(x,y) \leq  \Big( \Big(\frac{x}{y}\Big)^{\alpha} + \Big (\frac{y}{x}\Big)^{\alpha}\Big)\qquad \mbox{ for all } x,y>0\,.
\end{equation}
Notice that the assumption that $W$ is non-negative is not a restriction since otherwise we could take a different
constant kernel. In addition we need a regularity assumption for $K$. In this paper we assume an analyticity condition on $W$, that is we assume that
\begin{equation}\label{kernel0}
\begin{split}
W(\cdot,1)& \mbox{ has an analytic extension to } \C\setminus (-\infty,0],\\
\Re\left( W\left( \xi ,1\right) \right)& \geq 0\, \quad\ \mbox{ in }  \Re\left( \xi \right) \geq 0,\\
\left\vert W\left( \xi ,1\right) \right\vert &\leq C\left( \frac{1}{\left\vert \xi \right\vert ^{\alpha }}+\left\vert \xi \right\vert
^{\alpha }\right) \,,\qquad  \xi \in \C\setminus\left( -\infty ,0\right]\,. 
 \end{split}
\end{equation}
\begin{remark}
 As a consequence of~\eqref{kernel0} we also have the relation
 \begin{equation}\label{eq:complex:sym}
  W\left(\xi,1\right)=W\left(\xi^{-1},1\right)\quad\text{for }\xi\in\C\setminus(-\infty,0].
 \end{equation}
 This can be seen if one defines $\widehat{W}(\xi)\vcc=W(\xi^{-1},1)$ and notes that~\eqref{eq:complex:sym} holds for $\xi\in(0,\infty)$ due to the symmetry and homogeneity of $W$. The identity theorem of complex analysis then ensures that $\widehat{W}$ and $W(\cdot,1)$ coincide on $\C\setminus(-\infty,0]$.
\end{remark}
Furthermore, if $W_{+}(\xi)=W(\xi,1)$ for $\Im(\xi)>0$ and $W_{-}(\xi)=W(\xi,1)$ for $\Im(\xi)<0$, we assume that
                      \begin{align}
                       W_{+}\; \mbox{ and }& W_{-} \mbox{ have a } C^{1,\gamma}\mbox{-extension to }\{ \Im(\xi) \geq 0 \}\setminus\left\{0\right\} \mbox{ and }
                       \{\Im(\xi) \leq 0 \}\setminus\left\{0\right\} \mbox{ respectively }\nonumber\\
                       &\mbox{ for some } \gamma >0\,. \label{kernel0b}
                      \end{align}
The precise meaning of this is that $W_{+}$ and $W_{-}$ are differentiable and for $\abs{z_1-z_2}\leq \frac{1}{2}\min\left\{\abs{z_1},\abs{z_2}\right\}$ it holds
\begin{equation}\label{eq:hoelder}
 \frac{\abs{W_{\pm}'\left(z_1\right)-W_{\pm}'\left(z_2\right)}}{\abs{z_1-z_2}^{\gamma}}\leq C\begin{cases}
                                                                                              \min\left\{\abs{z_1},\abs{z_2}\right\}^{-\alpha-1-\gamma} &\min\left\{\abs{z_1},\abs{z_2}\right\}\leq 1\\
                                                                                              \min\left\{\abs{z_1},\abs{z_2}\right\}^{\alpha-1-\gamma} &\min\left\{\abs{z_1},\abs{z_2}\right\}\geq 1.
                                                                                            \end{cases}
\end{equation}
This means $W'_{\pm}$ satisfies a uniform local Hölder condition. Furthermore we assume a growth condition on $W_{+}$ and $W_{-}$ and their derivatives. Precisely, we need
\begin{equation}\label{eq:W:decay}
 \abs{W_{\pm}\left(z\right)}\leq C\left(\abs{z}^{-\alpha}+\abs{z}^{\alpha}\right)\quad \text{and}\quad \abs{W_{\pm}'\left(z\right)}\leq C\left(\abs{z}^{-\alpha-1}+\abs{z}^{\alpha-1}\right).
\end{equation} 
\begin{remark}
 Note that the estimate on $W_{\pm}$ in~\eqref{eq:W:decay} is a consequence of~\eqref{kernel0} and also follows from that one on $W'_{\pm}$ by integration.
\end{remark}
In addition, for $\alpha>0$, we have to assume some precise asymptotic behaviour of the kernel $W$ near the origin in the complex plane, i.e.\ we assume there exists some constant $C_W>0$ such that
\begin{equation}\label{eq:asymp:W}
 W\left(\xi,1\right)\sim C_W\xi^{-\alpha}\quad \text{as }\xi\to 0,
\end{equation}
while the precise meaning of the notation $\sim$ here is the following. Given some function $\varphi$ that is analytic in  $\C\setminus\left(-\infty,0\right]$, we write  
\begin{equation*}
 \varphi\left(z\right)\sim A z^{-\alpha} \quad \text{as } z\to 0
\end{equation*}
 if there exists some real function $\eta\geq 0$ such that
 \begin{equation}\label{eq:def:asymp}
  \begin{split}
   \abs{z^{\alpha}\varphi\left(z\right)-A}\leq \eta\left(z\right)\quad \text{and}\quad \lim_{\rho\to 0^{+}}\sup_{\stackrel{\abs{\xi}=\rho}{\xi\in \C\setminus\left(-\infty,0\right]}}\eta\left(\xi\right)=0.
  \end{split}
 \end{equation}
These rather strong assumptions are necessary due to our strategy for the proof. We will explain this point in more detail
below, but want to mention here, that this condition is satisfied by most kernels that 
one typically encounters in applications as for example~\eqref{smolukernel}.

Our result can now be formulated as follows.
\begin{theorem}\label{Thm:uniqueness}
 Assume that $K$ satisfies the assumptions~\eqref{kernel1}\textendash \eqref{eq:asymp:W} and let $f_1$ and $f_2$ be two self-similar 
 profiles that satisfy~\eqref{eq:norm:mass}. Then  $f_1=f_2$ if $\eps$ is sufficiently small.
\end{theorem}

\begin{remark}\label{Rem:alpha:pos}
 Since the statement of Proposition~\ref{P.repkernels} does not hold in this form for $\alpha=0$ the proof of Theorem~\ref{Thm:uniqueness} would in principle have to be adapted in this case. However, we can reduce the case $\alpha=0$ to that one with small positive $\tilde{\alpha}$ by noting that if $K$ satisfies~\eqref{kernel1}\textendash \eqref{eq:W:decay} with $\alpha=0$ the same assumptions also hold with any $\tilde{\alpha}\in(0,1/2)$. On the other hand, \eqref{eq:asymp:W} does not hold necessarily for a kernel $K$ satisfying~\eqref{kernel1}\textendash \eqref{eq:W:decay} with $\alpha=0$. However, \eqref{eq:asymp:W} is only used in the proof of Proposition~\ref{Prop:Q0:estimate}, and the result of this Proposition holds if $\alpha=0$ even if~\eqref{eq:asymp:W} is not assumed. For this, we have to choose $\tilde{\alpha}$ sufficiently small. Thus, to simplify the presentation, we will assume throughout this work that $\alpha\in(0,1/2)$ while we only give some comments on the necessary adaptations in the proof of Proposition~\ref{Prop:Q0:estimate} for the case that~\eqref{eq:asymp:W} does not hold. 
\end{remark}

The paper is organized as follows. In Section~\ref{Sec:main:ideas} we collect several results that were obtained in~\cite{NV13a,NV14} and we introduce some weighted $L^{\infty}$-norm on the Laplace transform on the real line. The reason for this choice is that we can show that any self-similar profile is close to the one for the constant kernel in this norm. Unfortunately this norm is very weak and thus, in order to show uniqueness by some contraction argument, it is necessary to represent all functionals that appear in terms of the Laplace transform. This is the main reason for the analyticity assumption on $W$ as well as~\eqref{kernel0b} which allows for a representation of $W$ as the Laplace transform of some function $\Gamma$ (see Section~\ref{Sec:rep:W}). Another main difficulty is due to the possible singularity of $W$ at the origin that induces a boundary layer. More precisely, in the region close to zero the self-similar profile for the perturbated kernel $K$ deviates significantly from the one for the constant kernel. This is the reason why we have to restrict ourselves to exponents $\alpha<1/2$ and why we need the assumption~\eqref{eq:asymp:W}. On the other hand, many kernels appearing in applications such as~\eqref{smolukernel} exhibit such a singular behaviour at the origin.

Furthermore in Section~\ref{Sec:prefactor} we also consider uniqueness for the prefactor equation, which is a related but much simpler problem, arising in the study of the asymptotic behaviour of solutions to~\eqref{eq:selfsim} for large cluster sizes.

In Section~\ref{Sec:proof:unique} we collect the key estimates to show uniqueness of self-similar profiles that are contained in Propositions~\ref{Prop:uniform:bound:muk}\textendash\ref{Prop:splitting:2}. Using these estimates we then give the proof of Theorem~\ref{Thm:uniqueness} (see Section~\ref{Sec:proof:alpha}).

Sections~\ref{Sec:est:m:large:x}\textendash\ref{Sec:asymptotics} are then dedicated to the proofs of these key estimates, while the hardest part is the boundary layer estimate (see Proposition~\ref{Prop:splitting:2}) and Sections~\ref{Sec:bl:proof}\textendash\ref{Sec:asymptotics} deal exclusively with this issue.

Finally in the Appendix we collect several properties concerning the norm (Section~\ref{Sec:prop:norm}) as well as the proof of the representation formula for $W$ together with some useful estimates (Section~\ref{Sec:appendix:rep:ker}) and the proof of the uniqueness result for the prefactor equation (Section~\ref{Sec:proof:prefactor}).

\section{Main ideas, Basic Definitions and Notation}\label{Sec:main:ideas}

\subsection{Rewriting the equation}

In order to prove uniqueness via a contraction type argument, it is first necessary to prove that any 
solution to \eqref{eq:selfsim} is close in a suitable norm to one of the solutions $\bar{f}_{a}=ae^{-ax}$ of the kernel $\bar{K}=2$. 
 
 In \cite{NV14} we have established this fact in a weighted $L^{\infty}$-norm of the corresponding Laplace transforms (see~\eqref{S2E5}).
 It turns out  that instead of fixing the mass it is more convenient to normalize the function $f$ such that it decays as $e^{-x}$ as $x\to \infty$. Accordingly we define $\mu(x)=f(x)e^x$ such that $\mu$ solves the equation
\begin{equation}\label{eq:mu}
 x^{2}\mu\left(x\right)=\int_{0}^{x}\int_{x-y}^{\infty}yK\left(y,z\right)\mu\left(y\right)\mu\left(z\right)\ee^{x-\left(y+z\right)}\dz\dy.
\end{equation}
In the case $K=2$ we have the explicit solution $\bar{\mu}=1$.

For a fixed kernel $K$ we furthermore define the bilinear form $B_{K}$ given by
\begin{equation*}
 B_{K}\left(\mu_1,\mu_2\right)\vcc=\frac{1}{x^2}\int_{0}^{x}\int_{x-y}^{\infty}yK\left(y,z\right)\mu_{1}\left(y\right)\mu_{2}\left(z\right)\ee^{x-\left(y+z\right)}\dz\dy.
\end{equation*}
Rewriting $\mu$ again as $\mu\left(x\right)=1+m\left(x\right)=\bar{\mu}\left(x\right)+m\left(x\right)$ we can rewrite equation~\eqref{eq:mu} as
\begin{equation}\label{eq:mu2}
 m\left(x\right)-B_{2}\left(\bar{\mu},m\right)-B_{2}\left(m,\bar{\mu}\right)=B_{2}\left(m,m\right)+\eps B_{W}\left(\bar{\mu}+m,\bar{\mu}+m\right).
\end{equation}

We furthermore denote in the following by 
\begin{equation*}
 \zeta\left(x\right)\vcc=\ee^{-x}
\end{equation*}
a cut-off function that is suitable since we work with Laplace transforms.

\subsection{Results of previous articles}\label{Sec:prev:res}

In this section we collect several results and estimates obtained in~\cite{NV13a} and~\cite{NV14}. Furthermore, here and in the rest of the paper all constants are independent of $\eps$.

 We have proved in \cite{NV13a} (cf. Theorem 1.2 and 1.3) that for any self-similar profile $f$
there exist positive constants $c_1$ and $c_2$ such that 
\begin{equation}\label{fdecay}
 0<e^{-c_1x} \leq f(x) \leq e^{-c_2x} \,, \qquad \mbox{ for all } x \geq 1
\end{equation}
and if $f(x)=e^{-xa(x)}$ that $\lim_{x \to \infty} a(x)$ exists.

\begin{remark}\label{Rem:rescaling}
 Note that in the following, instead of fixing the mass as done in~\eqref{eq:norm:mass}, we will always assume that all solutions to~\eqref{eq:selfsim} are rescaled in such a way that $\lim_{x \to \infty} a(x)=1$. More precisely, suppose that we have two solutions $f_1$ and $f_2$ of~\eqref{eq:selfsim} with the normalisation~\eqref{eq:norm:mass}. Then it might be shown that choosing $a_1,a_2$ (close to one if $\eps$ is small) we obtain that the Laplace transforms of the functions $g_1$ and $g_2$, defined by means of~\eqref{eq:rescaling} with $a_1$ and $a_2$, are analytic in $\Re\left(p\right)>-1$ and they are singular at $p=-1$. Notice that this is a way of stating (in Laplace transform variables) that $\frac{\log\left(g_{k}\left(x\right)\right)}{x}\to -1$ as $x\to\infty$ for $k=1,2$. We will always denote from now on the rescaled functions $g_1,g_2$ by $f_1,f_2$.
\end{remark}

In \cite{NV14} it was moreover proved (see \cite[Lemma 3.1]{NV14}) that for any given $\delta>0$ we have
\begin{equation}\label{S2E5}
\sup_{p>-1} \frac{1+p}{|p|}\Big| \int_0^{\infty} \big(1-e^{-px}\big) (f(x) - \bar f(x))\dx \Big| \leq \delta
\end{equation}
if $\eps$ is sufficiently small. Here $\bar f(x)=e^{-x}$ is the self-similar profile for $K=2$ with exponential decay.

In addition, we have also established a control on $f$ for small values of $x$. In fact, given $\eta>0$
there exists $R_0>0$ such that 
\begin{equation}
 \label{S2E3}
 \int_{R}^{2R} f(x)\dx \leq C_{\eta} R^{1-\eta} \,, \qquad \mbox{ for all } R \in (0,R_0],
\end{equation}
which implies $\int_0^1 x^{-\delta} f(x) \dx \leq C_{\delta}$ for all $\delta \in [0,1)$ (cf.~\cite[Lemma 2.4]{NV14}).

Furthermore it was also shown (see \cite[Lemma~{2.3}]{NV14}) that there exists a constant $C>0$ independent of $f$ such that
\begin{equation}\label{eq:uniform:est:f}
 \int_{0}^{\infty}f\left(x\right)\dx\leq C
\end{equation}
and in~\cite[Lemma~5.2]{NV13a} that for any $A>$ there exists some constant $C_A>0$ such that for each $\mu$ solving~\eqref{eq:mu} it holds
\begin{equation}\label{eq:apriori:large}
 \int_{1}^{AR}\mu\left(x\right)\dx\leq C_A R \quad \text{for all } R\geq 2. 
\end{equation}
From~\eqref{S2E3}, \eqref{eq:uniform:est:f} and~\eqref{eq:apriori:large} it follows that there exists $C_{\eta}>0$ such that 
\begin{equation}\label{eq:mean:est}
 \int_{R}^{2R}\mu\left(x\right)\dx\leq C_{\eta}\max\left\{R^{1-\eta},R\right\} \quad \text{and}\quad \int_{R}^{2R}\abs{m\left(x\right)}\dx\leq C_{\eta}\max\left\{R^{1-\eta},R\right\} 
\end{equation}
for all $R>0$, $\eta\in\left(0,1\right)$ and all solutions $\mu$ of~\eqref{eq:mu} and $m$ of~\eqref{eq:mu2}.

\subsection{Defining suitable norms and spaces}\label{Sec:norm:space}

In the following we will denote by $\mathcal{M}$ the space of signed Radon measures. We define suitable norms on the Laplace transforms in the following way. For a given measure $\omega\in \mathcal{M}$ we denote by $\Omega$ the Laplace transform of $\omega$, i.e.\ 
\begin{equation*}
 \Omega\left(p\right)\vcc=\int_{0}^{\infty}\omega\left(x\right)\ee^{-px}\dx.
\end{equation*}
To be precise, we have to restrict to the subclass $\mathcal{M}_t\subset\mathcal{M}$ of those measures whose Laplace transform exists in the sense that if we split $\omega=\omega_{+}-\omega_{-}$ it holds
\begin{equation*}
 \begin{split}
  \int_{0}^{\infty}\ee^{-px}\omega_{+}\left(x\right)\dx<\infty\quad \text{and}\quad \int_{0}^{\infty}\ee^{-px}\omega_{-}\left(x\right)\dx<\infty\quad \text{for all } p>0.
 \end{split}
\end{equation*}
Then for any weight $\chi>0$ and $k\in\left\{0,1,2\right\}$ we define the (semi-) norms
\begin{align}
 \snorm{k,\chi}{\omega}&\vcc=\sup_{p>0}\left(\frac{p^{1+k}}{\left(1+p\right)^{1-\chi}}\abs{\del_{p}^{k}\Omega\left(p\right)}\right)\label{eq:def:snorm}
\end{align}
and the corresponding norms
\begin{align*}
 \fnorm{k,\chi}{\omega}&\vcc=\sum_{\ell=0}^{k}\snorm{\ell,\chi}{\omega}.
\end{align*}
From this definition it directly follows for $k=0,1,2$ that
\begin{equation}\label{eq:norm:order}
 \fnorm{k,\chi_1}{\omega}\leq \fnorm{k,\chi_2}{\omega} \quad \text{and}\quad \snorm{k,\chi_1}{\omega}\leq\snorm{k,\chi_2}{\omega}\quad \text{if } \chi_1\leq \chi_2.
\end{equation}
We furthermore define the following subspace of $\mathcal{M}$
\begin{equation*}
 \esp_{k,\chi}\vcc=\left\{\omega\in \mathcal{M}_t\;\left|\; \fnorm{k,\chi}{\omega}<\infty\right.\right\}.
\end{equation*}
Additionally, we introduce the following norms on the level of Laplace transforms, that correspond to the norms $\fnorm{k,\chi}{\cdot}$: for a function $G\in C^k\left(0,\infty\right)$ and a weight $\chi>0$ we define 
\begin{equation*}
 \begin{split}
  \lnorm{0,\chi}{G}&\vcentcolon=\sup_{p>0}\frac{p}{\left(1+p\right)^{1-\chi}}\abs{G\left(p\right)},\\
  \lnorm{1,\chi}{G}&\vcc=\lnorm{0,\chi}{G}+\sup_{p>0}\frac{p^2}{\left(1+p\right)^{1-\chi}}\abs{\del_p G\left(p\right)},\\
  \lnorm{2,\chi}{G}&\vcc=\lnorm{1,\chi}{G}+\sup_{p>0}\frac{p^3}{\left(1+p\right)^{1-\chi}}\abs{\del_p^2G\left(p\right)}.
 \end{split}
\end{equation*}
For $k=0,1,2$ we then also define the Banach spaces
\begin{equation*}
 X_{k,\chi}\vcc=\left.\left\{G\in C^k\left(0,\infty\right)\;\right|\;\lnorm{k,\chi}{G}<\infty\right\}.
\end{equation*}
The relation between the two norms and spaces is then the following. Denoting by $\T$ the operator mapping a measure $\omega$ to its Laplace transform $\Omega$, it holds for $k=0,1,2$
\begin{equation*}
 \begin{split}
  \T\colon \esp_{k,\chi}&\to X_{k,\chi}\\
   \omega&\mapsto \Omega
 \end{split}
\end{equation*}
and $\T$ is then an isometry, i.e.\ we have
\begin{equation}\label{eq:isometry}
 \fnorm{k,\chi}{\omega}=\lnorm{k,\chi}{\T \omega}.
\end{equation}
The reason why we have to introduce these different norms is the following. The space $\esp_{k,\chi}$ is not a Banach space as its completion contains for example derivatives of Dirac-measures, i.e.\ considering for some $a> 0$ the sequence $\omega_n\vcc=n\left(\delta_{a+1/n}-\delta_a\right)$, where $\delta_x$ denotes the Dirac-measure at $x$, it easily follows
\begin{equation*}
 \omega_n\longrightarrow \delta_a'\quad \text{as } n\longrightarrow \infty \text{ with respect to the } \fnorm{k,\chi}{\cdot}\text{-norm}.
\end{equation*}
Instead, the completion of $\esp_{k,\chi}$ with respect to $\fnorm{k,\chi}{\cdot}$ is a space of generalised distributions, i.e.\ acting on a family of test functions that are analytic functions on the real line.

On the other hand the space $X_{k,\chi}$ with the corresponding norm is a Banach space. In Section~\ref{Sec:Inversion} we will have to invert the linearised coagulation operator (see~\eqref{eq:lin:coag} and~\eqref{eq:lin:coag:lap}) and this will be done by computing an explicit formula in terms of the Laplace transform. The problem is, that it is not immediately clear, if, starting with some right-hand side in $\esp_{k,\chi}$, this inverse again corresponds to some measure in $\esp_{k,\chi}$, i.e.\ that the operator can also be inverted on $\esp_{k,\chi}$. However we can obtain estimates for the Laplace transforms of the required quantities. 

 To simplify notation we also define the following function
\begin{equation*}
 \Lambda_{\chi}\left(s\right)\vcc=\begin{cases}
                         s^{-1} & s\leq 1\\
                         s^{-\chi} & s\geq 1
                        \end{cases}
\end{equation*}
and remark that it holds
\begin{equation}\label{eq:Lambda}
 \Lambda_{\chi}\left(s\right)\leq \frac{\left(1+s\right)^{1-\chi}}{s}\leq 2^{1-\chi}\Lambda_{\chi}\left(s\right).
\end{equation}
Furthermore, in the following the exponent $\chi$ in the weight will mostly be some fixed value
\begin{equation*}
 \theta\in\left(\alpha,1/2\right).
\end{equation*}
Therefore, to shorten the notation we will drop in the following the index if $\chi=\theta$ and just write
\begin{equation*}
 \fnorm{k}{\cdot}\vcc=\fnorm{k,\theta}{\cdot}, \qquad \snorm{k}{\cdot}\vcc=\snorm{k,\theta}{\cdot}, \qquad \lnorm{k}{\cdot}\vcc=\lnorm{k,\theta}{\cdot}\qquad \text{and} \qquad \Lambda\vcc=\Lambda_{\theta} \qquad \text{for } k\in\left\{0,1,2\right\}
\end{equation*}
as well as
\begin{equation*}
\esp_{k}\vcc=\esp_{k,\theta},\qquad \esp\vcc=\esp_{2,\theta},\qquad X_{k}\vcc=X_{k,\theta}\qquad \text{and}\qquad X\vcc=X_{2,\theta}\qquad \text{for } k\in\left\{0,1,2\right\}.  
\end{equation*}

\subsection{Representation formula for the kernel $W$}\label{Sec:rep:W}

In view of the norms given in~\eqref{eq:def:snorm} it will be necessary to express $W$ as the Laplace transform of some function. In fact, it will be more convenient to have this representation for $W\left(y,z\right)/(y+z)$ which is the content of the following Proposition.

\begin{proposition}\label{P.repkernels}
 Suppose that $W$ satisfies the assumptions~\eqref{kernel1}\textendash\eqref{kernel0b} with $\alpha\in(0,1)$.
Then there exists a measure $\Gamma$ on $\mathbb{R}_{+}\times \mathbb{R}_{+}$ that is homogeneous of degree $-1$ and satisfies
\begin{equation*}
\frac{W(y,z)}{y+z}=\int_{0}^{\infty}\int_{0}^{\infty}\Gamma(\xi ,\eta) e^{-\xi y-\eta z}\dxi\deta.
\end{equation*}
Furthermore, $\Gamma$ can be decomposed as $\Gamma(\xi,\eta)=\widetilde{\Gamma}(\xi,\eta)+W_{\pm}(-1)\delta(\xi-\eta)$, with a function $\widetilde{\Gamma} \colon \mathbb{R}_{+}\times \mathbb{R}_{+}\rightarrow \mathbb{R}$ satisfying
\begin{equation*}
 \abs{\widetilde{\Gamma}( \xi ,\eta )} \leq \frac{C}{(\xi +\eta )^{1-\alpha }}\left( \frac{1}{\xi^{\alpha}}+\frac{1}{\eta^{\alpha}}\right).
\end{equation*}
\end{proposition}

\begin{remark}\label{Rem:W:minus1}
 From~\eqref{eq:complex:sym} together with~\eqref{kernel0b} it follows that $W_{\pm}\left(-1\right)$ is a well-defined constant, i.e. $W_{+}(-1)=W_{-}(-1)=W_{\pm}(-1)$.
\end{remark}
The proof of Proposition~\ref{P.repkernels} can be found in the Appendix (Section~\ref{Sec:appendix:rep:ker}) together with several estimates on integrals involving $\Gamma$.
\begin{remark}
If  $W\left( x,y\right) =\big( \frac{x}{y}\big) ^{\alpha }+\big( 
\frac{y}{x}\big) ^{\alpha }$, then $\Gamma(\xi,\eta)= \frac{\sin(\pi\alpha)}{\pi} \frac{ \big( \frac{\xi}{\eta}\big)^{\alpha} - \big( \frac{\eta}{\xi}\big)^{\alpha}}{\xi-\eta}+2\cos(\pi\alpha)\delta(\xi-\eta)$.
\end{remark}

\subsection{The boundary layer for $x\to 0$.}\label{Ss.boundarylayer}

Another main difficulty, that originates from the fact that we allow for unbounded kernels $W$, concerns the situation that 
solutions of~\eqref{eq:selfsim} for small $x$ for unbounded kernels behave in general differently from solutions of~\eqref{eq:selfsim} with $K=2$. Indeed, $\bar{f}_{a}(x) $ converges to constant as $x\rightarrow 0^{+}$.
However, if we take $W( y,z) =\big( \frac{y}{z}\big) ^{\alpha}
+\big( \frac{z}{y}\big)^{\alpha }$ with $\alpha >0$ the expected
behaviour of the solutions of~\eqref{eq:selfsim} for small $x$ is $f\left(
x\right) \rightarrow 0$ as $x\rightarrow 0^{+}$ for any $\varepsilon >0$ 
and as a consequence we have  the onset of a boundary layer near $x=0$.
In order to examine this effect it is convenient to reformulate~\eqref{eq:selfsim} as
\[
x^{2}f(x) =\int_{0}^{x} yf(y) \int_{0}^{\infty
}K(y,z) f(z)\dz\dy
-\int_{0}^{x}dy\int_{0}^{x-y}dz\,K(y,z) yf(y)f(z)\dz\dy. 
\]
It turns out that the leading order terms of this equation for small $x$ are 
\[
x^{2}f(x) =\int_{0}^{x}yf(y) \int_{0}^{\infty
}K(y,z) f(z) \dz\dy.
\]
This equation can be reduced to a linear ODE whose solution is given by
\[
x^{2}f(x) =C\exp \left( -G\left( x;f\right) \right) \,, \qquad \mbox{ with } \quad
G\left( x;f\right) =\int_{x}^{1}\frac{\int_{0}^{\infty }K(t,z)
f(z) \dz}{t}\dt\,
\]
for a suitable real constant $C.$ If $W( y,z) =\big( \frac{y}{z}
\big) ^{\alpha }+\big( \frac{z}{y}\big) ^{\alpha }$,  we
have the following asymptotics for $G\left( x;f\right) $ as $
x\rightarrow 0$
\[
G( x ;f) \sim -\beta ( f) \log \left( x\right) +\frac{\varepsilon }{\alpha }\frac{1}{x^{\alpha }}\int_{0}^{\infty }z^{\alpha
}f(z) \dz-L( x;f)\,, \qquad \mbox{ with }\;
\beta( f) =2\int_{0}^{\infty }f(z) \dz\,,
\]
where we include in $L( x;f)$ all bounded terms. Notice that
the operator $L( x;f) $ contains functionals like $\int_{0}^{\infty }z^{-\alpha }f(z) dz.$
Then, near $x=0$, equation~\eqref{eq:selfsim} has similar
properties as
\begin{equation*}
f(x) =Cx^{\beta ( f) -2}\exp \left( -\frac{\varepsilon }{\alpha }\frac{1}{x^{\alpha }}\int_{0}^{\infty }z^{\alpha
}f(z)\dz +L\left( x;f\right) \right).
\end{equation*}
If $\varepsilon $ is small, we  expect $f$ to be close to one of the
functions $\bar{f}_{a}$ except in a neighbourhood of $x=0.$ Then we 
have $\beta(f) $ close to $\beta (\bar{f}_{a}) =2.$
Therefore, the main change of $f$ near $x=0$ due to the singular character
of $W$ is due to the exponential factor $\exp \left( -\frac{1}{\alpha }\frac{1}{x^{\alpha }}\int_{0}^{\infty }z^{\alpha }f( z) dz\right)$ and
we need
to use norms in spaces containing objects for which $\int_{0}^{\infty
}dz\,z^{-\alpha }f(z)$ is well defined, such as for example a weighted $L^1$-norm. 
Then the functional
\[
f\rightarrow \exp \left( -\frac{\varepsilon }{\alpha }\frac{1}{x^{\alpha }}\int_{0}^{\infty }z^{\alpha }f(z)\dz \right) 
\]
is Lipschitz continuous  if $\alpha <\frac{1}{2}.$ This is the reason for our  additional restriction on $W$ in 
\eqref{kernel3} and it seems that this problem can also not be removed using
different types of norms.

\subsection{The prefactor equation}\label{Sec:prefactor}

We would also like to point out another related but simpler problem, namely the uniqueness
of solutions to 
\begin{equation}\label{S1E4}
x\mu ( x) =\tfrac{1}{2}\int_{0}^{x} K( y,x-y) \mu( y) \mu ( x-y) \dy   
\end{equation}
where $K$ also satisfies \eqref{kernel1}\textendash\eqref{kernel0b}. This equation arises naturally in the study of the asymptotics of solutions of \eqref{eq:selfsim}. Indeed, it has been proved in \cite{NV13a} that solutions of \eqref{eq:selfsim} have the following property. We define $m( x) =\ee^{ax}f( x)$. Then, there exists $a>0$ such that, up to subsequences, we have that
$\lim_{R\rightarrow \infty }m( Rx) =\mu ( x)$, where $\mu $ solves \eqref{S1E4}. Moreover, it has been shown, that there exist constants $d_{*},D_{*}>0$ such that
\begin{equation}\label{eq:prefactor:scaling}
 d_{*}x\leq \int_{0}^{x}\mu(y)\dy\leq D_{*}x\quad \text{for all }x>0.
\end{equation}
 If we define, for $k=1,2$, in analogy with~\eqref{eq:def:snorm} the (semi-) norms
 \begin{equation}\label{norm1def}
\psnorm{k}{\omega}\vcc=\sup_{p>0}\left( p^k\abs{ \Omega( p)}\right) =\sup_{p>0} \left (p^k \abs{\int_{0}^{\infty }\omega( x)
\ee^{-xp}\dx}\right )\quad \text{and} \quad \pnorm{k}{\omega}\vcc=\sum_{\ell=0}^{k}\psnorm{\ell}{\omega}.
\end{equation}
then it turns out that for $K=2$ the only solution of (\ref{S1E4}) satisfying $\pnorm{1}{\mu}<\infty $ is $\bar{\mu}=1.$ The
results in \cite{NV13a} imply that $\pnorm{0}{\mu -\bar{\mu}}$ can be made small for small $\varepsilon $. Then we have the following uniqueness result.

\begin{lemma}\label{L.uniqueprefactor}
Let $K$ satisfy the assumptions~\eqref{kernel1}\textendash\eqref{kernel0b}, where here $\alpha \in [0,1)$.
Then, if $\eps>0$ is sufficiently small, the only solution to~\eqref{S1E4} which satisfies~\eqref{eq:prefactor:scaling} and which has finite norm $\pnorm{0}{\cdot}$ 
is the constant one.
\end{lemma}

The methods used to prove this result are similar to the ones indicated above, although much simpler. We do not need to assume that $\alpha <\frac{1}{2}$ since no boundary layer analysis is necessary here. In particular we just need to prove suitable
Lipschitz properties for the right-hand side of (\ref{S1E4}) in the norm $\pnorm{1}{\cdot}$. The interested reader can find the proof in Section~\ref{Sec:proof:prefactor}.  

\section{Uniqueness of self-similar profiles}\label{Sec:proof:unique}

In this section we will first outline the strategy to show uniqueness and collect the necessary estimates that are partly quite technical and will be proved in the subsequent sections. Furthermore we remark that all constants $C$ occurring are independent of $\eps$, while they might depend on $W$.

\subsection{Main estimates}\label{Sec:main:est}

To begin with we state two a-priori estimates, where the first one gives a uniform bound on solutions $\mu_k$ of~\eqref{eq:mu2}, while the second one shows that each solution $m_k$ of~\eqref{eq:mu3} is small in the $\fnorm{2}{\cdot}$-norm for $\eps\to 0$.

\begin{proposition}[Uniform bound on $\mu_k$]\label{Prop:uniform:bound:muk}
  For $\eps>0$ sufficiently small there exists some constant $C>0$ such that for each solution $\mu_k$ of~\eqref{eq:mu} it holds
 \begin{equation*}
  \fnorm{2}{\mu_k}<C.
 \end{equation*}
\end{proposition}

\begin{proposition}[Smallness of $m_k$]\label{Prop:smallness:mk}
  Given $\delta>0$ there exists $\eps>0$ such that for every solution $m_k$ of~\eqref{eq:mu2} it holds
 \begin{equation*}
  \fnorm{2}{m_k}\leq \delta.
 \end{equation*}
\end{proposition}
The proofs may be found in Sections~\ref{Sec:proof:uniform:bd:muk} and~\ref{Sec:proof:smallness} respectively.

We have furthermore the following estimates on the bilinear forms $B_K$ and $B_W$, which are shown in Section~\ref{Sec:est:m:large:x}.

\begin{proposition}[Estimates for the quadratic terms]\label{Prop:bilinear}
 For any $\theta\in\left(\alpha,1/2\right)$ there exists $C>0$ such that for all functions $\omega_1$ and $\omega_2$ with $\fnorm{2}{\omega_{i}}<\infty$ it holds
 \begin{align}
  \snorm{2}{B_{2}\left(\omega_1,\omega_2\right)}&\leq C\fnorm{1}{\omega_1}\fnorm{1}{\omega_2} \label{eq:bil:1}\\
  \fnorm{2,\theta-\alpha}{B_W\left(\omega_1,\omega_2\right)}&\leq C\fnorm{2}{\omega_1}\fnorm{2}{\omega_2}\label{eq:bil:2}.
 \end{align}
 \end{proposition}
 
 \begin{remark}\label{Rem:closedness:bilinear}
  Assuming $\omega_1,\omega_2\in \esp_{2}$ it holds
  \begin{equation*}
   B_2\left(\omega_1,\omega_2\right)\in \esp_{2}\quad \text{and}\quad B_W\left(\omega_1,\omega_2\right)\in \esp_{2,\theta-\alpha}.
  \end{equation*}
 \end{remark}

 We will also need to invert the linearised coagulation operator that is given by
\begin{equation}\label{eq:lin:coag}
 \LL\colon\omega\mapsto \omega-\left(B_2\left(\bar{\mu},\omega\right)+B_2\left(\omega,\bar{\mu}\right)\right)
\end{equation}
and reads in Laplace variables as
\begin{equation}\label{eq:lin:coag:lap}
 \widehat{\LL}\left(M\right)\vcc=M-2\int_{p}^{\infty}\int_{q}^{\infty}\frac{M\left(r\right)-M\left(1\right)}{\left(1-r\right)r^2}-\frac{M'\left(r\right)}{r}\dr\dq.
\end{equation}
We have the following result.
\begin{proposition}[Properties of $\widehat{\LL}^{-1}$]\label{Prop:Lin:inv}
 For $k=1,2$ and $\chi\in\left(0,1\right)$ the operator $\widehat{\LL}\colon X_{k,\chi}\to X_{k,\chi}$ is well-defined, bounded and invertible with bounded inverse, i.e.\ there exists $\widehat{\LL}^{-1}$ and constants $C_1$, $C_2$ such that 
 \begin{itemize}
  \item $\lnorm{k,\chi}{\widehat{\LL} M}\leq C_1\lnorm{k,\chi}{M}$
  \item $\lnorm{k,\chi}{\widehat{\LL}^{-1}G}\leq C_2 \lnorm{k,\chi}{G}$
 \end{itemize}
 for all $M,G\in X_{k,\chi}$. Furthermore we have the following commutator relation
 \begin{itemize}
  \item $\T \LL \omega=\widehat{\LL}\T\omega$ for all $\omega\in\esp_{\ell,\chi}$ with $\ell=1,2$.
 \end{itemize}
\end{proposition}

The proof of this Proposition is given in Section~\ref{Sec:Inversion}. Furthermore, for $\alpha>0$, we have the following regularising effect for the difference of two solutions of~\eqref{eq:mu3} that is proven in Section~\ref{Sec:boundary:layer}.
 
\begin{proposition}[Regularising effect]\label{Prop:reg}
 Assume $\alpha\in\left(0,1/2\right)$. There exists some constant $C>0$ such that for any two solutions $m_1$ and $m_2$ of~\eqref{eq:mu2} the difference $m=m_1-m_2$ satisfies
 \begin{equation*}
  \fnorm{2}{m}\leq C\fnorm{0}{m}.
 \end{equation*}
\end{proposition}

Finally we have the following boundary layer estimate which we need for $\alpha>0$ and whose proof is the most technical part of the paper and will be given in Sections~\ref{Sec:boundary:layer}\textendash\ref{Sec:asymptotics}.

\begin{proposition}[Boundary layer estimate]\label{Prop:splitting:2}
  Assume $\alpha\in\left(0,1/2\right)$. For any $\delta_{*}>0$ it holds, if $\eps>0$ is sufficiently small,  for the difference $m=m_1-m_2$ of two solutions $m_1$ and $m_2$
 \begin{equation*}
  \fnorm{2}{\zeta m}\leq \delta_{*}\fnorm{2}{m}+C_{\delta_{*}}\fnorm{0}{\left(1-\zeta\right)m}.
 \end{equation*}
\end{proposition}

The main idea behind this is a localisation principle, i.e.\ the quantity $\zeta m$ contains the information of $m$ for small values of $x$, while the information about large values of $x$ is contained in $\left(1-\zeta\right)m$. 

We come now to the proof of uniqueness of self-similar profiles. As in the case of bounded kernels $W$, i.e.\ $\alpha=0$, the boundary layer estimates are not needed we first consider the case $\alpha>0$, which is the complicated one and we will treat $\alpha=0$ afterwards separately.

\subsection{Proof of Theorem~\ref{Thm:uniqueness}}\label{Sec:proof:alpha}
Taking the difference $m=m_1-m_2$ of two solutions $m_1,m_2$ of~\eqref{eq:mu2} it satisfies
\begin{equation}\label{eq:diff:m}
 \begin{split}
  m-\left(B_{2}\left(\bar{\mu},m\right)+B_{2}\left(m,\bar{\mu}\right)\right)=\left(B_{2}\left(m,m_1\right)+B_{2}\left(m_2,m\right)\right)+\eps\left(B_{W}\left(\mu_1,m\right)+B_{W}\left(m,\mu_2\right)\right).
 \end{split}
\end{equation}

Applying the Laplace transform to this equation and recalling the definition of the linearised operator $\widehat{\LL}$ we obtain from this
\begin{equation*}
 \begin{split}
  M\vcc=\T m=\widehat{\LL}^{-1}\T\left(B_2\left(m_1,m\right)+B_2\left(m,m_2\right)\right)+\eps\widehat{\LL}^{-1}\T\left(B_W\left(\mu_1,m\right)+B_W\left(m,\mu_2\right)\right).
 \end{split}
\end{equation*}
Then, recalling~\eqref{eq:isometry}, and noting also the relation $\T\left(\left(1-\zeta\right)m\right)=M\left(\cdot\right)-M\left(\cdot+1\right)$ we obtain after taking the $\fnorm{0}{\cdot}$-norm 
\begin{equation*}
 \begin{split}
  &\quad \fnorm{0}{\left(1-\zeta\right)m}=\lnorm{0}{\T\left(\left(1-\zeta\right)m\right)}\\
  &\leq \lnorm{0}{\widehat{\LL}^{-1}\left(\T\left(B_2\left(m_1,m\right)+B_2\left(m,m_2\right)\right)\left(\cdot\right)\right)-\widehat{\LL}^{-1}\left(\T\left(B_2\left(m_1,m\right)+B_2\left(m,m_2\right)\right)\left(\cdot+1\right)\right)}\\
  &\quad +\eps \lnorm{0}{\widehat{\LL}^{-1}\left(\T\left(B_W\left(\mu_1,m\right)+B_W\left(m,\mu_2\right)\right)\left(\cdot\right)\right)-\widehat{\LL}^{-1}\left(\T\left(B_W\left(\mu_1,m\right)+B_W\left(m,\mu_2\right)\right)\left(\cdot+1\right)\right)}\\
  &\leq \lnorm{0}{\widehat{\LL}^{-1}\left(\T\left(B_2\left(m_1,m\right)+B_2\left(m,m_2\right)\right)\right)}+\lnorm{0}{\widehat{\LL}^{-1}\left(\T\left(B_2\left(m_1,m\right)+B_2\left(m,m_2\right)\right)\left(\cdot+1\right)\right)}\\
  &\quad +\eps\lnorm{0,1+\theta-\alpha}{\widehat{\LL}^{-1}\left(\T\left(B_W\left(\mu_1,m\right)+B_W\left(m,\mu_2\right)\right)\left(\cdot\right)\right)-\widehat{\LL}^{-1}\left(\T\left(B_W\left(\mu_1,m\right)+B_W\left(m,\mu_2\right)\right)\left(\cdot+1\right)\right)}\\
  &\leq \lnorm{0}{\widehat{\LL}^{-1}\left(\T\left(B_2\left(m_1,m\right)+B_2\left(m,m_2\right)\right)\right)}+\lnorm{0}{\widehat{\LL}^{-1}\left(\T\left(B_2\left(m_1,m\right)+B_2\left(m,m_2\right)\right)\left(\cdot+1\right)\right)}\\
  &\quad +\eps\lnorm{1,\theta-\alpha}{\widehat{\LL}^{-1}\T\left(B_W\left(\mu_1,m\right)+B_W\left(m,\mu_2\right)\right)},
 \end{split}
\end{equation*}
where we used $\lnorm{0}{\cdot}\leq \lnorm{0,1+\theta-\alpha}{\cdot}$ (see~\eqref{eq:norm:order}) in the second estimate as well as Remark~\ref{Rem:reg:weight} in the last one.

Using then the boundedness of $\widehat{\LL}^{-1}$ as shown in Proposition~\ref{Prop:Lin:inv} as well as Proposition~\ref{Prop:bilinear} we then obtain, recalling also~\eqref{eq:isometry} as well as Remark~\ref{Rem:elem:norm:est}, that
\begin{equation*}
 \begin{split}
  &\quad \fnorm{0}{\left(1-\zeta\right)m}\\
  &\leq C\left(\lnorm{1}{\T\left(B_2\left(m_1,m\right)+B_2\left(m,m_2\right)\right)}+\lnorm{1}{\T\left(B_2\left(m_1,m\right)+B_2\left(m,m_2\right)\right)\left(\cdot+1\right)}\right)\\
  &\quad +C\eps\lnorm{1,\theta-\alpha}{\T\left(B_W\left(\mu_1,m\right)+B_W\left(m,\mu_2\right)\right)}\\
  &\leq C\left(\fnorm{1}{B_2\left(m_1,m\right)}+\fnorm{1}{B_2\left(m,m_2\right)}\right)+C\eps\left(\fnorm{1,\theta-\alpha}{B_W\left(\mu_1,m\right)}+\fnorm{1,\theta-\alpha}{B_W\left(m,\mu_2\right)}\right)\\
  &\leq C\left(\fnorm{1}{m_1}+\fnorm{1}{m_2}\right)\fnorm{1}{m}+C\eps\left(\fnorm{2}{\mu_1}+\fnorm{2}{\mu_2}\right)\fnorm{2}{m}.
 \end{split}
\end{equation*}
Finally Propositions~\ref{Prop:uniform:bound:muk},~\ref{Prop:smallness:mk} and~\ref{Prop:reg} yield
\begin{equation}\label{eq:m:localised}
 \begin{split}
  \fnorm{0}{\left(1-\zeta\right)m}\leq C\delta\fnorm{1}{m}+C\eps\fnorm{2}{m}\leq C\left(\eps+\delta\right)\fnorm{0}{m}
 \end{split}
\end{equation}
for any $\delta>0$ if $\eps$ is sufficiently small. On the other hand from Proposition~\ref{Prop:splitting:2} together with the estimate just shown and Proposition~\ref{Prop:reg} we get
\begin{equation}\label{eq:m:loc:2}
 \begin{split}
  \fnorm{2}{\zeta m}&\leq \delta_{*}\fnorm{2}{m}+C_{\delta_{*}}\fnorm{0}{\left(1-\zeta\right)m}\leq \widehat{C}\delta_{*}\fnorm{0}{m}+CC_{\delta_{*}}\left(\eps+\delta\right)\\
  &=\left(\widehat{C}\delta_{*}+CC_{\delta_{*}}\left(\eps+\delta\right)\right)\fnorm{0}{m}.
 \end{split}
\end{equation}
Taking now~\eqref{eq:m:localised} and~\eqref{eq:m:loc:2} together we get
\begin{equation*}
 \begin{split}
  \fnorm{0}{m}&\leq C\left(\delta+\eps\right)\fnorm{0}{m}+\left(\widehat{C}\delta_{*}+CC_{\delta_{*}}\left(\eps+\delta\right)\right)\fnorm{0}{m},
 \end{split}
\end{equation*}
which implies uniqueness by choosing first $\delta_{*}$ and then $\delta$ and $\eps$ sufficiently small.

\section{Estimates for $B_2$ and $B_W$ - Proof of Proposition~\ref{Prop:bilinear}}\label{Sec:est:m:large:x}

Recalling~\eqref{eq:diff:m} we have that the difference $m\vcc=m_1-m_2$ of two solutions $m_1$ and $m_2$ of equation~\eqref{eq:mu2} satisfies
\begin{equation*}
 \begin{split}
  m-\left(B_{2}\left(\bar{\mu},m\right)+B_{2}\left(m,\bar{\mu}\right)\right)=\left(B_{2}\left(m,m_1\right)+B_{2}\left(m_2,m\right)\right)+\eps\left(B_{W}\left(\mu_1,m\right)+B_{W}\left(m,\mu_2\right)\right).
 \end{split}
\end{equation*}
Multiplying by $x^2\ee^{-px}$, integrating and denoting by $M_k$ and $M$ the Laplace transforms of $m_k$ and $m$ it follows
\begin{multline}\label{eq:M}
  M''\left(p\right)=\int_{0}^{\infty}x^{2}\ee^{-px}B_{2}\left(\bar{\mu},m\right)\dx+\int_{0}^{\infty}x^{2}\ee^{-px}B_{2}\left(m,\bar{\mu}\right)\dx\\
  +\int_{0}^{\infty}x^{2}\ee^{-px}B_{2}\left(m,m_1\right)\dx+\int_{0}^{\infty}x^{2}\ee^{-px}B_{2}\left(m_2,m\right)\dx\\
  +\eps\int_{0}^{\infty}x^{2}\ee^{-px}B_{W}\left(m,\mu_1\right)\dx+\eps\int_{0}^{\infty}x^{2}\ee^{-px}B_{W}\left(\mu_{2},m\right)\dx.
\end{multline}

\subsection{Estimates on $B_2$ - Proof of~\eqref{eq:bil:1}}\label{Sec:proof:B2}

\begin{proof}[Proof of~\eqref{eq:bil:1} in Proposition~\ref{Prop:bilinear}]
 By the definition of $B_2$ and using Fubini's theorem it holds
 \begin{equation*}
  \begin{split}
   \int_{0}^{\infty}x^{2}\ee^{-px}B_{2}\left(\omega_{1},\omega_{2}\right)\dx&=2\int_{0}^{\infty}\int_{0}^{x}\int_{x-y}^{\infty}\ee^{-px}y\omega_{1}\left(y\right)\omega_{2}\left(z\right)\ee^{x-\left(y+z\right)}\dz\dy\dx\\
   &=2\int_{0}^{\infty}\int_{0}^{\infty}y\omega_{1}\left(y\right)\omega_{2}\left(z\right)\ee^{-\left(y+z\right)}\int_{y}^{y+z}\ee^{\left(1-p\right)x}\dx\\
   &=\frac{2}{p-1}\int_{0}^{\infty}\int_{0}^{\infty}y\omega_1\left(y\right)\omega_{2}\left(z\right)\ee^{-p y}\left(\ee^{-z}-\ee^{-p z}\right)\dz\dy.
  \end{split}
 \end{equation*}
Rearranging and denoting as before by $\Omega_{i}$ the Laplace transform of $\omega_i$ we obtain
\begin{equation}\label{eq:bil1}
 \begin{split}
  \int_{0}^{\infty}x^{2}\ee^{-px}B_{2}\left(\omega_{1},\omega_{2}\right)\dx&= \frac{2}{p-1}\Omega_{1}'\left(p\right)\left(\Omega_{2}\left(p\right)-\Omega_{2}\left(1\right)\right).
 \end{split}
\end{equation}
To estimate the right-hand side we first consider $p\geq 2$ and note that due to the definition of $\snorm{1}{\cdot}$ we have
\begin{equation*}
 \begin{split}
  \abs{\Omega_2\left(p\right)-\Omega_{2}\left(1\right)}=\abs{\int_{1}^{p}\Omega_{2}'\left(s\right)\ds}\leq \snorm{1}{\omega_2}\int_{1}^{p}\frac{\left(1+s\right)^{1-\theta}}{s^{2}}\ds\leq C\snorm{1}{\omega_2}\int_{1}^{\infty}s^{-1-\theta}\ds\leq C\frac{\snorm{1}{\omega_{2}}}{\theta}.
 \end{split}
\end{equation*}
On the other hand we have
\begin{equation}\label{eq:sem1}
 \abs{\Omega_1'\left(p\right)}\leq \snorm{1}{\omega_1}\frac{\left(1+p\right)^{1-\theta}}{p^2} \quad \text{for all } p>0.
\end{equation}
Combining these estimates we get
\begin{equation*}
 \begin{split}
  \abs{\int_{0}^{\infty}x^{2}\ee^{-px}B_{2}\left(\omega_{1},\omega_{2}\right)\dx}\leq \frac{C}{p-1}\snorm{1}{\omega_2}\snorm{1}{\omega_1}\frac{\left(1+p\right)^{1-\theta}}{p^2}\leq C\snorm{1}{\omega_1}\snorm{1}{\omega_2}\frac{\left(1+p\right)^{1-\theta}}{p^3},
 \end{split}
\end{equation*}
using also $1/(p-1)\leq 2/p$ for $p\geq 2$. This then shows the claim for $p\geq 2$.

For $p\in\left[1,2\right]$ we get similarly
\begin{equation*}
 \begin{split}
  \abs{\Omega_{2}\left(p\right)-\Omega_{2}\left(1\right)}\leq\int_{1}^{p}\abs{\Omega_{2}'\left(s\right)}\ds\leq C \snorm{1}{\omega_2}\int_{1}^{p} s^{-1-\theta}\ds\leq C\snorm{1}{\omega_{2}}\left(p-1\right)
 \end{split}
\end{equation*}
and
\begin{equation*}
 \abs{\Omega_{1}'\left(p\right)}\leq\snorm{1}{\omega_{1}}\frac{\left(1+p\right)^{1-\theta}}{p^{2}}\leq C\snorm{1}{\omega_1}\frac{\left(1+p\right)^{1-\theta}}{p^{3}} \quad \text{for } p\in \left[1,2\right]. 
\end{equation*}
Using this together with~\eqref{eq:bil1} we obtain
\begin{equation*}
 \begin{split}
  \abs{\int_{0}^{\infty}x^{2}\ee^{-px}B_{2}\left(\omega_{1},\omega_{2}\right)\dx}\leq C\snorm{1}{\omega_{1}}\snorm{1}{\omega_{2}}\frac{\left(1+p\right)^{1-\theta}}{p^{3}}.
 \end{split}
\end{equation*}
This shows the claim for $p\in\left[1,2\right]$. 

We finally consider $p\in\left(0,1\right)$. Similarly as before we can estimate
\begin{equation*}
 \begin{split}
  \abs{\Omega_{2}\left(p\right)-\Omega_{2}\left(1\right)}=\abs{\int_{p}^{1}\Omega_{2}'\left(s\right)\ds}\leq \snorm{1}{\omega_2}\int_{p}^{1}\frac{\left(1+s\right)^{1-\theta}}{s^{2}}\ds\leq 2^{1-\theta}\int_{p}^{1}s^{-2}\ds\leq C\frac{\abs{1-p}}{p}.
 \end{split}
\end{equation*}
Using this together with~\eqref{eq:sem1} and~\eqref{eq:bil1} we obtain
\begin{equation*}
 \begin{split}
  \abs{\int_{0}^{\infty}x^{2}\ee^{-px}B_{2}\left(\omega_{1},\omega_{2}\right)\dx}\leq C\snorm{1}{\omega_1}\snorm{1}{\omega_2}\frac{\left(1+p\right)^{1-\theta}}{p^{3}}.
 \end{split}
\end{equation*}
This shows the claim also for $p\leq 1$ and thus finishes the proof.
\end{proof}

\subsection{Estimates on $B_W$ - Proof of~\eqref{eq:bil:2}}\label{Sec:proof:BW}

We next consider estimates for $B_{W}$ in the $\snorm{2}{\cdot}$-semi-norm and thus define
\begin{equation}\label{eq:def:L:BW}
 L\left(p\right)=L\left[\omega_1,\omega_2\right]\left(p\right)\vcc=\int_{0}^{\infty}x^{2}\ee^{-px}B_{W}\left(\omega_1,\omega_2\right)\dx,
\end{equation}
 where we use the notation $L\left[\omega_1,\omega_2\right]$ if we want to stress the dependence on $\omega_1$ and $\omega_2$. Using the definition of $B_W$ and Proposition~\ref{P.repkernels} together with Fubini's theorem we can rewrite
\begin{equation*}
 \begin{split}
  &\quad L\left(p\right)\\
  &=\int_{0}^{\infty}\int_{0}^{\infty}yW\left(y,z\right)\omega_{1}\left(y\right)\omega_2\left(z\right)\int_{y}^{y+z}\ee^{\left(1-p\right)x}\dx\ee^{-\left(y+z\right)}\dz\dy\\
  &=\int_{0}^{\infty}\int_{0}^{\infty}y\left(y+z\right)\int_{0}^{\infty}\int_{0}^{\infty}\Gamma\left(\xi,\eta\right)\ee^{-\xi y-\eta z}\dxi\deta\omega_{1}\left(y\right)\omega_{2}\left(z\right)\int_{y}^{y+z}\ee^{\left(1-p\right)x}\dx\ee^{-\left(y+z\right)}\dz\dy\\
  &=-\frac{1}{p-1}\int_{0}^{\infty}\int_{0}^{\infty}\Gamma\left(\xi,\eta\right)\int_{0}^{\infty}\int_{0}^{\infty}y\left(y+z\right)\ee^{-\xi y}\ee^{-\eta z}\omega_{1}\left(y\right)\omega_{2}\left(z\right)\left(\ee^{-p\left(y+z\right)}-\ee^{-py}\ee^{-z}\right)\dz\dy\dxi\deta.
 \end{split}
\end{equation*}
Rearranging we get
\begin{equation}\label{eq:BW:Laplace}
 \begin{multlined}
  \quad L\left(p\right)\\
   \shoveleft{=-\frac{1}{p-1}\int_{0}^{\infty}\int_{0}^{\infty}\Gamma\left(\xi,\eta\right)\left[\int_{0}^{\infty}y^{2}\omega_{1}\left(y\right)\ee^{-y\left(\xi+p\right)}\dy\int_{0}^{\infty}\omega_{2}\left(z\right)\left(\ee^{-\left(p+\eta\right)z}-\ee^{-\left(1+\eta\right)z}\right)\dz\right.}\\
   \shoveright{+\left.\int_{0}^{\infty}y\omega_{1}\left(y\right)\ee^{-\left(\xi+p\right)y}\dy\int_{0}^{\infty}z\omega_{2}\left(z\right)\left(\ee^{-\left(p+\eta\right)z}-\ee^{-\left(1+\eta\right)z}\right)\dz\right]\dxi\deta}\\
   \shoveleft{ =-\frac{1}{p-1}\int_{0}^{\infty}\int_{0}^{\infty}\Gamma\left(\xi,\eta\right)\left[\del_{\xi}^{2}\Omega_{1}\left(\xi+p\right)\left(\Omega_{2}\left(p+\eta\right)-\Omega_{2}\left(\eta+1\right)\right)\right]\dxi\deta}\\
 -\frac{1}{p-1}\int_{0}^{\infty}\int_{0}^{\infty}\Gamma\left(\xi,\eta\right)\left[\del_{\xi}\Omega_{1}\left(\xi+p\right)\del_{\eta}\left(\Omega_{2}\left(p+\eta\right)-\Omega_{2}\left(1+\eta\right)\right)\right]\dxi\deta=\vcc L_1\left(p\right)+L_{2}\left(p\right).
  \end{multlined}
\end{equation}

\begin{lemma}\label{Lem:est:L:small:p}
 There exists some constant $C>0$ such that for all $\omega_1, \omega_2\in\esp_{2}$ it holds
 \begin{equation*}
  \sup_{0<p\leq 2}\frac{p^3}{\left(1+p\right)^{1-\theta}}\abs{L\left(p\right)}\leq C\fnorm{2}{\omega_1}\fnorm{2}{\omega_2},
 \end{equation*}
 where $L$ is given by~\eqref{eq:def:L:BW}.
\end{lemma}

\begin{proof}
 We will use the following estimates that we will show at the end of the proof. It holds
\begin{align}
 \abs{\frac{\Omega_2\left(\eta+p\right)-\Omega_2\left(\eta+1\right)}{p-1}}\leq C\frac{\fnorm{1}{\omega_2}}{\left(\eta+p\right)\left(\eta+1\right)^{\theta}}, \label{eq:loss:diff:1}\\
 \abs{\frac{\Omega_2'\left(\eta+p\right)-\Omega_2'\left(\eta+1\right)}{p-1}}\leq C\frac{\fnorm{2}{\omega_2}}{\left(\eta+p\right)^2\left(\eta+1\right)^{\theta}}.\label{eq:loss:diff:2}
\end{align}
Furthermore we have for $p\leq 2$ that
\begin{equation*}
 \abs{\del_{\xi}^{k}\Omega_1\left(\xi+p\right)}\leq C\fnorm{k}{\omega_1}\begin{cases}
                                                                         \left(\xi+p\right)^{-k-1} & \xi\leq 1\\
                                                                         \left(\xi+p\right)^{-k-\theta} & \xi\geq 1.
                                                                        \end{cases}                                                                         
\end{equation*}
Thus for $p\leq 2$ we find
\begin{equation*}
 \begin{split}
  \abs{L\left(p\right)}&\leq C\fnorm{2}{\omega_1}\fnorm{2}{\omega_2}\int_{1}^{\infty}\int_{0}^{\infty}\left(\frac{\abs{\Gamma\left(\xi,\eta\right)}}{\left(\xi+p\right)^{2+\theta}\left(\eta+p\right)\left(\eta+1\right)^{\theta}}+\frac{\abs{\Gamma\left(\xi,\eta\right)}}{\left(\xi+p\right)^{1+\theta}\left(\eta+p\right)^2\left(\eta+1\right)^{\theta}}\right)\deta\dxi\\
  &\quad +C\fnorm{2}{\omega_1}\fnorm{2}{\omega_2}\int_{0}^{1}\int_{0}^{\infty}\left(\frac{\abs{\Gamma\left(\xi,\eta\right)}}{\left(\xi+p\right)^3\left(\eta+p\right)\left(\eta+1\right)^{\theta}}+\frac{\abs{\Gamma\left(\xi,\eta\right)}}{\left(\xi+p\right)^2\left(\eta+p\right)^2\left(\eta+1\right)^{\theta}}\right)\deta\dxi\\
  &=\vcc(I)+(II).
 \end{split}
\end{equation*}
We have to consider the different regions with $\xi$ and $\eta$ large or small separately and so first obtain
\begin{equation*}
 \begin{split}
  (I)&\leq \int_{1}^{\infty}\int_{0}^{\infty}\abs{\Gamma\left(\xi,\eta\right)}\left(\frac{1}{\xi^{2+\theta}\left(\eta+p\right)\left(\eta+1\right)^{\theta}}+\frac{1}{\xi^{1+\theta}\left(\eta+p\right)^{2}\left(\eta+1\right)^{\theta}}\right)\deta\dxi\\
  &= \int_{1}^{\infty}\int_{1}^{\infty}\left(\cdots\right)\deta\dxi+\int_{1}^{\infty}\int_{0}^{1}\left(\cdots\right)\deta\dxi=\vcc(I)_a+(I)_b.
 \end{split}
\end{equation*}
Using that for $x\geq 1$ we have $x^{-\ell}\leq C\left(x+1\right)^{-\ell}$ we see from Lemma~\ref{Lem:est:Gamma:int} that $(I)_a$ is just bounded by a constant. For $(I)_b$ we obtain together with Proposition~\ref{P.repkernels} and using that the support of $\delta(\xi-\eta)$ does not intersect $(1,\infty)\times(0,1)$ that
\begin{equation*}
 \begin{split}
  (I)_b&\leq C\int_{1}^{\infty}\int_{0}^{1}\frac{1}{\left(\xi+\eta\right)^{1-\alpha}}\left(\frac{1}{\xi^{\alpha}}+\frac{1}{\eta^{\alpha}}\right)\left(\frac{1}{\xi^{2+\theta}\left(\eta+p\right)\left(\eta+1\right)^{\theta}}+\frac{1}{\xi^{1+\theta}\left(\eta+p\right)^{2}\left(\eta+1\right)^{\theta}}\right)\deta\dxi\\
  &\leq C\int_{1}^{\infty}\frac{1}{\xi^{1-\alpha}\xi^{1+\theta}}\left(\int_{0}^{1}\frac{1}{\xi^{\alpha}}\left(\frac{1}{\eta+p}+\frac{1}{\left(\eta+p\right)^2}\right)\deta+\int_{0}^{1}\frac{1}{\eta^{\alpha}}\left(\frac{1}{\xi+p}+\frac{1}{\left(\xi+p\right)^2}\right)\deta\right)\dxi\\
  &\leq C\int_{1}^{\infty}\frac{1}{\xi^{2+\theta-\alpha}}\left(\frac{1}{\xi^{\alpha}}\left(\frac{1}{p^{\theta}}+\frac{1}{p^{1+\theta}}\right)+\frac{1}{p^{\theta}}+\frac{1}{p^{1+\theta}}\right)\dxi\leq \frac{C}{p^{1+\theta}}.
 \end{split}
\end{equation*}
For $(II)$ the homogeneity of $\Gamma$ and Lemma~\ref{Lem:est:Gamma:int} imply
\begin{equation*}
 \begin{split}
  (II)&\leq \int_{0}^{1}\int_{0}^{\infty}\frac{\abs{\Gamma\left(\xi,\eta\right)}}{\left(\xi+p\right)^2\left(\eta+p\right)}\left(\frac{1}{\xi+p}+\frac{1}{\eta+p}\right)\deta\dxi\\
  &\leq \frac{1}{p^3}\int_{0}^{\infty}\int_{0}^{\infty}\frac{\abs{\Gamma\left(\xi,\eta\right)}}{\left(\xi+1\right)^2\left(\eta+1\right)}\left(\frac{1}{\xi+1}+\frac{1}{\eta+1}\right)\deta\dxi\leq \frac{C}{p^3}.
 \end{split}
\end{equation*}
Thus for $p\leq 2$ we obtain
\begin{equation*}
 \abs{L\left(p\right)}\leq C\frac{\fnorm{2}{\omega_1}\fnorm{2}{\omega_2}}{p^3}
\end{equation*}
and together with the estimate for $p\geq 2$ obtained before this shows the claim. 

It thus remains to show the estimates~\eqref{eq:loss:diff:1} and~\eqref{eq:loss:diff:2}. We first note
\begin{equation*}
 \begin{split}
  \abs{\frac{\Omega_2\left(\eta+p\right)-\Omega_2\left(\eta+1\right)}{p-1}}=\frac{1}{\abs{p-1}}\abs{\int_{\eta+1}^{\eta+p}\Omega_2'\left(s\right)\ds}\leq \frac{\fnorm{1}{\omega_2}}{\abs{p-1}}\abs{\int_{\eta+1}^{\eta+p}\frac{\left(s+1\right)^{1-\theta}}{s^2}\ds}.
 \end{split}
\end{equation*}
For $\eta\leq 1$ and $p\leq 2$ we then obtain that
\begin{equation*}
 \begin{split}
  &\quad \abs{\frac{\Omega_2\left(\eta+p\right)-\Omega_2\left(\eta+1\right)}{p-1}}\leq C\frac{\fnorm{1}{\omega_2}}{\abs{p-1}}\abs{\int_{\eta+1}^{\eta+p}\frac{1}{s^2}\ds}\leq C\frac{\fnorm{1}{\omega_2}}{\abs{p-1}}\abs{\frac{1}{\eta+1}-\frac{1}{\eta+p}}\\
  &\leq C\frac{\fnorm{1}{\omega_2}}{\left(\eta+1\right)\left(\eta+p\right)}\leq C\frac{\fnorm{1}{\omega_2}}{\left(\eta+p\right)\left(\eta+1\right)^{\theta}}.
 \end{split}
\end{equation*}
On the other hand we get for $\eta\geq 1$ that
\begin{equation*}
 \begin{split}
  &\quad\abs{\frac{\Omega_2\left(\eta+p\right)-\Omega_2\left(\eta+1\right)}{p-1}}\leq C\frac{\fnorm{1}{\omega_2}}{\abs{p-1}}\abs{\int_{\eta+1}^{\eta+p}s^{-1-\theta}\ds}\\
  &\leq C\frac{\fnorm{1}{\omega_2}}{\abs{p-1}}\max\left\{\frac{1}{\left(\eta+p\right)^{1+\theta}},\frac{1}{\left(\eta+1\right)^{1+\theta}}\right\}\abs{p-1}\leq C\frac{\fnorm{1}{\omega_2}}{\left(\eta+p\right)\left(\eta+1\right)^{\theta}}.
 \end{split}
\end{equation*}
This shows~\eqref{eq:loss:diff:1}. The proof of~\eqref{eq:loss:diff:2} is similar using that it holds
\begin{equation*}
 \begin{split}
  \abs{\int_{\eta+1}^{\eta+p}\frac{1}{s^3}\ds}=\abs{\frac{\left(\eta+1\right)^2-\left(\eta+p\right)^2}{\left(\eta+1\right)^2\left(\eta+p\right)^2}}=\abs{p-1}\frac{2\eta+\left(1+p\right)}{\left(\eta+1\right)^2\left(\eta+p\right)^2}\leq C\frac{\abs{p-1}}{\left(\eta+p\right)^2\left(\eta+1\right)^{\theta}}
 \end{split}
\end{equation*}
for $\eta\leq 1$ and $p\leq 2$.
\end{proof}

\begin{lemma}\label{Lem:est:L:large:p:loss}
 There exists some constant $C>0$ such that for all $\omega_1,\omega_2\in\esp_{2}$ it holds
\begin{equation*}
 \sup_{p\geq 2}\frac{p^3}{\left(1+p\right)^{1-\left(\theta-\alpha\right)}}\abs{L\left(p\right)}\leq C\fnorm{2}{\omega_1}\fnorm{2}{\omega_2},
\end{equation*}
where $L$ is given in~\eqref{eq:def:L:BW}.
\end{lemma}

\begin{proof}
 From the definition of the norms and~\eqref{eq:Lambda} we have the estimates
\begin{align*}
 \abs{\Omega_1''\left(\xi+p\right)}\leq C\frac{\snorm{2}{\omega_1}}{\left(\xi+p\right)^{2+\theta}}, \qquad \abs{\Omega_1'\left(\xi+p\right)}\leq C\frac{\snorm{1}{\omega_1}}{\left(\xi+p\right)^{1+\theta}},\\
 \abs{\frac{\Omega_2\left(\eta+p\right)-\Omega_2\left(\eta+1\right)}{p-1}}\leq C\frac{\fnorm{0}{\omega_2}}{p}\left(\frac{1}{\left(\eta+p\right)^{\theta}}+\frac{1}{\left(\eta+1\right)^{\theta}}\right)\leq C\frac{\fnorm{0}{\omega_2}}{p\left(\eta+1\right)^{\theta}},\\
 \abs{\frac{\Omega_2'\left(\eta+p\right)-\Omega_2'\left(\eta+1\right)}{p-1}}\leq C\frac{\snorm{1}{\omega_2}}{p}\left(\frac{1}{\left(\eta+p\right)^{1+\theta}}+\frac{1}{\left(\eta+1\right)^{1+\theta}}\right)\leq C\frac{\snorm{1}{\omega_2}}{p\left(\eta+1\right)^{1+\theta}}.
\end{align*}
From this and~\eqref{eq:BW:Laplace} it follows that
\begin{equation}\label{eq:loss:BW:large:p}
 \begin{split}
  \abs{L\left(p\right)}&\leq C\frac{\fnorm{2}{\omega_1}\fnorm{1}{\omega_2}}{p}\int_{0}^{\infty}\int_{0}^{\infty}\abs{\Gamma\left(\xi,\eta\right)}\left(\frac{1}{\left(\xi+p\right)^{2+\theta}}\frac{1}{\left(\eta+1\right)^{\theta}}+\frac{1}{\left(\xi+p\right)^{1+\theta}}\frac{1}{\left(\eta+1\right)^{1+\theta}}\right)\deta\dxi\\
  &\leq C\frac{\fnorm{2}{\omega_1}\fnorm{1}{\omega_2}}{p}\int_{0}^{\infty}\frac{1}{\left(\xi+p\right)^{1+\theta}}\int_{0}^{\infty}\abs{\Gamma\left(\xi,\eta\right)}\left(\frac{1}{\left(\xi+p\right)\left(\eta+1\right)^{\theta}}+\frac{1}{\left(\eta+1\right)^{1+\theta}}\right)\deta\dxi.
 \end{split}
\end{equation}
Using that for $p\geq 2$ it holds $\left(\xi+p\right)^{-1}\leq \min\left\{1,\xi^{-1}\right\}$ we find together with Lemma~\ref{Lem:Gamma:eta:int} that
\begin{equation*}
 \begin{split}
  \int_{0}^{\infty}\abs{\Gamma\left(\xi,\eta\right)}\left(\frac{1}{\xi+p}\frac{1}{\left(\eta+1\right)^{\theta}}+\frac{1}{\left(\eta+1\right)^{1+\theta}}\right)\deta\leq C \min\left\{\xi^{-\alpha}+\log\left(\xi^{-1}\right),\xi^{\alpha-1}\right\}.
 \end{split}
\end{equation*}
Plugging this estimate into~\eqref{eq:loss:BW:large:p} we obtain for $p\geq 2$ that it holds
\begin{equation*}
 \begin{split}
  &\quad \abs{L\left(p\right)}\leq C\frac{\fnorm{2}{\omega_1}\fnorm{1}{\omega_1}}{p}\int_{0}^{\infty}\frac{1}{\left(\xi+p\right)^{1+\theta}}\min\left\{\frac{1}{\xi^{\alpha}}+\log\left(\frac{1}{\xi}\right),\frac{1}{\xi^{1-\alpha}}\right\}\dxi\\
  &\leq C\frac{\fnorm{2}{\omega_1}\fnorm{1}{\omega_2}}{p}\left(\int_{0}^{1}\frac{\xi^{-\alpha}+\log\left(\xi^{-1}\right)}{\left(\xi+p\right)^{1+\theta}}\dxi+\int_{1}^{p}\frac{1}{\left(\xi+p\right)^{1+\theta}\xi^{1-\alpha}}\dxi+\int_{p}^{\infty}\frac{1}{\left(\xi+p\right)^{1+\theta}\xi^{1-\alpha}}\dxi\right)\\
  &\leq C\frac{\fnorm{2}{\omega_1}\fnorm{1}{\omega_2}}{p}\left(\frac{1}{p^{1+\theta}}+\frac{1}{p^{1+\theta-\alpha}}+\frac{1}{p^{1+\theta-\alpha}}\right)\leq C\frac{\fnorm{2}{\omega_1}\fnorm{1}{\omega_2}}{p^{2+\theta-\alpha}}. 
 \end{split}
\end{equation*}
\end{proof}

\begin{proof}[Proof of~\eqref{eq:bil:2} of Proposition~\ref{Prop:bilinear}]
 This is a direct consequence of Lemma~\ref{Lem:est:L:small:p} and Lemma~\ref{Lem:est:L:large:p:loss}.
\end{proof}

\subsection{Estimates of $L\left(p\right)-L\left(p+1\right)$}

We also collect in the following some estimates on the difference $L\left(p\right)-L\left(p+1\right)$ that occur naturally when considering the Laplace transform of $\left(1-\zeta\right)m$ and that will be used later in the proof of Proposition~\ref{Prop:m:large:x}.

\begin{lemma}\label{Lem:est:L2:diff:large:p}
 There exists some constant $C>0$ such that for all $\omega_1, \omega_2\in\esp_{2}$ it holds \begin{equation*}
  \begin{split}
   \sup_{p\geq 2}\frac{p^3}{\left(1+p\right)^{1-\theta}}\abs{L_2\left(p\right)-L_2\left(p+1\right)}\leq C\left(\snorm{2}{\omega_1}\fnorm{1}{\omega_2}+\fnorm{1}{\omega_1}\snorm{2}{\omega_2}+\fnorm{1}{\omega_1}\fnorm{1}{\omega_2}\right),
  \end{split}
 \end{equation*}
 where $L_2$ is as in~\eqref{eq:BW:Laplace}.
 \end{lemma}

\begin{proof}
Let $p\geq 2$. Using the definition of $L_2$ we obtain
  \begin{equation}\label{eq:diff:L2}
  \begin{aligned}
   &\quad L_{2}\left(p\right)-L_2\left(p+1\right)\\
   &=\int_{0}^{\infty}\int_{0}^{\infty}\Gamma\left(\xi,\eta\right)\cdot\\
   &\qquad\qquad\cdot\left[\Omega_{1}'\left(\xi+p+1\right)\frac{\Omega_{2}'\left(p+1+\eta\right)-\Omega_{2}'\left(1+\eta\right)}{p}-\Omega_{1}'\left(\xi+p\right)\frac{\Omega_{2}'\left(p+\eta\right)-\Omega_2'\left(1+\eta\right)}{p-1}\right]\dxi\deta\\
     &\begin{aligned}
             =\int_{0}^{\infty}\int_{0}^{\infty}\Gamma\left(\xi,\eta\right)&\left[\left(\Omega_{1}'\left(\xi+p+1\right)-\Omega_{1}'\left(\xi+p\right)\right)\frac{\Omega_{2}'\left(p+1+\eta\right)-\Omega_{2}'\left(1+\eta\right)}{p}\right.\\
             &\quad \left.+\Omega_{1}'\left(\xi+p\right)\left(\frac{\Omega_{2}'\left(p+1+\eta\right)-\Omega_2'\left(1+\eta\right)}{p}-\frac{\Omega_{2}'\left(p+\eta\right)-\Omega_{2}'\left(1+\eta\right)}{p-1}\right)\right]\dxi\deta
            \end{aligned}\\
   &=\vcc I_1+I_2.
  \end{aligned}
 \end{equation}
 $I_1$ can be estimated as
 \begin{equation}\label{eq:L2:diff:I1}
  \begin{split}
   \abs{I_1}\leq \int_{0}^{\infty}\int_{0}^{\infty}\abs{\Gamma\left(\xi,\eta\right)}\abs{\Omega_{1}'\left(\xi+p+1\right)-\Omega_{1}'\left(\xi+p\right)}\frac{\abs{\Omega_{2}'\left(p+1+\eta\right)}+\abs{\Omega_{2}'\left(1+\eta\right)}}{p}\dxi\deta.
  \end{split}
 \end{equation}
 From the definition of $\snorm{2}{\cdot}$ and $\snorm{1}{\cdot}$, using also $p\geq 2$, we have that
 \begin{equation*}
  \begin{split}
   \abs{\Omega_{1}'\left(\xi+p+1\right)-\Omega_{1}'\left(\xi+p\right)}=\abs{\int_{\xi+p}^{\xi+p+1}\Omega_{1}''\left(x\right)\ds}\leq \snorm{2}{\omega_1}\int_{\xi+p}^{\xi+p+1}\frac{\left(1+z\right)^{1-\theta}}{s^{3}}\ds\leq C\frac{\snorm{2}{\omega_{1}}}{\left(\xi+p\right)^{2+\theta}}
  \end{split}
 \end{equation*}
 as well as
 \begin{equation*}
  \begin{split}
   \abs{\Omega_{2}'\left(p+1+\eta\right)}\leq \snorm{1}{\omega_2}\frac{\left(2+p+\eta\right)^{1-\theta}}{\left(1+p+\eta\right)^{2}}\leq C\frac{\snorm{1}{\omega_2}}{\left(\eta+1\right)^{1+\theta}}\quad \text{and} \quad \abs{\Omega_{2}'\left(1+\eta\right)}\leq C\frac{\snorm{1}{\omega_2}}{\left(1+\eta\right)^{1+\theta}}.
  \end{split}
 \end{equation*}
 Using this in~\eqref{eq:L2:diff:I1} we get
 \begin{equation}\label{eq:L2:diff:I1:2}
  \begin{split}
   \abs{I_1}&\leq C\frac{\snorm{2}{\omega_1}\snorm{1}{\omega_2}}{p}\int_{0}^{\infty}\int_{0}^{\infty}\frac{\abs{\Gamma\left(\xi,\eta\right)}}{\left(\xi+p\right)^{2+\theta}\left(1+\eta\right)^{1+\theta}}\dxi\deta.
  \end{split}
 \end{equation}
Together with Lemma~\ref{Lem:Gamma:eta:int} and using $p\geq 2$ we further obtain
\begin{equation}\label{eq:L2:diff:I1:3}
 \begin{split}
  \abs{I_1}&\leq C\frac{\snorm{2}{\omega_1}\snorm{1}{\omega_2}}{p}\left(\int_{0}^{1}\frac{1}{\left(\xi+p\right)^{2+\theta}}\left(\frac{1}{\xi^{\alpha}}+\log\left(\frac{1}{\xi}\right)\right)\dxi+\int_{1}^{\infty}\frac{1}{\left(\xi+p\right)^{2+\theta}}\frac{1}{\xi^{1-\alpha}}\dxi\right)\\
  &\leq C\frac{\snorm{2}{\omega_1}\snorm{1}{\omega_2}}{p}\left(\frac{C}{p^{2+\theta}}+p^{1-\left(1-\alpha\right)-2-\theta}\int_{0}^{\infty}\frac{1}{\left(\xi+1\right)^{2+\theta}}\frac{1}{\xi^{1-\alpha}}\dxi\right)\leq C\frac{\snorm{2}{\omega_1}\snorm{1}{\omega_2}}{p^{3+\theta-\alpha}}.
 \end{split}
\end{equation}
We now proceed by considering $I_2$ and we first rewrite
\begin{equation*}
 \begin{split}
  &\quad\frac{\Omega_{2}'\left(p+1+\eta\right)-\Omega_{2}'\left(1+\eta\right)}{p}-\frac{\Omega_{2}'\left(p+\eta\right)-\Omega_{2}'\left(1+\eta\right)}{p-1}=\frac{\Omega_{2}'\left(p+1+\eta\right)}{p}-\frac{\Omega_{2}'\left(p+\eta\right)}{p-1}+\frac{\Omega_{2}'\left(1+\eta\right)}{p\left(p-1\right)}\\
  &=\int_{p+\eta}^{p+1+\eta}\del_s\left(\frac{\Omega_{2}'\left(s\right)}{s-\eta-1}\right)\ds+\frac{\Omega_{2}'\left(1+\eta\right)}{p\left(p-1\right)}=\int_{p+\eta}^{p+1+\eta}\frac{\Omega_{2}''\left(s\right)}{s-\eta-1}-\frac{\Omega_{2}'\left(s\right)}{\left(s-\eta-1\right)^2}\ds+\frac{\Omega_{2}'\left(1+\eta\right)}{p\left(p-1\right)}.
 \end{split}
\end{equation*}
Thus we can estimate
\begin{equation*}
 \begin{split}
  &\quad \abs{\frac{\Omega_{2}'\left(p+1+\eta\right)-\Omega_{2}'\left(1+\eta\right)}{p}-\frac{\Omega_{2}'\left(p+\eta\right)-\Omega_{2}'\left(1+\eta\right)}{p-1}}\\
  &\leq \snorm{2}{\omega_2}\int_{p+\eta}^{1+p+\eta}\frac{\left(1+s\right)^{1-\theta}}{s^{3}\left(s-\eta-1\right)}\ds+\snorm{1}{\omega_2}\int_{p+\eta}^{1+p+\eta}\frac{\left(1+s\right)^{1-\theta}}{s^2\left(s-\eta-1\right)^2}\ds+\frac{\snorm{1}{\omega_2}}{p\left(p-1\right)}\frac{\left(2+\eta\right)^{1-\theta}}{\left(1+\eta\right)^2}\\
  &\leq C\frac{\snorm{2}{\omega_2}}{p-1}\int_{p+\eta}^{1+p+\eta}s^{-2-\theta}\ds+C\frac{\snorm{1}{\omega_2}}{\left(p-1\right)^2}\int_{p+\eta}^{1+p+\eta}s^{-1-\theta}\ds+C\frac{\snorm{1}{\omega_{2}}}{p\left(p-1\right)}\frac{1}{\left(1+\eta\right)^{1+\theta}}.
 \end{split}
\end{equation*}
Using $p\geq 2$ and elementary estimates for the integrals we further obtain
\begin{equation*}
 \begin{split}
  &\quad \abs{\frac{\Omega_{2}'\left(p+1+\eta\right)-\Omega_{2}'\left(1+\eta\right)}{p}-\frac{\Omega_{2}'\left(p+\eta\right)-\Omega_{2}'\left(1+\eta\right)}{p-1}}\\
  &\leq C\left(\frac{\snorm{2}{\omega_2}}{p}\frac{1}{\left(p+\eta\right)^{2+\theta}}+\frac{\snorm{1}{\omega_2}}{p^2}\frac{1}{\left(p+\eta\right)^{1+\theta}}+\frac{\snorm{1}{\omega_2}}{p^2}\frac{1}{\left(1+\eta\right)^{1+\theta}}\right)\\
  &\leq C\frac{\snorm{2}{\omega_2}}{p}\frac{1}{\left(p+\eta\right)^{2+\theta}}+C\frac{\snorm{1}{\omega_2}}{p^2}\frac{1}{\left(1+\eta\right)^{1+\theta}}.
 \end{split}
\end{equation*}
On the other hand we have
\begin{equation*}
 \begin{split}
  \abs{\Omega_{1}'\left(\xi+p\right)}\leq \snorm{1}{\omega_1}\frac{\left(1+\xi+p\right)^{1-\theta}}{\left(\xi+p\right)^2}\leq C\snorm{1}{\omega_1}\left(\xi+p\right)^{-1-\theta}.
 \end{split}
\end{equation*}
Using the last two estimates in the definition of $I_2$ in~\eqref{eq:diff:L2} gives
\begin{multline}\label{eq:L2:diff:I2:1}
  \abs{I_2}\leq C\frac{\snorm{1}{\omega_1}\snorm{2}{\omega_2}}{p}\int_{0}^{\infty}\int_{0}^{\infty}\frac{\abs{\Gamma\left(\xi,\eta\right)}}{\left(p+\eta\right)^{2+\theta}\left(\xi+p\right)^{1+\theta}}\deta\dxi\\
  +C\frac{\snorm{1}{\omega_1}\snorm{1}{\omega_2}}{p^2}\int_{0}^{\infty}\int_{0}^{\infty}\frac{\abs{\Gamma\left(\xi,\eta\right)}}{\left(1+\eta\right)^{1+\theta}\left(\xi+p\right)^{1+\theta}}\dxi\deta.
\end{multline}
 It thus remains to estimate the two integrals on the right-hand side of~\eqref{eq:L2:diff:I2:1}. To estimate the second term we note that the integral $\int_{0}^{\infty}\abs{\Gamma\left(\xi,\eta\right)}\left(1+\eta\right)^{-\left(1+\theta\right)}\deta$ already occurred in \eqref{eq:L2:diff:I1:2} and thus using Lemma~\ref{Lem:Gamma:eta:int} we get similarly as before
\begin{equation}\label{eq:L2:diff:I2:2}
 \begin{split}
  &\quad\int_{0}^{\infty}\int_{0}^{\infty}\frac{\abs{\Gamma\left(\xi,\eta\right)}}{\left(1+\eta\right)^{1+\theta}\left(\xi+p\right)^{1+\theta}}\dxi\deta\leq \int_{0}^{1}\frac{1}{\left(\xi+p\right)^{1+\theta}}\left(\xi^{-\alpha}+\log\left(\xi^{-1}\right)\right)\dxi+\int_{1}^{\infty}\frac{1}{\left(\xi+p\right)^{1+\theta}\xi^{1-\alpha}}\dxi\\
  &\leq \frac{C}{p^{1+\theta}}+p^{1-\left(1-\alpha\right)-\left(1+\theta\right)}\int_{0}^{\infty}\frac{1}{\left(\xi+1\right)^{1+\theta}\xi^{1-\alpha}}\dxi\leq \frac{C}{p^{1+\theta-\alpha}}.
 \end{split}
\end{equation}
To estimate the first integral in~\eqref{eq:L2:diff:I2:1} we change variables and use Proposition~\ref{P.repkernels} to obtain
\begin{equation}\label{eq:L2:diff:I2:3}
 \begin{split}
  &\quad \int_{0}^{\infty}\frac{\abs{\Gamma\left(\xi,\eta\right)}}{\left(p+\eta\right)^{2+\theta}\left(\xi+p\right)^{1+\theta}}\deta\dxi\leq p^2\frac{1}{p}\frac{1}{p^{2+\theta+1+\theta}}\int_{0}^{\infty}\int_{0}^{\infty}\frac{\abs{\Gamma\left(\xi,\eta\right)}}{\left(1+\eta\right)^{2+\theta}\left(\xi+1\right)^{1+\theta}}\deta\dxi\leq \frac{C}{p^{2+2\theta}},
 \end{split}
\end{equation}
where the last integral is bounded due to Lemma~\ref{Lem:est:Gamma:int}.

Using the estimates obtained in~\eqref{eq:L2:diff:I2:2} and \eqref{eq:L2:diff:I2:3} in~\eqref{eq:L2:diff:I2:1} we get
\begin{equation*}
 \begin{split}
  \abs{I_{2}}\leq C\frac{\snorm{1}{\omega_1}\snorm{2}{\omega_2}}{p^{3+2\theta}}+C\frac{\snorm{1}{\omega_1}\snorm{1}{\omega_2}}{p^{3+\theta-\alpha}}.
 \end{split}
\end{equation*}
Together with~\eqref{eq:diff:L2} and~\eqref{eq:L2:diff:I1:3} and using again $p\geq 2$ this then shows
\begin{equation*}
 \begin{split}
  \abs{L_{2}\left(p\right)-L_{2}\left(p+1\right)}&\leq  C\frac{\snorm{2}{\omega_1}\snorm{1}{\omega_2}}{p^{3+\theta-\alpha}}+C\frac{\snorm{1}{\omega_1}\snorm{2}{\omega_2}}{p^{3+2\theta}}+C\frac{\snorm{1}{\omega_1}\snorm{1}{\omega_2}}{p^{3+\theta-\alpha}}\\
  &\leq C\frac{\left(1+p\right)^{1-\theta}}{p^3}\left(\snorm{2}{\omega_1}\snorm{1}{\omega_2}+\snorm{1}{\omega_1}\snorm{2}{\omega_2}+\snorm{1}{\omega_1}\snorm{1}{\omega_2}\right),
 \end{split}
\end{equation*}
which finishes the proof.
\end{proof}

\begin{lemma}\label{Lem:est:L1:diff:large:p}
 There exists some constant $C>0$ such that for all $\omega_1,\omega_2\in \esp_{2}$ it holds
  \begin{equation*}
  \begin{split}
   \sup_{p\geq 2}\frac{p^3}{\left(1+p\right)^{1-\theta}}\abs{L_1\left(p\right)-L_1\left(p+1\right)}\leq C\left(\snorm{2}{\omega_1}\fnorm{0}{\omega_2}+\fnorm{2}{\omega_1}\snorm{1}{\omega_2}\right),
  \end{split}
 \end{equation*}
 where $L_1$ is as in~\eqref{eq:BW:Laplace}.
\end{lemma}

\begin{proof}
 From the definition of $L_1$ we first obtain for $p\geq 2$ that
 \begin{equation*}
  \begin{split}
   &\quad L_1\left(p\right)-L_1\left(p+1\right)\\
   &=\int_{0}^{\infty}\int_{0}^{\infty}\Gamma\left(\xi,\eta\right)\left[\left(\Omega_1''\left(\xi+p\right)-\Omega_1''\left(\xi+p+1\right)\right)\frac{\Omega_2\left(\eta+p\right)-\Omega_2\left(\eta+1\right)}{p-1}\right]\dxi\deta\\
   &\quad + \int_{0}^{\infty}\int_{0}^{\infty}\Gamma\left(\xi,\eta\right)\frac{\Omega_1''\left(\xi+p+1\right)}{p-1}\left(\left(\Omega_2\left(\eta+p\right)-\Omega_2\left(\eta+1\right)\right)-\left(\Omega_2\left(\eta+p+1\right)-\Omega_2\left(\eta+1\right)\right)\right)\dxi\deta\\
   &\quad +\int_{0}^{\infty}\int_{0}^{\infty}\Gamma\left(\xi,\eta\right)\Omega_{1}''\left(\xi+p+1\right)\left(\Omega_2\left(\eta+p+1\right)-\Omega_2\left(\eta+1\right)\right)\left(\frac{1}{p-1}-\frac{1}{p}\right)\dxi\deta.
  \end{split}
 \end{equation*}
Rearranging then gives
 \begin{equation*}
  \begin{split}
   &\quad L_1\left(p\right)-L_1\left(p+1\right)\\
   &=\int_{0}^{\infty}\int_{0}^{\infty}\Gamma\left(\xi,\eta\right)\left[\left(\Omega_1''\left(\xi+p\right)-\Omega_1''\left(\xi+p+1\right)\right)\frac{\Omega_2\left(\eta+p\right)-\Omega_2\left(\eta+1\right)}{p-1}\right]\dxi\deta\\
   &\quad+ \frac{1}{p-1}\int_{0}^{\infty}\int_{0}^{\infty}\Gamma\left(\xi,\eta\right)\Omega_1''\left(\xi+p+1\right)\left(\Omega_2\left(\eta+p\right)-\Omega_2\left(\eta+p+1\right)\right)\dxi\deta\\
   & \quad+\frac{1}{p\left(p-1\right)}\int_{0}^{\infty}\int_{0}^{\infty}\Gamma\left(\xi,\eta\right)\Omega_{1}''\left(\xi+p+1\right)\left(\Omega_2\left(\eta+p+1\right)-\Omega_2\left(\eta+1\right)\right)\dxi\deta=\vcc(I)+(II)+(III).
  \end{split}
  \end{equation*}
  We estimate the three terms separately. Using that $p\geq 2$ together with the monotonicity of  $\Lambda$ and~\eqref{eq:Lambda} we get 
\begin{equation*}
 \begin{split}
  \abs{\Omega_{1}''\left(\xi+p\right)}\leq C\frac{\snorm{2}{\omega_1}}{\left(\xi+p\right)^{2+\theta}} \quad \text{and} \quad \abs{\Omega_{1}''\left(\xi+p+1\right)}\leq C\frac{\snorm{2}{\omega_1}}{\left(\xi+p\right)^{2+\theta}}.
 \end{split}
\end{equation*}
Similarly 
\begin{equation*}
 \begin{split}
  \abs{\Omega_2\left(\eta+p\right)}\leq C\frac{\fnorm{0}{\omega_2}}{\left(\eta+p\right)^{\theta}}\leq C\frac{\fnorm{0}{\omega_2}}{\eta^{\theta}} \quad \text{and} \quad \abs{\Omega_2\left(\eta+1\right)}\leq C\frac{\fnorm{0}{\omega_2}}{\left(\eta+1\right)^{\theta}}\leq C\frac{\fnorm{0}{\omega_2}}{\eta^{\theta}}.
 \end{split}
\end{equation*}
Thus altogether we get
\begin{equation*}
 \begin{split}
  \abs{\left(I\right)}&\leq \int_{0}^{\infty}\int_{0}^{\infty}\abs{\Gamma\left(\xi,\eta\right)}\left(\abs{\Omega_{1}''\left(\xi+p\right)}+\abs{\Omega_1''\left(\xi+p+1\right)}\right)\frac{\abs{\Omega_2\left(\eta+p\right)}+\abs{\Omega_2\left(\eta+1\right)}}{p-1}\deta\dxi\\
  &\leq C\frac{\snorm{2}{\omega_1}\fnorm{0}{\omega_2}}{p}\int_{0}^{\infty}\int_{0}^{\infty}\frac{\abs{\Gamma\left(\xi,\eta\right)}}{\left(\xi+p\right)^{2+\theta}\eta^{\theta}}\deta\dxi\\
  &=C\frac{\snorm{2}{\omega_1}\fnorm{0}{\omega_2}}{p}\frac{p^2}{p^{2+\theta+1+\theta}}\int_{0}^{\infty}\int_{0}^{\infty}\frac{\abs{\Gamma\left(\xi,\eta\right)}}{\left(\xi+1\right)^{2+\theta}\eta^{\theta}}\deta\dxi\leq C\frac{\snorm{2}{\omega_1}\fnorm{0}{\omega_2}}{p^{2+2\theta}},
 \end{split}
\end{equation*}
 where we changed variables in the third step and used that the remaining integral on the right-hand side is bounded by a constant due to Lemma~\ref{Lem:est:Gamma:int:2}.

We next estimate the term $(II)$. Similarly as before we have for $p\geq 2$ that
\begin{equation*}
 \begin{split}
  \abs{\Omega_1''\left(\xi+p+1\right)}\leq C\frac{\snorm{2}{\omega_1}}{\left(\xi+p\right)^{2+\theta}}
 \end{split}
\end{equation*}
as well as 
\begin{equation*}
 \begin{split}
  \abs{\Omega_{2}\left(\eta+p\right)-\Omega_2\left(\eta+p+1\right)}\leq \snorm{1}{\omega_2}\int_{\eta+p}^{\eta+p+1}\frac{\left(1+s\right)^{1-\theta}}{s^2}\ds\leq C\frac{\snorm{1}{\omega_2}}{\left(\eta+p\right)^{1+\theta}}.
 \end{split}
\end{equation*}
Using this gives
\begin{equation*}
 \begin{split}
  \abs{(II)}&\leq \frac{1}{p-1}\int_{0}^{\infty}\int_{0}^{\infty}\abs{\Gamma\left(\xi,\eta\right)}\abs{\Omega_{1}''\left(\xi+p+1\right)}\abs{\Omega_{2}\left(\eta+p\right)-\Omega_2\left(\eta+p+1\right)}\dxi\deta\\
  &\leq C\frac{\snorm{2}{\omega_1}\snorm{1}{\omega_1}}{p}\int_{0}^{\infty}\int_{0}^{\infty}\abs{\Gamma\left(\xi,\eta\right)}\frac{1}{\left(\xi+p\right)^{2+\theta}}\frac{1}{\left(\eta+p\right)^{1+\theta}}\dxi\deta\\
  &=C\frac{\snorm{2}{\omega_1}\snorm{1}{\omega_1}}{p}\frac{p^2}{p^{1+2+\theta+1+\theta}}\int_{0}^{\infty}\int_{0}^{\infty}\abs{\Gamma\left(\xi,\eta\right)}\frac{1}{\left(1+\xi\right)^{2+\theta}}\frac{1}{\left(\eta+1\right)^{1+\theta}}\dxi\deta\leq C\frac{\snorm{2}{\omega_1}\snorm{1}{\omega_1}}{p^{3+2\theta}},
 \end{split}
\end{equation*}
as the last integral on the right-hand side is bounded by some constant due to Lemma~\ref{Lem:est:Gamma:int}.

It remains to estimate $(III)$. For this we use similarly as before
\begin{equation*}
 \begin{split}
  \abs{\Omega_1''\left(\xi+p+1\right)}\leq C\frac{\snorm{2}{\omega_1}}{\left(\xi+p\right)^{2+\theta}} \quad \text{and} \quad \abs{\Omega_2\left(\eta+p+1\right)}+\abs{\Omega_2\left(\eta+1\right)}\leq C\frac{\fnorm{0}{\omega_2}}{\eta^{\theta}}
 \end{split}
\end{equation*}
to obtain together with Lemma~\ref{Lem:est:Gamma:int:2} that
\begin{equation*}
 \begin{split}
  \abs{(III)}&\leq \frac{1}{\left(p-1\right)p}\int_{0}^{\infty}\int_{0}^{\infty}\abs{\Gamma\left(\xi,\eta\right)}\abs{\Omega_{1}''\left(\xi+p+1\right)}\left(\abs{\Omega_2\left(\eta+p+1\right)}+\abs{\Omega_2\left(\eta+1\right)}\right)\dxi\deta\\
  &\leq C\frac{\snorm{2}{\omega_1}\fnorm{0}{\omega_2}}{p^2}\int_{0}^{\infty}\int_{0}^{\infty}\abs{\Gamma\left(\xi,\eta\right)}\frac{1}{\left(\xi+p\right)^{2+\theta}}\frac{1}{\eta^{\theta}}\dxi\deta\\
  &=C\frac{\snorm{2}{\omega_1}\fnorm{0}{\omega_2}}{p^{3+2\theta}}\int_{0}^{\infty}\int_{0}^{\infty}\abs{\Gamma\left(\xi,\eta\right)}\frac{1}{\left(\xi+1\right)^{2+\theta}}\frac{1}{\eta^{\theta}}\deta\dxi\leq C\frac{\snorm{2}{\omega_1}\fnorm{0}{\omega_2}}{p^{3+2\theta}}.
 \end{split}
\end{equation*}
Summarizing the estimates obtained for $(I)-(III)$ we have
\begin{equation*}
 \begin{split}
  \abs{L_1\left(p\right)-L_1\left(p+1\right)}&\leq C\left(\frac{\snorm{2}{\omega_1}\fnorm{0}{\omega_2}}{p^{2+2\theta}}+\frac{\snorm{2}{\omega_1}\snorm{1}{\omega_1}}{p^{3+2\theta}}+\frac{\snorm{2}{\omega_1}\fnorm{0}{\omega_2}}{p^{3+2\theta}}\right)\\
  &\leq C\frac{\snorm{2}{\omega_1}\fnorm{0}{\omega_2}+\fnorm{2}{\omega_1}\snorm{1}{\omega_2}}{p^{2+2\theta}}
 \end{split}
\end{equation*}
for $p\geq 2$, finishing the proof.
\end{proof}

As a consequence of the previous estimates we obtain
\begin{lemma}\label{Lem:est:L}
 There exists some constant $C>0$ such that for all $\omega_1,\omega_2\in \esp_{2}$ it holds
 \begin{equation*}
  \begin{split}
   \sup_{p>0}\frac{p^3}{\left(1+p\right)^{1-\theta}}\abs{L\left(p\right)-L\left(p+1\right)}\leq C\fnorm{2}{\omega_1}\fnorm{2}{\omega_2},
  \end{split}
 \end{equation*}
 where $L$ is as in~\eqref{eq:def:L:BW}. 
\end{lemma}
\begin{proof}
 The proof immediately follows from the estimates in Lemma~\ref{Lem:est:L:small:p}, Lemma~\ref{Lem:est:L2:diff:large:p} and Lemma~\ref{Lem:est:L1:diff:large:p}.
\end{proof}

\subsection{Estimates on $\left(1-\zeta\right)m$}

\begin{lemma}\label{Lem:est:m:loc:large}
 There exists some constant such that for any two solutions $m_1$ and $m_2$ of~\eqref{eq:mu2} the difference $m\vcc=m_1-m_2$ satisfies
 \begin{equation*}
  \begin{split}
   \snorm{2}{\left(1-\zeta\right)m}\leq C\left(\fnorm{1}{m}+\eps\fnorm{2}{m}\right).
  \end{split}
 \end{equation*}
\end{lemma}

\begin{proof}
 Recalling that the Laplace transform $M$ of $m$ satisfies~\eqref{eq:M}, we have
\begin{equation*}
 \begin{split}
  \snorm{2}{\left(1-\zeta\right)m}&=\sup_{p>0}\frac{p^3}{\left(1+p\right)^{1-\theta}}\abs{\int_{0}^{\infty}x^2\left(1-\ee^{-x}\right)\ee^{-px}m\left(x\right)\dx}=\sup_{p>0}\frac{p^3}{\left(1+p\right)^{1-\theta}}\abs{M''\left(p\right)-M''\left(p+1\right)}.
 \end{split}
\end{equation*}
Using definition~\eqref{eq:def:L:BW}, we can rewrite and estimate
\begin{multline*}
  \quad\abs{M''\left(p\right)-M''\left(p+1\right)}\\
  \shoveleft{\leq \abs{\int_{0}^{\infty}x^{2}\ee^{-px}B_{2}\left(\bar{\mu},m\right)\dx}+\abs{\int_{0}^{\infty}x^{2}\ee^{-\left(p+1\right)x}B_{2}\left(\bar{\mu},m\right)\dx}
  +\abs{\int_{0}^{\infty}x^{2}\ee^{-px}B_{2}\left(m,\bar{\mu}\right)\dx}}\\
  \quad+\abs{\int_{0}^{\infty}x^{2}\ee^{-\left(p+1\right)x}B_{2}\left(m,\bar{\mu}\right)\dx}+\abs{\int_{0}^{\infty}x^{2}\ee^{-px}B_{2}\left(m,m_1\right)\dx}+\abs{\int_{0}^{\infty}x^{2}\ee^{-\left(p+1\right)x}B_{2}\left(m,m_1\right)\dx}\\
  \quad+\abs{\int_{0}^{\infty}x^{2}\ee^{-px}B_{2}\left(m_2,m\right)\dx}+\abs{\int_{0}^{\infty}x^{2}\ee^{-\left(p+1\right)x}B_{2}\left(m_2,m\right)\dx}\\
  \quad+\eps\abs{L\left[m,\mu_1\right]\left(p\right)-L\left[m,\mu_1\right]\left(p+1\right)}+\eps\abs{L\left[\mu_2,m\right]\left(p\right)-L\left[\mu_2,m\right]\left(p+1\right)}.
\end{multline*}
From Proposition~\ref{Prop:bilinear} and Lemma~\ref{Lem:est:L} we now obtain, using also $\frac{\left(2+p\right)^{1-\theta}}{\left(1+p\right)^3}\leq 2^{1-\theta}\frac{\left(1+p\right)^{1-\theta}}{p^3}$, that
\begin{multline*}
   \quad\abs{M''\left(p\right)-M''\left(p+1\right)}\\
   \shoveleft{\leq C\left(\fnorm{1}{\bar{\mu}}\fnorm{1}{m}+\fnorm{1}{m}\fnorm{1}{m_1}+\fnorm{1}{m_2}\fnorm{1}{m}\right)\frac{\left(1+p\right)^{1-\theta}}{p^3}}\\
   +\eps\left(\fnorm{1}{m}\fnorm{2}{\mu_1}+\snorm{2}{m}\fnorm{1}{\mu_1}+\fnorm{1}{\mu_2}\fnorm{2}{m}+\snorm{2}{\mu_2}\fnorm{1}{m}\right)\frac{\left(1+p\right)^{1-\theta}}{p^3}.
\end{multline*}
Using that $\fnorm{2}{m_i}=\fnorm{2}{\mu_i-\bar{\mu}}\leq \fnorm{2}{\mu_i}+\fnorm{2}{\bar{\mu}}$ together with Lemma~\ref{Lem:norm:est:pos} and Proposition~\ref{Prop:uniform:bound:muk} we get
\begin{equation*}
 \begin{split}
  \abs{M''\left(p\right)-M''\left(p+1\right)}\leq C\frac{\left(1+p\right)^{1-\theta}}{p^3}\left(\fnorm{1}{m}+\eps\fnorm{2}{m}\right),
 \end{split}
\end{equation*}
finishing the proof.
\end{proof}

The following result will be used later to show some regularisation effect for the difference $m$ of two solutions to~\eqref{eq:mu2} (see Proposition~\ref{Prop:reg} and the corresponding proof in Section~\ref{Sec:boundary:layer}). It follows directly from Lemma~\ref{Lem:est:m:loc:large} by interpolation.

\begin{proposition}\label{Prop:m:large:x}
 Let $\delta>0$. Then there exists some constant $C>0$ such that for $\eps>0$ sufficiently small and any two solutions $m_1$ and $m_2$ of~\eqref{eq:mu2} the difference $m\vcc=m_1-m_2$ satisfies
 \begin{equation*}
  \fnorm{2}{\left(1-\zeta\right)m}\leq C\fnorm{0}{m}+\delta\fnorm{2}{m}.
 \end{equation*}
\end{proposition}

\begin{proof}
 Applying first Lemma~\ref{Lem:interpolation} with $\delta=1$ and then Lemma~\ref{Lem:est:m:loc:large} gives (using also Lemma~\ref{Lem:elem:est:norm})
 \begin{equation*}
  \begin{split}
   \fnorm{2}{\left(1-\zeta\right)m}&=\fnorm{1}{\left(1-\zeta\right)m}+\snorm{2}{\left(1-\zeta\right)m}\leq \left(1+C_1\right)\fnorm{0}{\left(1-\zeta\right)m}+2\snorm{2}{\left(1-\zeta\right)m}\\
   &\leq \left(1+C_1\right)\fnorm{0}{\left(1-\zeta\right)m}+2\widehat{C}\left(\fnorm{1}{m}+\eps\fnorm{2}{m}\right)\\
   &\leq 2\left(1+C_1\right)\fnorm{0}{m}+2\widehat{C}C_{\widetilde{\delta}}\fnorm{0}{m}+2\widehat{C}\left(\eps+\widetilde{\delta}\right)\fnorm{2}{m}\\
   &=2\left(1+C_1+\widehat{C}C_{\widetilde{\delta}}\right)\fnorm{0}{m}+2\widehat{C}\left(\eps+\widetilde{\delta}\right)\fnorm{2}{m},
  \end{split}
 \end{equation*}
 where we applied Lemma~\ref{Lem:interpolation} again (with $\widetilde{\delta}$) in the last step. For $\eps,\widetilde{\delta}\leq \delta/(4\widehat{C})$ the claim follows. 
\end{proof}

\section{Inversion of the linearised coagulation operator}\label{Sec:Inversion}

In this section we give the proof of Proposition~\ref{Prop:Lin:inv}, i.e.\ we show that the linearised coagulation operator is well defined and can be inverted in the space $X_{k,\chi}$.

\begin{proof}[Proof of Proposition~\ref{Prop:Lin:inv}]
 We first remark, that the relation $\T\LL=\widehat{\LL}\T$ is just an easy computation.
 
 To show the boundedness of $\widehat{\LL}$ and the existence and boundedness of the inverse, we proceed in three main steps. First we show that $\widehat{\LL}$ is well-defined and bounded, next we show that $\widehat{\LL}$ is invertible on $X_{k,\chi}$ by computing the inverse explicitly and finally, using this explicit solution formula, we show that the inverse is also bounded.
 
 {\bf Step 1:} To show that $\widehat{\LL}$ is well-defined and bounded it suffices to show that for any $M\in X_{k,\chi}$ we have
 \begin{equation*}
  \lnorm{2,\chi}{\int_{p}^{\infty}\int_{q}^{\infty}\left(\frac{M\left(r\right)-M\left(1\right)}{\left(1-r\right)r^2}-\frac{M'\left(r\right)}{r}\right)\dr\dq}\leq C \lnorm{1,\chi}{M}.
 \end{equation*}
while due to the definition of $\lnorm{k,\chi}{\cdot}$ and~\eqref{eq:Lambda} this is equivalent to 
\begin{equation*}
 \abs{\frac{\dd^\ell}{\dd p^\ell}\int_{p}^{\infty}\int_{q}^{\infty}\left(\frac{M\left(r\right)-M\left(1\right)}{\left(1-r\right)r^2}-\frac{M'\left(r\right)}{r}\right)\dr\dq}\leq C\lnorm{1,\chi}{M}\Lambda\left(p\right)p^{-\ell}
\end{equation*}
 for all $p>0$ and $\ell=0,1,2$. We first consider $p>1$ and note that for $r\geq 1$ we have
 \begin{equation*}
  \begin{split}
   \abs{M\left(r\right)-M\left(1\right)}&=\abs{\int_{1}^{r}M'\left(s\right)\ds}\leq \lnorm{1,\chi}{M}\int_{1}^{r}s^{-1-\chi}\ds\leq \lnorm{1,\chi}{M}\left(r-1\right)\\
   \abs{M\left(r\right)-M\left(1\right)}&\leq C\lnorm{0}{M}\left(r^{-\chi}+1\right)\leq C\lnorm{0}{M}.
  \end{split}
 \end{equation*}
Together this gives 
 \begin{equation}\label{eq:est:diff:M}
  \begin{split}
   \abs{M\left(r\right)-M\left(1\right)}\leq C\lnorm{1,\chi}{M}\min\left\{\left(r-1\right),1\right\}.
  \end{split}
 \end{equation}
 Using this, as well as $\frac{\min\left\{\left(r-1\right),1\right\}}{r-1}\leq \frac{2}{r}$ for $r>1$, we find for $p>1$ that
 \begin{equation*}
  \begin{split}
   &\quad\abs{\int_{p}^{\infty}\int_{q}^{\infty}\left(\frac{M\left(r\right)-M\left(1\right)}{\left(1-r\right)r^2}-\frac{M'\left(r\right)}{r}\right)\dr\dq}\leq C\lnorm{1,\chi}{M}\int_{p}^{\infty}\int_{q}^{\infty}\left(\frac{\min\left\{\left(r-1\right),1\right\}}{\left(r-1\right)r^2}+r^{-2-\chi}\right)\dr\\
   &\leq C\lnorm{1,\chi}{M}\int_{p}^{\infty}\int_{q}^{\infty}\left(r^{-3}+r^{-2-\chi}\right)\dr\dq\leq C\lnorm{1}{M}p^{-\chi}.
  \end{split}
 \end{equation*}
 For the first derivative we similarly get
 \begin{equation*}
  \begin{split}
   &\quad \abs{\int_{p}^{\infty}\left(\frac{M\left(r\right)-M\left(1\right)}{\left(1-r\right)r^2}-\frac{M'\left(r\right)}{r}\right)\dr}\leq C\lnorm{1,\chi}{M}\int_{p}^{\infty}\left(\frac{\min\left\{\left(r-1\right),1\right\}}{\left(r-1\right)r^2}+r^{-2-\chi}\right)\dr\\
   &\leq C\lnorm{1,\chi}{M}\int_{p}^{\infty}\left(r^{-3}+r^{-2-\chi}\right)\dr\leq C\lnorm{1,\chi}{M}p^{-1-\chi}.
  \end{split}
 \end{equation*}
 In the same way we obtain for the second derivative
 \begin{equation*}
  \begin{split}
   &\quad \abs{\frac{M\left(p\right)-M\left(1\right)}{\left(p-1\right)p^2}-\frac{M'\left(p\right)}{p}}\leq C\lnorm{1,\chi}{M}p^{-2-\chi}.
  \end{split}
 \end{equation*}

For $p\leq 1$ we proceed similarly first noticing that for $r\leq 1$ we have
\begin{equation}\label{eq:inverse:small:p}
 \begin{split}
  \abs{M\left(1\right)-M\left(r\right)}=\abs{\int_{r}^{1}M'\left(s\right)\ds}\leq C\lnorm{1,\chi}{M}\int_{r}^{1}s^{-2}\ds=C\lnorm{1,\chi}{M}\frac{1-r}{r}.
 \end{split}
\end{equation}
Thus we get by splitting the integral
\begin{equation*}
 \begin{split}
  &\quad \abs{\int_{p}^{\infty}\int_{q}^{\infty}\left(\frac{M\left(r\right)-M\left(1\right)}{\left(1-r\right)r^2}-\frac{M'\left(r\right)}{r}\right)\dr\dq}\\
  &\leq \abs{\int_{p}^{1}\int_{q}^{1}\left(\cdots\right)\dr\dq}+\abs{\int_{p}^{1}\int_{1}^{\infty}\left(\cdots\right)\dr\dq}+\abs{\int_{1}^{\infty}\int_{q}^{\infty}\left(\cdots\right)\dr\dq}\\
  &\begin{multlined}
    \leq C\lnorm{1,\chi}{M}\int_{p}^{\infty}\int_{q}^{\infty}r^{-3}\dr\dq+C\lnorm{1,\chi}{M}\int_{0}^{1}\int_{1}^{\infty}\left(\frac{\min\left\{\left(r-1\right),1\right\}}{\left(r-1\right)r^2}+r^{-2-\chi}\right)\dr\dq\mbox{\qquad\qquad}\\
   +C\lnorm{1,\chi}{M}\int_{1}^{\infty}\int_{1}^{\infty}\left(\frac{\min\left\{\left(r-1\right),1\right\}}{\left(r-1\right)r^2}+r^{-2-\chi}\right)\dr\dq
   \end{multlined}\\
  &\leq C\lnorm{1,\chi}{M}p^{-1}+C\lnorm{1,\chi}{M}\leq C\lnorm{1,\chi}{M}p^{-1}.
 \end{split}
\end{equation*}
Similarly we obtain for the first and second derivative
\begin{equation*}
 \begin{split}
  \abs{\int_{p}^{\infty}\left(\frac{M\left(r\right)-M\left(1\right)}{\left(1-r\right)r^2}-\frac{M'\left(r\right)}{r}\right)\dr}\leq C\lnorm{1,\chi}{M}p^{-2}
 \end{split}
\end{equation*}
and
\begin{equation*}
 \begin{split}
  \abs{\frac{M\left(p\right)-M\left(1\right)}{\left(1-p\right)p^2}-\frac{M'\left(p\right)}{p}}\leq C\lnorm{1,\chi}{M}p^{-3}.
 \end{split}
\end{equation*}

Together this then shows the first step, i.e.\ that $\widehat{\LL}\colon X_{k,\chi}\to X_{k,\chi}$ is well-defined and $\lnorm{k,\chi}{\widehat{\LL} M}\leq C_1 \lnorm{k,\chi}{M}$ for some constant $C_1>0$.

{\bf Step 2:} We next show that the inverse operator $\widehat{\LL}^{-1}\colon X_{k,\chi}\to X_{k,\chi}$ exists. Therefore let $G\in X_{k,\chi}$. We then have to find $M\in X_{k,\chi}$ such that 
\begin{equation*}
 G=\widehat{\LL} M=M-2\int_{p}^{\infty}\int_{q}^{\infty}\left(\frac{M\left(r\right)-M\left(1\right)}{\left(1-r\right)r^2}-\frac{M'\left(r\right)}{r}\right)\dr\dq
\end{equation*}
 holds. Differentiating two times on both sides leads to
 \begin{equation*}
  G''\left(p\right)=M''\left(p\right)-\frac{2}{\left(1-p\right)p^2}\left(M\left(p\right)-M\left(1\right)\right)+\frac{2}{p}M'\left(p\right).
 \end{equation*}
Defining $H\left(p\right)\vcc=M\left(p\right)-M\left(1\right)$ we can rewrite this as
\begin{equation*}
 \left(pH'\right)'-H''+H'-\left(\frac{2}{p}H\right)'=\left(p-1\right)G''.
\end{equation*}
Using the decay properties at infinity of functions in $X_{k,\chi}$ we can integrate this equation on $\left(p,\infty\right)$ to obtain 
\begin{equation*}
 -pH'\left(p\right)+H'\left(p\right)-H\left(p\right)-M\left(1\right)+\frac{2}{p}H\left(p\right)=\int_{p}^{\infty}\left(s-1\right)G''\left(s\right)\ds.
\end{equation*}
 Using $\frac{p^2}{\left(1-p\right)^2}$ as integrating factor we can further rewrite the equation to get
 \begin{equation*}
  \left(\frac{p^2}{1-p}H\right)'-\frac{p^2}{\left(1-p\right)^2}M\left(1\right)=\frac{p^2}{\left(1-p\right)^2}\int_{p}^{\infty}\left(s-1\right)G''\left(s\right)\ds.
 \end{equation*}
Integrating now over $\left(0,p\right)$ gives
\begin{equation*}
 H\left(p\right)=M\left(1\right)\frac{1-p}{p^2}\int_{0}^{p}\frac{\xi^2}{\left(1-\xi\right)^2}\dxi+\frac{1-p}{p^2}\int_{0}^{p}\frac{\xi^2}{\left(1-\xi\right)^2}\int_{\xi}^{\infty}\left(s-1\right)G''\left(s\right)\ds\dxi.
\end{equation*}
Note that due to the singularity $\left(1-\xi\right)^{-2}$ in the integrals this formula is a-priori only well defined for $p<1$. In order to get an expression that is valid for all $p>0$ we have to rearrange the expression for $H$ (or $M$ respectively). We furthermore note that the value $M\left(1\right)$ is part of the problem and has also to be determined. Using the definition of $H$ we now rewrite the solution formula in terms of $M$ to get
\begin{equation}\label{eq:expr:M}
 \begin{split}
  M\left(p\right)=M\left(1\right)+\frac{1-p}{p^2}\int_{0}^{p}\frac{\xi^2}{\left(1-\xi\right)^2}\left(M\left(1\right)+\int_{\xi}^{\infty}\left(s-1\right)G''\left(s\right)\ds\right)\dxi.
 \end{split}
\end{equation}
Considering now the limit $p\to 1$ it turns out, by using for example l'H\^{o}pital's rule, that
\begin{equation*}
 \lim_{p\to 1}\frac{1-p}{p^2}\int_{0}^{p}\frac{\xi^2}{\left(1-\xi\right)^2}\left(M\left(1\right)+\int_{\xi}^{\infty}\left(s-1\right)G''\left(s\right)\ds\right)\dxi=M\left(1\right)+\int_{1}^{\infty}\left(s-1\right)G''\left(s\right)\ds.
\end{equation*}
 Thus, to get a consistent solution, we have to choose the value $M\left(1\right)$ such that 
 \begin{equation*}
  M\left(1\right)=-\int_{1}^{\infty}\left(s-1\right)G''\left(s\right)\ds=-G\left(1\right).
 \end{equation*}
 Using this in~\eqref{eq:expr:M} we obtain 
 \begin{equation}\label{eq:sol:form:M}
  \begin{split}
   M\left(p\right)=-G\left(1\right)+\frac{1-p}{p^2}\int_{0}^{p}\frac{\xi^2}{\left(1-\xi\right)^2}\int_{\xi}^{1}\left(s-1\right)G''\left(s\right)\ds\dxi.
  \end{split}
 \end{equation}
 One can easily see (by arguments similar to the one below) that this expression is valid for all $p>0$. In fact, assuming $G$ to be analytic in $\C\setminus \left[0,-\infty\right)$ one can even show that~\eqref{eq:sol:form:M} defines an analytic function $M$.

 We can further simplify this expression using integration by parts and the decay properties of $G$ to obtain
 \begin{equation*}
  \begin{split}
      &\quad M\left(p\right)=-G\left(1\right)+\frac{1-p}{p^2}\int_{0}^{p}\frac{\xi^2}{\left(1-\xi\right)^2}\int_{\xi}^{1}\left(s-1\right)\frac{\dd}{\ds}\left(G'\left(s\right)-G'\left(1\right)\right)\ds\dxi\\
   &=-G\left(1\right)+\frac{1-p}{p^2}\int_{0}^{p}\frac{\xi^2}{\left(1-\xi\right)^2}\left(-\left(\xi-1\right)\left(G'\left(\xi\right)-G'\left(1\right)\right)-\int_{\xi}^{1}G'\left(s\right)-G'\left(1\right)\ds\right)\dxi\\
   &=-G\left(1\right)+\frac{1-p}{p^2}\int_{0}^{p}\frac{\xi^2}{1-\xi}\left(G'\left(\xi\right)-G'\left(1\right)\right)\dxi-\frac{1-p}{p^2}\int_{0}^{p}\xi^2\int_{\xi}^{1}G'\left(s\right)-G'\left(1\right)\ds\frac{\dd}{\dxi}\left(1-\xi\right)^{-1}\dxi.
  \end{split}
 \end{equation*}
 Integrating by parts again in the last integral on the right-hand side gives
 \begin{multline*}
  M\left(p\right)=-G\left(1\right)+\frac{1-p}{p^2}\int_{0}^{p}\frac{\xi^2}{1-\xi}\left(G'\left(\xi\right)-G'\left(1\right)\right)\dxi-\frac{1-p}{p^2}\left(\frac{p^2}{1-p}\int_{p}^{1}G'\left(s\right)-G'\left(1\right)\ds\right.\\
   \left.-\int_{0}^{p}\frac{2\xi}{1-\xi}\int_{\xi}^{1}G'\left(s\right)-G'\left(1\right)\ds\dxi+\int_{0}^{p}\frac{\xi^2}{1-\xi}\left(G'\left(\xi\right)-G'\left(1\right)\right)\dxi\right).
 \end{multline*}
   Summarizing and rearranging we can simplify this expression to get
   \begin{equation}\label{eq:inverse:1}
   \begin{split}
   &\quad M\left(p\right)=-G\left(1\right)-\int_{p}^{1}G'\left(s\right)-G'\left(1\right)\ds+\frac{1-p}{p^2}\int_{0}^{p}\frac{2\xi}{1-\xi}\int_{\xi}^{1}G'\left(s\right)-G'\left(1\right)\ds\dxi\\
   &=-2G\left(1\right)+G\left(p\right)+G'\left(1\right)\left(1-p\right)+2\frac{1-p}{p^2}\int_{0}^{p}\frac{\xi}{1-\xi}\left(G\left(1\right)-G\left(\xi\right)\right)\dxi-G'\left(1\right)\frac{1-p}{p^2}\int_{0}^{p}2\xi\dxi\\
   &=G\left(p\right)-2G\left(1\right)+2\frac{1-p}{p^2}\int_{0}^{p}\frac{\xi}{1-\xi}\left(G\left(1\right)-G\left(\xi\right)\right)\dxi.
  \end{split}
 \end{equation}
{\bf Step 3:} We now finish the proof by showing that $M$ given by~\eqref{eq:sol:form:M} satisfies $M\in X_{k,\chi}$ and $\lnorm{k,\chi}{M}\leq C_2\lnorm{k,\chi}{G}$, i.e.\ $\lnorm{k,\chi}{\widehat{\LL}^{-1}G}\leq C_2\lnorm{k,\chi}{G}$. To do so we proceed similarly to the first step and again note that it suffices to consider $M\left(p\right)-G\left(p\right)$ or to be be more precise to show
\begin{equation*}
 \begin{split}
  \lnorm{2,\chi}{-2G\left(1\right)+\frac{2\left(1-p\right)}{p^2}\int_{0}^{p}\frac{\xi}{1-\xi}\left(G\left(1\right)-G\left(\xi\right)\right)\dxi}\leq C\lnorm{1,\chi}{G}
 \end{split}
\end{equation*}
which is equivalent to
\begin{equation*}
 \begin{split}
  \abs{\frac{\dd^\ell}{\dd p^\ell}\left(-2G\left(1\right)+\frac{2\left(1-p\right)}{p^2}\int_{0}^{p}\frac{\xi}{1-\xi}\left(G\left(1\right)-G\left(\xi\right)\right)\dxi\right)}\leq C\lnorm{1,\chi}{G}\Lambda_{\chi}\left(p\right)p^{-\ell}
 \end{split}
\end{equation*}
for all $p>0$ and $\ell=0,1,2$. We consider again the cases $p\leq 1$ and $p>1$ separately and begin with $p\leq 1$. For $\ell=0$ we obtain together with~\eqref{eq:inverse:small:p}
\begin{equation*}
 \begin{split}
  &\quad\abs{-2G\left(1\right)+\frac{2\left(1-p\right)}{p^2}\int_{0}^{p}\frac{\xi}{1-\xi}\left(G\left(1\right)-G\left(\xi\right)\right)\dxi}\\
  &\leq 2\abs{G\left(1\right)}+C\frac{1-p}{p^2}\lnorm{1,\chi}{G}\int_{0}^{p}\dxi\leq C\lnorm{0,\chi}{G}+C\lnorm{1,\chi}{G}\frac{1-p}{p}\leq C\lnorm{1,\chi}{G}p^{-1}.
 \end{split}
\end{equation*}
For the first derivative (i.e.\ $\ell=1$) we obtain similarly
\begin{equation*}
 \begin{split}
  &\quad\abs{\frac{\dd}{\dd p}\left(-2G\left(1\right)+\frac{2\left(1-p\right)}{p^2}\int_{0}^{p}\frac{\xi}{1-\xi}\left(G\left(1\right)-G\left(\xi\right)\right)\dxi\right)}\\
  &=\abs{-2\left(\frac{1}{p^2}-\frac{2}{p^3}\right)\int_{0}^{p}\frac{\xi}{1-\xi}\left(G\left(1\right)-G\left(\xi\right)\right)\dxi+\frac{2}{p}\left(G\left(1\right)-G\left(p\right)\right)}\leq C\lnorm{1,\chi}{G}\left(\frac{1}{p^2}+\frac{1}{p}\right)\\
  &\leq C\lnorm{1,\chi}{G}p^{-2}.
 \end{split}
\end{equation*}
For the second derivative ($\ell=2$) we find in the same way
\begin{equation*}
 \begin{split}
  &\quad\abs{\frac{\dd^2}{\dd p^2}\left(-2G\left(1\right)+\frac{2\left(1-p\right)}{p^2}\int_{0}^{p}\frac{\xi}{1-\xi}\left(G\left(1\right)-G\left(\xi\right)\right)\dxi\right)}\\
  &=2\abs{\left(\frac{6}{p^4}-\frac{2}{p^3}\right)\int_{0}^{p}\frac{\xi}{1-\xi}\left(G\left(1\right)-G\left(\xi\right)\right)\dxi-\left(\frac{1}{p^2}-\frac{2}{p^3}\right)\frac{p\left(G\left(1\right)-G\left(p\right)\right)}{1-p}-\frac{G\left(1\right)-G\left(p\right)}{p^2}-\frac{G'\left(p\right)}{p}}\\
  &\leq C\lnorm{1,\chi}{G}\left(\frac{1}{p^3}+\frac{1}{p^2}\right)\leq C\lnorm{1,\chi}{G}p^{-3}.
 \end{split}
\end{equation*}
This shows the expected behaviour for $p\leq 1$. It thus remains to estimate $M$ for $p>1$. Due to the constant $G\left(1\right)$ occurring in~\eqref{eq:inverse:1} we have to exploit some cancellation in this case and it is thus convenient to rewrite the expression for $M$ again. More precisely for $p>1$ by splitting the integral and rearranging we get
\begin{multline*}
  \quad -2G\left(1\right)+\frac{2\left(1-p\right)}{p^2}\int_{0}^{p}\frac{\xi}{1-\xi}\left(G\left(1\right)-G\left(\xi\right)\right)\dxi\\
  \shoveleft{=-2G\left(1\right)+\frac{2\left(1-p\right)}{p^2}\int_{0}^{1}\frac{\xi}{1-\xi}\left(G\left(1\right)-G\left(\xi\right)\right)\dxi+\frac{2\left(1-p\right)}{p^2}\int_{1}^{p}\left(1+\frac{1}{\xi-1}\right)\left(G\left(\xi\right)-G\left(1\right)\right)\dxi}\\
  \shoveleft{=-2G\left(1\right)+2C_G\frac{1-p}{p^2}+\frac{2\left(1-p\right)}{p^2}+\frac{2\left(1-p\right)}{p^2}\int_{1}^{p}\frac{G\left(\xi\right)-G\left(1\right)}{\xi-1}\dxi}\\
  +\frac{2\left(1-p\right)}{p^2}\int_{1}^{p}G\left(\xi\right)\dxi+2G\left(1\right)\frac{\left(1-p\right)^2}{p^2},
\end{multline*}
where we defined $C_G\vcc=\int_{0}^{1}\frac{\xi}{1-\xi}\left(G\left(1\right)-G\left(\xi\right)\right)\dxi$. Summarizing we then get
\begin{equation*}
 \begin{split}
  &\quad -2G\left(1\right)+\frac{2\left(1-p\right)}{p^2}\int_{0}^{p}\frac{\xi}{1-\xi}\left(G\left(1\right)-G\left(\xi\right)\right)\dxi\\  
  &=G\left(1\right)\left(\frac{2}{p^2}-\frac{4}{p}\right)+2C_G\left(\frac{1}{p^2}-\frac{1}{p}\right)+\frac{2\left(1-p\right)}{p^2}\int_{1}^{p}\frac{G\left(\xi\right)-G\left(1\right)}{\xi-1}\dxi+\frac{2\left(1-p\right)}{p^2}\int_{1}^{p}G\left(\xi\right)\dxi.
 \end{split}
\end{equation*}
Noting that due to~\eqref{eq:inverse:small:p} it holds
\begin{equation*}
 \begin{split}
  \abs{C_G}\leq \int_{0}^{1}\frac{\xi}{1-\xi}\abs{G\left(1\right)-G\left(\xi\right)}\dxi\leq C\lnorm{1,\chi}{G}\quad \text{and}\quad \abs{G\left(1\right)}\leq C\lnorm{0,\chi}{G}
 \end{split}
\end{equation*}
and 
\begin{equation*}
 \begin{split}
  \sup_{p>1}\frac{p^{\ell+1}}{\left(1+p\right)^{1-\chi}}\abs{\frac{\dd^\ell}{\dd p^\ell}\left(\frac{1}{p^n}\right)}\leq C=C\left(n,\ell,\chi\right)\quad \text{for all } n\in\N \text{ and } \ell=0,1,2\,,
 \end{split}
\end{equation*}
it suffices to consider only the term
\begin{equation*}
 \frac{2\left(1-p\right)}{p^2}\int_{1}^{p}\frac{G\left(\xi\right)-G\left(1\right)}{\xi-1}\dxi+\frac{2\left(1-p\right)}{p^2}\int_{1}^{p}G\left(\xi\right)\dxi.
\end{equation*}
Defining
\begin{equation*}
 A\left(p\right)\vcc=2\left(\frac{1}{p^2}-\frac{1}{p}\right)
\end{equation*}
we can rewrite this as
\begin{equation}\label{eq:inverse:large:p}
  A\left(p\right)\int_{1}^{p}\frac{G\left(\xi\right)-G\left(1\right)}{\xi-1}\dxi+A\left(p\right)\int_{1}^{p}G\left(\xi\right)\dxi.
\end{equation}
Before we estimate this term we collect some estimates. For $p\geq 1$ we have
\begin{equation*}
 \abs{\frac{\dd^{k}}{\dd p^k}A\left(p\right)}\leq \frac{C}{p^{k+1}} \quad \text{for } k=0,1,2.
\end{equation*}
Furthermore taking also~\eqref{eq:est:diff:M} into account we have for $p\geq 1$
\begin{equation*}
 \abs{\int_{1}^{p}\frac{G\left(\xi\right)-G\left(1\right)}{\xi-1}\dxi}\leq C\lnorm{1,\chi}{G}\int_{1}^{p}\frac{\min\left\{\left(\xi-1\right),1\right\}}{\xi-1}\dxi\leq C\lnorm{1,\chi}{G}\int_{1}^{p}\frac{2}{\xi}\dxi\leq C\lnorm{1,\chi}{G}\log\left(p\right)
\end{equation*}
and similarly
\begin{equation*}
 \abs{\frac{G\left(p\right)-G\left(1\right)}{p-1}}\leq C\lnorm{1,\chi}{G}p^{-1}.
\end{equation*}
Finally for $p\geq 1$ we also have
\begin{equation*}
 \begin{split}
  \abs{\int_{1}^{p}G\left(\xi\right)\dxi}\leq C\lnorm{0,\chi}{G}\int_{1}^{p}\xi^{-\chi}\dxi\leq C\lnorm{1,\chi}{G}p^{1-\chi}\quad \text{and} \quad \abs{\frac{\dd^{k}}{\dd p^k}G\left(p\right)}\leq C\lnorm{k,\chi}{G}p^{-\chi-k}
 \end{split}
\end{equation*}
for $k=0,1,2$. Using these preliminary considerations we now go back to~\eqref{eq:inverse:large:p} and estimate this expression as well as the first and second derivative. We have for $p>1$
\begin{equation*}
 \begin{split}
  \abs{A\left(p\right)\int_{1}^{p}\frac{G\left(\xi\right)-G\left(1\right)}{\xi-1}\dxi+A\left(p\right)\int_{1}^{p}G\left(\xi\right)\dxi}\leq C\lnorm{1,\chi}{G}\frac{1}{p}\left(\log\left(p\right)+p^{1-\chi}\right)\leq C\lnorm{1,\chi}{G}p^{-\chi}.
 \end{split}
\end{equation*}
For the first derivative we get
\begin{equation*}
 \begin{split}
  &\quad \abs{A'\left(p\right)\int_{0}^{p}\frac{G\left(\xi\right)-G\left(1\right)}{\xi-1}\dxi+A\left(p\right)\frac{G\left(p\right)-G\left(1\right)}{p-1}+A'\left(p\right)\int_{0}^{p}G\left(\xi\right)\dxi+A\left(p\right)G\left(p\right)}\\
  &\leq C\lnorm{1,\chi}{G}\left(\frac{\log\left(p\right)}{p^2}+\frac{1}{p^2}+\frac{p^{1-\chi}}{p^2}+\frac{p^{-\chi}}{p}\right)\leq C\lnorm{1,\chi}{G}p^{-1-\chi}.
 \end{split}
\end{equation*}
Finally for the second derivative we have (taking into account also the definition of $A\left(p\right)$)
\begin{equation*}
 \begin{split}
  &\quad \abs{\frac{\dd^2}{\dd p^2}\left(A\left(p\right)\int_{0}^{p}\frac{G\left(\xi\right)-G\left(1\right)}{\xi-1}\dxi+A\left(p\right)\int_{1}^{p}G\left(\xi\right)\dxi\right)}\\
  &\leq \abs{A''\left(p\right)\int_{0}^{p}\frac{G\left(\xi\right)-G\left(1\right)}{\xi-1}\dxi}+2\abs{A'\left(p\right)\frac{G\left(p\right)-G\left(1\right)}{p-1}}+\abs{A\left(p\right)\left(\frac{G'\left(p\right)}{p-1}-\frac{G\left(p\right)-G\left(1\right)}{\left(p-1\right)^2}\right)}\\
  &\qquad+\abs{A''\left(p\right)\int_{0}^{p}G\left(\xi\right)\dxi}+2\abs{A'\left(p\right)G\left(p\right)}+\abs{A\left(p\right)G'\left(p\right)}\\
  &\leq C\lnorm{1,\chi}{G}\left(\frac{\log\left(p\right)}{p^3}+\frac{1}{p^3}\right)+\abs{-\frac{G'\left(p\right)}{p^2}+\frac{G\left(p\right)-G\left(1\right)}{p^2\left(p-1\right)}}+C\lnorm{1,\chi}{G}\left(\frac{p^{1-\chi}}{p^3}+\frac{p^{-\chi}}{p^2}+\frac{p^{-1-\chi}}{p}\right)\\
  &\leq C\lnorm{1,\chi}{G}p^{-1-\chi}.
 \end{split}
\end{equation*}
This then shows that $\widehat{\LL}^{-1}$ is well-defined and bounded in $X_{k,\chi}$ and thus finishes the proof.
\end{proof}

\section{Uniform bounds for $\mu_k$ and smallness for $m_k$}\label{Sec:apriori}

In this section we show that each solution of~\eqref{eq:mu} is uniformly bounded and each solution of~\eqref{eq:mu2} is small in the $\fnorm{2}{\cdot}$-norm under the appropriate rescaling as explained in Remark~\ref{Rem:rescaling}. The first step for this is the following technical Lemma.

\begin{lemma}\label{Lem:lap:est:mu}
 Let $n\in\N_{0}$. For any $\delta>0$ there exists some constant $p_{*}=p_{*}\left(\delta,n\right)>0$ such that it holds for any solution $m_k$ of~\eqref{eq:mu2}
 \begin{equation*}
  p^{n+\theta}\abs{\int_{0}^{\infty}x^n\ee^{-px}m_k\left(x\right)\dx}\leq \delta\quad \text{for any } p>p_{*}.
 \end{equation*}
 \end{lemma}
 
  \begin{proof}
  We first note that we can write $m_k=\mu_k-\bar{\mu}$ with $\bar{\mu}=1$ and $\mu_k$ solving~\eqref{eq:mu}. For $p\geq 2$ we then get
  \begin{equation}\label{eq:smallness:1:n:c}
 \begin{split}
  \abs{\int_{0}^{\infty}x^n\ee^{-px}m_k\left(x\right)\dx}&\leq \int_{0}^{1/p}x^n\ee^{-px}\mu_k\left(x\right)\dx+\int_{1/p}^{\infty}x^n\ee^{-px}\mu_k\left(x\right)\dx+\int_{0}^{\infty}x^n\ee^{-px}\dx\\
  &\leq \int_{0}^{1/p}x^n\ee^{-px}\mu_k\left(x\right)\dx+\int_{1/p}^{\infty}x^n\ee^{-px}\mu_k\left(x\right)\dx+\frac{n!}{p^{n+1}}.
 \end{split}
\end{equation}
We now estimate the integrals on the right-hand side. On the one hand using~\eqref{eq:mean:est} together with a dyadic decomposition we have
\begin{equation}\label{eq:smallness:2:n:c}
 \begin{split}
  \int_{1/p}^{\infty}x^n\ee^{-px}\mu_k\left(x\right)\dx=\sum_{\ell=0}^{\infty}\int_{\frac{2^\ell}{p}}^{\frac{2^{\ell+1}}{p}}x^n\ee^{-px}\mu_{k}\left(x\right)\dx\leq C\sum_{\ell=0}^{\infty}\left(\frac{2^{\ell+1}}{p}\right)^{n+1}\ee^{-2^{\ell}}\leq C_n p^{-n-1}.
 \end{split}
\end{equation}
 On the other hand we get by means of~\eqref{eq:mean:est} that
 \begin{equation}\label{eq:smallness:3:n:c}
  \begin{split}
   \int_{0}^{1/p}x^n\ee^{-px}\mu_k\left(x\right)\dx&\leq p^{-n}\int_{0}^{1/p}\mu_{k}(x)\dx\leq C_{\eta}p^{\eta-1-n}.
  \end{split}
 \end{equation}
 Then plugging~\eqref{eq:smallness:2:n:c} and~\eqref{eq:smallness:3:n:c} into~\eqref{eq:smallness:1:n:c} we obtain
 \begin{equation*}
  \abs{\int_{0}^{\infty}x^n\ee^{-px}m_k\left(x\right)\dx}\leq C_{\eta,n}\left(p^{\eta-n-1}+p^{-n-1}\right)\leq \delta p^{-n-\theta}
 \end{equation*}
 for $\eta$ sufficiently small and $p>p_{*}=p_{*}\left(\delta,n\right)$ sufficiently large.
 \end{proof}

With Lemma~\ref{Lem:lap:est:mu} we can show the smallness of $m_k=\mu_k-\bar{\mu}$ in the $\fnorm{0}{\cdot}$-norm for $k=1,2$.

\begin{lemma}\label{Lem:m:small:0}
 For given $\delta>0$ there exists $\eps_0>0$ such that the following holds for any $\eps<\eps_0$ and any solution $\mu_k$ of~\eqref{eq:mu} it holds
 \begin{equation}\label{eq:S2E1}
  \fnorm{0}{\mu_k-\bar{\mu}}\leq \delta.
 \end{equation}
\end{lemma}

\begin{proof}
Recalling~\eqref{fdecay}, \eqref{S2E5} and~\eqref{S2E3} we will derive~\eqref{eq:S2E1}. If $U_k$ and $\bar U$ denote the Laplace transforms of $\mu_k$ and $\bar \mu=1$ 
respectively, and ${\cal F}_k$ and $\bar {\cal F}$ the desingularised Laplace transforms of $f_k$ and $\bar f$ respectively, i.e. ${\cal F}_k
=\int_0^{\infty} \big(1-e^{-px}\big) f_k(x)\dx$, then
we have the relation 
\begin{equation}\label{Urelation}
U_k(p)=-{\cal F}_k(p{-}1) + U_k(1)
\end{equation}
 etc., while \eqref{S2E5} means that
$\sup_{p>-1}\frac{1+p}{p}\Big| {\cal F}_k(p)-\bar {\cal F}(p)\Big| \leq \delta$.

We first notice that $U_k-\bar{U}$ is the Laplace transform of $m_k=\mu_k-\bar{\mu}$ and thus due to Lemma~\ref{Lem:lap:est:mu} we have
\[
p^{\theta }\left\vert U_k\left( p\right) -\bar{U}\left( p\right) \right\vert
\leq \delta \, ,\qquad p\geq p_{0 } \,.
\]
On the other hand~\eqref{S2E5} implies that $\Big| \int_0^{\infty} (f_k(x)-e^{-x})\psi(x)\,dx\Big|$ can be made small for 
small $\eps$ for any smooth test function $\psi$ with compact support. Due to \eqref{fdecay} and \eqref{S2E3} the regions for large and small $x$
can be made arbitrarily small and hence it follows that
\begin{equation*}
 \Big| U_k(1)-\bar U_k(1)\Big| = \Big| \int_0^{\infty} \big( f_k(x)-e^{-x}\big)\,dx \Big |\leq \delta
\end{equation*}
if $\eps$ is sufficiently small. Then it also follows from \eqref{S2E5} and \eqref{Urelation} that
\begin{equation*}
 \sup_{p \in (0, p_0)} \frac{p}{(1+p)^{1-\theta}}|U_k(p)-\bar U(p)|\leq \delta
\end{equation*}
for sufficiently small $\eps$.
\end{proof}

\subsection{Proof of Proposition~\ref{Prop:uniform:bound:muk}}\label{Sec:proof:uniform:bd:muk}

\begin{proof}[Proof of Proposition~\ref{Prop:uniform:bound:muk}]
 The proof immediately follows from Lemma~\ref{Lem:m:small:0} together with Lemma~\ref{Lem:norm:est:pos} noting that $\mu_k$ is non-negative.
\end{proof}

\begin{remark}\label{Rem:m:bounded:uniform}
 Noting that $m_k=\mu_k-\bar{\mu}$ and the Laplace transform of $\bar{\mu}=1$ is just $1/p$ we also get from Proposition~\ref{Prop:uniform:bound:muk} that $\fnorm{2}{m_k}$ is uniformly bounded for any $m_k$ solving~\eqref{eq:mu2}.
\end{remark}

\subsection{Proof of Proposition~\ref{Prop:smallness:mk}}\label{Sec:proof:smallness}

We next show that the $\snorm{2}{\cdot}$-norm of any solution $m$ of~\eqref{eq:mu2} can be estimated by the $\fnorm{0}{\cdot}$-norm and some small constant.

\begin{lemma}\label{Lem:m:small:semi}
 For given $\delta>0$ there exists $C>0$ and $\eps_0>0$ such that for any $\eps<\eps_0$ and any $m_k$ solving~\eqref{eq:mu2} it holds
 \begin{equation*}
  \snorm{2}{m_k}\leq C\fnorm{0}{m_k}+\delta.
 \end{equation*}
\end{lemma}
\begin{proof}
 We will consider two different regions for $p$; the region where $p$ is large can be controlled by Lemma~\ref{Lem:lap:est:mu}, while for small $p$ we will use~\eqref{eq:mu2} to control the norm of $m_k$.

 We start with the region where $p$ is large and get from Lemma~\ref{Lem:lap:est:mu} that
 \begin{equation}\label{eq:est:m:large:p}
  \abs{\int_{0}^{\infty}x^2\ee^{-px}m_k\left(x\right)\dx}\leq \frac{\delta}{2}p^{-2-\theta}
 \end{equation}
 for sufficiently large $p>p_{*}=p_{*}\left(\delta\right)$.

In order to control the region where $p<p_{*}$, we use~\eqref{eq:mu2} together with Proposition~\ref{Prop:bilinear} to obtain
\begin{equation}\label{eq:est:BW:small:p}
 \begin{split}
  \abs{\int_{0}^{\infty}x^2\ee^{-px}B_W\left(\omega_1,\omega_2\right)\left(x\right)\dx}&\leq C\fnorm{2}{\omega_1}\fnorm{2}{\omega_2}\frac{\left(p+1\right)^{1-\left(\theta-\alpha\right)}}{p^3}\\
  &\leq C\left(1+p_{*}\right)^{\alpha}\fnorm{2}{\omega_1}\fnorm{2}{\omega_2}\frac{\left(1+p\right)^{1-\theta}}{p^3}. 
 \end{split}
\end{equation}
Using $\bar{\mu}+m_k=\mu_k$,~\eqref{eq:est:BW:small:p} and Proposition~\ref{Prop:bilinear}, we obtain from~\eqref{eq:mu2} for $p\in\left(0,p_{*}\right)$ that
\begin{equation*}
 \begin{split}
  \abs{\int_{0}^{\infty}x^2\ee^{-px}m_k\left(x\right)\dx}&\leq \abs{\int_{0}^{\infty}x^2\ee^{-px}B_2\left(\bar{\mu},m_k\right)\dx}+\abs{\int_{0}^{\infty}x^2\ee^{-px}B_2\left(m_k,\bar{\mu}\right)\dx}\\
  &\quad + \abs{\int_{0}^{\infty}x^2\ee^{-px}B_2\left(m_k,m_k\right)\dx}+\eps\abs{\int_{0}^{\infty}x^2\ee^{-px}B_W\left(\mu_k,\mu_k\right)\dx}\\
  &\leq C\left(\fnorm{1}{\bar{\mu}}\fnorm{1}{m_k}+\fnorm{1}{m_k}^2\right)\frac{\Lambda\left(p\right)}{p^2}+C\left(p_{*}\right)\eps\fnorm{2}{\mu_k}^2\frac{\Lambda\left(p\right)}{p^2}.
 \end{split}
\end{equation*}
Using furthermore that $\fnorm{1}{m_k}$, $\fnorm{2}{\mu_k}$ and $\fnorm{1}{\bar{\mu}}$ are uniformly bounded by some constant due to Proposition~\ref{Prop:uniform:bound:muk} and Remark~\ref{Rem:m:bounded:uniform} we further get
\begin{equation*}
 \begin{split}
  \sup_{p\in\left(0,p_{*}\right)}\frac{\left(1+p\right)^{1-\theta}}{p^3}\abs{\int_{0}^{\infty}x^2\ee^{-px}m_k\left(x\right)\dx}\leq C\fnorm{1}{m_k}+C\left(p_{*}\right)\eps.
 \end{split}
\end{equation*}
Applying the interpolation result from Lemma~\ref{Lem:interpolation} with $\delta_1>0$ gives
\begin{equation*}
 \begin{split}
  \sup_{p\in\left(0,p_{*}\right)}\frac{\left(1+p\right)^{1-\theta}}{p^3}\abs{\int_{0}^{\infty}x^2\ee^{-px}m_k\left(x\right)\dx}&\leq C\left(\delta_1\right)\fnorm{0}{m_k}+C\delta_1\snorm{2}{m_k}+C\left(p_{*}\right)\eps\\
  &\leq  C\left(\delta_1\right)\fnorm{0}{m_k}+C\delta_1+C\left(p_{*}\right)\eps.
 \end{split}
\end{equation*}
Choosing then $\delta_1$ and $\eps$ sufficiently small the claim follows with~\eqref{eq:est:m:large:p}.
\end{proof}

\begin{proof}[Proof of Proposition~\ref{Prop:smallness:mk}]
 This follows directly from Lemma~\ref{Lem:m:small:0} and Lemma~\ref{Lem:m:small:semi} taking also Lemma~\ref{Lem:norm:equ} into account.
\end{proof}

\section{Boundary layer estimate and regularising effect}\label{Sec:boundary:layer}

In this section we derive an estimate for $\zeta m$, where $m=m_1-m_2$ is the difference of two solutions $m_1$ and $m_2$ of equation~\eqref{eq:mu2}. Note that from now on we always assume $\alpha>0$ since for $\alpha=0$ the analysis of the boundary layer as well as the regularising effect is not necessary. 

\begin{proposition}\label{Prop:est:zeta:m}
 For any $\delta_{*}>0$ it holds
 \begin{equation*}
  \fnorm{2}{\zeta m}\leq \delta_{*}\fnorm{2}{m}+C_{\delta_{*}}\fnorm{1}{\left(1-\zeta\right)m},
 \end{equation*}
 if $\eps$ is sufficiently small.
\end{proposition}

As the proof of this result is quite technical and the main part of this work we first derive the proofs of Proposition~\ref{Prop:splitting:2} and Proposition~\ref{Prop:reg}, while the proof of Proposition~\ref{Prop:est:zeta:m} will be postponed to Sections~\ref{Sec:bl:proof}\textendash\ref{Sec:asymptotics}.
 
\begin{proof}[Proof of Proposition~\ref{Prop:splitting:2}]
The claim follows directly from Proposition~\ref{Prop:est:zeta:m} by interpolation, i.e.\ using Lemmas~\ref{Lem:interpolation} and~\ref{Lem:elem:est:norm}.
\end{proof}

\begin{proof}[Proof of Proposition~\ref{Prop:reg}]
Writing $m=\left(1-\zeta\right)m+\zeta m$ the claim follows from Propositions~\ref{Prop:splitting:2} and~\ref{Prop:m:large:x} together with Lemma~\ref{Lem:elem:est:norm}.
\end{proof}

\section{The boundary layer equation - Proof of Proposition~\ref{Prop:est:zeta:m}}\label{Sec:bl:proof}

To obtain estimates for $\fnorm{0}{\zeta m}$  we rewrite~\eqref{eq:mu} as
\begin{equation}\label{nuequation}
x\nu_j \left( x\right) =\int_{0}^{x}\beta _{K}\left( y;\mu_j \right) \nu_j \left(
y\right) \dy-R_{K}\left( x\right) 
\end{equation}
with the notations $\nu_j(x)=xe^{-x}\mu_j(x)$,
\begin{align*}
\beta _{K}( y;\mu_j ) &=\int_{0}^{\infty }K( y,z) \mu_j( z) e^{-z}\dz,\\
R_{K}( x) &=\int_{0}^{x}\int_{0}^{x-y}K( y,z)
y\mu_j ( y) \mu_j  (z) e^{-( y+z) }\dz\dy.\nonumber
\end{align*}

In order to cancel singular terms at the origin we need to use the fact that $\beta _{2}( x;\mu_j )=\beta_2(\mu_j) $ is close to $2$. We reformulate~\eqref{nuequation} in a more convenient form by differentiating in \eqref{nuequation} to obtain 
$\left( x\nu_j (x) \right)_x =\beta _{2}( \mu_j) \nu_j (x) +\varepsilon \beta _{W}( x;\mu_j ) \nu_j(x) -\left(R_{K}( x)\right)_x$.
We  split the terms
\begin{align*}
&\left( x\nu_j(x) \right)_x 
 =\frac{\beta _{2}\left( \mu_j
\right) }{x}\left( x\nu_j (x)\right) +\frac{\varepsilon \beta
_{W}( x;\mu_j ) }{x}e^{-x}\left( x\nu_j (x) \right)
+\varepsilon \beta _{W}( x;\mu_j ) \left( 1{-}e^{-x}\right) \nu_j
( x) -\left(R_{K}(x)\right)_x \,.
\end{align*}
Using the integrating factor 
\begin{equation}\label{Phidef}\Phi ( x;\mu_j ) \vcc=\varepsilon \int_{x}^{\infty }\frac{\beta
_{W}( t;\mu_j ) }{t}e^{-t}\dt 
\end{equation}
and that $ x\nu_j(x) \exp \left( \Phi( x;\mu_j) \right) \big| _{x=\infty }=0$ 
we can integrate the equation. Dividing  by $\exp\left( \Phi ( x;\mu_j) \right)$ and
recalling that $\nu_j (x) =xe^{-x}\mu_j \left( x\right)$ it follows that
\begin{align*}
\zeta\left(x\right)\mu_j (x) &=-x^{\beta _{2}\left( \mu_j \right)
-2}\int_{x}^{\infty }\xi ^{-\beta _{2}\left( \mu_j \right) }\exp \left( -\Phi
( x;\mu_j ) +\Phi ( \xi ;\mu_j ) \right) \varepsilon \beta
_{W}\left( \xi ;\mu_j \right) ( 1{-}e^{-\xi }) \xi e^{-\xi }\mu_j
( \xi ) \dxi  \\
&\qquad +x^{\beta _{2}\left( \mu_j \right) -2}\int_{x}^{\infty }\xi ^{-\beta
_{2}\left( \mu_j \right) }\exp \left( -\Phi( x;\mu_j ) +\Phi (\xi ;\mu_j) \right) \left(R_{K}(\xi)\right)_{\xi}
\dxi\,.
\end{align*}
For a shorter notation we introduce
\[
\kappa _{j}=\beta
_{2}\left( \mu _{j}\right) -2\,, \qquad \Phi_j(\cdot)=\Phi( \cdot ;\mu _{j})\,,\qquad j=1,2 \,.
\]
\begin{remark}\label{Rem:smallness:kappa}
 We notice that $\kappa_k=2\left(U_k\left(1\right)-\bar{U}\left(1\right)\right)$ and thus we obtain from Lemma~\ref{Lem:m:small:0} that
 \begin{equation*}
  \abs{\kappa_k}\leq C\fnorm{0}{\mu_k-\bar{\mu}}\to 0\quad \text{as } \eps\to 0.
 \end{equation*}
\end{remark}

Then, also evaluating $\left(R_{K}(\xi)\right)_{\xi}$, we find
\begin{equation}\label{eq:mu3}
 \begin{split}
  \zeta\left(x\right)\mu_j (x)
&=-\varepsilon \int_{x}^{\infty }\left( \frac{x}{\xi }\right) ^{\kappa_j }\frac{e^{\Phi_j ( \xi ) }}{e^{\Phi_j (x) }}\beta _{W}( \xi ;\mu_j ) \left( \frac{1{-}e^{-\xi }}{\xi }
\right) e^{-\xi }\mu_j \left( \xi \right) \dxi  \\
&\qquad \qquad  +\int_{x}^{\infty }\frac{e^{-\xi } }{\xi ^{2}}\left( \frac{x}{\xi }\right) ^{\kappa_j }\frac{e^{\Phi_j ( \xi ) }}{e^{\Phi_j (x) }}\int_{0}^{\xi }K( y,\xi {-}y) y\mu_j (y)\mu_j( \xi {-}y) \dy\dxi.
 \end{split}
\end{equation}

In order to prove Proposition~\ref{Prop:est:zeta:m}, we need to obtain the difference of the Laplace transform of the left-hand side of~\eqref{eq:mu3} for $j=1,2$. We denote by $U_j$ and $\bar{U}$ the Laplace transforms of $\mu_j$ and $\bar{\mu}$ for $j=1,2$ and note that due to the definition of $m_j$ it holds $m_1-m_2=\mu_1-\mu_2$, thus we can consider $\dmu\vcc=\mu_1-\mu_2$ instead of $m$. Then it holds
\begin{equation*}
  U_j\left(p\right)=\int_{0}^{\infty}\mu_j\left(x\right)\ee^{-px}\dx \quad \text{and} \quad U_j''\left(p+1\right)=\int_{0}^{\infty}x^2\zeta\left(x\right)\mu_j\left(x\right)\ee^{-px}\dx.
 \end{equation*}
Multiplying equation~\eqref{eq:mu3} with $\left(x\right)x^2\ee^{-px}$ and integrating we thus find
\begin{equation*}
 \begin{split}
  U_j''\left(p+1\right)&=-\eps\int_{0}^{\infty}\int_{0}^{\xi}x^2\ee^{-px}\left(\frac{x}{\xi}\right)^{\kappa_j}\frac{\ee^{\Phi_j\left(\xi\right)}}{\ee^{\Phi_j\left(\xi\right)}}\beta_W\left(\xi;\mu_j\right)\left(\frac{1-\ee^{-\xi}}{\xi}\right)\ee^{-\xi}\mu_j\left(\xi\right)\dx\dxi\\
  &\quad+\int_{0}^{\infty}\int_{y}^{\infty}\int_{0}^{\xi}\frac{\ee^{-\xi}}{\xi^2}x^2\ee^{-px}\left(\frac{x}{\xi}\right)^{\kappa_j}\frac{\ee^{\Phi_j\left(\xi\right)}}{\ee^{\Phi_j\left(x\right)}}K\left(y,\xi-y\right)y\mu_j\left(y\right)\mu_j\left(\xi-y\right)\dx\dxi\dy\\
  &=-\eps\int_{0}^{\infty}\int_{0}^{\xi}x^2\ee^{-px}\left(\frac{x}{\xi}\right)^{\kappa_j}\frac{\ee^{\Phi_j\left(\xi\right)}}{\ee^{\Phi_j\left(\xi\right)}}\beta_W\left(\xi;\mu_j\right)\left(\frac{1-\ee^{-\xi}}{\xi}\right)\ee^{-\xi}\mu_j\left(\xi\right)\dx\dxi\\
  &\quad +\int_{0}^{\infty}\int_{0}^{\infty}\int_{0}^{y+z}\frac{\ee^{-\left(y+z\right)}}{\left(y+z\right)^2}x^2\ee^{-px}\left(\frac{x}{y+z}\right)^{\kappa_j}\frac{\ee^{\Phi_j\left(y+z\right)}}{\ee^{\Phi_j\left(x\right)}}K\left(y,z\right)y\mu_j\left(y\right)\mu_j\left(z\right)\dx\dz\dy.
 \end{split}
\end{equation*}
 By taking the difference of the equations for $j=1$ and $j=2$ and rearranging we then obtain
 \begin{equation*}
  \begin{split}
   U_1''\left(p+1\right)-U_2''\left(p+1\right)=-\eps\left[K_1+K_2+K_3+K_4\right]+\left[J_1+J_2+J_3\right]
  \end{split}
 \end{equation*}
 with
 \begin{align*}
  K_1&=\int_{0}^{\infty}\int_{0}^{\xi}x^2\ee^{-px}\left(\left(\frac{x}{\xi}\right)^{\kappa_1}-\left(\frac{x}{\xi}\right)^{\kappa_2}\right)\frac{\ee^{\Phi_1\left(\xi\right)}}{\ee^{\Phi_1\left(x\right)}}\beta_W\left(\xi;\mu_1\right)\left(\frac{1-\ee^{-\xi}}{\xi}\right)\ee^{-\xi}\mu_1\left(\xi\right)\dx\dxi,\\
  K_2&=\int_{0}^{\infty}\int_{0}^{\xi}x^2\ee^{-px}\left(\frac{x}{\xi}\right)^{\kappa_2}\left(\frac{\ee^{\Phi_1\left(\xi\right)}}{\ee^{\Phi_1\left(x\right)}}-\frac{\ee^{\Phi_2\left(\xi\right)}}{\ee^{\Phi_2\left(x\right)}}\right)\beta_W\left(\xi;\mu_1\right)\left(\frac{1-\ee^{-\xi}}{\xi}\right)\ee^{-\xi}\mu_1\left(\xi\right)\dx\dxi,\\
  K_3&=\int_{0}^{\infty}\int_{0}^{\xi}x^2\ee^{-px}\left(\frac{x}{\xi}\right)^{\kappa_2}\frac{\ee^{\Phi_2\left(\xi\right)}}{\ee^{\Phi_2\left(x\right)}}\left(\beta_W\left(\xi;\mu_1\right)-\beta_W\left(\xi;\mu_2\right)\right)\left(\frac{1-\ee^{-\xi}}{\xi}\right)\ee^{-\xi}\mu_1\left(\xi\right)\dx\dxi,\\
  K_4&=\int_{0}^{\infty}\left(1-\ee^{-y}\right)\ee^{-y}\left(1+y\right)H\left(y,p\right)\beta_W\left(y;\mu_2\right)\left(\mu_1\left(y\right)-\mu_2\left(y\right)\right)\dy
  \end{align*}
 and (using the notation $\int_{\Sigma}\left(\cdots\right)\vcc=\int_{0}^{\infty}\int_{0}^{\infty}\int_{0}^{y+z}\left(\cdots\right)\dx\dz\dy$)
\begin{align*}
 J_1&=\int_{\Sigma}\frac{\ee^{-\left(y+z\right)}}{\left(y+z\right)^2}x^2\ee^{-px}\left(\left(\frac{x}{y+z}\right)^{\kappa_1}-\left(\frac{x}{y+z}\right)^{\kappa_2}\right)\frac{\ee^{\Phi_1\left(y+z\right)}}{\ee^{\Phi_1\left(x\right)}}K\left(y,z\right)y\mu_1\left(y\right)\mu_1\left(z\right),\\
 J_2&=\int_{\Sigma}\frac{\ee^{-\left(y+z\right)}}{\left(y+z\right)^2}x^2\ee^{-px}\left(\frac{x}{y+z}\right)^{\kappa_2}\left(\frac{\ee^{\Phi_1\left(y+z\right)}}{\ee^{\Phi_1\left(x\right)}}-\frac{\ee^{\Phi_2\left(y+z\right)}}{\ee^{\Phi_2\left(x\right)}}\right)K\left(y,z\right)y\mu_1\left(y\right)\mu_1\left(z\right),\\
 J_3&=\int_{0}^{\infty}\int_{0}^{\infty}\ee^{-\left(y+z\right)}H\left(y+z,p\right)\left(1+y+z\right)\frac{y}{y+z}K\left(y,z\right)\left(\mu_1\left(y\right)\mu_1\left(z\right)-\mu_2\left(y\right)\mu_2\left(z\right)\right)\dz\dy,
\end{align*}
where we also defined the function
\begin{equation*}
 H\left(\Y,p\right)\vcc=\frac{1}{\Y\left(1+\Y\right)}\int_{0}^{\Y}x^2\ee^{-px}\left(\frac{x}{\Y}\right)^{\kappa_2}\frac{\ee^{\Phi_2\left(\Y\right)}}{\ee^{\Phi_2\left(x\right)}}\dx.
\end{equation*}

To prove Proposition~\ref{Prop:est:zeta:m} we then have to control the terms $K_1-K_4$ and $J_1-J_3$, while it is relatively easy to control $K_1-K_3$ and $J_1$ and $J_2$. The hard part is then to estimate the expressions $K_4$ and $J_3$ as these are Laplace transforms of products. Therefore they might be written as convolutions of Laplace transforms. The main difficulties in estimating theses terms are then that we have to show that the inverse Laplace transform of $H$ exists to obtain suitable estimates for it. 

In the following we will denote by $\delta>0$ some small constant and assume implicitly that we take $\eps$ sufficiently small such that $\abs{\kappa_j}\leq \delta$.

\begin{remark}\label{Rem:est:beta:W}
 From the definition of $\beta_W$ together with~\eqref{kernel3}, \eqref{fdecay}, \eqref{S2E3} and~\eqref{eq:uniform:est:f} we have
\begin{equation*}
 0\leq \beta_W\left(\xi;\mu_j\right)\leq C\left(\xi^{-\alpha}+\xi^{\alpha}\right).
\end{equation*}
\end{remark}

\begin{proof}[Proof of Proposition~\ref{Prop:est:zeta:m}]
In the following we will estimate the terms $K_i$ and $J_i$ separately, while for the terms $K_i$ it suffices to show that they are bounded by $\fnorm{2}{m}$ due to the additional factor $\eps$ in front of them.

\subsection{Estimates for $K_1$ and $J_1$}

We start by estimating $K_1$ and $J_1$ and we will show for any $\delta_{0}>0$
\begin{equation}\label{eq:bl:1}
\eps\snorm{2}{K_1}+\snorm{2}{J_1}\leq \delta_0\fnorm{0}{\mu_1-\mu_2}+C_{\delta_0}\fnorm{0}{\left(1-\zeta\right)\left(\mu_1-\mu_2\right)}.
\end{equation}
which suffices due to Lemma~\ref{Lem:norm:equ}. For this we first note that from the definition of $\Phi_j$ as well as the non-negativity of $\beta_W$ it follows that $\Phi_j$ is non-increasing, i.e.\ for $x\leq \xi$ it holds
\begin{equation*}
 \Phi_j\left(\xi\right)-\Phi_j\left(x\right)\leq 0\quad \text{and} \quad \frac{\ee^{\Phi_j\left(\xi\right)}}{\ee^{\Phi_j\left(x\right)}}\leq 1.
\end{equation*}
 Furthermore, using the elementary inequality $\abs{\ee^{-a}-\ee^{-b}}\leq \abs{a-b}$ for $a,b>0$ we find for $0<x\leq \xi$ and sufficiently small $\eps>0$, using also Remark~\ref{Rem:smallness:kappa}, that
 \begin{equation*}
  \begin{split}
   \abs{\left(\frac{x}{\xi}\right)^{\kappa_1}-\left(\frac{x}{\xi}\right)^{\kappa_2}}&=\ee^{-\log\left(\frac{\xi}{x}\right)\min\left\{\kappa_1,\kappa_2\right\}}\abs{\ee^{-\log\left(\frac{\xi}{x}\right)\left(\kappa_1-\min\left\{\kappa_1,\kappa_2\right\}\right)}-\ee^{-\log\left(\frac{\xi}{x}\right)\left(\kappa_2-\min\left\{\kappa_1,\kappa_2\right\}\right)}}\\
   &\leq \left(\frac{\xi}{x}\right)^{\delta}\log\left(\frac{\xi}{x}\right)\abs{\kappa_1-\kappa_2}.
  \end{split}
 \end{equation*}
In the same way we get for $x\leq y+z$ that
\begin{equation*}
 \begin{split}
  \frac{\ee^{\Phi_j\left(y+z\right)}}{\ee^{\Phi_j\left(x\right)}}\leq 1\quad \text{and}\quad \abs{\left(\frac{x}{y+z}\right)^{\kappa_1}-\left(\frac{x}{y+z}\right)^{\kappa_2}}\leq \left(\frac{y+z}{x}\right)^{\delta}\log\left(\frac{y+z}{x}\right)\abs{\kappa_1-\kappa_2}.
 \end{split}
\end{equation*}
Using this, Remark~\ref{Rem:est:beta:W}, as well as $\frac{1-\ee^{-\xi}}{\xi}\leq C$ and $\log w\leq Cw^{\delta}$ for $w\geq 1$ we obtain
\begin{equation*}
 \begin{split}
  \abs{K_1}&\leq C\int_{0}^{\infty}\int_{0}^{\xi}x^2\ee^{-px}\left(\frac{\xi}{x}\right)^{\delta}\log\left(\frac{\xi}{x}\right)\left[\xi^{-\alpha}+\xi^{\alpha}\right]\ee^{-\xi}\mu_1\left(\xi\right)\dx\dxi\abs{\kappa_1-\kappa_2}\\
  &\leq C\int_{0}^{\infty}\int_{0}^{\xi}x^2\ee^{-px}x^{-2\delta}\xi^{2\delta}\left(\xi^{-\alpha}+\xi^{\alpha}\right)\ee^{-\xi}\mu_1\left(\xi\right)\dxi\abs{\kappa_1-\kappa_2}\\
  &=C\int_{0}^{\infty}\xi^{2\delta}\left(\xi^{-\alpha}+\xi^{\alpha}\right)\ee^{-\xi}\mu_1\left(\xi\right)\int_{0}^{\xi}x^{2-2\delta}\ee^{-px}\dx\dxi\abs{\kappa_1-\kappa_2}.
 \end{split}
\end{equation*}
To control the remaining integral we note that
\begin{equation*}
 \int_{0}^{\xi}x^{2-2\delta}\ee^{-px}\dx\leq\begin{cases}
                                             \xi^{3-2\delta} & p\leq 1\\
                                             \frac{C}{p^{3-2\delta}} & p\geq 1.
                                            \end{cases}
\end{equation*}
Using this we find together with~\eqref{fdecay} and~\eqref{S2E3} that
\begin{equation*}
 \begin{split}
  \abs{K_1}\leq \int_{0}^{\infty}\xi^{3}\left(\xi^{-\alpha}+\xi^{\alpha}\right)\ee^{-\xi}\mu_1\left(\xi\right)\dxi\abs{\kappa_1-\kappa_2}\leq C\abs{\kappa_1-\kappa_2} \quad \text{for } p\leq 1
 \end{split}
\end{equation*}
as well as
\begin{equation*}
 \begin{split}
  \abs{K_1}\leq \frac{C}{p^{3-2\delta}}\int_{0}^{\infty}\xi ^{2\delta}\left(\xi^{-\alpha}+\xi^{\alpha}\right)\ee^{-\xi}\mu_1\left(\xi\right)\dxi\abs{\kappa_1-\kappa_2}\leq Cp^{-2-\theta}\abs{\kappa_1-\kappa_2} \quad \text{for } p\geq 1. 
 \end{split}
\end{equation*}
Together this shows
\begin{equation*}
 \abs{K_1}\leq C\abs{\kappa_1-\kappa_2}\frac{\left(1+p\right)^{1-\theta}}{p^3}.
\end{equation*}
Similarly using that we have $K\left(y,z\right)\leq C\left(\bigl(\frac{y}{z}\bigr)^{\alpha}+\bigl(\frac{z}{y}\bigr)^{\alpha}\right)$ we get
\begin{equation*}
 \begin{split}
  \abs{J_1}&\leq C\int_{0}^{\infty}\int_{0}^{\infty}\int_{0}^{y+z}\frac{\ee^{-\left(y+z\right)}}{\left(y+z\right)^2}x^2\ee^{-px}\left(\frac{y+z}{x}\right)^{2\delta}\left(\Bigl(\frac{y}{z}\Bigr)^{\alpha}+\Bigl(\frac{z}{y}\Bigr)^{\alpha}\right)y\mu_1\left(y\right)\mu_1\left(z\right)\dx\dz\dy\abs{\kappa_1-\kappa_2}\\
  &\leq C\int_{0}^{\infty}\int_{0}^{\infty}\frac{\ee^{-\left(y+z\right)}}{\left(y+z\right)^{1-2\delta}}\left(\Bigl(\frac{y}{z}\Bigr)^{\alpha}+\Bigl(\frac{z}{y}\Bigr)^{\alpha}\right)\mu_1\left(y\right)\mu_1\left(z\right)\int_{0}^{y+z}x^{2-2\delta}\ee^{-px}\dx\dz\dy\abs{\kappa_1-\kappa_2}\\
  &\leq C\abs{\kappa_1-\kappa_2}\frac{\left(1+p\right)^{1-\theta}}{p^3}.
 \end{split}
\end{equation*}
Thus it remains to estimate $\abs{\kappa_1-\kappa_2}$. Since $\abs{\kappa_1-\kappa_2}=2\abs{U_1\left(1\right)-U_2\left(1\right)}$ it suffices to estimate $U_1\left(1\right)-U_2\left(1\right)$ which will be done by some interpolation argument via
\begin{equation*}
 \begin{split}
  \abs{U_1\left(1\right)-U_2\left(1\right)}&=\abs{\int_{0}^{\infty}\left(\mu_1\left(z\right)-\mu_2\left(z\right)\right)\ee^{-z}\dz}\\
  &\leq \abs{\int_{0}^{\infty}\left(\mu_1\left(z\right)-\mu_2\left(z\right)\right)\ee^{-z}\ee^{-nz}\dz}+\abs{\int_{0}^{\infty}\left(\mu_1\left(z\right)-\mu_2\left(z\right)\right)\ee^{-z}\left(1-\ee^{-nz}\right)\dz}\\
  &\leq \fnorm{0}{\mu_1-\mu_2}\frac{\left(2+n\right)^{1-\theta}}{\left(1+n\right)}+\abs{\int_{0}^{\infty}\left(\mu_1\left(z\right)-\mu_2\left(z\right)\right)\left(1-\zeta\left(z\right)\right)\frac{1-\ee^{-nz}}{1-\ee^{-z}}\ee^{-z}\dz}.
 \end{split}
\end{equation*}
Using 
\begin{equation*}
 \frac{1-\ee^{-nz}}{1-\zeta}=\sum_{k=0}^{n-1}\ee^{-kz}
\end{equation*}
we can further rewrite
\begin{equation*}
 \begin{split}
  \abs{U_1\left(1\right)-U_2\left(1\right)}&\leq C\frac{\fnorm{0}{\mu_1-\mu_2}}{\left(1+n\right)^{\theta}}+\sum_{k=0}^{n-1}\abs{\int_{0}^{\infty}\left(\mu_1\left(z\right)-\mu_2\left(z\right)\right)\left(1-\zeta\left(z\right)\right)\ee^{-\left(k+1\right)z}\dz}\\
  &\leq C\frac{\fnorm{0}{\mu_1-\mu_2}}{\left(1+n\right)^{\theta}}+\fnorm{0}{\left(1-\zeta\right)\left(\mu_1-\mu_2\right)}\sum_{k=0}^{n-1}\frac{\left(2+k\right)^{1-\theta}}{\left(1+k\right)}\\
  &\leq \delta_0\fnorm{0}{\mu_1-\mu_2}+C_{\delta_0}\fnorm{0}{\left(1-\zeta\right)\left(\mu_1-\mu_2\right)}
 \end{split}
\end{equation*}
for any $\delta_0>0$. Summarizing we have thus shown
\begin{equation*}
 \eps\abs{K_1}+\abs{J_1}\leq \left(\delta_0\fnorm{0}{\mu_1-\mu_2}+C_{\delta_0}\fnorm{0}{\left(1-\zeta\right)\left(\mu_1-\mu_2\right)}\right)\frac{\left(1+p\right)^{1-\theta}}{p^3},
\end{equation*}
which is just~\eqref{eq:bl:1}.

\subsection{Estimate of $K_2$ and $J_2$}
 
Next we estimate $K_2$ and $J_2$, i.e.\ we show 
\begin{equation}\label{eq:bl:2}
\eps\snorm{2}{K_2}+\snorm{2}{J_2}\leq C\eps\fnorm{2}{\mu_1-\mu_2}.
\end{equation} For this we will use the following estimate which is a consequence of the monotonicity of $\Phi_j$ as well as the elementary inequality $\abs{\ee^{-a}-\ee^{-b}}\leq \abs{a-b}$ for $a,b\geq 0$. For $0<x\leq \xi$ we have $0\leq \Phi_j\left(\xi\right)\leq \Phi_j\left(x\right)$ and thus
\begin{equation*}
 \begin{split}
  \abs{\frac{\ee^{\Phi_1\left(\xi\right)}}{\ee^{\Phi_1\left(x\right)}}-\frac{\ee^{\Phi_2\left(\xi\right)}}{\ee^{\Phi_2\left(x\right)}}}&=\abs{\ee^{-\left(\Phi_1\left(x\right)-\Phi_1\left(\xi\right)\right)}-\ee^{-\left(\Phi_2\left(x\right)-\Phi_2\left(\xi\right)\right)}}\leq \abs{\Phi_1\left(x\right)-\Phi_1\left(\xi\right)-\Phi_2\left(x\right)+\Phi_2\left(\xi\right)}\\
  &\leq \abs{\Phi_1\left(\xi\right)-\Phi_2\left(\xi\right)}+\abs{\Phi_1\left(x\right)-\Phi_2\left(x\right)}.
 \end{split}
\end{equation*}
In the same way we get for $x\leq y+z$
\begin{equation*}
 \abs{\frac{\ee^{\Phi_1\left(y+z\right)}}{\ee^{\Phi_1\left(x\right)}}-\frac{\ee^{\Phi_2\left(y+z\right)}}{\ee^{\Phi_2\left(x\right)}}}\leq \abs{\Phi_1\left(y+z\right)-\Phi_2\left(y+z\right)}+\abs{\Phi_1\left(x\right)-\Phi_2\left(x\right)}.
\end{equation*}
Using this for $K_2$ and $J_2$ we obtain
\begin{equation}\label{eq:K2:1}
 \abs{K_2}\leq \int_{0}^{\infty}\int_{0}^{\xi}x^2\ee^{-px}\left(\frac{\xi}{x}\right)^{\delta}\left(\abs{\Phi_1\left(x\right)-\Phi_2\left(x\right)}+\abs{\Phi_1\left(\xi\right)-\Phi_2\left(\xi\right)}\right)\beta_W\left(\xi;\mu_1\right)\ee^{-\xi}\mu_1\left(\xi\right)\dx\dxi
\end{equation}
and
\begin{equation}\label{eq:J2:1}
 \begin{split}
  \abs{J_2}\leq \int_{\Sigma}\frac{\ee^{-\left(y+z\right)}}{\left(y+z\right)^2}x^2\ee^{-px}\left(\frac{y+z}{x}\right)^{\delta}\left(\abs{\Phi_1\left(y+z\right)-\Phi_2\left(y+z\right)}+\abs{\Phi_1\left(x\right)-\Phi_2\left(x\right)}\right)K\left(y,z\right)y\mu_1\left(y\right)\mu_1\left(z\right).
 \end{split}
\end{equation}
In order to continue we have to estimate $\Phi_1-\Phi_2$ and we thus rewrite this difference in terms of the Laplace-transforms of $\mu_1$ and $\mu_2$. From the definition of $\Phi_j$ it holds together with Proposition~\ref{P.repkernels} that
\begin{equation*}
 \begin{split}
  \Phi_j\left(x\right)&=\eps\int_{x}^{\infty}\frac{\beta_W\left(t,\mu_j\right)}{t}\ee^{-t}\dt=\eps\int_{x}^{\infty}\int_{0}^{\infty}W\left(t,z\right)\mu_j\left(z\right)\ee^{-z}\dz\frac{\ee^{-t}}{t}\dt\\
  &=\eps\int_{x}^{\infty}\int_{0}^{\infty}\int_{0}^{\infty}\int_{0}^{\infty}\Gamma\left(\xi,\eta\right)\ee^{-\xi t-\eta z}\dxi\deta\mu_j\left(z\right)\ee^{-z}\ee^{-t}\left(1+\frac{z}{t}\right)\dz\dt\\
  &=\eps\int_{x}^{\infty}\int_{0}^{\infty}\int_{0}^{\infty}\Gamma\left(\xi,\eta\right)\ee^{-\left(\xi+1\right)t}U_j\left(\eta+1\right)\deta\dxi\dt\\
  &\quad-\eps\int_{x}^{\infty}\int_{0}^{\infty}\int_{0}^{\infty}\del_{\xi}\int_{0}^{\xi}\Gamma\left(s,\eta\right)\ds\ee^{-\left(\xi+1\right)t}\frac{U_j'\left(\eta+1\right)}{t}\deta\dxi\dt.
\end{split}
\end{equation*}
Evaluating the integrals in $t$ gives
 \begin{equation*}
 \begin{split}
 \Phi_{j}\left(x\right)&=\eps\int_{0}^{\infty}\int_{0}^{\infty}\frac{\Gamma\left(\xi,\eta\right)}{\xi+1}U_j\left(\eta+1\right)\ee^{-\left(\xi+1\right)x}\deta\dxi\\
  &\quad-\eps\int_{x}^{\infty}\int_{0}^{\infty}\int_{0}^{\infty}\int_{0}^{\xi}\Gamma\left(s,\eta\right)\ds U_j'\left(\eta+1\right)\ee^{-\left(\xi+1\right)t}\deta\dxi\dt\\
  &=\eps\int_{0}^{\infty}\int_{0}^{\infty}\frac{\Gamma\left(\xi,\eta\right)}{\xi+1}U_j\left(\eta+1\right)\ee^{-\left(\xi+1\right)x}\deta\dxi-\eps\int_{0}^{\infty}\int_{0}^{\infty}\frac{\ee^{-\left(1+\xi\right)x}}{1+\xi}U_j'\left(\eta+1\right)\int_{0}^{\xi}\Gamma\left(s,\eta\right)\ds\dxi\deta.
 \end{split}
\end{equation*}
From this we obtain
\begin{equation*}
 \begin{split}
  &\quad \abs{\Phi_1\left(X\right)-\Phi_2\left(X\right)}\\
  &=\eps\abs{\int_{0}^{\infty}\frac{\ee^{-\left(\xi+1\right)X}}{\xi+1}\int_{0}^{\infty}\Gamma\left(\xi,\eta\right)\left(U_1\left(\eta+1\right)-U_2\left(\eta+1\right)\right)-\left(U_1'\left(\eta+1\right)-U_2'\left(\eta+1\right)\right)\int_{0}^{\xi}\Gamma\left(s,\eta\right)\ds\dxi\deta}\\
  &\leq \eps\int_{0}^{\infty}\frac{\ee^{-\left(\xi+1\right)X}}{\xi+1}\int_{0}^{\infty}\abs{\Gamma\left(\xi,\eta\right)}\abs{U_1\left(\eta+1\right)-U_2\left(\eta+1\right)}\deta\dxi\\
  &\quad +\eps\int_{0}^{\infty}\frac{\ee^{-\left(\xi+1\right)X}}{\xi+1}\int_{0}^{\infty}\int_{0}^{\xi}\abs{\Gamma\left(s,\eta\right)}\ds\abs{U_1'\left(\eta+1\right)-U_2'\left(\eta+1\right)}\deta\dxi\\
  &\leq \eps\fnorm{0}{\dmu}\int_{0}^{\infty}\frac{\ee^{-\left(\xi+1\right)X}}{\xi+1}\int_{0}^{\infty}\abs{\Gamma\left(\xi,\eta\right)}\frac{\left(2+\eta\right)^{1-\theta}}{1+\eta}\deta\dxi\\
  &\quad+\eps\snorm{1}{\dmu}\int_{0}^{\infty}\frac{\ee^{-\left(\xi+1\right)X}}{\xi+1}\int_{0}^{\infty}\int_{0}^{\xi}\abs{\Gamma\left(s,\eta\right)}\ds\frac{\left(2+\eta\right)^{1-\theta}}{\left(1+\eta\right)^2}\deta\dxi.
 \end{split}
\end{equation*}
Thus we need to control the remaining integrals. For the first one we find together with Proposition~\ref{P.repkernels} that
\begin{equation*}
 \begin{split}
  &\quad \int_{0}^{\infty}\int_{0}^{\infty}\frac{\ee^{-\left(\xi+1\right)X}}{\xi+1}\abs{\Gamma\left(\xi,\eta\right)}\frac{\left(2+\eta\right)^{1-\theta}}{1+\eta}\deta\dxi\\
  &\leq C\int_{0}^{\infty}\int_{0}^{\infty}\frac{\ee^{-\left(1+\xi\right)X}}{\xi+1}\frac{1}{\left(\xi+\eta\right)^{1-\alpha}}\left(\frac{1}{\xi^{\alpha}}+\frac{1}{\eta^{\alpha}}\right)\frac{1}{\left(1+\eta\right)^{\theta}}\dxi\deta+C\ee^{-X}\int_{0}^{\infty}\ee^{-\xi X}\frac{\left(2+\xi\right)^{1-\theta}}{\left(1+\xi\right)^{2}}\dxi\\
  &= C\ee^{-X}\int_{0}^{\infty}\frac{\ee^{-\xi X}}{\xi+1}\left(\frac{1}{\xi^{\alpha}}\int_{0}^{\infty}\frac{1}{\left(\xi+\eta\right)^{1-\alpha}}\frac{1}{\left(1+\eta\right)^{\theta}}\deta+\int_{0}^{\infty}\frac{1}{\left(\xi+\eta\right)^{1-\alpha}}\frac{1}{\eta^{\alpha}}\frac{1}{\left(1+\eta\right)^{\theta}}\deta\right)\dxi+C\ee^{-X}.
 \end{split}
\end{equation*}
Estimating $\left(\xi+\eta\right)^{\alpha-1}\leq \eta^{\alpha-1}$ as well as $\left(1+\eta\right)^{-\theta}\leq \eta^{-\theta}$ it follows
\begin{equation*}
 \begin{split}
  &\quad \int_{0}^{\infty}\int_{0}^{\infty}\frac{\ee^{-\left(\xi+1\right)X}}{\xi+1}\abs{\Gamma\left(\xi,\eta\right)}\frac{\left(2+\eta\right)^{1-\theta}}{1+\eta}\deta\dxi\\
  &\leq C\ee^{-X}\int_{0}^{\infty}\frac{\ee^{-\xi X}}{\xi+1}\left(\frac{1}{\xi^{\alpha}}\int_{0}^{\infty}\frac{1}{\eta^{1-\alpha}}\frac{1}{\left(1+\eta\right)^{\theta}}\deta+\int_{0}^{\infty}\frac{1}{\left(\xi+\eta\right)^{1-\alpha}}\frac{1}{\eta^{\alpha+\theta}}\deta\right)\dxi+C\ee^{-X}\\
  &\leq C\ee^{-X}\int_{0}^{\infty}\frac{\ee^{-\xi X}}{\xi+1}\left(\frac{C}{\xi^{\alpha}}+\frac{1}{\xi^{\theta}}\int_{0}^{\infty}\frac{1}{\left(1+\eta\right)^{1-\alpha}}\frac{1}{\eta^{\alpha+\theta}}\deta\right)\dxi+C\ee^{-X}\\
  &\leq C\ee^{-X}\int_{0}^{\infty}\frac{\ee^{-\xi X}}{\xi+1}\left(\frac{1}{\xi^{\alpha}}+\frac{1}{\xi^{\theta}}\right)\dxi+C\ee^{-X}\leq C\ee^{-X}.
 \end{split}
\end{equation*}
Furthermore we have
\begin{equation*}
 \begin{split}
  \int_{0}^{\xi}\abs{\Gamma\left(s,\eta\right)}\ds&\leq C\int_{0}^{\xi}\frac{1}{\left(s+\eta\right)^{1-\alpha}}\left(\frac{1}{s^{\alpha}}+\frac{1}{\eta^{\alpha}}\right)\ds+C\int_{0}^{\infty}\delta_{s=\eta}\ds=C\int_{0}^{\frac{\xi}{\eta}}\frac{1}{\left(s+1\right)^{1-\alpha}}\left(\frac{1}{s^{\alpha}}+1\right)\ds+C\\
  &\leq C\int_{0}^{\frac{\xi}{\eta}}\frac{1}{s^{1-\alpha}}\ds+C=C\left(\left(\frac{\xi}{\eta}\right)^{\alpha}+1\right),
 \end{split}
\end{equation*}
where we used $\left(s+1\right)^{\alpha-1}\left(s^{-\alpha}+1\right)\leq 2s^{\alpha-1}$ in the second step. From this we get
\begin{equation*}
 \begin{split}
  &\quad\int_{0}^{\infty}\frac{\ee^{-\left(\xi+1\right)X}}{\xi+1}\int_{0}^{\infty}\int_{0}^{\xi}\abs{\Gamma\left(s,\eta\right)}\ds\frac{\left(2+\eta\right)^{1-\theta}}{\left(1+\eta\right)^2}\deta\dxi\\
  &\leq C\ee^{-X}\int_{0}^{\infty}\xi^{\alpha-1}\ee^{-\xi X}\int_{0}^{\infty}\frac{1}{\left(1+\eta\right)^{1+\theta}}\left(\frac{1}{\eta^{\alpha}}+1\right)\deta\dxi\leq C\ee^{-X}X^{-\alpha}\int_{0}^{\infty}\ee^{-\xi}\xi^{\alpha-1}\dxi\leq C\frac{\ee^{-X}}{X^{\alpha}}.
 \end{split}
\end{equation*}
Summarizing we find
\begin{equation*}
 \begin{split}
  \abs{\Phi_1\left(X\right)-\Phi_2\left(X\right)}\leq C\eps\ee^{-X}\left(\fnorm{0}{\dmu}+\frac{\snorm{1}{\dmu}}{X^{\alpha}}\right)\leq C\eps \fnorm{1}{\dmu}\frac{\ee^{-X/2}}{X^{\alpha}}.
 \end{split}
\end{equation*}
Coming now back to the estimate of $K_2$ and $J_2$ we obtain from~\eqref{eq:K2:1} and Remark~\ref{Rem:est:beta:W}, applying also Fubini's Theorem, that
\begin{equation*}
 \begin{split}
  \abs{K_2}&\leq C\eps\fnorm{1}{\dmu} \int_{0}^{\infty}\int_{0}^{\xi}x^2\ee^{-px}\left(\frac{\xi}{x}\right)^{\delta}\left(\frac{\ee^{-x/2}}{x^{\alpha}}+\frac{\ee^{-\xi/2}}{\xi^{\alpha}}\right)\beta_W\left(\xi;\mu_1\right)\ee^{-\xi}\mu_1\left(\xi\right)\dx\dxi\\
  &\leq C\eps\fnorm{1}{\dmu}\int_{0}^{\infty}x^2\ee^{-px}\int_{x}^{\infty}\left(\frac{\xi}{x}\right)^{\delta}\left(\frac{\ee^{-x/2}}{x^{\alpha}}+\frac{\ee^{-\xi/2}}{\xi^{\alpha}}\right)\left(\xi^{-\alpha}+\xi^{\alpha}\right)\ee^{-\xi}\mu_1\left(\xi\right)\dxi\dx\\
  &\leq 2C\eps\fnorm{1}{\dmu}\int_{0}^{\infty}x^2\ee^{-px}\frac{\ee^{-x/2}}{x^{\alpha+\delta}}\int_{x}^{\infty}\xi^{\delta}\left(\xi^{-\alpha}+\xi^{\alpha}\right)\ee^{-\xi}\mu_1\left(\xi\right)\dxi\dx,
 \end{split}
\end{equation*}
where we used $\xi^{-\alpha}\ee^{-\xi/2}\leq x^{-\alpha}\ee^{-x/2}$ for $\xi\geq x$ in the last step. Estimating the remaining integrals using also the a-priori estimates~\eqref{fdecay} and~\eqref{S2E3} we get
\begin{equation*}
 \begin{split}
  \abs{K_2}&\leq C\eps\fnorm{1}{\dmu}\int_{0}^{\infty}x^{2-\alpha-\delta}\ee^{-\left(p+1/2\right)x}\dx=C\eps\frac{\fnorm{1}{\dmu}}{\left(p+1/2\right)^{3-\alpha-\delta}}\int_{0}^{\infty}x^{2-\alpha-\delta}\ee^{-x}\dx\\
  &\leq C\eps\fnorm{1}{\dmu}\frac{\left(1+p\right)^{1-\theta}}{p^3},
 \end{split}
\end{equation*}
if we choose $\delta$ sufficiently small. It remains to estimate $J_2$ and we similarly get from~\eqref{eq:J2:1}, using also that $K\left(y,z\right)$ can be estimated by $\left(\left(y/z\right)^{\alpha}+\left(z/y\right)^{\alpha}\right)$, that
\begin{equation*}
 \begin{split}
  \abs{J_2}&\leq C\eps\fnorm{1}{\dmu}\int_{0}^{\infty}\int_{0}^{\infty}\int_{0}^{y+z}\frac{\ee^{-\left(y+z\right)}}{\left(y+z\right)^2}x^2\ee^{-px}\left(\frac{y+z}{x}\right)^{\delta}\cdot\\
  &\qquad\qquad\qquad\qquad\qquad\qquad\qquad\cdot\left(\frac{\ee^{-\frac{y+z}{2}}}{\left(y+z\right)^{\alpha}}+\frac{\ee^{-\frac{x}{2}}}{x^{\alpha}}\right)K\left(y,z\right)y\mu_1\left(y\right)\mu_1\left(z\right)\dx\dz\dy\\
  &\leq C\eps\fnorm{1}{\dmu}\int_{0}^{\infty}x^2\ee^{-px}\frac{\ee^{-x/2}}{x^{\alpha+\delta}}\int_{0}^{\infty}\int_{\max\left\{0,x-y\right\}}^{\infty}\frac{\ee^{-\left(y+z\right)}\mu_1\left(y\right)\mu_2\left(z\right)}{\left(y+z\right)^{1-\delta}}\left(\Bigl(\frac{y}{z}\Bigr)^{\alpha}+\Bigl(\frac{z}{y}\Bigr)^{\alpha}\right)\dz\dy\dx.
 \end{split}
\end{equation*}
Considering first the two inner integrals and using the symmetry,~\eqref{fdecay} and~\eqref{S2E3} it follows
\begin{equation*}
 \begin{split}
  &\quad\int_{0}^{\infty}\int_{\max\left\{0,x-y\right\}}^{\infty}\frac{\ee^{-\left(y+z\right)}\mu_1\left(y\right)\mu_2\left(y\right)}{\left(y+z\right)^{1-\delta}}\left(\Bigl(\frac{y}{z}\Bigr)^{\alpha}+\Bigl(\frac{z}{y}\Bigr)^{\alpha}\right)\dz\dy\\
  &=2\int_{0}^{\infty}\int_{0}^{\infty}\frac{\ee^{-\left(y+z\right)}\mu_1\left(y\right)\mu_2\left(y\right)}{\left(y+z\right)^{1-\delta}}\left(\frac{z}{y}\right)^{\alpha}\dz\dy\leq 2\int_{0}^{\infty}\frac{\ee^{-y}\mu_1\left(y\right)}{y^{\alpha}}\dy\int_{0}^{\infty}\ee^{-z}\mu_1\left(z\right)z^{\alpha+\delta-1}\dz\leq C.
 \end{split}
\end{equation*}
Using this we obtain for $J_2$ in the same way as for $K_2$ that
\begin{equation*}
 \begin{split}
  \abs{J_2}\leq C\eps\fnorm{1}{\dmu}\int_{0}^{\infty}x^{2-\alpha-\delta}\ee^{-\left(p+1/2\right)x}\dx\leq C\eps\fnorm{1}{\dmu}\frac{\left(1+p\right)^{1-\theta}}{p^3},
 \end{split}
\end{equation*}
which proves~\eqref{eq:bl:2}.

\subsection{Estimate for $K_3$}

We now continue by estimating $K_3$ and we will show
\begin{equation}\label{eq:bl:3}
 \eps\snorm{2}{K_3}\leq C\eps\fnorm{1}{\mu_1-\mu_2}.
\end{equation}
 Recalling that $K_3$ is given by
\begin{equation*}
  K_3=\int_{0}^{\infty}\int_{0}^{\xi}x^2\ee^{-px}\left(\frac{x}{\xi}\right)^{\kappa_2}\frac{\ee^{\Phi_2\left(\xi\right)}}{\ee^{\Phi_2\left(x\right)}}\left(\beta_W\left(\xi;\mu_1\right)-\beta_W\left(\xi;\mu_2\right)\right)\left(\frac{1-\ee^{-\xi}}{\xi}\right)\ee^{-\xi}\mu_1\left(\xi\right)\dx\dxi
\end{equation*}
we first derive an estimate for the difference $\beta_W\left(\xi;\mu_1\right)-\beta_W\left(\xi;\mu_2\right)$ using Proposition~\ref{P.repkernels}. We have
\begin{equation}\label{eq:rep:beta:W}
 \begin{split}
  \beta_W\left(y;\mu_j\right)&=\int_{0}^{\infty}\ee^{-\xi y}\int_{0}^{\infty}\Gamma\left(\xi,\eta\right)\int_{0}^{\infty}\left(y+z\right)\mu_j\left(z\right)\ee^{-\left(\eta+1\right)z}\dz\deta\dxi\\
  &=\int_{0}^{\infty}\ee^{-\xi y}\int_{0}^{\infty}\Gamma\left(\xi,\eta\right)\left(yU_j\left(\eta+1\right)-U_j'\left(\eta+1\right)\right)\deta\dxi.
 \end{split}
\end{equation}
Thus we obtain with Proposition~\ref{P.repkernels} that
\begin{equation*}
 \begin{split}
  &\quad \abs{\beta_W\left(y;\mu_1\right)-\beta_W\left(y;\mu_2\right)}\\
  &\leq \int_{0}^{\infty}\ee^{-\xi y}\int_{0}^{\infty}\abs{\Gamma\left(\xi,\eta\right)}\left(y\abs{U_1\left(\eta+1\right)-U_2\left(\eta+1\right)}+\abs{U_1'\left(\eta+1\right)-U_2'\left(\eta+1\right)}\right)\deta\dxi\\
  &\leq C\int_{0}^{\infty}\ee^{-\xi y}\int_{0}^{\infty}\abs{\Gamma\left(\xi,\eta\right)}\left(\fnorm{0}{\dmu}\frac{y}{\left(\eta+1\right)^{\theta}}+\snorm{1}{\dmu}\frac{1}{\left(\eta+1\right)^{1+\theta}}\right)\deta\dxi\\
  &\leq C\fnorm{1}{\dmu}\left(1+y\right)\int_{0}^{\infty}\ee^{-\xi y}\left(\int_{0}^{\infty}\frac{1}{\left(\xi+\eta\right)^{1-\alpha}}\left(\frac{1}{\xi^{\alpha}}+\frac{1}{\eta^{\alpha}}\right)\frac{1}{\eta^{\theta}}\deta+\frac{1}{\xi^{\theta}}\right)\dxi.
 \end{split}
\end{equation*}
Changing variables, we then get
\begin{equation*}
 \begin{split}
   \abs{\beta_W\left(y;\mu_1\right)-\beta_W\left(y;\mu_2\right)}&\leq C\fnorm{1}{\dmu}\left(1+y\right)\int_{0}^{\infty}\ee^{-\xi y}\xi^{-\theta}\left(\int_{0}^{\infty}\frac{1}{\left(1+\eta\right)^{1-\alpha}}\left(1+\frac{1}{\eta^{\alpha}}\right)\frac{1}{\eta^{\theta}}\deta+1\right)\dxi\\
  &\leq C\fnorm{1}{\dmu}\left(1+y\right)y^{\theta-1}\int_{0}^{\infty}\ee^{-\xi}\xi^{-\theta}\dxi\leq C\fnorm{1}{\dmu}y^{\theta-1}\left(1+y\right).
 \end{split}
\end{equation*}
Using this in the expression for $K_3$ together with similar arguments as before, such as monotonicity for $\Phi_j$ and smallness of $\kappa_j$, we can estimate
\begin{equation*}
 \begin{split}
  \abs{K_3}&\leq C\fnorm{1}{\dmu}\int_{0}^{\infty}\int_{0}^{\xi}x^2\ee^{-px}\left(\frac{\xi}{x}\right)^{\delta}\xi^{\theta-1}\left(1+\xi\right)\ee^{-\xi}\mu_1\left(\xi\right)\dx\dxi\\
  &\leq C\fnorm{1}{\dmu}\int_{0}^{\infty}x^{2-\delta}\ee^{-px}\int_{x}^{\infty}\xi^{\theta+\delta-1}\left(1+\xi\right)\ee^{-\xi}\mu_1\left(\xi\right)\dxi\dx\leq C\fnorm{1}{\dmu}\frac{1}{p^{3-\delta}}\\
  &\leq C\fnorm{1}{\dmu}\frac{\left(1+p\right)^{1-\theta}}{p^3}
 \end{split}
\end{equation*}
which shows~\eqref{eq:bl:3}.

\subsection{Estimates for $K_4$ and $J_3$}

It now remains to estimate the terms $K_4$ and $J_3$ which is the most difficult part. We recall
\begin{equation*}
 K_4=\int_{0}^{\infty}\left(1-\ee^{-y}\right)\ee^{-y}\left(1+y\right)H\left(y,p\right)\beta_W\left(y;\mu_2\right)\left(\mu_1\left(y\right)-\mu_2\left(y\right)\right)\dy
\end{equation*}
and using also the symmetry
\begin{equation*}
 \begin{split}
  J_3&=\int_{0}^{\infty}\int_{0}^{\infty}\ee^{-\left(y+z\right)}H\left(y+z,p\right)\left(1+y+z\right)\frac{y}{y+z}K\left(y,z\right)\left(\mu_1\left(y\right)\mu_1\left(z\right)-\mu_2\left(y\right)\mu_2\left(z\right)\right)\dz\dy\\
  &=\frac{1}{2}\int_{0}^{\infty}\int_{0}^{\infty}\ee^{-\left(y+z\right)}H\left(y+z,p\right)\left(1+y+z\right)K\left(y,z\right)\left(\mu_1\left(y\right)\mu_2\left(z\right)-\mu_2\left(y\right)\mu_2\left(z\right)\right)\dz\dy.
 \end{split}
\end{equation*}
Before estimating these terms we first rewrite $H$ integrating by parts, i.e.\
\begin{equation*}
 \begin{split}
  H\left(y,p\right)&=-\frac{1}{p\Y\left(1+\Y\right)}\int_{0}^{\Y}\del_{x}\left(\ee^{-px}\right)\left(x^2\left(\frac{x}{\Y}\right)^{\kappa_2}\frac{\ee^{\Phi_2\left(\Y\right)}}{\ee^{\Phi_2\left(x\right)}}\right)\dx\\
  &=-\frac{\ee^{-p\Y}}{p}\frac{\Y}{1+\Y}+H_{0,1}\left(\Y,p\right)+H_{0,2}\left(\Y,p\right),
 \end{split}
\end{equation*}
where we define
\begin{equation}\label{eq:def:H0}
 \begin{split}
  H_{0,1}\left(\Y,p\right)&\vcc=\frac{2+\kappa_2}{p\Y\left(1+\Y\right)}\int_{0}^{\Y}\ee^{-px}x\left(\frac{x}{\Y}\right)^{\kappa_2}\frac{\ee^{\Phi_2\left(\Y\right)}}{\ee^{\Phi_2\left(x\right)}}\dx,\\
  H_{0,2}\left(\Y,p\right)&\vcc=-\frac{1}{p\Y\left(1+\Y\right)}\int_{0}^{\Y}\ee^{-px}\Phi_2'\left(x\right)x^2\left(\frac{x}{\Y}\right)^{\kappa_2}\frac{\ee^{\Phi_2\left(\Y\right)}}{\ee^{\Phi_2\left(x\right)}}\dx,\\
  H_{0}\left(\Y,p\right)&\vcc=H_{0,1}\left(\Y,p\right)+H_{0,2}\left(\Y,p\right).
 \end{split}
\end{equation}
Using this splitting of $H$ we obtain a corresponding splitting of $K_4$ and $J_3$, i.e.\
\begin{equation*}
 K_4=K_{4,1}+K_{4,0}\qquad \text{and}\qquad J_3=J_{3,1}+J_{3,0}
\end{equation*}
with
\begin{align*}
 K_{4,1}&=-\frac{1}{p}\int_{0}^{\infty}\left(1-\ee^{-y}\right)\ee^{-\left(1+p\right)y}y\beta_W\left(y;\mu_2\right)\left(\mu_1\left(y\right)-\mu_2\left(y\right)\right)\dy,\\
 K_{4,0}&=\int_{0}^{\infty}\left(1-\ee^{-y}\right)\ee^{-y}\left(1+y\right)H_0\left(y,p\right)\beta_W\left(y;\mu_2\right)\left(\mu_1\left(y\right)-\mu_2\left(y\right)\right)\dy,\\
 J_{3,1}&=-\frac{1}{2p}\int_{0}^{\infty}\int_{0}^{\infty}\ee^{-\left(y+z\right)\left(p+1\right)}\left(y+z\right)K\left(y,z\right)\left(\mu_1\left(y\right)\mu_1\left(z\right)-\mu_2\left(y\right)\mu_2\left(z\right)\right)\dz\dy,\\
 J_{3,0}&=\frac{1}{2}\int_{0}^{\infty}\int_{0}^{\infty}\ee^{-\left(y+z\right)}H_{0}\left(y+z,p\right)\left(1+y+z\right)K\left(y,z\right)\left(\mu_1\left(y\right)\mu_1\left(z\right)-\mu_2\left(y\right)\mu_2\left(z\right)\right)\dz\dy.
\end{align*}

\subsection{Estimates of $K_{4,1}$ and $J_{3,1}$}

We first consider the terms $K_{4,1}$ and $J_{3,1}$ which are easier to estimate. Recalling from~\eqref{eq:rep:beta:W} that we have
\begin{equation*}
 \begin{split}
  \beta_W\left(y;\mu_2\right)=\int_{0}^{\infty}\int_{0}^{\infty}\Gamma\left(\xi,\eta\right)\ee^{-\xi y}\left(yU_2\left(\eta+1\right)-U_2'\left(\eta+1\right)\right)\deta\dxi
 \end{split}
\end{equation*}
as well as $m=m_1-m_2=\mu_1-\mu_2=\dmu$, we can rewrite $K_{4,1}$ by
\begin{equation*}
 \begin{split}
  K_{4,1}&=-\frac{1}{p}\int_{0}^{\infty}\int_{0}^{\infty}\Gamma\left(\xi,\eta\right)\int_{0}^{\infty}\left(yU_2\left(\eta+1\right)-U_2'\left(\eta+1\right)\right)\left(1-\ee^{-y}\right)\ee^{-\left(1+p+\xi\right)y}ym\left(y\right)\dy\dxi\deta\\
  &=-\frac{1}{p}\int_{0}^{\infty}\int_{0}^{\infty}\Gamma\left(\xi,\eta\right)U_2\left(\eta+1\right)\left(M''\left(1+p+\xi\right)-M''\left(2+\xi+p\right)\right)\dxi\deta\\
  &\qquad +\frac{1}{p}\int_{0}^{\infty}\int_{0}^{\infty}\Gamma\left(\xi,\eta\right)U_{2}'\left(\eta+1\right)\left(M'\left(1+p+\xi\right)-M'\left(2+p+\xi\right)\right)\dxi\deta.
 \end{split}
\end{equation*}
Recalling additionally that $K=2+\eps W$ and the representation formula for $W$ from Proposition~\ref{P.repkernels} we can rewrite $J_{3,1}$ as
\begin{equation*}
 \begin{split}
  &\quad J_{3,1}\\
  &=-\frac{1}{p}\int_{0}^{\infty}\int_{0}^{\infty}\ee^{-\left(1+p\right)\left(y+z\right)}\left(y+z\right)\left(\mu_1\left(y\right)\mu_1\left(z\right)-\mu_2\left(y\right)\mu_2\left(z\right)\right)\dz\dy\\
  &\quad -\frac{\eps}{2p}\int_{0}^{\infty}\int_{0}^{\infty}\ee^{-\left(1+p\right)\left(y+z\right)}\left(y+z\right)W\left(y,z\right)\left(\mu_1\left(y\right)\mu_1\left(z\right)-\mu_2\left(y\right)\mu_2\left(z\right)\right)\dz\dy\\
  &=\frac{1}{p}\frac{\dd}{\dd p}\left(\left(U_1\left(p+1\right)\right)^2-\left(U_2\left(p+1\right)\right)^2\right)\\
  &\quad-\frac{\eps}{2p}\int_{0}^{\infty}\int_{0}^{\infty}\Gamma\left(\xi,\eta\right)\int_{0}^{\infty}\int_{0}^{\infty}\ee^{-\left(p+1\right)\left(y+z\right)}\left(y+z\right)^2\ee^{-\xi y-\eta z}\left(\mu_1\left(y\right)\mu_1\left(z\right)-\mu_2\left(y\right)\mu_2\left(z\right)\right)\dz\dy\deta\dxi.
 \end{split}
\end{equation*}
Rearranging this yields
\begin{equation*}
 \begin{split}
  &\quad J_{3,1}\\
  &=\frac{1}{p}\frac{\dd}{\dd p}\left(\left(U_1\left(p+1\right)\right)^2-\left(U_2\left(p+1\right)\right)^2\right)\\
  &\quad -\frac{\eps}{2p}\int_{0}^{\infty}\int_{0}^{\infty}\Gamma\left(\xi,\eta\right)\left(\del_{\xi}+\del_{\eta}\right)^2\left(U_1\left(\xi+p+1\right)U_1\left(\eta+p+1\right)-U_2\left(\xi+p+1\right)U_2\left(\eta+p+1\right)\right)\deta\dxi\\
  &=\vcc J_{3,1,1}-J_{3,1,2}.
 \end{split}
\end{equation*}

\subsubsection{Estimate for $J_{3,1}$}

We start by estimating $J_{3,1}$ and we will show
\begin{equation}\label{eq:bl:4:1}
 \snorm{2}{J_{3,1}}\leq \delta_{0}\fnorm{1}{\dmu}+C_{\delta_0}\fnorm{1}{\left(1-\zeta\right)\dmu}.
\end{equation}
We consider $J_{3,1,1}$ and $J_{3,1,2}$ separately and we begin with $J_{3,1,1}$ for which we need some splitting argument similarly to the estimate of $J_1$. For $n\in\N$ we insert 
\begin{equation*}
 1=\ee^{-yn}\ee^{-zn}+\left(1-\ee^{-yn}\right)\ee^{-zn}+\ee^{-yn}\left(1-\ee^{-zn}\right)+\left(1-\ee^{-yn}\right)\left(1-\ee^{-zn}\right)
\end{equation*}
into the integral expression for $J_{3,1,1}$ to get after rearranging
\begin{equation*}
 \begin{split}
  J_{3,1,1}&=\frac{1}{p}\del_{p}\left[\int_{0}^{\infty}\ee^{-\left(p+n+1\right)y}\dmu\left(y\right)\dy\int_{0}^{\infty}\ee^{-\left(p+n+1\right)z}\left(\mu_1\left(z\right)+\mu_2\left(z\right)\right)\dz\right.\\
  &\quad\left.+2\int_{0}^{\infty}\left(1-\ee^{-ny}\right)\ee^{-\left(p+1\right)y}\mu_1\left(y\right)\dy\int_{0}^{\infty}\ee^{-\left(p+1+n\right)z}\dmu\left(z\right)\dz\right.\\
  &\quad\left.+2\int_{0}^{\infty}\left(1-\ee^{-ny}\right)\ee^{-\left(p+1\right)y}\dmu\left(y\right)\dy\int_{0}^{\infty}\ee^{-\left(p+1+n\right)z}\mu_2\left(z\right)\dz\right.\\
  &\quad+\left. \int_{0}^{\infty}\left(1-\ee^{-ny}\right)\ee^{-\left(p+1\right)y}\dmu\left(y\right)\dy\int_{0}^{\infty}\left(1-\ee^{-nz}\right)\ee^{-\left(p+1\right)z}\left(\mu_1\left(z\right)+\mu_{2}\left(z\right)\right)\dz\right].
 \end{split}
\end{equation*}
Thus it follows
\begin{multline*}
  \quad\abs{J_{3,1,1}}\\
  \shoveleft{\leq \frac{1}{p}\left[\frac{\snorm{1}{\dmu}\left(\fnorm{0}{\mu_1}+\fnorm{0}{\mu_2}\right)}{\left(p+n+1\right)^{1+2\theta}}+\frac{\fnorm{0}{\dmu}\left(\snorm{1}{\mu_1}+\snorm{1}{\mu_2}\right)}{\left(p+n+1\right)^{1+2\theta}}\right]}\\
  \shoveleft{+\frac{2}{p}\left[\frac{\snorm{1}{\left(1-\ee^{-n\cdot}\right)\mu_1}\fnorm{0}{\dmu}}{\left(p+1\right)^{1+\theta}\left(p+n+1\right)^{\theta}}+\frac{\fnorm{0}{\left(1-\ee^{-n\cdot}\right)\mu_1}\snorm{1}{\dmu}}{\left(p+1\right)^{\theta}\left(p+n+1\right)^{1+\theta}}+ \right.}\\
  \shoveright{\left.+\frac{\snorm{1}{\left(1-\ee^{-n\cdot}\right)\dmu}\fnorm{0}{\mu_2}}{\left(p+1\right)^{1+\theta}\left(p+n+1\right)^{\theta}}+\frac{\fnorm{0}{\left(1-\ee^{-n\cdot}\right)\dmu}\snorm{1}{\mu_2}}{\left(p+1\right)^{\theta}\left(p+n+1\right)^{1+\theta}}\right]}\\
  \shoveleft{+\frac{1}{p\left(p+1\right)^{1+2\theta}}\left[\snorm{1}{\left(1-\ee^{-n\cdot}\right)\dmu}\left(\fnorm{0}{\left(1-\ee^{-n\cdot}\right)\mu_1}+\fnorm{0}{\left(1-\ee^{-n\cdot}\right)\mu_2}\right)\right.}\\
  \left.+\fnorm{0}{\left(1-\ee^{-n\cdot}\right)\dmu}\left(\snorm{1}{\left(1-\ee^{-n\cdot}\right)\mu_1}+\snorm{1}{\left(1-\ee^{-n\cdot}\right)\mu_2}\right)\right].
\end{multline*}
From Lemmas~\ref{Lem:elem:est:norm} and~\ref{Lem:norm:est:split} together with the uniform boundedness of $\fnorm{2}{\mu_1}$ and $\fnorm{2}{\mu_2}$ shown in Proposition~\ref{Prop:uniform:bound:muk} it follows
\begin{equation*}
 \begin{split}
  &\quad\abs{J_{3,1,1}}\\
  &\leq C\fnorm{1}{\dmu}\left(\fnorm{1}{\mu_1}+\fnorm{1}{\mu_2}+\fnorm{1}{\left(1-\ee^{-n\cdot}\right)\mu_1}\right)\frac{1}{n^{\theta}}\frac{1}{p\left(1+p\right)^{1+\theta}}\\
  &\qquad+C\fnorm{1}{\left(1-\ee^{-n\cdot}\right)\dmu}\left(\fnorm{1}{\mu_2}+\fnorm{1}{\left(1-\ee^{-n\cdot}\right)\mu_1}+\fnorm{1}{\left(1-\ee^{-n\cdot}\right)\mu_2}\right)\frac{1}{p\left(1+p\right)^{1+2\theta}}\\
  &\leq C\fnorm{1}{\dmu}\frac{1}{n^{\theta}}\frac{\left(1+p\right)^{1-\theta}}{p^3}+C\left(n\right)\fnorm{1}{\left(1-\zeta\right)\dmu}\frac{\left(1+p\right)^{1-\theta}}{p^3}.
 \end{split}
\end{equation*}
 By choosing $n$ sufficiently large we can then obtain
 \begin{equation}\label{eq:J:3:1:1}
  \abs{J_{3,1,1}}\leq \left(\frac{\delta_0}{2}\fnorm{1}{\dmu}+\widetilde{C}_{\delta_0}\fnorm{1}{\left(1-\zeta\right)\dmu}\right)\frac{\left(1+p\right)^{1-\theta}}{p^3}.
 \end{equation}
Next we estimate $J_{3,1,2}$ for which no splitting is necessary as this term already contains a factor $\eps$. We have
 \begin{multline*}
\quad\abs{J_{3,1,2}}\\
  \shoveleft{=\frac{\eps}{2p}\abs{\int_{0}^{\infty}\int_{0}^{\infty}\Gamma\left(\xi,\eta\right)\left(\del_{\xi}+\del_{\eta}\right)^2\left(U_1\left(\xi+p+1\right)U_1\left(\eta+p+1\right)-U_2\left(\xi+p+1\right)U_2\left(\eta+p+1\right)\right)\deta\dxi}}\\
  \shoveleft=\frac{\eps}{2p}\abs{\int_{0}^{\infty}\int_{0}^{\infty}\Gamma\left(\xi,\eta\right)\left(U_1''\left(\xi+p+1\right)U_1\left(\eta+p+1\right)-U_2''\left(\xi+p+1\right)U_2\left(\eta+p+1\right)\right.\right.\\
  \left.\left.+2U_1'\left(\xi+p+1\right)U_1'\left(\eta+p+1\right)-2U_2'\left(\xi+p+1\right)U_2'\left(\eta+p+1\right)\right.\right.\\
  \left.\left.+U_1\left(\xi+p+1\right)U_1''\left(\eta+p+1\right)-U_2\left(\xi+p+1\right)U_2''\left(\eta+p+1\right)\right)\vphantom{\int_{0}^{\infty}}\deta\dxi}.  
 \end{multline*}
Inserting additional terms and rearranging gives
\begin{multline*}
  \quad \abs{J_{3,1,2}}=\frac{\eps}{2p}\abs{\int_{0}^{\infty}\int_{0}^{\infty}\Gamma\left(\xi,\eta\right)\left(\left(U_1''\left(\xi+p+1\right)-U_2''\left(\xi+p+1\right)\right)U_1\left(\eta+p+1\right)\right.\right.\\
  \left.\left.+U_2''\left(\xi+p+1\right)\left(U_1\left(\eta+p+1\right)-U_2\left(\eta+p+1\right)\right)\right.\right.\\
  \left.\left.+2\left(U_1'\left(\xi+p+1\right)-U_2'\left(\xi+p+1\right)\right)U_1'\left(\eta+p+1\right)\right.\right.\\
  \left.\left.+2U_2'\left(\xi+p+1\right)\left(U_1'\left(\eta+p+1\right)-U_2'\left(\eta+p+1\right)\right)\right.\right.\\
  \left.\left.+\left(U_1\left(\xi+p+1\right)-U_2\left(\xi+p+1\right)\right)U_1''\left(\eta+p+1\right)\right.\right.\\
  \left.\left.+U_2\left(\xi+p+1\right)\left(U_1''\left(\eta+p+1\right)-U_2''\left(\eta+p+1\right)\right)\right)\vphantom{\int_{0}^{\infty}}}
\end{multline*}
leading to the following estimate
\begin{multline*}
 \abs{J_{3,1,2}}\leq\frac{\eps}{2p}\int_{0}^{\infty}\int_{0}^{\infty}\abs{\Gamma\left(\xi,\eta\right)}\left(\left(\snorm{2}{\dmu}\fnorm{0}{\mu_1}+\snorm{2}{\mu_2}\fnorm{0}{\dmu}\right)\frac{\left(2+\xi+p\right)^{1-\theta}}{\left(1+\xi+p\right)^{3}}\frac{\left(2+\eta+p\right)^{1-\theta}}{\eta+p+1}\right.\\
  \left.+2\left(\snorm{2}{\dmu}\snorm{1}{\mu_1}+\snorm{1}{\mu_2}\snorm{1}{\dmu}\right)\frac{\left(2+\xi+p\right)^{1-\theta}}{\left(1+\xi+p\right)^2}\frac{\left(2+\eta+p\right)^{1-\theta}}{\left(1+\eta+p\right)^2}\right.\\
  \left.+\left(\fnorm{0}{\dmu}\snorm{2}{\mu_1}+\fnorm{0}{\mu_2}\snorm{2}{\dmu}\right)\frac{\left(2+\xi+p\right)^{1-\theta}}{\xi+p+1}\frac{\left(2+\eta+p\right)^{1-\theta}}{\left(\eta+p+1\right)^3}\right)\deta\dxi.
\end{multline*}
Estimating by the highest norms and the most dominant terms we get
\begin{equation*}
 \begin{split}
  &\quad\abs{J_{3,1,2}}\\
  &\leq C\eps\frac{\fnorm{2}{\dmu}\left(\fnorm{2}{\mu_1}+\fnorm{2}{\mu_2}\right)}{p}\int_{0}^{\infty}\int_{0}^{\infty}\abs{\Gamma\left(\xi,\eta\right)}\cdot\\
    &\cdot\left(\frac{1}{\left(1+\xi+p\right)^{2+\theta}\left(1+\eta+p\right)^{\theta}}+\frac{1}{\left(1+\xi+p\right)^{1+\theta}\left(1+\eta+p\right)^{1+\theta}}+\frac{1}{\left(1+\xi+p\right)^{\theta}\left(1+\eta+p\right)^{2+\theta}}\right)\deta\dxi.
 \end{split}
\end{equation*}
Using the homogeneity of $\Gamma$ we further obtain by changing variables and taking into account the uniform boundedness of $\fnorm{2}{\mu_1}$ and $\fnorm{2}{\mu_2}$ that
\begin{equation*}
 \begin{split}
  &\quad \abs{J_{3,1,2}}\\
   & \leq \frac{C\eps\fnorm{2}{\dmu}}{p\left(1+p\right)^{1+2\theta}}\int_{0}^{\infty}\int_{0}^{\infty}\left(\frac{\abs{\Gamma\left(\xi,\eta\right)}}{\left(1+\xi\right)^{2+\theta}\left(1+\eta\right)^{\theta}}+\frac{\abs{\Gamma\left(\xi,\eta\right)}}{\left(1+\xi\right)^{1+\theta}\left(1+\eta\right)^{1+\theta}}+\frac{\abs{\Gamma\left(\xi,\eta\right)}}{\left(1+\xi\right)^{\theta}\left(1+\eta\right)^{2+\theta}}\right)\deta\dxi\\
  &\leq C\eps\fnorm{2}{\dmu}\frac{\left(1+p\right)^{1-\theta}}{p^3},
 \end{split}
\end{equation*}
where last integral is bounded by a constant due to Lemma~\ref{Lem:est:Gamma:int}. Together with~\eqref{eq:J:3:1:1} this shows~\eqref{eq:bl:4:1} for $\eps$ sufficiently small.

\subsubsection{Estimate for $K_{4,1}$}

We next consider $K_{4,1}$ and we will show
\begin{equation}\label{eq:bl:4:2}
 \eps\snorm{2}{K_{4,1}}\leq C\eps\fnorm{2}{\dmu}.
\end{equation}
 Recall
\begin{equation*}
 \begin{split}
  K_{4,1}&=-\frac{1}{p}\int_{0}^{\infty}\int_{0}^{\infty}\Gamma\left(\xi,\eta\right)U_2\left(\eta+1\right)\left(M''\left(1+p+\xi\right)-M''\left(2+\xi+p\right)\right)\dxi\deta\\
  &\qquad +\frac{1}{p}\int_{0}^{\infty}\int_{0}^{\infty}\Gamma\left(\xi,\eta\right)U_{2}'\left(\eta+1\right)\left(M'\left(1+p+\xi\right)-M'\left(2+p+\xi\right)\right)\dxi\deta
 \end{split}
\end{equation*}
as well as
\begin{equation*}
 \abs{M'\left(p+\xi+2\right)-M'\left(\xi+p+1\right)}\leq C\snorm{2}{\dmu}\left(1+p+\xi\right)^{-2-\theta}.
\end{equation*}
Using this we obtain together with~\eqref{eq:Lambda} similarly as before that
\begin{equation*}
 \begin{split}
  \abs{K_{4,1}}&\leq \frac{C}{p}\fnorm{0}{\mu_2}\snorm{2}{\dmu}\int_{0}^{\infty}\int_{0}^{\infty}\frac{\abs{\Gamma\left(\xi,\eta\right)}}{\left(\eta+1\right)^{\theta}\left(1+p+\xi\right)^{2+\theta}}\deta\dxi\\
  &\quad +\frac{C}{p}\snorm{1}{\mu_2}\snorm{2}{\dmu}\int_{0}^{\infty}\int_{0}^{\infty}\frac{\abs{\Gamma\left(\xi,\eta\right)}}{\left(\eta+1\right)^{1+\theta}\left(1+\xi+p\right)^{2+\theta}}\deta\dxi\\
  &\leq \frac{C}{p}\fnorm{1}{\mu_2}\snorm{2}{\dmu}\int_{0}^{\infty}\int_{0}^{\infty}\frac{\abs{\Gamma\left(\xi,\eta\right)}}{\eta^{\theta}\left(1+\xi+p\right)^{2+\theta}}\deta\dxi.
 \end{split}
\end{equation*}
Changing variables we find together with Lemma~\ref{Lem:est:Gamma:int:2} and Proposition~\ref{Prop:uniform:bound:muk} that
\begin{equation*}
 \begin{split}
  \abs{K_{4,1}}&\leq \frac{C\fnorm{1}{\mu_2}\snorm{2}{\dmu}}{p\left(1+p\right)^{1+2\theta}}\int_{0}^{\infty}\int_{0}^{\infty}\frac{\abs{\Gamma\left(\xi,\eta\right)}}{\eta^{\theta}\left(1+\xi\right)^{2+\theta}}\deta\dxi\leq C\snorm{2}{\dmu}\frac{\left(1+p\right)^{1-\theta}}{p^3}.
 \end{split}
\end{equation*}
This shows~\eqref{eq:bl:4:2}

\subsection{Estimates for $K_{4,0}$ and $J_{3,0}$}

Thus it remains to estimate the terms $K_{4,0}$ and $J_{3,0}$ which is the most difficult part. We recall
\begin{align*}
 K_{4,0}&=\int_{0}^{\infty}\left(1-\ee^{-y}\right)\ee^{-y}\left(1+y\right)H_0\left(y,p\right)\beta_W\left(y;\mu_2\right)\left(\mu_1\left(y\right)-\mu_2\left(y\right)\right)\dy\\
 J_{3,0}&=\frac{1}{2}\int_{0}^{\infty}\int_{0}^{\infty}\ee^{-\left(y+z\right)}H_{0}\left(y+z,p\right)\left(1+y+z\right)K\left(y,z\right)\left(\mu_1\left(y\right)\mu_1\left(z\right)-\mu_2\left(y\right)\mu_2\left(z\right)\right)\dz\dy.
\end{align*}
In order to get estimates for $K_{4,0}$ and $J_{3,0}$ we need to rewrite these expressions in terms of the Laplace-transforms of $\mu_1$ and $\mu_2$. For this we use the representation of $H_0$ given by
\begin{equation}\label{eq:def:Q0}
 H_{0}\left(\Y,p\right)=\int_{0}^{\infty}Q_0\left(\xi,p\right)\ee^{-\xi \Y}\dxi \quad \text{for all } \Y,p>0
\end{equation}
that is proven in Proposition~\ref{Prop:rep:H0}. Here $Q_0\left(\cdot,p\right)\in L^2\left(\R_{+}\right)$ for all $p>0$. We will also use the following estimate on $Q_0\left(\xi,p\right)$ that is shown in Proposition~\ref{Prop:Q0:estimate}. For some $\nu>0$ sufficiently small it holds
\begin{equation*}
   \int_{0}^{\infty}\frac{\abs{Q_0\left(\xi,p\right)}}{\left(1+\xi\right)^{2\theta-\nu}}\dxi\leq C\frac{\left(1+p\right)^{1-\theta}}{p^3} \quad \text{for all } p>0.
\end{equation*}

\subsection{Estimate for $J_{3,0}$}

We will show
\begin{equation}\label{eq:bl:5:1}
 \snorm{2}{J_{3,0}}\leq \delta_{0}\fnorm{2}{\dmu}+C_{\delta_{0}}\fnorm{1}{\left(1-\zeta\right)\dmu}.
\end{equation}
From~\eqref{eq:def:Q0} and $K=2+\eps W$ we can rewrite
\begin{equation*}
 \begin{split}
  J_{3,0}&=\frac{1}{2}\int_{0}^{\infty}\int_{0}^{\infty}\ee^{-\left(y+z\right)}H_{0}\left(y+z,p\right)\left(1+y+z\right)\left(2+\eps W\left(y,z\right)\right)\left(\mu_1\left(y\right)\mu_1\left(z\right)-\mu_2\left(y\right)\mu_2\left(z\right)\right)\dz\dy\\
  &=\vcc J_{3,0,1}+J_{3,0,2}
 \end{split}
\end{equation*}
with
\begin{align*}
 J_{3,0,1}&=\int_{0}^{\infty}Q_0\left(\xi,p\right)\int_{0}^{\infty}\int_{0}^{\infty}\ee^{-\left(\xi+1\right)\left(y+z\right)}\left(1+y+z\right)\left(\mu_1\left(y\right)\mu_1\left(z\right)-\mu_2\left(y\right)\mu_{2}\left(z\right)\right)\dz\dy\dxi,\\
 J_{3,0,2}&=\frac{\eps}{2}\int_{0}^{\infty}Q_0\left(\xi,p\right)\int_{0}^{\infty}\int_{0}^{\infty}\ee^{-\left(\xi+1\right)\left(y+z\right)}\left(1+y+z\right)W\left(y,z\right)\left(\mu_1\left(y\right)\mu_1\left(z\right)-\mu_2\left(y\right)\mu_2\left(z\right)\right)\dz\dy\dxi.
\end{align*}
We first consider $J_{3,0,1}$ and insert again the splitting
\begin{equation*}
 1=\ee^{-yn}\ee^{-zn}+\left(1-\ee^{-yn}\right)\ee^{-zn}+\ee^{-yn}\left(1-\ee^{-zn}\right)+\left(1-\ee^{-yn}\right)\left(1-\ee^{-zn}\right)
\end{equation*}
to obtain, by rearranging and inserting additional terms, that
\begin{multline*}
  \quad J_{3,0,1}\\
  \shoveleft{=\int_{0}^{\infty}Q_0\left(\xi,p\right)\left(1-\del_{\xi}\right)\int_{0}^{\infty}\int_{0}^{\infty}\ee^{-\left(\xi+1\right)\left(y+z\right)}\left(\mu_1\left(y\right)\mu_1\left(z\right)-\mu_2\left(y\right)\mu_2\left(z\right)\right)\dz\dy\dxi}\\
 \shoveleft{=\int_{0}^{\infty}Q_0\left(\xi,p\right)\left(1-\del_{\xi}\right)\Delta U\left(\xi+1+n\right)\left(U_1\left(\xi+1+n\right)+U_2\left(\xi+1+n\right)\right)\dxi}\\
 \shoveleft{\quad+2\int_{0}^{\infty}Q_0\left(\xi,p\right)\left(1-\del_{\xi}\right)\int_{0}^{\infty}\ee^{-\left(\xi+1\right)y}\left(1-\ee^{-yn}\right)\left[\mu_1\left(y\right)\left(U_1\left(\xi+1+n\right)-U_2\left(\xi+1+n\right)\right)\phantom{\dy\dxi}\right.}\\
         \shoveright{\left.+\left(\mu_1\left(y\right)-\mu_{2}\left(y\right)\right)U_2\left(\xi+1+n\right)\right]\dy\dxi}\\
 \shoveleft{\quad+\int_{0}^{\infty}Q_0\left(\xi,p\right)\left(1-\del_{\xi}\right)\cdot}\\
 \cdot\left(\int_{0}^{\infty}\left(1-\ee^{-yn}\right)\ee^{-\left(\xi+1\right)y}\dmu\left(y\right)\dy\int_{0}^{\infty}\left(1-\ee^{-zn}\right)\ee^{-\left(1+\xi\right)z}\left(\mu_1\left(z\right)+\mu_2\left(z\right)\right)\dz\right)\dxi.
\end{multline*}
From this we get the following estimate
\begin{equation*}
 \begin{split}
  &\quad \abs{J_{3,0,1}}\\
  &\begin{multlined}
    \leq C\int_{0}^{\infty}\abs{Q_0\left(\xi,p\right)}\left(\fnorm{0}{\dmu}\left(\fnorm{0}{\mu_1}+\fnorm{0}{\mu_2}\right)\frac{1}{\left(\xi+1+n\right)^{2\theta}}\right.\\
  \left.\phantom{\qquad\qquad}+\left(\snorm{1}{\dmu}\left(\fnorm{0}{\mu_1}+\fnorm{0}{\mu_2}\right)+\fnorm{0}{\dmu}\left(\snorm{1}{\mu_1}+\snorm{1}{\mu_2}\right)\right)\frac{1}{\left(1+\xi+n\right)^{1+2\theta}}\right)\dxi
   \end{multlined}\\
  &\quad\begin{multlined}
         +2C\int_{0}^{\infty}\frac{\abs{Q_0\left(\xi,p\right)}}{\left(\xi+n+1\right)^{\theta}\left(1+\xi\right)^{\theta}}\left(\left(\fnorm{0}{\left(1-\ee^{-n\cdot}\right)\mu_1}\fnorm{0}{\dmu}+\fnorm{0}{\left(1-\ee^{-n\cdot}\right)\dmu}\fnorm{0}{\mu_2}\right)\right.\\
  \left.+\left(\snorm{1}{\left(1-\ee^{-n\cdot}\right)\mu_1}\fnorm{0}{\dmu}+\snorm{1}{\left(1-\ee^{-n\cdot}\right)\dmu}\fnorm{0}{\mu_2}\right)\frac{1}{\xi+1}\right.\\
  \left.+\left(\fnorm{0}{\left(1-\ee^{-n\cdot}\right)\mu_1}\snorm{1}{\dmu}+\fnorm{0}{\left(1-\ee^{-n\cdot}\right)\dmu}\snorm{1}{\mu_2}\right)\frac{1}{\xi+n+1}\right)\dxi
        \end{multlined}\\
   &\quad\begin{multlined}
          +\int_{0}^{\infty}\abs{Q_0\left(\xi,p\right)}\left(\fnorm{0}{\left(1-\ee^{-n\cdot}\right)\dmu}\left(\fnorm{0}{\left(1-\ee^{-n\cdot}\right)\mu_1}+\fnorm{0}{\left(1-\ee^{-n\cdot}\right)\mu_2}\right)\frac{1}{\left(1+\xi\right)^{2\theta}}\phantom{\Big)\dxi}\right.\\
  \left.+\left(\snorm{1}{\left(1-\ee^{-n\cdot}\right)\dmu}\left(\fnorm{0}{\left(1-\ee^{-n\cdot}\right)\mu_1}+\fnorm{0}{\left(1-\ee^{-n\cdot}\right)\mu_2}\right)\right.\right.\\
  \left.\left.+\fnorm{0}{\left(1-\ee^{-n\cdot}\right)\dmu}\left(\snorm{1}{\left(1-\ee^{-n\cdot}\right)\mu_1}+\snorm{1}{\left(1-\ee^{-n\cdot}\right)\mu_2}\right)\right)\frac{1}{\left(1+\xi\right)^{1+2\theta}}\right)\dxi.
         \end{multlined}
 \end{split}
\end{equation*}
Estimating by the most dominant terms we further obtain
\begin{multline*}
   \abs{J_{3,0,1}}\leq C\fnorm{1}{\dmu}\left(\fnorm{1}{\mu_1}+\fnorm{1}{\mu_2}\right)\int_{0}^{\infty}\frac{\abs{Q_{0}\left(\xi,p\right)}}{\left(1+\xi\right)^{2\theta}}\dxi\\
  +C\fnorm{1}{\dmu}\fnorm{1}{\left(1-\ee^{-n\cdot}\right)\mu_1}\int_{0}^{\infty}\frac{\abs{Q_0\left(\xi,p\right)}}{\left(1+\xi+n\right)^{\theta}\left(1+\xi\right)^{\theta}}\dxi\\
  +C\fnorm{1}{\left(1-\ee^{-n\cdot}\right)\dmu}\fnorm{1}{\mu_2}\int_{0}^{\infty}\frac{\abs{Q_0\left(\xi,p\right)}}{\left(1+\xi\right)^{2\theta}}\dxi\\
  +C\fnorm{1}{\left(1-\ee^{-n\cdot}\right)\dmu}\left(\fnorm{1}{\left(1-\ee^{-n\cdot}\right)\mu_1}+\fnorm{1}{\left(1-\ee^{-n\cdot}\right)\mu_2}\right)\int_{0}^{\infty}\frac{\abs{Q_0\left(\xi,p\right)}}{\left(\xi+1\right)^{2\theta}}\dxi.
\end{multline*}
Choosing now $\nu>0$ sufficiently small and using Lemma~\ref{Lem:elem:est:norm}, Lemma~\ref{Lem:norm:est:split} and Proposition~\ref{Prop:uniform:bound:muk} we can further estimate
\begin{equation*}
 \begin{split}
  \abs{J_{3,0,1}}\leq \frac{C}{n^{\nu}}\fnorm{1}{\dmu}\int_{0}^{\infty}\frac{\abs{Q_0\left(\xi,p\right)}}{\left(1+\xi\right)^{2\theta-\nu}}\dxi+C\left(n\right)\fnorm{1}{\left(1-\zeta\right)\dmu}\int_{0}^{\infty}\frac{\abs{Q_0\left(\xi,p\right)}}{\left(1+\xi\right)^{2\theta}}\dxi.
 \end{split}
\end{equation*}
Thus for sufficiently large $n$ we obtain from Proposition~\ref{Prop:Q0:estimate}
\begin{equation}\label{eq:J301}
 \abs{J_{3,0,1}}\leq \left(\delta_{0}\fnorm{1}{\dmu}+C_{\delta_{0}}\fnorm{1}{\left(1-\zeta\right)\dmu}\right)\frac{\left(1+p\right)^{1-\theta}}{p^3}.
\end{equation}
We next consider the term $J_{3,0,2}$ and obtain together with Proposition~\ref{P.repkernels} that
\begin{multline*}
  \quad J_{3,0,2}\\
  \shoveleft{=\frac{\eps}{2}\int_{0}^{\infty}\int_{0}^{\infty}\Gamma\left(\sigma,\eta\right)\int_{0}^{\infty}Q_0\left(\xi,p\right)\int_{0}^{\infty}\int_{0}^{\infty}\ee^{-\left(\xi+1\right)\left(y+z\right)}\ee^{-\sigma y-\eta z}\left(1+y+z\right)\left(y+z\right)\cdot}\\
  \shoveright{\cdot\left(\mu_1\left(y\right)\mu_1\left(z\right)-\mu_2\left(y\right)\mu_2\left(z\right)\right)\dz\dy\dxi\deta\dsig}\\
  \shoveleft{=-\frac{\eps}{2}\int_{0}^{\infty}\int_{0}^{\infty}\Gamma\left(\sigma,\eta\right)\int_{0}^{\infty}Q_0\left(\xi,p\right)\left(1-\del_{\xi}\right)\del_{\xi}\int_{0}^{\infty}\int_{0}^{\infty}\ee^{-\left(\xi+1+\sigma\right)y}\ee^{-\left(\xi+\eta+1\right)z}\cdot}\\
  \shoveright{\cdot\left(\mu_1\left(y\right)\mu_1\left(z\right)-\mu_2\left(y\right)\mu_2\left(z\right)\right)\dz\dy\dxi\deta\dsig}\\
  \shoveleft{=-\frac{\eps}{2}\int_{0}^{\infty}\int_{0}^{\infty}\Gamma\left(\sigma,\eta\right)\int_{0}^{\infty}Q_0\left(\xi,p\right)\left(1-\del_{\xi}\right)\del_{\xi}\left(\left(U_{1}\left(\xi+\sigma+1\right)-U_2\left(\xi+\sigma+1\right)\right)U_1\left(\xi+\eta+1\right)\right.}\\
  \left.+U_2\left(\xi+\sigma+1\right)\left(U_1\left(\xi+\eta+1\right)-U_2\left(\xi+\eta+1\right)\right)\right)\dxi\deta\dsig.
\end{multline*}
 From this, taking Leibniz' rule into account, we obtain the following estimate
\begin{multline*}
  \quad \abs{J_{3,0,2}}\\
  \shoveleft{\leq C\eps\int_{0}^{\infty}\int_{0}^{\infty}\int_{0}^{\infty}\abs{\Gamma\left(\sigma,\eta\right)}\abs{Q_0\left(\xi,p\right)}\left\{\frac{\snorm{1}{\dmu}\fnorm{0}{\mu_1}+\snorm{1}{\mu_2}\fnorm{0}{\dmu}}{\left(\xi+\sigma+1\right)^{1+\theta}\left(\xi+\eta+1\right)^{\theta}}\right.}\\
  \left.+\frac{\fnorm{0}{\dmu}\snorm{1}{\mu_1}+\fnorm{0}{\mu_2}\snorm{1}{\dmu}}{\left(\xi+\sigma+1\right)^{\theta}\left(\xi+\eta+1\right)^{1+\theta}}+\frac{\snorm{2}{\dmu}\fnorm{0}{\mu_1}+\snorm{2}{\mu_2}\fnorm{0}{\dmu}}{\left(\xi+\sigma+1\right)^{2+\theta}\left(\xi+\eta+1\right)^{\theta}}\right.\\
  \left.+2\frac{\snorm{1}{\dmu}\snorm{1}{\mu_1}+\snorm{1}{\mu_2}\snorm{1}{\dmu}}{\left(\xi+\sigma+1\right)^{1+\theta}\left(\xi+\eta+1\right)^{1+\theta}}+\frac{\fnorm{0}{\dmu}\snorm{2}{\mu_1}+\fnorm{0}{\mu_2}\snorm{2}{\dmu}}{\left(\xi+\sigma+1\right)^{\theta}\left(\xi+\eta+1\right)^{2+\theta}}\right\}\dxi\deta\dsig.
\end{multline*}
Estimating by the highest norms, changing variables in the $\sigma$- and $\eta$-integral and using the uniform boundedness of $\mu_1$ and $\mu_2$ (see Proposition~\ref{Prop:uniform:bound:muk}) gives together with the homogeneity of $\Gamma$ that
\begin{equation*}
 \begin{split}
  &\quad\abs{J_{3,0,2}}\\
  &\leq C\eps\fnorm{1}{\dmu}\int_{0}^{\infty}\frac{\abs{Q_0\left(\xi,p\right)}}{\left(1+\xi\right)^{2\theta}}\int_{0}^{\infty}\int_{0}^{\infty}\left(\frac{\abs{\Gamma\left(\sigma,\eta\right)}}{\left(1+\sigma\right)^{1+\theta}\left(1+\eta\right)^{\theta}}+\frac{\abs{\Gamma\left(\sigma,\eta\right)}}{\left(1+\sigma\right)^{\theta}\left(1+\eta\right)^{1+\theta}}\right)\deta\dsig\dxi\\
  &\begin{multlined}
    \quad+ C\eps\fnorm{2}{\dmu}\int_{0}^{\infty}\frac{\abs{Q_0\left(\xi,p\right)}}{\left(1+\xi\right)^{1+2\theta}}\cdot\\
    \cdot\int_{0}^{\infty}\int_{0}^{\infty}\left(\frac{\abs{\Gamma\left(\sigma,\eta\right)}}{\left(1+\sigma\right)^{2+\theta}\left(1+\eta\right)^{\theta}}+\frac{\abs{\Gamma\left(\sigma,\eta\right)}}{\left(1+\sigma\right)^{1+\theta}\left(1+\eta\right)^{1+\theta}}+\frac{\abs{\Gamma\left(\sigma,\eta\right)}}{\left(1+\sigma\right)^{\theta}\left(1+\eta\right)^{2+\theta}}\right)\deta\dsig\dxi.
   \end{multlined}
 \end{split}
\end{equation*}
 Estimating by the most dominant term in $\xi$ we obtain together with Lemma~\ref{Lem:est:Gamma:int} and Proposition~\ref{Prop:Q0:estimate}
\begin{equation}\label{eq:J302}
 \begin{split}
  \abs{J_{3,0,2}}\leq C\eps\fnorm{2}{\dmu}\int_{0}^{\infty}\frac{\abs{Q_0\left(\xi,p\right)}}{\left(1+\xi\right)^{2\theta}}\dxi\leq C\eps\fnorm{2}{\dmu}\frac{\left(1+p\right)^{1-\theta}}{p^3}.
 \end{split}
\end{equation}
Thus taking~\eqref{eq:J301} and~\eqref{eq:J302} together we obtain~\eqref{eq:bl:5:1} if $\eps$ is sufficiently small.

\subsubsection{Estimate for $K_{4,0}$}

It only remains to estimate $K_{4,0}$, i.e.\ we will show
\begin{equation}\label{eq:bl:5:2}
 \eps\snorm{2}{K_{4,0}}\leq C\eps\fnorm{2}{\dmu}.
\end{equation}
Therefore, with the help of $Q_0$ as in~\eqref{eq:def:Q0}, we first rewrite  
\begin{equation*}
 \begin{split}
  K_{4,0}&=\int_{0}^{\infty}\left(1-\ee^{-y}\right)\ee^{-y}\left(1+y\right)H_0\left(y,p\right)\beta_W\left(y;\mu_2\right)\left(\mu_1\left(y\right)-\mu_2\left(y\right)\right)\dy\\
  &=\int_{0}^{\infty}Q_0\left(\xi,p\right)\left(1-\del_{\xi}\right)\int_{0}^{\infty}\left(1-\ee^{-y}\right)\ee^{-\left(1+\xi\right)y}\beta_W\left(y,\mu_2\right)\left(\mu_1\left(y\right)-\mu_2\left(y\right)\right)\dy\dxi.
 \end{split}
\end{equation*}
Recalling also the representation of $\beta_W$ from~\eqref{eq:rep:beta:W} we further get
\begin{equation*}
 \begin{split}
  &\quad K_{4,0}\\
  &=\int_{0}^{\infty}\int_{0}^{\infty}\Gamma\left(\sigma,\eta\right)\int_{0}^{\infty}Q_0\left(\xi,p\right)\left(1-\del_{\xi}\right)\int_{0}^{\infty}\left(1-\ee^{-y}\right)\ee^{-\left(1+\xi+\sigma\right)y}\left(yU_2\left(\eta+1\right)-U_2'\left(\eta+1\right)\right)\cdot\\
  &\qquad\qquad\cdot\left(\mu_1\left(y\right)-\mu_2\left(y\right)\right)\dy\dxi\dsig\deta\\
  &=-\int_{0}^{\infty}\int_{0}^{\infty}\Gamma\left(\sigma,\eta\right)U_2\left(\eta+1\right)\int_{0}^{\infty}Q_0\left(\xi,p\right)\left(1-\del_{\xi}\right)\del_{\xi}\left(M\left(\xi+\sigma+1\right)-M\left(\xi+\sigma+2\right)\right)\dxi\dsig\deta\\
  &\quad-\int_{0}^{\infty}\int_{0}^{\infty}\Gamma\left(\sigma,\eta\right)U_2'\left(\eta+1\right)\int_{0}^{\infty}Q_0\left(\xi,p\right)\left(1-\del_{\xi}\right)\left(M\left(\xi+\sigma+1\right)-M\left(\xi+\sigma+2\right)\right)\dxi\dsig\deta.
 \end{split}
\end{equation*}
Taking into account that we have
\begin{equation*}
 \abs{M\left(r+2\right)-M\left(r+1\right)}\leq C\frac{\snorm{1}{\dmu}}{\left(r+1\right)^{1+\theta}}\quad \text{and}\quad \abs{M'\left(r+2\right)-M'\left(r+1\right)}\leq C\frac{\snorm{2}{\dmu}}{\left(r+1\right)^{2+\theta}}
\end{equation*}
as well as
\begin{equation*}
 \abs{M''\left(r+2\right)}\leq C\frac{\snorm{2}{\dmu}}{\left(r+1\right)^{2+\theta}}
\end{equation*}
we obtain 
\begin{equation*}
 \begin{split}
  &\quad\abs{K_{4,0}}\\
  &\leq C\int_{0}^{\infty}\int_{0}^{\infty}\int_{0}^{\infty}\abs{\Gamma\left(\sigma,\eta\right)}\abs{Q_0\left(\xi,p\right)}\frac{\fnorm{0}{\mu_2}}{\left(\eta+1\right)^{\theta}}\frac{\snorm{2}{\dmu}}{\left(1+\xi+\sigma\right)^{2+\theta}}\dxi\deta\dsig\\
  &\quad+C\int_{0}^{\infty}\int_{0}^{\infty}\int_{0}^{\infty}\abs{\Gamma\left(\sigma,\eta\right)}\abs{Q_0\left(\xi,p\right)}\frac{\snorm{1}{\mu_2}}{\left(1+\eta\right)^{1+\theta}}\left(\frac{\snorm{1}{\dmu}}{\left(1+\xi+\sigma\right)^{1+\theta}}+\frac{\snorm{2}{\dmu}}{\left(1+\xi+\sigma\right)^{2+\theta}}\right)\dxi\deta\dsig.
 \end{split}
\end{equation*}
Estimating, similarly as before, by the highest norms and the most dominant terms in the $\eta$- and $\sigma$-integral and using also Proposition~\ref{Prop:uniform:bound:muk}, we obtain that
\begin{equation*}
 \begin{split}
  \abs{K_{4,0}}&\leq C\fnorm{2}{\dmu}\snorm{1}{\mu_2}\int_{0}^{\infty}\int_{0}^{\infty}\int_{0}^{\infty}\frac{\abs{\Gamma\left(\sigma,\eta\right)}\abs{Q_0\left(\xi,p\right)}}{\eta^{\theta}\left(1+\xi+\sigma\right)^{1+\theta}}\dxi\deta\dsig\\
  &\leq C\fnorm{2}{\dmu}\int_{0}^{\infty}\frac{\abs{Q_0\left(\xi,p\right)}}{\left(1+\xi\right)^{2\theta}}\int_{0}^{\infty}\int_{0}^{\infty}\frac{\abs{\Gamma\left(\sigma,\eta\right)}}{\eta^{\theta}\left(1+\sigma\right)^{1+\theta}}\deta\dsig\dxi\leq C\fnorm{2}{\dmu}\int_{0}^{\infty}\frac{\abs{Q_0\left(\xi,p\right)}}{\left(1+\xi\right)^{2\theta}}\dxi,
 \end{split}
\end{equation*}
where we used Lemma~\ref{Lem:est:Gamma:int:2} in the last step. From Proposition~\ref{Prop:Q0:estimate} it then follows~\eqref{eq:bl:5:2}.
\end{proof}

\section{Existence of $Q_0$}\label{Sec:rep:H0}

In this section we show that the function $H_0$ as given by~\eqref{eq:def:H0} can be represented as Laplace transform of some function $Q_0$ as indicated by~\eqref{eq:def:Q0}. For this we will rely on the Paley-Wiener Theorem that may be found in~\cite[Ch.VI, 4.]{Yos78}. Following~\cite{Yos78} we denote by $H^2\left(0\right)$ the Hardy-Lebesgue class consisting of functions $\varphi$ such that
\begin{enumerate}
 \item $\varphi\left(y\right)$ is holomorphic in the right half-plane $\Re(y)>0$
 \item for each fixed $x>0$, $\varphi\left(x+\im z\right)$ as a function of $z$ belongs to $L^2\left(\R\right)$ in such a way that
 \begin{equation}\label{eq:int:prop}
  \sup_{x>0}\left(\int_{-\infty}^{\infty}\abs{\varphi\left(x+\im z\right)}^2\dz\right)<\infty.
 \end{equation}
\end{enumerate}

With this we cite the following Theorem  (see~\cite[Ch.VI, 4, Theorem 2]{Yos78}).

\begin{theorem}[Paley-Wiener]\label{Thm:Paley-Wiener}
 Let $\varphi\left(y\right)\in H^2\left(0\right)$. Then the boundary function $\varphi\left(\im z\right)\in L^2\left(-\infty,\infty\right)$ of $\varphi\left(x+\im z\right)$ exists in the sense that 
 \begin{equation*}
  \lim_{x\downarrow 0}\int_{-\infty}^{\infty}\abs{\varphi\left(\im z\right)-\varphi\left(x+\im z\right)}^2\dz=0
 \end{equation*}
 in such a way that the inverse Fourier transform
 \begin{equation*}
  \psi\left(t\right)=\frac{1}{2\pi}\lim_{N\to \infty}\int_{-N}^{N}\varphi\left(\im z\right)\ee^{\im tz}\dz
 \end{equation*}
 vanishes for $t<0$ and $\varphi\left(y\right)$ may be obtained as the Laplace transform of $\psi\left(t\right)$. 
\end{theorem}

With this we will show the representation of $H_0$ as Laplace transform of $Q_0$.

\begin{proposition}\label{Prop:rep:H0}
 The function $H_0\left(\Y,p\right)$ from~\eqref{eq:def:H0} satisfies
 \begin{equation*}
  H_0\left(\Y,p\right)=\int_{0}^{\infty}Q_0\left(\xi,p\right)\ee^{-\xi \Y}\dxi,
 \end{equation*}
 where the function $Q_0$ is given by
 \begin{equation}\label{eq:Q0}
  Q_0\left(\xi,p\right)=\frac{1}{2\pi \im}\lim_{R\to \infty}\int_{-R \im}^{R\im}H_0\left(\Y,p\right)\ee^{\Y\xi}\dd{\Y}
 \end{equation}
 and $Q_0\left(\cdot,p\right)\in L^2\left(\R\right)$ for each $p>0$. Note that this limit is taken with respect to the $L^2$-topology.
\end{proposition}

\subsection{Analytic extension of $\beta_W$ and $\Phi$}

Before we can prove Proposition~\ref{Prop:rep:H0}, we have to extend the functions $\beta_W\left(\cdot, \mu_j\right)$ and $\Phi\left(\cdot,\mu_j\right)$ analytically to the right half plane in $\C$ and collect some properties. Furthermore we note that in the following, to simplify notation, we will always write $\mu$, $\beta_W$ as well as $\Phi$ instead
of $\mu_k$, $\beta_W\left(\cdot,\mu_k\right)$ and $\Phi_k$. Thus assuming $\mu$ to be fixed in the following we have the following result.
\begin{lemma}\label{Lem:prop:betaW}
 The function $\beta_W$ can be extended analytically to $\C\setminus\left(-\infty,0\right]$ and we have for all $x\in \C\setminus\left(-\infty,0\right]$ the following estimates
 \begin{align}
    \abs{\beta_W\left(x\right)}&\leq C\left(\abs{x}^{-\alpha}+\abs{x}^{\alpha}\right)\label{eq:est:beta:W:0},\\
    \abs{\frac{\dd^{\ell}}{\dx^\ell}\beta_W\left(x\right)}&\leq C_{\ell}\left(\abs{x}^{-\ell-\alpha}+\abs{x}^{\alpha-\ell}\right) &&\text{if } \Re\left(x\right)\geq 0 \text{ for all }  \ell\in\N \label{eq:est:beta:W},\\
     \Re\left(\beta_W\left(x\right)\right)&\geq 0 &&\text{if } \Re\left(x\right)\geq 0. \label{eq:real:part:betaW}
 \end{align}
\end{lemma}
\begin{proof}
 Using the homogeneity of $W$ we can rewrite
\begin{equation}\label{eq:hom:betaW}
 \begin{split}
  \beta_W\left(y\right)=\int_{0}^{\infty}W\left(y,z\right)\mu\left(z\right)\ee^{-z}\dz=\int_{0}^{\infty}W\left(\frac{y}{z},1\right)\mu\left(z\right)\ee^{-z}\dz.
 \end{split}
\end{equation}
By assumption~\eqref{kernel0} it then follows that $\beta_W$ can be extended to an analytic function in $\C\setminus\left(-\infty,0\right]$. Furthermore using
\begin{equation*}
 \Re\left(\beta_W\left(x\right)\right)=\int_{0}^{\infty}\Re\left(W\left(\frac{x}{z},1\right)\mu\left(z\right)\ee^{-z}\right)\dz
\end{equation*}
we also get immediately~\eqref{eq:real:part:betaW}. On the other hand it follows from~\eqref{eq:hom:betaW} and~\eqref{kernel0} together with~\eqref{eq:mean:est} that
\begin{equation*}
 \abs{\beta_W\left(x\right)}\leq C\int_{0}^{\infty}\left(\abs{\frac{z}{x}}^{\alpha}+\abs{\frac{x}{z}}^{\alpha}\right)\mu\left(z\right)\ee^{-z}\dz\leq C\left(\abs{x}^{-\alpha}+\abs{x}^{\alpha}\right),
\end{equation*}
which shows~\eqref{eq:est:beta:W:0}. Then, using the analyticity of $\beta_W$ in $\C\setminus\left(-\infty,0\right]$, we get~\eqref{eq:est:beta:W} for $\ell\in\N$  by means of Cauchy estimates.
\end{proof}

\begin{remark}\label{Rem:path:int}
From now on we will use repeatedly the notation $\int_{x}^{\sgn\left(\Im x\right)\im\infty}\left(\cdots\right)\dt$ to denote the path integral $\int_{\gamma_x}\left(\cdots\right)\dt$ where we denote 
\begin{equation*}
 \begin{split}
  \gamma_x\vcc=\left\{\Re\left(x\right)+\im\sgn\left(\Im x\right)\xi\;|\; \xi\in\left(\Im x,\infty\right)\right\} \quad \text{if } \Im\left(x\right)\neq 0
 \end{split}
\end{equation*}
while for $\Im\left(x\right)=0$ we can take both signs, i.e.\
\begin{equation*}
 \gamma_x\vcc=\left\{\Re\left(x\right)\pm\im\xi\;|\; \xi\in\left(0,\infty\right)\right\} \quad \text{if } \Im\left(x\right)= 0.
\end{equation*}
Similarly, for $x\in \C\setminus\left(-\infty,0\right]$, integrals of the form $\int_{x}^{\infty}\left(\cdots\right)\dt$ have to be understood as integrals $\int_{\gamma_x}\left(\cdots\right)\dt$ with 
\begin{equation*}
 \gamma_x=\left\{s+\im\Im\left(x\right)\;|\; s\in\left(\Re x,\infty\right)\right\}
\end{equation*}
as illustrated in Figure~\ref{fig:gamma}. Finally for $x,\Y\in \C$ we denote by $\int_{x}^{\Y}\left(\cdots\right)\dt$ the path integral over the segment from $x$ to $\Y$. In all the following computations the segment connecting $x$ and $\Y$ does not intersect the negative real line $\left(-\infty,0\right)$.
\end{remark}
  
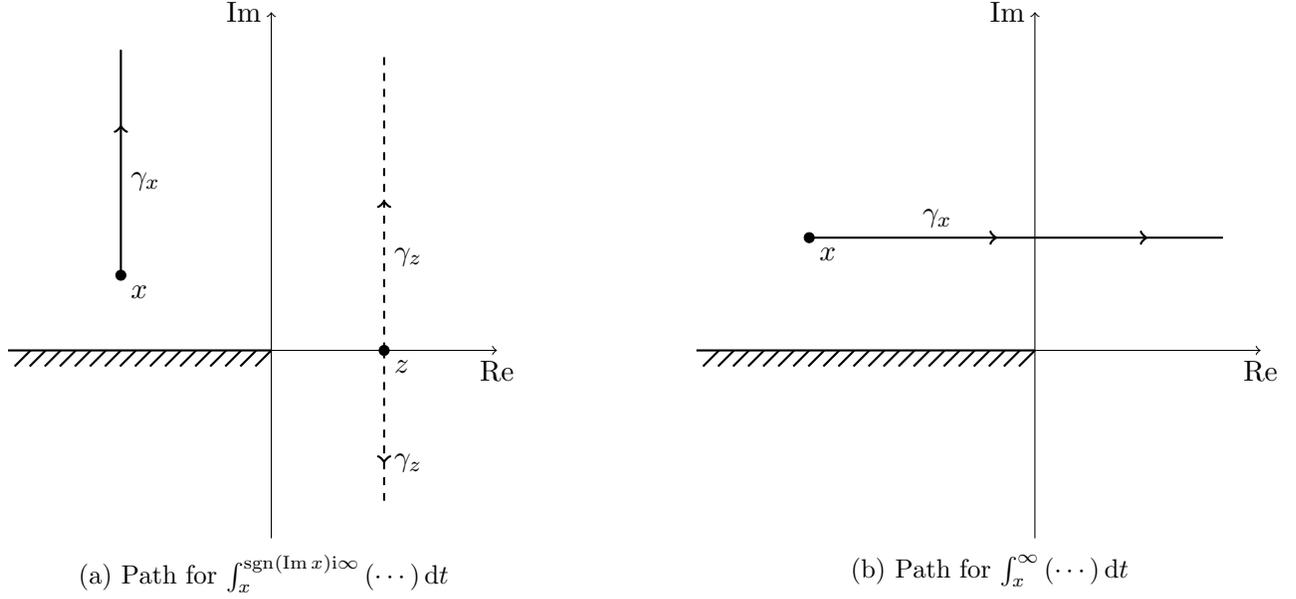
\begin{figure}
  \centering
  \begin{minipage}[t]{.45\linewidth}
    \centering
     \begin{tikzpicture}[contour/.style={postaction={decorate, decoration={markings,
mark=at position 2cm with {\arrow[line width=1pt]{>}}
}}},
contourtwo/.style={postaction={decorate, decoration={markings,
mark=at position 1.5cm with {\arrow[line width=1pt]{>}}
}}},
    interface/.style={postaction={draw,decorate,decoration={border,angle=45,
                    amplitude=0.3cm,segment length=2mm}}},
]
\draw[->] (0,0) -- (3,0) node[below] {$\Re$};
\draw[->] (0,-2.5) -- (0,4.5) node[left] {$\Im$};

\draw[line width=.8pt,interface](0,0)--(-3.5,0);

\path[draw,line width=0.8pt,contour] (-2,1)  -- (-2,4);
\path[draw,dashed,line width=0.8pt,contour] (1.5,0)  -- (1.5,4);
\path[draw,dashed,line width=0.8pt,contourtwo] (1.5,0)  -- (1.5,-2);

\draw[black,fill=black] (-2,1) circle (.4ex);
\draw[black, fill=black] (1.5,0) circle (.4ex);
\node[below right] at (-2,1) {$x$};
\node[below right] at (1.5,0) {$z$};
\node[right] at (-2,2.25) {$\gamma_{x}$};
\node[right] at (1.5,1.25) {$\gamma_{z}$};
\node[right] at (1.5,-1.5) {$\gamma_{z}$};
\end{tikzpicture}
\subcaption{Path for $\int_{x}^{\sgn\left(\Im x\right)\im\infty}\left(\cdots\right)\dt$}
  \end{minipage}
  \hfill
  \begin{minipage}[t]{.45\linewidth}
    \centering
    \begin{tikzpicture}[contour/.style={postaction={decorate, decoration={markings,
mark=at position 4.5cm with {\arrow[line width=1pt]{>}},
mark=at position 2.5cm with {\arrow[line width=1pt]{>}}
}}},
    interface/.style={postaction={draw,decorate,decoration={border,angle=45,
                    amplitude=0.3cm,segment length=2mm}}},
]
\draw[->] (0,0) -- (3,0) node[below] {$\Re$};
\draw[->] (0,-2.5) -- (0,4.5) node[left] {$\Im$};

\draw[line width=.8pt,interface](0,0)--(-4.5,0);

\path[draw,line width=0.8pt,contour] (-3,1.5)  -- (2.5,1.5);

\draw[black,fill=black] (-3,1.5) circle (.4ex);
\node[below right] at (-3,1.5) {$x$};
\node[above] at (-1.3,1.5) {$\gamma_{x}$};
\end{tikzpicture}
      \subcaption{Path for $\int_{x}^{\infty}\left(\cdots\right)\dt$}
  \end{minipage}
  \caption{Integration path $\gamma_{x}$}
  \label{fig:gamma}
\end{figure}

 In the following Lemma we collect several properties and estimates for $\Phi$ that will be useful later in the Proofs of Propositions~\ref{Prop:rep:H0} and~\ref{Prop:Q0:estimate}.

\begin{lemma}\label{Lem:properties:Phi}
 The function $\Phi$ as given in~\eqref{Phidef} can be extended analytically to $\C\setminus\left(-\infty,0\right]$. Furthermore, for $\Phi$ and $x\in \C\setminus\left(-\infty,0\right]$ we have the representation formulas 
\begin{align}
 \Phi\left(x\right)&=\eps\int_{x}^{ \sgn\left(\Im x\right)\im\infty}\frac{\beta_W\left(t\right)}{t}\ee^{-t}\dt, \label{eq:Phi:complex:1}\\
 \Phi\left(x\right)&=\eps\frac{\beta_W\left(x\right)\ee^{-x}}{x}+\eps\int_{x}^{\sgn\left(\Im x\right)\im\infty}\frac{\dd}{\dt}\left(\frac{\beta_W\left(t\right)}{t}\right)\ee^{-t}\dt.\label{eq:Phi:complex:2}
\end{align}
 Additionally for $x\in\C\setminus\left(-\infty,0\right]$ with $\Re\left(x\right)\geq 0$ we have
 \begin{align}
 \abs{\frac{\dd^{\ell}}{\dx^\ell}\Phi\left(x\right)}&\leq C\eps\left(\abs{x}^{-\ell-\alpha}+\abs{x}^{\alpha-1}\right) \quad \text{for all } \ell\in\N,  \label{eq:der:Phi:bound}\\
  \abs{\Phi\left(x\right)}&\leq C\eps \min\left\{\abs{x}^{-\alpha},\abs{x}^{\alpha-1}\right\}\label{eq:bound:Phi},
   \end{align}
 while for $\Re\left(\Y\right)\geq 0$ and $x\in \left\{\lambda \Y\; |\; \lambda\in\left[0,1\right]\right\}$ there exists $C>0$ such that 
 \begin{equation}\label{eq:difference:Phi}
  \Re\left(\Phi\left(\Y\right)-\Phi\left(x\right)\right)\leq C\eps \left(\abs{\Y}^{1+\alpha}+\abs{\Y}^{1-\alpha}\right).
 \end{equation}
\end{lemma}

\begin{proof}
From Lemma~\ref{Lem:prop:betaW} it follows immediately that the expression~\eqref{Phidef} can be extended analytically to $\C\setminus\left(-\infty,0\right]$ where the integral has to be interpreted as explained in Remark~\ref{Rem:path:int}. Formula~\eqref{eq:Phi:complex:1} then follows by contour deformation taking into account also Lemma~\ref{Lem:prop:betaW}. Note that the integral in~\eqref{eq:Phi:complex:1} has to be understood in the sense of improper Riemann integrals.

On the other hand formula~\eqref{eq:Phi:complex:2} is obtained by means of integration by parts, while the integral on the right hand side exists in the usual Lebesgue sense.

 The estimates~\eqref{eq:der:Phi:bound} and \eqref{eq:bound:Phi} easily follow from the definitions, i.e.\ as $\Phi'\left(x\right)=-\eps\frac{\beta_W\left(x\right)}{x}\ee^{-x}$ the estimate~\eqref{eq:der:Phi:bound} for $\ell=1$ immediately follows from~\eqref{eq:est:beta:W:0}, while for $\ell>1$ we just differentiate repeatedly using~\eqref{eq:est:beta:W:0} and~\eqref{eq:est:beta:W} and taking into account also the oscillatory term $\ee^{-x}$. To obtain~\eqref{eq:bound:Phi} for $\abs{x}\leq 1$ we just integrate~\eqref{eq:der:Phi:bound} for $\ell=1$. For $\abs{x}\geq 1$ the estimate directly follows from~\eqref{eq:Phi:complex:2} together with~\eqref{eq:est:beta:W:0} and~\eqref{eq:est:beta:W}.
 
 It remains to prove~\eqref{eq:difference:Phi} which is more complicated. From the definition of $\Phi$ we obtain by contour deformation
\begin{equation*}
 \begin{split}
  \Phi\left(\Y\right)-\Phi\left(x\right)&=\eps \int_{\Y}^{\infty}\frac{\beta_W\left(t\right)}{t}\ee^{-t}\dt-\eps \int_{x}^{\infty}\frac{\beta_W\left(t\right)}{t}\ee^{-t}\dt=-\eps\int_{x}^{\Y}\frac{\beta_W\left(t\right)}{t}\ee^{-t}\dt\\
  &=-\eps\int_{0}^{\infty}\mu\left(z\right)\ee^{-z}\int_{x}^{\Y}W\left(\frac{t}{z},1\right)\frac{\ee^{-t}}{t}\dt\dx.
 \end{split}
\end{equation*}
Using the splitting $\ee^{-t}=1+\left(\ee^{-t}-1\right)$ we can write
\begin{equation*}
 \begin{split}
  \Phi\left(\Y\right)-\Phi\left(x\right)&=-\eps\int_{0}^{\infty}\mu\left(z\right)\ee^{-z}\int_{x}^{\Y}\frac{W\left(\frac{t}{z},1\right)}{t}\dt\dz-\eps\int_{0}^{\infty}\mu\left(z\right)\ee^{-z}\int_{x}^{\Y}\frac{W\left(\frac{t}{z},1\right)}{t}\left(\ee^{-t}-1\right)\dt\dz\\
  &=\vcc(I)+(II).
 \end{split}
\end{equation*}
Using that $\abs{\frac{\ee^{-t}-1}{t}}\leq C$, the a-priori estimates~\eqref{fdecay} and~\eqref{S2E3} and writing $x=\lambda \Y$ with $\lambda\in\left[0,1\right]$ we can estimate the second term by
\begin{equation*}
 \begin{split}
  \abs{(II)}&\leq C\eps\int_{0}^{\infty}\mu\left(z\right)\ee^{-z}\abs{\Y}\left(1-\lambda\right)\int_{0}^{1}\abs{W\left(\frac{\Y\left(\lambda+s\left(1-\lambda\right)\right)}{z},1\right)}\ds\dz\\
  &\leq C\eps\abs{\Y}\int_{0}^{\infty}\mu\left(z\right)\ee^{-z}(1-\lambda)\int_{0}^{1}\frac{\abs{\Y}^{\alpha}\left(s\left(1-\lambda\right)+\lambda\right)^{\alpha}}{z^{\alpha}}+\frac{z^{\alpha}}{\abs{\Y}^{\alpha}\left(s\left(1-\lambda\right)+\lambda\right)^{\alpha}}\ds\dz\\
  &\leq C\eps\left(\abs{\Y}^{1+\alpha}+\abs{\Y}^{1-\alpha}\right).
 \end{split}
\end{equation*}
Furthermore, writing $\Y=r\ee^{\im \theta}$ with $\theta\in\left[-\frac{\pi}{2},\frac{\pi}{2}\right]$ we can rewrite
\begin{equation*}
 \begin{split}
  (I)&=-\eps\int_{0}^{\infty}\mu\left(z\right)\ee^{-z}\int_{\lambda \Y}^{\Y}\frac{W\left(\frac{t}{z},1\right)}{t}\dt\dz=-\eps\int_{0}^{\infty}\mu\left(z\right)\ee^{-z}\int_{r\ee^{\im \theta}\lambda}^{r\ee^{\im \theta}}\frac{W\left(\frac{t}{z},1\right)}{t}\dt\dz\\
  &=-\eps\int_{0}^{\infty}\mu\left(z\right)\ee^{-z}\int_{r\lambda}^{r}\frac{W\left(\frac{s\ee^{\im \theta}}{z},1\right)}{s}\ds\dz,
 \end{split}
\end{equation*}
where the last integral is an integral in the real line. Using then~\eqref{kernel0} and $\mu\geq 0$ we find
\begin{equation*}
 \Re\left(-\eps\int_{0}^{\infty}\mu\left(z\right)\ee^{-z}\int_{\Y\lambda}^{\Y}\frac{W\left(\frac{t}{z},1\right)}{t}\dt\dz\right)\leq 0.
\end{equation*}
Together this then shows~\eqref{eq:difference:Phi}.
\end{proof}

In the following we will frequently use that 
\begin{equation*}
 \begin{split}
  \abs{\ee^{\Phi\left(\Y\right)-\Phi\left(x\right)}}\leq C\quad \text{for } x\in\left\{\lambda \Y\;|\; \lambda\in\left[0,1\right]\right\}\quad \text{and}\quad \abs{\ee^{\Phi\left(\Y\right)-\Phi\left(x\right)}}\leq C_{R}\quad \text{for any } \abs{x},\abs{\Y}\geq R>0.
 \end{split}
\end{equation*}

\begin{remark}\label{Rem:H0:analytic}
 From the estimates obtained in Lemma~\ref{Lem:properties:Phi} it follows immediately that we can extend the expression for $H_0$ to the right half-plane by integrating along the segment connecting $0$ to $\Y$. From the analyticity of the integrand the value of the integral is then independent of the choice of the path.
\end{remark}

Before we come to the proof of Proposition~\ref{Prop:rep:H0} we also note that due to Remark~\ref{Rem:smallness:kappa} we have $\abs{\kappa}\to 0$ as $\eps\to 0$. Thus, in the following we may always assume that for each $\delta>0$ it holds $\abs{\kappa}\leq \delta$ if $\eps$ is sufficiently small, while this will not be mentioned each time. 

\subsection{Proof of Proposition~\ref{Prop:rep:H0}}

We come now to the proof of the representation formula for $H_0$.

\begin{proof}[Proof of Proposition~\ref{Prop:rep:H0}]
 Due to Theorem~\ref{Thm:Paley-Wiener} it suffices to show that $H_0\in H^2\left(0\right)$. Recalling that $H_0\left(\cdot,p\right)$ is analytic in the right half plane due to Remark~\ref{Rem:H0:analytic}, it remains to establish the $L^2$-condition~\eqref{eq:int:prop} in the definition of $H^2\left(0\right)$. We first show that for $\abs{\Y}\leq 1$ and $\Re \left(\Y\right)\geq 0$ we have that $H_0\left(\Y,p\right)$ is uniformly bounded with respect to $\Y$. In fact, assuming $\abs{\Y}\leq 1$, choosing as path of integration in~\eqref{eq:def:H0} the segment from zero to $\Y$ and using also~\eqref{eq:der:Phi:bound} and~\eqref{eq:difference:Phi}, we have
 \begin{equation}\label{eq:H0:L2:bound:0}
  \begin{split}
   \abs{H_{0}\left(\Y,p\right)}&\leq\abs{\frac{\left(2+\kappa\right)\Y}{p\Y\left(1+\Y\right)}\int_{0}^{1}\ee^{-ps\Y}s\Y s^{\kappa}\frac{\ee^{\Phi\left(\Y\right)}}{\ee^{\Phi\left(s\Y\right)}}\ds}+\abs{\frac{\Y}{p\Y\left(1+\Y\right)}\int_{0}^{1}\Phi'\left(s\Y\right)\ee^{-ps\Y}s^2\Y^2s^{\kappa}\frac{\ee^{\Phi\left(\Y\right)}}{\ee^{\Phi\left(s\Y\right)}}\ds}\\
   &\leq \frac{C\abs{\Y}}{p\abs{1+\Y}}+\frac{C}{p\abs{1+\Y}}\int_{0}^{1}\left(\abs{s\Y}^{-1-\alpha}+\abs{s\Y}^{\alpha-1}\right)s^2\abs{\Y}^2s^{\kappa}\ds\\
   &\leq \frac{C}{p\abs{1+\Y}}\left(\abs{\Y}+\abs{\Y}^{1-\alpha}+\abs{\Y}^{1+\alpha}\right)\leq \frac{C}{p}.
  \end{split}
 \end{equation}
 We next consider $\abs{\Y}>1$. To do this we treat the terms $H_{0,1}$ and $H_{0,2}$ (see~\eqref{eq:def:H0}) separately and begin with $H_{0,1}$. By splitting the integral, we find
 \begin{equation*}
  \begin{split}
   H_{0,1}\left(\Y,p\right)&=\frac{2+\kappa}{p\Y\left(1+\Y\right)}\int_{0}^{\frac{\Y}{\abs{\Y}}}\ee^{-px}x\left(\frac{x}{\Y}\right)^{\kappa}\frac{\ee^{\Phi\left(\Y\right)}}{\ee^{\Phi\left(x\right)}}\dx+\frac{2+\kappa}{p\Y\left(1+\Y\right)}\int_{\frac{\Y}{\abs{\Y}}}^{\Y}\ee^{-px}x\left(\frac{x}{\Y}\right)^{\kappa}\frac{\ee^{\Phi\left(\Y\right)}}{\ee^{\Phi\left(x\right)}}\dx\\
   &=\vcc H_{0,1,1}+H_{0,1,2},
  \end{split}
 \end{equation*}
 where in both cases we can just take a segment as path of integration. Using the estimate
 \begin{equation}\label{eq:splitting:Phi}
  \Re\left(\Phi\left(\Y\right)-\Phi\left(x\right)\right)=\Re\left(\Phi\left(\frac{\Y}{\abs{\Y}}\right)-\Phi\left(x\right)\right)+\Re\left(\Phi\left(\Y\right)-\Phi\left(\frac{\Y}{\abs{\Y}}\right)\right)\leq C,
 \end{equation}
 which is due to Lemma~\ref{Lem:properties:Phi}, together with~\eqref{eq:bound:Phi} and~\eqref{eq:difference:Phi}, we obtain for $H_{0,1,1}$ that
 \begin{equation}\label{eq:H0:L2:bound:1}
  \begin{split}
   \abs{H_{0,1,1}}\leq \frac{C}{p\abs{\Y}^{1-\delta}\abs{1+\Y}}.
  \end{split}
 \end{equation}
 For $H_{0,1,2}$, integrating by parts, we obtain 
 \begin{equation*}
  \begin{split}
   H_{0,1,2}&=-\frac{2+\kappa}{p^2\Y\left(1+\Y\right)}\int_{\frac{\Y}{\abs{\Y}}}^{\Y}\del_{x}\left(\ee^{-px}\right)\left(x\left(\frac{x}{\Y}\right)^{\kappa}\frac{\ee^{\Phi\left(\Y\right)}}{\ee^{\Phi\left(x\right)}}\right)\dx\\
   &=-\frac{2+\kappa}{p^2\Y\left(1+\Y\right)}\left(\ee^{-p\Y}\Y-\frac{\Y}{\abs{\Y}}\frac{\ee^{-p\frac{\Y}{\abs{\Y}}}}{\abs{\Y}^{\kappa}}\frac{\ee^{\Phi\left(\Y\right)}}{\ee^{\Phi\left(\frac{\Y}{\abs{\Y}}\right)}}\right)\\
   &\qquad+\frac{2+\kappa}{p^2\Y\left(1+\Y\right)}\int_{\frac{\Y}{\abs{\Y}}}^{\Y}\ee^{-px}\left(\frac{1+\kappa}{x}-\Phi'\left(x\right)\right)x\left(\frac{x}{\Y}\right)^{\kappa}\frac{\ee^{\Phi\left(\Y\right)}}{\ee^{\Phi\left(x\right)}}\dx.
  \end{split}
 \end{equation*}
 Thus similarly as before we get
 \begin{equation}\label{eq:H012}
  \begin{split}
   \abs{H_{0,1,2}}&\leq \frac{C}{p^2\abs{1+\Y}}+\frac{C}{p^2\abs{\Y}^{1-\delta}\abs{1+\Y}}\\
   &\qquad+\frac{C}{p^2\abs{\Y}\abs{1+\Y}}\int_{1}^{\abs{\Y}}\abs{\ee^{-ps\frac{\Y}{\abs{\Y}}}\left(\frac{1+\kappa}{s\frac{\Y}{\abs{\Y}}}+\Phi'\left(s\frac{\Y}{\abs{\Y}}\right)\right)s\frac{\Y}{\abs{\Y}}\left(\frac{s}{\abs{\Y}}\right)^{\kappa}\frac{\ee^{\Phi\left(\Y\right)}}{\ee^{\Phi\left(\frac{s\Y}{\abs{\Y}}\right)}}}\ds\\
   &\leq \frac{C}{p^2\abs{1+\Y}}+\frac{C}{p^2\abs{\Y}^{1+\kappa}\abs{1+\Y}}\int_{0}^{\abs{\Y}}\left(s^{\kappa}+s^{\kappa-\alpha}+s^{\alpha+\kappa}\right)\ds\\
   &\leq \frac{C}{p^2\abs{1+\Y}}+\frac{C\abs{\Y}^{\alpha}}{p^2\abs{1+\Y}}\leq \frac{C\abs{\Y}^{\alpha}}{p^2\abs{1+\Y}}.
  \end{split}
 \end{equation}
 We next consider $H_{0,2}$ and proceed similarly, while we also use $\Phi'\left(x\right)=-\eps\beta_W\left(x\right)x^{-1}\ee^{-x}$, i.e.
 \begin{equation*}
  \begin{split}
   H_{0,2}\left(\Y,p\right)&=-\frac{1}{p\Y\left(1+\Y\right)}\int_{0}^{\frac{\Y}{\abs{\Y}}}\Phi'\left(x\right)\ee^{-px}x^2\left(\frac{x}{\Y}\right)^{\kappa}\frac{\ee^{\Phi\left(\Y\right)}}{\ee^{\Phi\left(x\right)}}\dx\\
   &\qquad+\frac{\eps}{p\Y\left(1+\Y\right)}\int_{\frac{\Y}{\abs{\Y}}}^{\Y}\beta_W\left(x\right)\ee^{-\left(1+p\right)x}x\left(\frac{x}{\Y}\right)^{\kappa}\frac{\ee^{\Phi\left(\Y\right)}}{\ee^{\Phi\left(x\right)}}\dx=\vcc H_{0,2,1}+H_{0,2,2}.
  \end{split}
 \end{equation*}
 Using again~\eqref{eq:splitting:Phi} and Lemma~\ref{Lem:properties:Phi} we obtain
 \begin{equation}\label{eq:H021}
  \begin{split}
   \abs{H_{0,2,1}}&\leq \frac{C\abs{\frac{\Y}{\abs{\Y}}}}{p\abs{\Y}\abs{1+\Y}}\int_{0}^{1}\left(\abs{s\frac{\Y}{\abs{\Y}}}^{-1-\alpha}+\abs{s\frac{\Y}{\abs{\Y}}}^{\alpha-1}\right)\abs{s\frac{\Y}{\abs{\Y}}}^2\left(\frac{s}{\abs{\Y}}\right)^{\kappa}\ds\leq \frac{C}{p\abs{\Y}^{1-\delta}\abs{1+\Y}}.
  \end{split}
 \end{equation}
To estimate $H_{0,2,2}$ we have to integrate by parts twice (at least for some terms). Precisely we have
\begin{multline}\label{eq:H022:split} 
  H_{0,2,2}=-\frac{\eps}{p\left(1+p\right)\Y\left(1+\Y\right)}\int_{\frac{\Y}{\abs{\Y}}}^{\Y}\del_{x}\left(\ee^{-\left(1+p\right)x}\right)\beta_{W}\left(x\right)x\left(\frac{x}{\Y}\right)^{\kappa}\frac{\ee^{\Phi\left(\Y\right)}}{\ee^{\Phi\left(x\right)}}\dx\\
  =-\frac{\eps}{p\left(1+p\right)\Y\left(1+\Y\right)}\left(\beta_W\left(\Y\right)\ee^{-\left(1+p\right)\Y}\Y-\beta_W\left(\frac{\Y}{\abs{\Y}}\right)\ee^{-\left(1+p\right)\frac{\Y}{\abs{\Y}}}\frac{\Y}{\abs{\Y}^{1+\kappa}}\frac{\ee^{\Phi\left(\Y\right)}}{\ee^{\Phi\left(\frac{\Y}{\abs{\Y}}\right)}}\right)\\
  \qquad+\frac{\eps}{p\left(1+p\right)\Y\left(1+\Y\right)}\int_{\frac{\Y}{\abs{\Y}}}^{\Y}\ee^{-\left(1+p\right)x}\left(\beta_W'\left(x\right)+\left(1+\kappa\right)\frac{\beta_W\left(x\right)}{x}\right)x\left(\frac{x}{\Y}\right)^{\kappa}\frac{\ee^{\Phi\left(\Y\right)}}{\ee^{\Phi\left(x\right)}}\dx\\
  -\frac{\eps}{p\left(1+p\right)\Y\left(1+\Y\right)}\int_{\frac{\Y}{\abs{\Y}}}^{\Y}\ee^{-\left(1+p\right)x}\beta_W\left(x\right)\Phi'\left(x\right)x\left(\frac{x}{\Y}\right)^{\kappa}\frac{\ee^{\Phi\left(\Y\right)}}{\ee^{\Phi\left(x\right)}}\dx=\vcc(I)+(II)
\end{multline}
with 
\begin{equation*}
 \begin{split}
  (I)&=-\frac{\eps}{p\left(1+p\right)\Y\left(1+\Y\right)}\left(\beta_W\left(\Y\right)\ee^{-\left(1+p\right)\Y}\Y-\beta_W\left(\frac{\Y}{\abs{\Y}}\right)\ee^{-\left(1+p\right)\frac{\Y}{\abs{\Y}}}\frac{\Y}{\abs{\Y}^{1+\kappa}}\frac{\ee^{\Phi\left(\Y\right)}}{\ee^{\Phi\left(\frac{\Y}{\abs{\Y}}\right)}}\right)\\
  &\quad +\frac{\eps}{p\left(1+p\right)\Y\left(1+\Y\right)}\int_{\frac{\Y}{\abs{\Y}}}^{\Y}\ee^{-\left(1+p\right)x}\left(\beta_W'\left(x\right)+\left(1+\kappa\right)\frac{\beta_W\left(x\right)}{x}\right)x\left(\frac{x}{\Y}\right)^{\kappa}\frac{\ee^{\Phi\left(\Y\right)}}{\ee^{\Phi\left(x\right)}}\dx
 \end{split}
\end{equation*}
and 
\begin{equation*}
 (II)=\frac{\eps}{p\left(1+p\right)\Y\left(1+\Y\right)}\int_{\frac{\Y}{\abs{\Y}}}^{\Y}\ee^{-\left(1+p\right)x}\beta_W\left(x\right)\Phi'\left(x\right)x\left(\frac{x}{\Y}\right)^{\kappa}\frac{\ee^{\Phi\left(\Y\right)}}{\ee^{\Phi\left(x\right)}}\dx.
\end{equation*}
We first estimate $(I)$ and as before we get
\begin{equation}\label{eq:H021:I}
 \begin{split}
  \abs{(I)}&\leq \frac{C}{p\left(1+p\right)\abs{\Y}\abs{1+\Y}}\left(\left(\abs{\Y}^{\alpha}+\abs{\Y}^{-\alpha}\right)\abs{\Y}+\abs{\Y}^{\delta}\right)\\
  &\quad+\frac{C}{p\left(1+p\right)\abs{\Y}\abs{1+\Y}}\int_{1}^{\abs{\Y}}\left(s^{-\alpha-1}+s^{\alpha-1}\right)s\left(\frac{s}{\abs{\Y}}\right)^{\kappa}\ds\leq \frac{C\abs{\Y}^{\alpha}}{p\left(1+p\right)\abs{1+\Y}},
 \end{split}
\end{equation}
where we estimated by the most dominant terms in the last step. It remains to consider $(II)$ while here we have to integrate by parts again to get enough decay in $\Y$ at infinity. Thus expressing $\Phi'$ again by $\beta_W$ yields
\begin{equation*}
 \begin{split}
  &\quad (II)\\
  &=\frac{\eps}{p\left(1+p\right)\Y\left(1+\Y\right)}\int_{\frac{\Y}{\abs{\Y}}}^{\Y}\ee^{-\left(2+p\right)x}\beta_W\left(x\right)\Phi'\left(x\right)\left(\frac{x}{\Y}\right)^{\kappa}\frac{\ee^{\Phi\left(\Y\right)}}{\ee^{\Phi\left(x\right)}}\dx\\
  &=-\frac{\eps^2}{p\left(1+p\right)\left(2+p\right)\Y\left(1+\Y\right)}\int_{\frac{\Y}{\abs{\Y}}}^{\Y}\del_{x}\left(\ee^{-\left(2+p\right)x}\right)\beta_W^2\left(x\right)\left(\frac{x}{\Y}\right)^{\kappa}\frac{\ee^{\Phi\left(\Y\right)}}{\ee^{\Phi\left(x\right)}}\dx\\
  &=-\frac{\eps^2}{p\left(1+p\right)\left(2+p\right)\Y\left(1+\Y\right)}\left(\ee^{-\left(2+p\right)\Y}\beta_W^2\left(\Y\right)-\ee^{-\left(2+p\right)\frac{\Y}{\abs{\Y}}}\beta_W^2\left(\frac{\Y}{\abs{\Y}}\right)\frac{1}{\abs{\Y}^{\kappa}}\frac{\ee^{\Phi\left(\Y\right)}}{\ee^{\Phi\left(\frac{\Y}{\abs{\Y}}\right)}}\right)\\
  &\quad +\frac{\eps^2}{p\left(1+p\right)\left(2+p\right)\Y\left(1+\Y\right)}\int_{\frac{\Y}{\abs{\Y}}}^{\Y}\ee^{-\left(2+p\right)x}\left(\del_{x}\left(\beta_W\left(x\right)\right)^2+\kappa\frac{\beta_W^2\left(x\right)}{x}-\beta_W^2\left(x\right)\Phi'\left(x\right)\right)\left(\frac{x}{\Y}\right)^{\kappa}\frac{\ee^{\Phi\left(\Y\right)}}{\ee^{\Phi\left(x\right)}}\dx.
 \end{split}
\end{equation*}
Now proceeding as before we find
\begin{equation}\label{eq:H0:L2:bound:5}
 \begin{split}
  \abs{(II)}&\leq \frac{C}{p\left(1+p\right)\left(2+p\right)\abs{\Y}\abs{1+\Y}}\left(\abs{\Y}^{2\alpha}+\abs{\Y}^{\delta}\right)\\
  &\quad+\frac{C}{p\left(1+p\right)\left(2+p\right)\abs{\Y}\abs{1+\Y}}\int_{1}^{\abs{\Y}}\left(s^{2\alpha-1}+s^{2\alpha}s^{\alpha-1}\right)\left(\frac{s}{\abs{\Y}}\right)^{\kappa}\ds\\
  &\leq \frac{C}{p\left(1+p\right)\left(2+p\right)\abs{\Y}^{1-2\alpha}\abs{1+\Y}}+\frac{C}{p\left(1+p\right)\left(2+p\right)\abs{\Y}^{1-3\alpha}\abs{1+\Y}}.
 \end{split}
\end{equation}
Summarizing~\eqref{eq:H0:L2:bound:1}\textendash\eqref{eq:H0:L2:bound:5} and taking into account $\abs{\Y}>1$ we find
\begin{equation*}
 \begin{split}
  \abs{H_0\left(\Y,p\right)}\leq C\left(p\right)\left(\abs{\Y}^{\delta-2}+\abs{\Y}^{\alpha-1}+\abs{\Y}^{3\alpha-2}\right),
 \end{split}
\end{equation*}
while on the other hand from~\eqref{eq:H0:L2:bound:0} we have $\abs{H_0\left(\Y,p\right)}\leq C\left(p\right)$ for $\abs{\Y}\leq 1$ and thus also $\abs{H_0\left(\Y,p\right)}\leq C\left(p\right)$ for $\abs{\Y}\leq 2$. Writing then $\Y=x+\im z$ we get for $x\leq 1$ that
\begin{equation*}
 \begin{split}
  &\quad\int_{-\infty}^{\infty}\abs{H_0\left(x+\im z\right)}^2\dz\leq C\int_{-1}^{1}\abs{H_0\left(x+\im z\right)}^2\dz+\int_{\R\setminus \left(-1,1\right)}\abs{H_0\left(x+\im z\right)}^2\dz\\
  &\leq C\left(p\right)+2C\left(p\right)\int_{1}^{\infty}\left(\frac{1}{\left(x^2+z^2\right)^{\frac{2-\delta}{2}}}+\frac{1}{\left(x^2+z^2\right)^{\frac{1-\alpha}{2}}}+\frac{1}{\left(x^2+z^2\right)^{\frac{2-3\alpha}{2}}}\right)^2\dz\\
  &\leq C\left(p\right)+2C\left(p\right)\int_{1}^{\infty}\frac{1}{z^{2\left(2-\delta\right)}}+\frac{1}{z^{2\left(1-\alpha\right)}}+\frac{1}{z^{2\left(2-3\alpha\right)}}\dz\leq C\left(p\right),
 \end{split}
\end{equation*}
where the constant is independent of $x$ and the last integral is finite as $\alpha<1/2$. For $x\geq 1$ we find similarly
\begin{equation*}
 \begin{split}
  &\int_{-\infty}^{\infty}\abs{H_0\left(x+\im z\right)}^2\dz\leq 2C\left(p\right)\int_{0}^{\infty}\left(\frac{1}{\left(x^2+z^2\right)^{\frac{2-\delta}{2}}}+\frac{1}{\left(x^2+z^2\right)^{\frac{1-\alpha}{2}}}+\frac{1}{\left(x^2+z^2\right)^{\frac{2-3\alpha}{2}}}\right)^2\dz\\
  &=C\left(p\right)\int_{0}^{1}\left(\cdots\right)\dz+C\left(p\right)\int_{1}^{\infty}\left(\cdots\right)\dz\\
  &\leq C\left(p\right)\left(\frac{1}{x^{2\left(2-\delta\right)}}+\frac{1}{x^{2\left(1-\alpha\right)}}+\frac{1}{x^{2\left(2-3\alpha\right)}}\right)+C\left(p\right)\int_{1}^{\infty}\frac{1}{z^{2\left(2-\delta\right)}}+\frac{1}{z^{2\left(1-\alpha\right)}}+\frac{1}{z^{2\left(2-3\alpha\right)}}\dz\leq C\left(p\right)
 \end{split}
\end{equation*}
independent of $x$, where in the last step we used $x\geq 1$ as well as $\alpha<1/2$ to get a finite integral. This then shows $H_0\in H^2\left(0\right)$ and thus finishes the proof. 
\end{proof}

\section{Integral estimate on $Q_{0}$}\label{Sec:int:est:Q0}

In this section we prove an integral estimate on $Q_0\left(\cdot,p\right)$ that is used in the proof of Proposition~\ref{Prop:est:zeta:m}.

\begin{proposition}\label{Prop:Q0:estimate}
 There exists a constant $C>0$ and $\nu=\nu\left(\alpha,\theta\right)>0$ sufficiently small such that the function $Q_0$ given by~\eqref{eq:Q0} in Proposition~\ref{Prop:rep:H0} satisfies
 \begin{equation}\label{eq:int:Q0}
  \int_{0}^{\infty}\frac{\abs{Q_0\left(\xi,p\right)}}{\left(1+\xi\right)^{2\theta-\nu}}\dxi\leq C\frac{\left(1+p\right)^{1-\theta}}{p^3} \quad \text{for all } p>0.
 \end{equation}
\end{proposition}

\begin{proof}
We first note that $Q_0\left(\cdot,p\right)$ is obtained as a limit as $R\to\infty$ in~\eqref{eq:Q0}. Thus we need to obtain estimates uniform in $R$ in order to show the claim. We recall the definition of $H_{0,1}$ and $H_{0,2}$ from~\eqref{eq:def:H0} and define in view of~\eqref{eq:Q0} for $R>0$ the corresponding truncated functions
\begin{align*}
 Q_{k,R}\left(\xi,p\right)&\vcc=\frac{1}{2\pi\im}\int_{-\im R}^{\im R}\ee^{\Y\xi}H_{0,k}\left(\Y,p\right)\dd{\Y}, \quad k=1,2.
\end{align*}
The general strategy of the proof is then the following. We will split the expressions $Q_{k,R}$ as
\begin{equation}\label{eq:QkR:splitting}
 Q_{k,R}\left(\xi,p\right)=L_{k}\left(\xi,p\right)+B_{k,R}\left(\xi,p\right)
\end{equation}
and show that for $L_k\left(\xi,p\right)$ we obtain suitable decay estimates in $\xi$ and $p$. On the other hand we will see that for all $p>0$ the expressions $B_{R,k}\left(\xi,p\right)$ converge to zero for a.e. $\xi>0$ as $R\to \infty$.

From~\eqref{eq:Q0} we know that for all $p>0$ there exists some subsequence $R_n=R_n\left(p\right)\to\infty$ as $n\to\infty$ such that 
\begin{equation*}
 Q_{0}\left(\xi,p\right)=\lim_{n\to\infty}\left(Q_{1,R_n}\left(\xi,p\right)+Q_{2,R_n}\left(\xi,p\right)\right) \quad \text{for a.e. } \xi\in\R_{+}.
\end{equation*}
Then using the splitting~\eqref{eq:QkR:splitting} it follows for all $p>0$ and a.e. $\xi\in\R_{+}$
\begin{equation*}
 \begin{split}
   \abs{Q_0\left(\xi,p\right)}&\leq \lim_{n\to\infty}\left(\abs{B_{1,R_n}\left(\xi,p\right)}+\abs{B_{2,R_n}\left(\xi,p\right)}\right)+\abs{L_1\left(\xi,p\right)}+\abs{L_2\left(\xi,p\right)}=\abs{L_1\left(\xi,p\right)}+\abs{L_2\left(\xi,p\right)}.
 \end{split}
\end{equation*}
The general strategy in order to obtain estimates for $L_{k}$ and $B_{k,R}$ will now be as follows. First we integrate by parts in the expressions $Q_{k,R}\left(\xi,p\right)$ in the variable $x$ (in $H_{0,k}$). Then we have to consider the regions $\xi<p$ and $\xi>p$ separately and integrate by parts either in $x$ again in the first case (i.e.\ for $\xi<p$) or in $\Y$ in the second one (i.e.\ for $\xi>p$) in order to obtain suitable decay estimates in $\xi$ and $p$.

Before we start estimating the different terms we will collect here some notation and general results that will be used frequently in the following. In order to separate the regions close to the origin and far from the origin in some integrals, we will have to introduce some cut-off functions both in the variable $x$ and $\Y$ that will be denoted by $\coup$ and $\cut$. Precisely, we take smooth functions $\coup,\cut\in C^{\infty}$ such that
\begin{equation}\label{eq:cutoff}
 \begin{split}
 0\leq \coup\leq 1,\quad \coup\left(-s\right)=\coup\left(s\right), \quad \coup\left(s\right)=0 \text{ for } \abs{s}\in\left[0,1/2\right] \quad \text{and} \quad \coup\left(s\right)=1 \text{ for } \abs{s}\in\left[1,\infty\right)\\
  0\leq \cut\leq 1,\quad \cut\left(-\Y\right)=\cut\left(\Y\right), \quad \cut\left(\Y\right)=1 \text{ for } \abs{\Y}\in\left[0,1\right] \quad \text{and} \quad \cut\left(\Y\right)=0 \text{ for } \abs{\Y}\in\left[2,\infty\right).  
 \end{split}
\end{equation}
 Furthermore we will use that for any $a\in \left[0,1\right]$ we have
 \begin{equation*}
  \abs{\ee^{-x}-1}\leq C_{a}\abs{x}^{a}\quad \text{for all } x \text{ with }\Re\left(x\right)\geq 0.
 \end{equation*}
 We will also often use that due to Fubini's Theorem we have
 \begin{equation*}
  \int_{-R}^{R}\int_{0}^{\Y}\left(\cdots\right)\ds\dd{\Y}=\int_{-R}^{R}\int_{s}^{\sgn\left(s\right)R}\left(\cdots\right)\dd{\Y}\ds.
 \end{equation*}
 Additionally, we will use several results concerning the precise asymptotic behaviour of the function $\Phi$ at zero. These estimates are collected and proved in Section~\ref{Sec:asymptotics}. We emphasize here that this is the only part where we will use the assumption~\eqref{eq:asymp:W} on the asymptotic behaviour of $W$ at the origin.
  
\subsection{Contribution of $H_{0,1}$}

We begin with the consideration of $Q_{1,R}$ and first integrate by parts in the expression for $H_{0,1}$ to get
\begin{equation*}
 \begin{split}
  H_{0,1}\left(\Y,p\right)&=\frac{2+\kappa}{p\Y\left(1+\Y\right)}\int_{0}^{\Y}\ee^{-px}x\left(\frac{x}{\Y}\right)^{\kappa}\frac{\ee^{\Phi\left(\Y\right)}}{\ee^{\Phi\left(x\right)}}\dx=-\frac{2+\kappa}{p^2\Y\left(1+\Y\right)}\int_{0}^{\Y}\del_{x}\left(\ee^{-px}\right)x\left(\frac{x}{\Y}\right)^{\kappa}\frac{\ee^{\Phi\left(\Y\right)}}{\ee^{\Phi\left(x\right)}}\dx\\
  &=-\frac{2+\kappa}{p^2\left(1+\Y\right)}\ee^{-p\Y}+\frac{\left(2+\kappa\right)\left(1+\kappa\right)}{p^2\left(1+\Y\right)\Y}\int_{0}^{\Y}\ee^{-px}\left(\frac{x}{\Y}\right)^{\kappa}\frac{\ee^{\Phi\left(\Y\right)}}{\ee^{\Phi\left(x\right)}}\dx\\
  &\quad-\frac{2+\kappa}{p^2\Y\left(1+\Y\right)}\int_{0}^{\Y}\ee^{-px}\Phi'\left(x\right)x\left(\frac{x}{\Y}\right)^{\kappa}\frac{\ee^{\Phi\left(\Y\right)}}{\ee^{\Phi\left(x\right)}}\dx=\vcc W_{1,0}\left(\Y,p\right)+W_{1,1}\left(\Y,p\right)+W_{1,2}\left(\Y,p\right).
 \end{split}
\end{equation*}

\subsubsection{Contribution of $W_{1,0}$}

We consider first the boundary part $W_{1,0}$ whose contribution to $Q_{1,R}$ is given by
\begin{equation}\label{eq:rand:1}
 \begin{split}
  &\quad-\frac{2+\kappa}{2p^2\pi\im}\int_{-\im R}^{\im R}\frac{\ee^{\Y\left(\xi-p\right)}}{1+\Y}\dd{\Y}=-\frac{2+\kappa}{2p^2\pi\im}\frac{1}{\xi-p}\int_{-\im R}^{\im R}\frac{\del_\Y\left(\ee^{\Y\left(\xi-p\right)}-1\right)}{1+\Y}\dd{\Y}\\
  &=-\frac{2+\kappa}{2p^2\pi\im}\frac{1}{\xi-p}\left(\frac{\ee^{\im R\left(\xi-p\right)}-1}{1+\im R}-\frac{\ee^{-\im R\left(\xi-p\right)}-1}{1-\im R}\right)-\frac{2+\kappa}{2p^2\pi\im}\frac{1}{\xi-p}\int_{-\im R}^{\im R}\frac{\ee^{\Y\left(\xi-p\right)}-1}{\left(1+\Y\right)^2}\dd{\Y}.
 \end{split}
\end{equation}
We first note that for $\xi\neq p$ we can estimate
\begin{equation*}
 \begin{split}
  \abs{\frac{2+\kappa}{2p^2\pi\im}\frac{1}{\xi-p}\left(\frac{\ee^{\im R\left(\xi-p\right)}-1}{1+\im R}-\frac{\ee^{-\im R\left(\xi-p\right)}-1}{1-\im R}\right)}\leq \frac{C}{p^2}\frac{1}{R\abs{\xi-p}}\longrightarrow 0\quad \text{as } R\longrightarrow\infty.
 \end{split}
\end{equation*}

To estimate the remaining integral in~\eqref{eq:rand:1} we have to consider the regions $\xi<p$ and $\xi>p$ separately using contour deformation and the analyticity of the integrand in $\C\setminus\left\{-1\right\}$.

  \begin{figure}
\begin{center}
  \begin{tikzpicture}[contour/.style={postaction={decorate, decoration={markings,
mark=at position 1.75cm with {\arrow[line width=1pt]{>}},
mark=at position 5.25cm with {\arrow[line width=1pt]{>}},
mark=at position 9.625cm with {\arrow[line width=1pt]{>}},
mark=at position 14.875cm with {\arrow[line width=1pt]{>}}}}},
contourtwo/.style={postaction={decorate, decoration={markings, mark=at position 2.625cm with {\arrow[line width=1pt]{>}},mark=at position 7.875cm with {\arrow[line width=1pt]{>}}}}}]

\draw[->] (-4,0) -- (5,0) node[below] {$\Re$};
\draw[->] (0,-4) -- (0,4.5) node[left]  {$\Im$};

\path[draw,line width=0.8pt, contour] (0,-3.5) node[above left]{$-R$}  -- (0,3.5) node[below left]{$R$}  arc (90:-90:3.5);

\path[draw, dashed, line width=0.8pt,contourtwo] (0,3.5) arc (90:270:3.5);

\node[below right] at (3.5,0) {$R$};
\draw[black,fill=black] (-1,0) circle (.4ex);
\node[below] at (-1,0){$-1$};
\end{tikzpicture}
\caption{Contour for $W_{1,0}$}
\label{fig:W10}
 \end{center}
\end{figure}
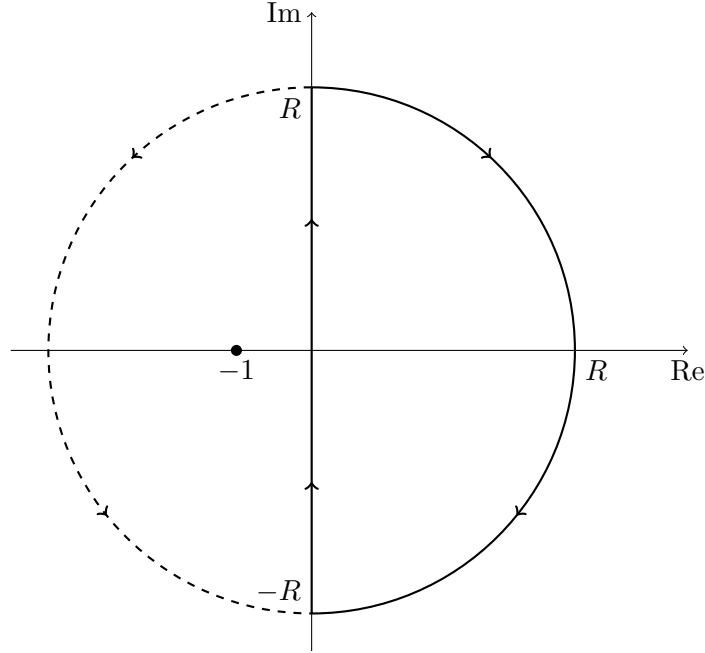

Starting with $\xi<p$ we will show that the integral over the half-circle of radius $R$ in the right half-plane (see Figure~\ref{fig:W10}) converges to $0$ if $R\to\infty$ which finally shows then that there is no contribution of $W_{1,0}\left(\cdot,p\right)$ to $Q_{1,R}$ in the limit $R\to \infty$ for $\xi<p$. Precisely, proceeding similarly as before we see
\begin{equation}\label{eq:rand:3}
 \begin{split}
  \abs{\frac{2+\kappa}{2p^2\pi\im}\frac{1}{\xi-p}\int_{-\im R}^{\im R}\frac{\ee^{\Y\left(\xi-p\right)}-1}{\left(1+\Y\right)^2}\dd{\Y}}=\abs{\frac{2+\kappa}{2p^2\pi\im} \frac{\im R}{\xi-p}\int_{-\pi/2}^{\pi/2}\frac{\ee^{R\ee^{\im s}\left(\xi-p\right)}-1}{\left(1+R\ee^{\im s}\right)^2}\ee^{\im s}\ds}\leq \frac{C}{p^2}\frac{1}{R\abs{\xi-p}},
 \end{split}
\end{equation}
where the right-hand side converges to zero for $R\to\infty$.

It remains to consider the region $\xi>p$ and to estimate the expression
\begin{equation*}
 \frac{2+\kappa}{2p^2\pi\im}\frac{1}{\xi-p}\int_{-\im R}^{\im R}\frac{\ee^{\Y\left(\xi-p\right)}-1}{\left(1+\Y\right)^2}\dd{\Y}
\end{equation*}
there, which will be done similarly as before. As $\xi>p$ we have to choose an integration path in the left half-plane and we can show in exactly the same way as in~\eqref{eq:rand:3} that the integral over the half-circle of radius $R$ in the left half-plane converges to zero for $R\to\infty$. Thus, by the residue theorem we find
\begin{equation*}
 \begin{split}
  -\frac{2+\kappa}{2p^2\pi\im}\frac{1}{\xi-p}\int_{-\im R}^{\im R}\frac{\ee^{\Y\left(\xi-p\right)}-1}{\left(1+\Y\right)^2}\dd{\Y}\longrightarrow -\frac{2+\kappa}{p^2}\ee^{-\left(\xi-p\right)},\quad \text{for }\xi>p\text{ and } R\longrightarrow \infty.
 \end{split}
\end{equation*}
 From this we see that the contribution of $W_{1,0}\left(\cdot,p\right)$ to the integral in~\eqref{eq:int:Q0} can be estimated by
 \begin{equation*}
  \begin{split}
   \frac{\abs{2+\kappa}}{p^2}\int_{p}^{\infty}\frac{\ee^{-\left(\xi-p\right)}}{\left(1+\xi\right)^{2\theta-\nu}}\dxi\leq \frac{C}{p^{2+2\theta-\nu}}\int_{0}^{\infty}\ee^{-\xi}\dxi\leq C\frac{\left(1+p\right)^{1-\theta}}{p^3}
  \end{split}
 \end{equation*}
 for sufficiently small $\nu>0$.

 \subsubsection{Contribution of $W_{1,1}$ for $\xi<p$}
 
Next we consider the contribution coming from $W_{1,1}$ and we again first assume $\xi<p$. Thus we integrate by parts in $x$ to obtain
\begin{equation*}
 \begin{split}
  W_{1,1}\left(\Y,p\right)&=-\frac{\left(2+\kappa\right)\left(1+\kappa\right)}{p^3\Y\left(1+\Y\right)}\int_{0}^{\Y}\del_{x}\left(\ee^{-px}-1\right)\left(\frac{x}{\Y}\right)^{\kappa}\frac{\ee^{\Phi\left(\Y\right)}}{\ee^{\Phi\left(x\right)}}\dx\\
  &=-\frac{\left(2+\kappa\right)\left(1+\kappa\right)}{p^3\Y\left(1+\Y\right)}\left(\ee^{-p\Y}-1\right)+\frac{\left(2+\kappa\right)\left(1+\kappa\right)}{p^3\Y\left(1+\Y\right)}\int_{0}^{\Y}\left(\ee^{-px}-1\right)\del_{x}\left(\left(\frac{x}{\Y}\right)^{\kappa}\frac{\ee^{\Phi\left(\Y\right)}}{\ee^{\Phi\left(x\right)}}\right)\dx\\
  &=\vcc W_{1,1,1}\left(\Y,p\right)+W_{1,1,2}\left(\Y,p\right).
 \end{split}
\end{equation*}
We consider the two terms separately and first note that $W_{1,1,1}$ is analytic in $\C\setminus\left\{-1\right\}$. Thus, instead of integrating along the imaginary axis (from $-\im R$ to $\im R$) we can deform the integration path in such a way that, around the origin, we integrate over some half-circle (of radius one for example) in the right half-plane (see Figure~\ref{fig:W11}).

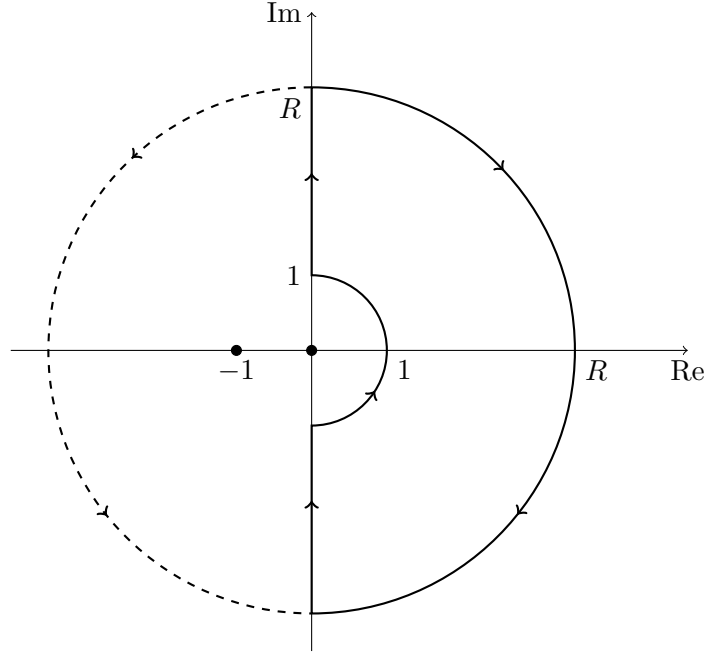
\begin{figure}
\begin{center}
  \begin{tikzpicture}[contour/.style={postaction={decorate,decoration={markings,
mark=at position 1.5cm with {\arrow[line width=1pt]{>}},
mark=at position 3.5cm with {\arrow[line width=1pt]{>}},
mark=at position 7cm with {\arrow[line width=1pt]{>}},
mark=at position 11cm with {\arrow[line width=1pt]{>}},
mark=at position 16cm with {\arrow[line width=1pt]{>}}}}},
contourtwo/.style={postaction={decorate,decoration={markings, mark=at position 2.625cm with {\arrow[line width=1pt]{>}},mark=at position 7.875cm with {\arrow[line width=1pt]{>}}}}}]

\draw[->] (-4,0) -- (5,0) node[below] {$\Re$};
\draw[->] (0,-4) -- (0,4.5) node[left] {$\Im$};

\path[draw,line width=0.8pt,contour] (0,-3.5)  -- (0,-1)  arc (-90:90:1) node[left] {$1$} -- (0, 3.5)node[below left] {$R$} arc (90:-90:3.5);

\path[draw, dashed, line width=0.8pt,contourtwo] (0,3.5) arc (90:270:3.5);

\node[below right] at (3.5,0) {$R$};
\node[below right] at (1,0) {$1$};
\draw[black,fill=black] (0,0) circle (.4ex);
\draw[black,fill=black] (-1,0) circle (.4ex);
\node[below] at (-1,0) {$-1$};
\end{tikzpicture}
\caption{Contour for $W_{1,1,1}$}
\label{fig:W11}
 \end{center}
\end{figure}

Splitting furthermore the two terms, we have to consider (while the integral has to be understood along the path just explained)
\begin{equation*}
 \begin{split}
  \frac{1}{2\pi\im}\int_{-\im R}^{\im R}\frac{\ee^{\Y\left(\xi-p\right)}}{\Y\left(1+\Y\right)}\dd{\Y}-\frac{1}{2\pi\im}\int_{-\im R}^{\im R}\frac{\ee^{\Y\xi}}{\Y\left(1+\Y\right)}\dd{\Y}.
 \end{split}
\end{equation*}
We evaluate these integrals now using residues and in the first integral we close the contour by some half-circle of radius $R$ in the right half-plane while in the second one we take the half-circle in the left half-plane. Due to the decay of the integrand at $\infty$ one immediately sees that in the limit $R\to\infty$ the integrals over these large half-circles vanish. Thus we obtain by the Residue Theorem that
\begin{equation*}
 \begin{split}
  \lim_{R\to\infty}\left(\frac{1}{2\pi\im}\int_{-\im R}^{\im R}\frac{\ee^{\Y\left(\xi-p\right)}}{\Y\left(1+\Y\right)}\dd{\Y}-\frac{1}{2\pi\im}\int_{-\im R}^{\im R}\frac{\ee^{\Y\xi}}{\Y\left(1+\Y\right)}\dd{\Y}\right)=\ee^{-\xi}-1.
 \end{split}
\end{equation*}
 Thus, we find as contribution from $W_{1,1,1}$ to the integral in~\eqref{eq:int:Q0}, if $\nu>0$ is sufficiently small, that
 \begin{equation*}
  \begin{split}
   \frac{C}{p^3}\int_{0}^{p}\frac{1-\ee^{-\xi}}{\left(1+\xi\right)^{2\theta-\nu}}\dxi\leq \frac{C}{p^3}p^{1-2\theta+\nu}\leq C\frac{\left(1+p\right)^{1-\theta}}{p^3}.
  \end{split}
 \end{equation*}
We next consider the contribution of $W_{1,1,2}$ and first estimate for $a\in\left(\alpha,\theta\right)$ (e.g. $a=(\alpha+\theta)/2$) that
\begin{equation*}
 \begin{split}
  \abs{W_{1,1,2}\left(\Y,p\right)}&=\abs{\frac{\left(2+\kappa\right)\left(1+\kappa\right)}{p^3\Y\left(1+\Y\right)}\int_{0}^{\Y}\left(\ee^{-px}-1\right)\left(\frac{\kappa}{x}-\Phi'\left(x\right)\right)\left(\frac{x}{\Y}\right)^{\kappa}\frac{\ee^{\Phi\left(\Y\right)}}{\ee^{\Phi\left(x\right)}}\dx}\\
  &\leq C\frac{1}{p^3\abs{1+\Y}}\int_{0}^{1}\abs{\ee^{-ps\Y}-1}\max\left\{\abs{s\Y}^{-\alpha-1},\abs{s\Y}^{\alpha-1}\right\}s^{\kappa}\ds\\
  &\leq C\frac{\abs{\Y}^{a}}{p^3\abs{1+\Y}}p^a\int_{0}^{1}s^{a+\kappa}\max\left\{\abs{s\Y}^{-\alpha-1},\abs{s\Y}^{\alpha-1}\right\}\ds\\
  &\leq C\left(a\right)\frac{1}{p^{3-a}}\frac{\max\left\{\abs{\Y}^{a-\alpha},\abs{\Y}^{a+\alpha}\right\}}{\abs{\Y}\abs{1+\Y}}.
 \end{split}
\end{equation*}
Note that the last integral is bounded as $a>\alpha$ and this is enough to have $a+\kappa>\alpha$ for sufficiently small $\eps$ due to Remark~\ref{Rem:smallness:kappa}. Using this we see that we can take the limit $R\to \infty$ in the integral $\int_{-\im R}^{\im R} \ee^{\Y\xi}W_{1,1,2}\left(\Y,p\right)\dd{\Y}$ and then further estimate
\begin{equation*}
 \begin{split}
  \abs{\frac{1}{2\pi\im}\int_{-\im\infty}^{\im\infty}\ee^{\Y\xi}W_{1,1,2}\left(\Y,p\right)\dd{\Y}}\leq \frac{C}{p^{3-a}}\int_{-\infty}^{\infty}\frac{\abs{\ee^{\im \Y\xi}}}{\abs{\Y}\abs{1+\Y}}\max\left\{\abs{\Y}^{a-\alpha},\abs{\Y}^{a+\alpha}\right\}\dd{\Y}\leq \frac{C\left(a,\delta\right)}{p^{3-a}}
 \end{split}
\end{equation*}
by the choice of $a$. For $a$ fixed we can then choose $\nu>0$ sufficiently small such that the contribution of $W_{1,1,2}$ to the integral in~\eqref{eq:int:Q0} can be estimated by (note we assume $\xi<p$)
\begin{equation*}
 \begin{split}
  \frac{C\left(a,\delta\right)}{p^{3-a}}\int_{0}^{p}\frac{1}{\xi^{2\theta-\nu}}\dxi\leq \frac{C\left(a,\delta\right)}{p^{2+2\theta-a-\nu}}\leq C\frac{\left(1+p\right)^{1-\theta}}{p^3}.
 \end{split}
\end{equation*}

\subsubsection{Contribution of $W_{1,1}$ for $\xi>p$}

Now we consider the contribution from $W_{1,1}$ in the region $\xi>p$ and therefore we integrate by parts in $\Y$ to obtain
\begin{equation*}
 \begin{split}
  &\quad\frac{1}{2\pi\im}\int_{-\im R}^{\im R} \ee^{\Y\xi}W_{1,1}\left(\Y,p\right)\dd{\Y}=-\frac{\left(2+\kappa\right)\left(1+\kappa\right)}{2\pi\im p^2 \xi}\int_{-\im R}^{\im R}\frac{\del_{\Y}\left(\ee^{\Y\left(\xi-1\right)}-1\right)}{\Y^{1+\kappa}\left(1+\Y\right)}\int_{0}^{\Y}\ee^{-px}x^{\kappa}\frac{\ee^{\Phi\left(\Y\right)}}{\ee^{\Phi\left(x\right)}}\dx\dd{\Y}\\
  &=-\frac{\left(2+\kappa\right)\left(1+\kappa\right)}{2\pi\im p^2 \xi}\left.\left(\frac{\del_{\Y}\left(\ee^{\Y\left(\xi-1\right)}-1\right)}{\Y^{1+\kappa}\left(1+\Y\right)}\int_{0}^{\Y}\ee^{-px}x^{\kappa}\frac{\ee^{\Phi\left(\Y\right)}}{\ee^{\Phi\left(x\right)}}\dx\right)\right|_{\Y=-\im R}^{\Y=\im R}\\
  &\quad -\frac{\left(2+\kappa\right)\left(1+\kappa\right)}{2\pi\im p^2 \xi}\left(\int_{-\im R}^{\im R}\left(\ee^{\Y\xi}-1\right)\left(\frac{1+\kappa}{\Y^{2+\kappa}\left(1+\Y\right)}+\frac{1}{\Y^{1+\kappa}\left(1+\Y\right)^2}\right)\int_{0}^{\Y}\ee^{-px}x^{\kappa}\frac{\ee^{\Phi\left(\Y\right)}}{\ee^{\Phi\left(x\right)}}\dx\dd{\Y}\right.\\
  &\quad \left. -\int_{-\im R}^{\im R}\frac{\ee^{\Y\xi}-1}{\Y\left(1+\Y\right)}\ee^{-p\Y}\dd{\Y}-\int_{\im R}^{\im R}\frac{\ee^{\Y\xi}-1}{\Y^{1+\kappa}\left(1+\Y\right)}\Phi'\left(\Y\right)\int_{0}^{\Y}\ee^{-px}x^{\kappa}\frac{\ee^{\Phi\left(\Y\right)}}{\ee^{\Phi\left(x\right)}}\dx\dd{\Y}  \right)\\
  &=(I)+\big((II)+(III)+(IV)\big).
 \end{split}
\end{equation*}
We estimate these terms now separately. First we have
\begin{equation}\label{eq:W11:I}
 \begin{split}
  \abs{(I)}\leq \frac{C}{p^2\xi}\left(\frac{\abs{\ee^{\im R\xi}-1}}{R^{1+\kappa}\abs{1+\im R}}+\frac{\abs{\ee^{-\im R\xi}-1}}{R^{1+\kappa}\abs{1-\im R}}\right)\int_{0}^{R}s^{\kappa}\ds\leq \frac{C}{R}\frac{1}{p^2\xi}\longrightarrow 0\quad \text{as } R\longrightarrow\infty.
 \end{split}
\end{equation}
In order to estimate the remaining terms we fix
\begin{equation}\label{eq:b:W11}
 b\in\left(\alpha,\theta\right).
\end{equation}
Then to estimate $(II)$ we use that for $\Y\in \im \R$ we have
\begin{equation*}
 \abs{\int_{0}^{\Y}\ee^{-px}x^{\kappa}\frac{\ee^{\Phi\left(\Y\right)}}{\ee^{\Phi\left(x\right)}}\dx}\leq C\abs{\Y}^{1+\kappa}.
\end{equation*}
Then we obtain
\begin{equation}\label{eq:W11:II}
 \begin{split}
  \abs{(II)}\leq \frac{C}{p^2\xi}\int_{-R}^{R}\frac{\abs{\ee^{\im \Y\xi}-1}}{\abs{\Y}^{2+\kappa}\left(1+\abs{\Y}\right)}\abs{\Y}^{1+\kappa}\dd{\Y}\leq \frac{C}{p^2\xi^{1-b}}\int_{0}^{R}\frac{\Y^b}{\Y\left(1+\Y\right)}\dd{\Y}\leq \frac{C}{p^2\xi^{1-b}}.
 \end{split}
\end{equation}
The term $(III)$ can be treated exactly as the term $W_{1,1,1}$ before and we obtain
\begin{equation*}
 \abs{(III)}\leq \frac{C}{p^2\xi^{1-b}}.
\end{equation*}
 It thus remains to consider the terms $(IV)$ and we find similarly as before
\begin{equation}\label{eq:W11:IV}
 \begin{split}
  \abs{(IV)}&=\frac{C}{p^2\xi}\abs{\int_{\im R}^{\im R}\frac{\ee^{\Y\xi}-1}{\Y^{1+\kappa}\left(1+\Y\right)}\Phi'\left(\Y\right)\int_{0}^{\Y}\ee^{-px}x^{\kappa}\frac{\ee^{\Phi\left(\Y\right)}}{\ee^{\Phi\left(x\right)}}\dx\dd{\Y}}\\
  &\leq \frac{C\eps}{p^2\xi^{1-b}}\int_{-R}^{R}\frac{\abs{\Y}^{b}}{\abs{1+\Y}}\max\left\{\abs{\Y}^{\alpha-1},\abs{\Y}^{-\alpha-1}\right\}\dd{\Y}\leq \frac{C}{p^2\xi^{1-b}},
 \end{split}
\end{equation}
where the last integral is bounded uniformly in $R$ due to the choice of $b$ in~\eqref{eq:b:W11}. Thus taking~\eqref{eq:W11:I} and \eqref{eq:W11:II}\textendash\eqref{eq:W11:IV} together we see that in the limit $R\to\infty$ we have the estimate
\begin{equation*}
 \limsup_{R\to\infty}\abs{\frac{1}{2\pi\im}\int_{-\im R}^{\im R} \ee^{\Y\xi}W_{1,1}\left(\Y,p\right)\dd{\Y}}\leq \frac{C}{p^2\xi^{1-b}}.
\end{equation*}
 Thus, the contribution of $W_{1,1}$ to the integral in~\eqref{eq:int:Q0} can then be estimated by
\begin{equation*}
 \begin{split}
  \frac{C}{p^2}\int_{p}^{\infty}\frac{1}{\xi^{2\theta-\nu+1-b}}\dxi\leq \frac{C}{p^{2+2\theta-b-\nu}}\leq C\frac{\left(1+p\right)^{1-\theta}}{p^3}
 \end{split}
\end{equation*}
if $\nu>0$ is sufficiently small (for $b$ fixed as in~\eqref{eq:b:W11}).

\subsection{Contribution of $W_{1,2}$}

Next we consider the contribution coming from $W_{1,2}\left(\Y,p\right)$ and first recall
\begin{equation*}
 W_{1,2}\left(\Y,p\right)=-\frac{2+\kappa}{p^2\Y\left(1+\Y\right)}\int_{0}^{\Y}\ee^{-px}\Phi'\left(x\right)x\left(\frac{x}{\Y}\right)^{\kappa}\frac{\ee^{\Phi\left(\Y\right)}}{\ee^{\Phi\left(x\right)}}\dx.
\end{equation*}
The strategy here is in principle the same as before, i.e.\ we consider the regions $\xi<p$ and $\xi>p$ separately and integrate by parts to get enough decay in $p$ and, depending on the region, also in $\xi$. But due to the occurrence of $\Phi'$ things get more involved and we have to take into account also the precise asymptotic behaviour of $\Phi$ at the origin. 

\subsubsection{Contribution of $W_{1,2}$ for $\xi<p$}

We start again by looking at the region $\xi<p$ and we first rewrite $W_{1,2}$ using the cut-off $\coup$ introduced in~\eqref{eq:cutoff} to obtain
\begin{equation*}
 \begin{split}
  W_{1,2}\left(\Y,p\right)&=-\frac{2+\kappa}{p^2\Y\left(1+\Y\right)}\int_{0}^{\Y}\ee^{-px}\Phi'\left(x\right)x\left(\frac{x}{\Y}\right)^{\kappa}\frac{\ee^{\Phi\left(\Y\right)}}{\ee^{\Phi\left(x\right)}}\coup\left(\im x\right)\dx\\
  &\quad-\frac{2+\kappa}{p^2\Y\left(1+\Y\right)}\int_{0}^{\Y}\ee^{-px}\Phi'\left(x\right)x\left(\frac{x}{\Y}\right)^{\kappa}\frac{\ee^{\Phi\left(\Y\right)}}{\ee^{\Phi\left(x\right)}}\left(1-\coup\left(\im x\right)\right)\dx=\vcc W_{1,2,1}\left(\Y,p\right)+W_{1,2,2}\left(\Y,p\right).
 \end{split}
\end{equation*}
Note that $\coup$ is a real function on the real line. We first consider $W_{1,2,1}$ which is the easier part because due to the cut-off $\coup$, we can integrate by parts again. Using furthermore that $\Phi'\left(x\right)=-\eps\beta_{W}\left(x\right)x^{-1}\ee^{-x}$ we obtain
\begin{equation*}
 \begin{split}
  W_{1,2,1}\left(\Y,p\right)&=\frac{\left(2+\kappa\right)\eps}{p^2\Y\left(1+\Y\right)}\int_{0}^{\Y}\ee^{-\left(p+1\right)x}\beta_W\left(x\right)\left(\frac{x}{\Y}\right)^{\kappa}\frac{\ee^{\Phi\left(\Y\right)}}{\ee^{\Phi\left(x\right)}}\coup\left(\im x\right)\dx\\
  &=-\frac{\left(2+\kappa\right)\eps}{p^2\left(1+p\right)}\frac{1}{\Y\left(1+\Y\right)}\int_{0}^{\Y}\del_{x}\left(\ee^{-\left(1+p\right)x}\right)\beta_W\left(x\right)\left(\frac{x}{\Y}\right)^{\kappa}\frac{\ee^{\Phi\left(\Y\right)}}{\ee^{\Phi\left(x\right)}}\coup\left(\im x\right)\dx.
 \end{split}
\end{equation*}
Integrating by parts this gives
\begin{equation*}
 \begin{split}
  W_{1,2,1}\left(\Y,p\right)&=-\frac{\left(2+\kappa\right)\eps}{p^2\left(1+p\right)}\frac{\ee^{-\left(1+p\right)\Y}\coup\left(\Y\right)}{\Y\left(1+\Y\right)}\beta_W\left(\Y\right)\\
  &\quad+\frac{\left(2+\kappa\right)\eps}{p^2\left(1+p\right)}\frac{1}{\Y\left(1+\Y\right)}\int_{0}^{\Y}\ee^{-\left(1+p\right)x}\del_{x}\left(\beta_W\left(x\right)\left(\frac{x}{\Y}\right)^{\kappa}\frac{\ee^{\Phi\left(\Y\right)}}{\ee^{\Phi\left(x\right)}}\coup\left(\im x\right)\right)\dx\\
  &=\vcc W_{1,2,1,1}\left(\Y,p\right)+W_{1,2,1,2}\left(\Y,p\right).
 \end{split}
\end{equation*}
First we have due to the estimates on $\beta_W$ from~\eqref{eq:est:beta:W:0},~\eqref{eq:est:beta:W} together with the properties of $\coup$ that
\begin{align*}
  \abs{W_{1,2,1,1}\left(\Y,p\right)}&\leq C\frac{\abs{\Y}^{\alpha}}{\abs{\Y}\abs{1+\Y}}\chi_{\left\{\abs{\Y}\geq 1/2\right\}}.
\end{align*}
Furthermore expanding the derivative we find
\begin{equation*}
 \begin{split}
  &\quad \del_{x}\left(\beta_W\left(x\right)\left(\frac{x}{\Y}\right)^{\kappa}\frac{\ee^{\Phi\left(\Y\right)}}{\ee^{\Phi\left(x\right)}}\coup\left(\im x\right)\right)\\
  &=\left(\beta_W'\left(x\right)\coup\left(\im x\right)+\frac{\kappa}{x}\beta_W\left(x\right)\coup\left(\im x\right)-\beta_W\left(x\right)\Phi'\left(x\right)\coup\left(\im x\right)+\im\beta_W\left(x\right)\coup'\left(\im x\right)\right)\left(\frac{x}{\Y}\right)^{\kappa}\frac{\ee^{\Phi\left(\Y\right)}}{\ee^{\Phi\left(x\right)}}.
 \end{split}
\end{equation*}
Estimating by the most singular term at infinity and using that $\coup$ is supported away from the origin we find
\begin{equation*}
 \begin{split}
  \abs{\del_{x}\left(\beta_W\left(x\right)\left(\frac{x}{\Y}\right)^{\kappa}\frac{\ee^{\Phi\left(\Y\right)}}{\ee^{\Phi\left(x\right)}}\coup\left(\im x\right)\right)}\leq \frac{C}{\abs{\Y}^{\kappa}}\abs{x}^{\kappa-1+2\alpha}\chi_{\left\{\abs{x}\geq \frac{1}{2}\right\}}.
 \end{split}
\end{equation*}
This then yields
\begin{equation*}
 \begin{split}
  \abs{\int_{0}^{\Y}\ee^{-\left(1+p\right)x}\del_{x}\left(\beta_W\left(x\right)\left(\frac{x}{\Y}\right)^{\kappa}\frac{\ee^{\Phi\left(\Y\right)}}{\ee^{\Phi\left(x\right)}}\coup\left(\im x\right)\right)\dx}\leq C\abs{\Y}^{2\alpha}\chi_{\left\{\abs{\Y}\geq \frac{1}{2}\right\}}.
 \end{split}
\end{equation*}
Thus, estimating again by the most singular terms we conclude
\begin{equation*}
 \begin{split}
  \abs{\frac{1}{2\pi\im}\int_{-R\im}^{R\im}\ee^{\Y\xi}W_{1,2,1}\left(\Y,p\right)\dd{\Y}}\leq \frac{C}{p^2\left(1+p\right)}\int_{\left(-R,R\right)\setminus\left(-\frac{1}{2},\frac{1}{2}\right)}\frac{1}{\abs{\Y}\abs{1+\Y}}\left(\abs{\Y}^{\alpha}+\abs{\Y}^{2\alpha}\right)\dd{\Y}\leq \frac{C}{p^2\left(1+p\right)},
 \end{split}
\end{equation*}
where the last estimate holds uniformly in $R$ and $\xi$. The contribution of $W_{1,2,1}$ to the integral in~\eqref{eq:int:Q0} can then be estimated by 
\begin{equation}\label{eq:contr:W121}
 \begin{split}
  \frac{C}{p^2\left(1+p\right)}\int_{0}^{p}\frac{1}{\xi^{2\theta-\nu}}\dxi\leq \frac{C}{p^{1+2\theta-\nu}\left(1+p\right)}\leq C\frac{\left(1+p\right)^{1-\theta}}{p^3}
 \end{split}
\end{equation}
if $\nu$ is sufficiently small.

We next estimate the expression coming from $W_{1,2,2}$ which is more involved. Note first that we can directly take the limit $R\to\infty$ due to the cut-off $1-\coup(\im x)$. Thus we have to consider
\begin{equation*}
 \begin{split}
  Q_{1,2,2}\vcc=-\frac{2+\kappa}{2\pi\im}\frac{1}{p^2}\int_{-\im\infty}^{\im\infty}\frac{\ee^{\Y\xi}}{\Y\left(1+\Y\right)}\int_{0}^{\Y}\ee^{-px}\Phi'\left(x\right)x\left(\frac{x}{\Y}\right)^{\kappa}\frac{\ee^{\Phi\left(\Y\right)}}{\ee^{\Phi\left(x\right)}}\left(1-\coup\left(\im x\right)\right)\dx\dd{\Y}.
 \end{split}
\end{equation*}
Transforming the integrals into integrals over real intervals and using Fubini's Theorem we can rewrite this as
\begin{equation*}
 \begin{split}
  Q_{1,2,2}&=-\frac{2+\kappa}{2\pi}\frac{1}{p^2}\int_{-\infty}^{\infty}\frac{\ee^{\im \Y\xi}}{\im \Y\left(1+\im \Y\right)}\int_{0}^{\im \Y}\ee^{-px}\Phi'\left(x\right)x\left(\frac{x}{\im \Y}\right)^{\kappa}\frac{\ee^{\Phi\left(\im \Y\right)}}{\ee^{\Phi\left(x\right)}}\left(1-\coup\left(\im x\right)\right)\dx\dd{\Y}\\
  &=-\frac{2+\kappa}{2\pi}\frac{1}{p^2}\int_{-\infty}^{\infty}\frac{\ee^{\im \Y\xi}}{\Y\left(1+\im \Y\right)}\int_{0}^{\Y}\ee^{-\im ps}\Phi'\left(\im x\right)\left(\im s\right)\left(\frac{s}{\Y}\right)^{\kappa}\frac{\ee^{\Phi\left(\im \Y\right)}}{\ee^{\Phi\left(\im s\right)}}\left(1-\coup\left(s\right)\right)\ds\dd{\Y}\\
  &=-\frac{2+\kappa}{2\pi}\frac{\im}{p^2}\int_{-\infty}^{\infty}\ee^{-\im ps}\Phi'\left(\im s\right)s^{1+\kappa}\left(1-\coup\left(s\right)\right)\ee^{-\Phi\left(\im s\right)}\int_{s}^{\sgn\left(s\right)\infty}\frac{\ee^{\im \Y \xi}}{ \Y^{1+\kappa}\left(1+\im \Y\right)}\ee^{\Phi\left(\im \Y\right)}\dd{\Y}\ds.
 \end{split}
\end{equation*}
Before continuing we have to insert the cut-off $\cut$ (see~\eqref{eq:cutoff}) in the variable $\Y$ to separate the regions $\abs{\Y}\leq 1$ and $\abs{\Y}\geq 1$, i.e.\ we split $Q_{1,2,2}=Q_{1,2,2,1}+Q_{1,2,2,2}$ with 
\begin{align*}
 Q_{1,2,2,1}&\vcc=-\frac{2+\kappa}{2\pi}\frac{\im}{p^2}\int_{-\infty}^{\infty}\ee^{-\im ps}\Phi'\left(\im s\right)s^{1+\kappa}\left(1-\coup\left(s\right)\right)\ee^{-\Phi\left(\im s\right)}\int_{s}^{\sgn\left(s\right)\infty}\frac{\ee^{\im \Y \xi}\ee^{\Phi\left(\im \Y\right)}\cut\left(\Y\right)}{\Y^{1+\kappa}\left(1+\im \Y\right)}\dd{\Y}\ds\\
 Q_{1,2,2,2}&\vcc=-\frac{2+\kappa}{2\pi}\frac{\im}{p^2}\int_{-\infty}^{\infty}\ee^{-\im ps}\Phi'\left(\im s\right)s^{1+\kappa}\left(1-\coup\left(s\right)\right)\ee^{-\Phi\left(\im s\right)}\int_{s}^{\sgn\left(s\right)\infty}\frac{\ee^{\im \Y \xi}\ee^{\Phi\left(\im \Y\right)}\left(1-\cut\left(\Y\right)\right)}{\Y^{1+\kappa}\left(1+\im \Y\right)}\dd{\Y}\ds.
\end{align*}
The second term, $Q_{1,2,2,2}$, is again easier to treat because we can integrate by parts in $s$. Doing this, using also the definition of $\Phi$, we get
\begin{equation}\label{eq:Q1222}
 \begin{split}
  Q_{1,2,2,2}&=\frac{\left(2+\kappa\right)\eps}{2\pi}\frac{1}{p^2}\int_{-\infty}^{\infty}\ee^{-\im s\left(p+1\right)}\beta_W\left(\im s\right)s^{\kappa}\left(1-\coup\left(s\right)\right)\ee^{-\Phi\left(\im s\right)}\int_{s}^{\sgn\left(s\right)\infty}\frac{\ee^{\im \Y \xi}\left(1-\cut\left(\Y\right)\right)}{\Y^{1+\kappa}\left(1+\im \Y\right)}\ee^{\Phi\left(\im \Y\right)}\dd{\Y}\ds\\
  &=-\frac{\left(2+\kappa\right)\eps}{2\pi}\frac{1}{p^2\left(1+p\right)}\int_{-\infty}^{\infty}\del_{s}\left(\ee^{-\im s\left(p+1\right)}-1\right)\beta_W\left(\im s\right)s^{\kappa}\left(1-\coup\left(s\right)\right)\ee^{-\Phi\left(\im s\right)}\int_{s}^{\sgn\left(s\right)\infty}\left(\cdots\right)\dd{\Y}\ds\\
  &=\frac{\left(2+\kappa\right)\eps}{2\pi}\frac{1}{p^2\left(1+p\right)}\int_{-\infty}^{\infty}\left(\ee^{-\im s\left(p+1\right)}-1\right)\del_{s}\left(\beta_W\left(\im s\right)s^{\kappa}\left(1-\coup\left(s\right)\right)\ee^{-\Phi\left(\im s\right)}\right)\int_{s}^{\sgn\left(s\right)\infty}\left(\cdots\right)\dd{\Y}\ds\\
  &\quad -\frac{\left(2+\kappa\right)\eps}{2\pi}\frac{1}{p^2\left(1+p\right)}\int_{-\infty}^{\infty}\left(\ee^{-\im s\left(p+1\right)}-1\right)\beta_W\left(\im s\right)\left(1-\coup\left(s\right)\right)\frac{\ee^{\im s \xi}}{s\left(1+\im s\right)}\left(1-\cut\left(s\right)\right)\ds. 
 \end{split}
\end{equation}
On the other hand it holds
\begin{equation*}
 \abs{\int_{s}^{\sgn\left(s\right)\infty}\frac{\ee^{\im \Y \xi}}{\Y^{1+\kappa}\left(1+\im \Y\right)}\ee^{\Phi\left(\im \Y\right)}\left(1-\cut\left(\Y\right)\right)\dd{\Y}}\leq C \quad \text{for all } s
\end{equation*}
 due to the cut-off $1-\cut(\Y)$ which is zero in $\left[-1,1\right]$ and the boundedness of $\ee^{\Phi\left(\im \Y\right)}$ for $\abs{\Y}\geq 1$. Furthermore, the cut-off $1-\coup\left(s\right)$ and Lemma~\ref{Lem:asymptotics:exp:Phi} yield
 \begin{equation*}
  \eps\abs{\del_{s}\left(\beta_W\left(\im s\right)s^{\kappa}\left(1-\coup\left(s\right)\right)\ee^{-\Phi\left(\im s\right)}\right)}\leq C\abs{s}^{\kappa-1}\chi_{\left\{\abs{s}\leq 1\right\}}.
 \end{equation*}
 Finally, using again that $1-\cut\left(s\right)=0$ for $-1\leq s\leq 1$ as well as $1-\coup\left(s\right)=0$ for $\abs{s}\geq 1$ we also get
\begin{equation}\label{eq:W12:v3}
 \abs{\ee^{-\im s\left(p+1\right)}\beta_W\left(\im s\right)\left(1-\coup\left(s\right)\right)\frac{\ee^{\im s \xi}}{s\left(1+\im s\right)}\left(1-\cut\left(s\right)\right)}=0  \quad\text{for all } s.
\end{equation}
Thus, taking these estimates together we conclude from~\eqref{eq:Q1222} together with $\abs{\ee^{-\im s\left(p+1\right)}-1}\leq C(p+1)^{a}\abs{s}^{a}$ with $a\in(\delta,\theta)$ that
\begin{equation*}
 \abs{Q_{1,2,2,2}}\leq \frac{C}{p^2\left(1+p\right)^{1-a}}
\end{equation*}
and similarly as in~\eqref{eq:contr:W121} this gives the desired contribution to the integral in~\eqref{eq:int:Q0} due to the choice of $a$.

\begin{remark}\label{Rem:est:bd:pert:v3}
 For bounded perturbations, i.e.\@ for $\alpha=0$ we estimate $W$ by a kernel $\widetilde{W}$ with positive $\tilde{\alpha}>0$. However, as explained in Remark~\ref{Rem:alpha:pos} we cannot recover the assumption~\ref{eq:asymp:W} and therefore we cannot rely on Lemma~\ref{Lem:asymptotics:exp:Phi} here. Nevertheless, we still obtain the desired estimate in this case due to the fact, that we can choose $\tilde{\alpha}>0$ as small as we want. In fact, the left-hand side of~\eqref{eq:W12:v3} can also be estimated by $C\abs{s}^{\kappa-1-\tilde{\alpha}}\chi_{\{\abs{s}\leq 1\}}$. If we choose $\tilde{\alpha}<\delta$ with $\delta<\theta/2$ we can proceed as above if we take $a\in(2\delta,\theta)$.
\end{remark}

Thus it remains to consider $Q_{1,2,2,1}$. Changing variables, $\Y=s+t$, and rearranging we obtain
\begin{equation*}
 \begin{split}
  &\quad Q_{1,2,2,1}\\
  &=-\frac{2+\kappa}{2\pi}\frac{\im}{p^2}\int_{-\infty}^{\infty}\ee^{-\im ps}\Phi'\left(\im s\right)s^{1+\kappa}\left(1-\coup\left(s\right)\right)\ee^{-\Phi\left(\im s\right)}\int_{s}^{\sgn\left(s\right)\infty}\cut\left(\Y\right)\frac{\ee^{\im \Y\xi}}{\Y^{1+\kappa}\left(1+\im \Y\right)}\ee^{\Phi\left(\im \Y\right)}\dd{\Y}\ds\\
  &=-\frac{2+\kappa}{2\pi}\frac{\im}{p^2}\int_{-\infty}^{\infty}\ee^{\im s\left(\xi-p\right)}\Phi'\left(\im s\right)s^{1+\kappa}\left(1-\coup\left(s\right)\right)\int_{0}^{\sgn\left(s\right)\infty}\frac{\cut\left(s+t\right)\ee^{\im t\xi}}{\left(s+t\right)^{1+\kappa}\left(1+\im\left(s+t\right)\right)}\ee^{\Phi\left(\im\left(s+t\right)\right)-\Phi\left(\im s\right)}\dt\ds.
 \end{split}
\end{equation*}
Integrating by parts in $s$ we get
\begin{multline}\label{eq:Q1221}
  \quad Q_{1,2,2,1}\\
  \shoveleft{=-\frac{2+\kappa}{2\pi}\frac{1}{p^2\left(\xi-p\right)}\int_{-\infty}^{\infty}\del_{s}\left(\ee^{\im s\left(\xi-p\right)}-1\right)\Phi'\left(\im s\right)s^{1+\kappa}\left(1-\coup\left(s\right)\right)\cdot}\\
  \shoveright{\cdot\int_{0}^{\sgn\left(s\right)\infty}\frac{\cut\left(s+t\right)\ee^{\im t\xi}}{\left(s+t\right)^{1+\kappa}\left(1+\im\left(s+t\right)\right)}\ee^{\Phi\left(\im\left(s+t\right)\right)-\Phi\left(\im s\right)}\dt\ds}\\
  \shoveleft{=\frac{2+\kappa}{2\pi}\frac{1}{p^2\left(\xi-p\right)}\int_{-\infty}^{\infty}\left(\ee^{\im s\left(\xi-p\right)}-1\right)\cdot}\\
  \cdot\del_{s}\left(\Phi'\left(\im s\right)s^{1+\kappa}\left(1-\coup\left(s\right)\right)\int_{0}^{\sgn\left(s\right)\infty}\frac{\cut\left(s+t\right)\ee^{\im t\xi}}{\left(s+t\right)^{1+\kappa}\left(1+\im\left(s+t\right)\right)}\ee^{\Phi\left(\im\left(s+t\right)\right)-\Phi\left(\im s\right)}\dt\right)\ds.
\end{multline}
Expanding the derivative it follows
\begin{equation*}
 \begin{split}
  &\quad \del_{s}\left(\Phi'\left(\im s\right)s^{1+\kappa}\left(1-\coup\left(s\right)\right)\int_{0}^{\sgn\left(s\right)\infty}\frac{\cut\left(s+t\right)\ee^{\im t\xi}}{\left(s+t\right)^{1+\kappa}\left(1+\im\left(s+t\right)\right)}\ee^{\Phi\left(\im\left(s+t\right)\right)-\Phi\left(\im s\right)}\dt\right)\\
  &=\del_{s}\left(\Phi'\left(\im s\right)s^{1+\kappa}\left(1-\coup\left(s\right)\right)\right)\int_{0}^{\sgn\left(s\right)\infty}\frac{\cut\left(s+t\right)\ee^{\im t\xi}}{\left(s+t\right)^{1+\kappa}\left(1+\im\left(s+t\right)\right)}\ee^{\Phi\left(\im\left(s+t\right)\right)-\Phi\left(\im s\right)}\dt\\
  &\quad +\Phi'\left(\im s\right)s^{1+\kappa}\left(1-\coup\left(s\right)\right)\int_{0}^{\sgn\left(s\right)\infty}\del_{s}\left(\frac{\cut\left(s+t\right)\ee^{\im t\xi}}{\left(s+t\right)^{1+\kappa}\left(1+\im\left(s+t\right)\right)}\right)\ee^{\Phi\left(\im\left(s+t\right)\right)-\Phi\left(\im s\right)}\dt\\
 &\quad +\im\Phi'\left(\im s\right)s^{1+\kappa}\left(1-\coup\left(s\right)\right)\int_{0}^{\sgn\left(s\right)\infty}\frac{\cut\left(s+t\right)\ee^{\im t\xi}\left(\Phi'\left(\im\left(s+t\right)\right)-\Phi'\left(\im s\right)\right)}{\left(s+t\right)^{1+\kappa}\left(1+\im\left(s+t\right)\right)}\ee^{\Phi\left(\im\left(s+t\right)\right)-\Phi\left(\im s\right)}\dt.
 \end{split}
\end{equation*}
From Lemmas~\ref{Lem:properties:Phi} and~\ref{Lem:asymptotics:xi:leq:p} together with $\eps\abs{s}^{-\alpha}\ee^{-\frac{D\eps}{\abs{s}^{\alpha}}}\leq C$ we then get by estimating by the most singular terms that
\begin{equation}\label{eq:est:W:bd:v3}
  \abs{\del_{s}\left(\Phi'\left(\im s\right)s^{1+\kappa}\left(1-\coup\left(s\right)\right)\int_{0}^{\sgn\left(s\right)\infty}\frac{\cut\left(s+t\right)\ee^{\im t\xi}}{\left(s+t\right)^{1+\kappa}\left(1+\im\left(s+t\right)\right)}\ee^{\Phi\left(\im\left(s+t\right)\right)-\Phi\left(\im s\right)}\dt\right)}\leq C\frac{\chi_{\left\{\abs{s}\leq 2\right\}}}{\abs{s}^{1+2\delta}}.
\end{equation}
Using this in~\eqref{eq:Q1221} together with $\abs{\ee^{\im s\left(\xi-p\right)}-1}\leq C\abs{s}^{a}\abs{\xi-p}^a$ for some $a\in\left(2\delta,\theta\right)$ we find
\begin{equation*}
 \begin{split}
  \abs{Q_{1,2,2,1}}\leq C\frac{1}{p^2\abs{\xi-p}^{1-a}}\int_{-2}^{2}\abs{s}^{a-1-2\delta}\ds\leq \frac{C\left(a\right)}{p^2\abs{\xi-p}^{1-a}}.
 \end{split}
\end{equation*}
The contribution to the integral in~\eqref{eq:int:Q0} can then be estimated by
\begin{equation}\label{eq:contr:Q1221}
 \begin{split}
  &\quad\frac{C}{p^2}\int_{0}^{p}\frac{1}{\left(\xi+1\right)^{2\theta-\nu}}\frac{1}{\abs{\xi-p}^{1-a}}\dxi\leq \frac{C}{p^2}\left(\frac{1}{p^{1-a}}\int_{0}^{p/2}\xi^{-2\theta+\nu}\dxi+\frac{1}{p^{2\theta}}\int_{p/2}^{p}\abs{\xi-p}^{a-1}\dxi\right)\\
  &\leq \frac{C}{p^{2+2\theta-a-\nu}}\leq C\frac{\left(1+p\right)^{1-\theta}}{p^3}
 \end{split}
\end{equation}
if we choose $\nu>0$ sufficiently small for fixed $a$.

\begin{remark}\label{Rem:adapt:11.4}
 For $\alpha=0$ as small change in the argument is necessary, similarly to Remark~\ref{Rem:est:bd:pert:v3}. Precisely, in this case the left-hand side of~\eqref{eq:est:W:bd:v3} can be estimated by $C\abs{s}^{\kappa-\tilde{\alpha}-1}\chi_{\{\abs{s}\leq 2\}}$. Choosing $\tilde{\alpha}<\delta$ and $\delta<\theta/2$, we can proceed as above.
\end{remark}

\subsubsection{Contribution of $W_{1,2}$ for $\xi>p$}

Next we have to consider the region where $\xi>p$ and we proceed similarly, i.e.\ we first introduce a cut-off in the variable $\Y$ to separate the regions of small and large values of $\Y$. Then in one part we can again integrate by parts directly, while in the other one we have to take more care using the precise asymptotic behaviour of $\Phi$. Precisely we have to consider
\begin{equation*}
 \begin{split}
  &\quad \frac{1}{2\pi\im}\int_{-\im R}^{\im R}\frac{2+\kappa}{p^2\Y\left(1+\Y\right)}\ee^{\Y\xi}\int_{0}^{\Y}\ee^{-px}\Phi'\left(x\right)x\left(\frac{x}{\Y}\right)^{\kappa}\frac{\ee^{\Phi\left(\Y\right)}}{\ee^{\Phi\left(x\right)}}\dx\dd{\Y}\\
  &=\frac{\left(2+\kappa\right)\im}{2\pi p^2}\int_{-R}^{R}\frac{\ee^{\im \Y\xi}}{\Y^{1+\kappa}\left(1+\im \Y\right)}\int_{0}^{\Y}\ee^{-\im ps}\Phi'\left(\im s\right)s^{1+\kappa}\frac{\ee^{\Phi\left(\im \Y\right)}}{\ee^{\Phi\left(\im s\right)}}\ds\left(\coup\left(\Y\right)+1-\coup\left(\Y\right)\right)\dd{\Y}=\widehat{Q}_{1,2,1}+\widehat{Q}_{1,2,2}
 \end{split}
\end{equation*}
with 
\begin{align*}
 \widehat{Q}_{1,2,1}&\vcc= \frac{\left(2+\kappa\right)\im}{2\pi p^2}\int_{-R}^{R}\coup\left(\Y\right)\frac{\ee^{\im \Y\xi}}{\Y^{1+\kappa}\left(1+\im \Y\right)}\int_{0}^{\Y}\ee^{-\im ps}\Phi'\left(\im s\right)s^{1+\kappa}\frac{\ee^{\Phi\left(\im \Y\right)}}{\ee^{\Phi\left(\im s\right)}}\ds\dd{\Y},\\
 \widehat{Q}_{1,2,2}&\vcc=\frac{\left(2+\kappa\right)\im}{2\pi p^2}\int_{-R}^{R}\left(1-\coup\left(\Y\right)\right)\frac{\ee^{\im \Y\xi}}{\Y^{1+\kappa}\left(1+\im \Y\right)}\int_{0}^{\Y}\ee^{-\im ps}\Phi'\left(\im s\right)s^{1+\kappa}\frac{\ee^{\Phi\left(\im \Y\right)}}{\ee^{\Phi\left(\im s\right)}}\ds\dd{\Y}.
\end{align*}
We first consider $\widehat{Q}_{1,2,1}$ where it is possible to integrate by parts and we thus obtain
\begin{equation*}
 \begin{split}
  \widehat{Q}_{1,2,1}&=\frac{2+\kappa}{2\pi p^2\xi}\int_{-R}^{R}\frac{\del_{\Y}\left(\ee^{\im \Y\xi}-1\right)}{\Y^{1+\kappa}\left(1+\im \Y\right)}\int_{0}^{\Y}\ee^{-\im ps}\Phi'\left(\im s\right)s^{1+\kappa}\ee^{-\Phi\left(\im s\right)}\ds\ee^{\Phi\left(\im \Y\right)}\coup\left(\Y\right)\dd{\Y}\\
  &=\frac{2+\kappa}{2\pi p^2\xi}\left.\left(\frac{\ee^{\im \Y\xi}-1}{\Y^{1+\kappa}\left(1+\im \Y\right)}\coup\left(\Y\right)\int_{0}^{\Y}\ee^{-\im ps}\Phi'\left(\im s\right)s^{1+\kappa}\frac{\ee^{\Phi\left(\im \Y\right)}}{\ee^{\Phi\left(\im s\right)}}\ds\right)\right|_{\Y=-R}^{\Y=R}\\
  &\quad -\frac{2+\kappa}{2\pi p^2\xi}\int_{-R}^{R}\left(\ee^{\im \Y\xi}-1\right)\del_{\Y}\left(\frac{\coup\left(\Y\right)}{\Y^{1+\kappa}\left(1+\im \Y\right)}\right)\int_{0}^{\Y}\ee^{-\im ps}\Phi'\left(\im s\right)s^{1+\kappa}\frac{\ee^{\Phi\left(\im \Y\right)}}{\ee^{\Phi\left(\im s\right)}}\ds \dd{\Y}\\
  &\quad -\frac{2+\kappa}{2\pi p^2\xi}\int_{-R}^{R}\frac{\ee^{\im \Y\xi}-1}{1+\im \Y}\ee^{-\im p\Y}\Phi'\left(\im \Y\right)\coup\left(\Y\right)\dd{\Y}\\
  &\quad -\frac{\left(2+\kappa\right)\im}{2\pi p^2\xi}\int_{-R}^{R}\frac{\ee^{\im \Y\xi}-1}{\Y^{1+\kappa}\left(1+\im \Y\right)}\coup\left(\Y\right)\Phi'\left(\im \Y\right)\int_{0}^{\Y}\ee^{-\im ps}\Phi'\left(\im s\right)\im s^{1+\kappa}\frac{\ee^{\Phi\left(\im \Y\right)}}{\ee^{\Phi\left(\im s\right)}}\ds\dd{\Y}\\
  &=\vcc(I)+(II)+(III)+(IV).
 \end{split}
\end{equation*}
Using similarly as before the boundedness of $\frac{\ee^{\Phi\left(\im \Y\right)}}{\ee^{\Phi\left(\im s\right)}}$ as well $\abs{\Phi'\left(Z\right)}\leq C\eps\left(\abs{Z}^{-\alpha-1}+\abs{Z}^{\alpha-1}\right)$ we find
\begin{align*}
 \abs{\int_{0}^{\Y}\ee^{-\im ps}\Phi'\left(\im s\right)s^{1+\kappa}\frac{\ee^{\Phi\left(\im \Y\right)}}{\ee^{\Phi\left(\im s\right)}}\ds}&\leq C\eps &\text{for } \abs{\Y}&\leq 1\\
 \abs{\int_{0}^{\Y}\ee^{-\im ps}\Phi'\left(\im s\right)s^{1+\kappa}\frac{\ee^{\Phi\left(\im \Y\right)}}{\ee^{\Phi\left(\im s\right)}}\ds}&\leq C\eps\abs{\Y}^{1+\alpha+\kappa} &\text{for } \abs{\Y}&\geq 1.
\end{align*}
From this estimate we then obtain
\begin{equation*}
 \begin{split}
  \abs{(I)}\leq \frac{C}{p^2\xi}\frac{R^{\alpha}}{1+R}\longrightarrow 0 \quad \text{for } R\longrightarrow \infty,
 \end{split}
\end{equation*}
where we also used the boundedness of $\coup$. Next we get, by taking also the limit $R\to\infty$, which is possible due to the sufficient large decay at infinity of the integrand, that
\begin{equation*}
 \begin{split}
  \abs{(II)}\leq \frac{C}{p^2\xi}\int_{-\infty}^{\infty}\frac{\chi_{\left\{\abs{\Y}\geq\frac{1}{2}\right\}}}{\abs{\Y}^{2+\kappa}\left(1+\abs{\Y}\right)}\max\left\{1,\abs{\Y}^{1+\alpha+\kappa}\right\}\dd{\Y}\leq \frac{C}{p^2\xi},
 \end{split}
\end{equation*}
using also that $\alpha<1/2$. Similarly we obtain
\begin{equation*}
 \begin{split}
  \abs{(III)}\leq \frac{C}{p^2\xi}\int_{-\infty}^{\infty}\frac{\chi_{\left\{\abs{\Y}\geq\frac{1}{2}\right\}}}{1+\abs{\Y}}\left(\abs{\Y}^{-\alpha-1}+\abs{\Y}^{\alpha-1}\right)\dd{\Y}\leq \frac{C}{p^2 \xi}.
 \end{split}
\end{equation*}
Finally for $(IV)$ we get
\begin{equation*}
 \begin{split}
  \abs{(IV)}\leq \frac{C}{p^2\xi}\int_{-\infty}^{\infty}\frac{\chi_{\left\{\abs{\Y}\geq\frac{1}{2}\right\}}}{\abs{\Y}^{1+\kappa}\left(1+\abs{\Y}\right)}\left(\abs{\Y}^{-\alpha-1}+\abs{\Y}^{\alpha-1}\right)\max\left\{1,\abs{\Y}^{1+\alpha+\kappa}\right\}\dd{\Y}\leq \frac{C}{p^2\xi},
 \end{split}
\end{equation*}
where it is crucial in the last step that $\alpha<1/2$. Thus for $\xi>p$ we obtain the estimate
\begin{equation*}
 \abs{\widehat{Q}_{1,2,2}}\leq \frac{C}{p^2\xi}
\end{equation*}
and therefore the contribution to~\eqref{eq:int:Q0} can be estimated by
\begin{equation}\label{eq:contr:Q122}
 \begin{split}
  \frac{C}{p^2}\int_{p}^{\infty}\frac{1}{\xi^{2\theta-\nu+1}}\dxi\leq \frac{C}{p^{2+2\theta-\nu}}\leq C\frac{\left(1+p\right)^{1-\theta}}{p^3}
 \end{split}
\end{equation}
as $\nu$ can be made as small as needed.

It now remains to estimate the contribution coming from $\widehat{Q}_{1,2,1}$. Noting that, due to the cut-off $1-\coup\left(\Y\right)$, we can take the limit $R\to\infty$ in $\widehat{Q}_{1,2,1}$, we have to consider
\begin{equation*}
 \begin{split}
  \lim_{R\to\infty}\widehat{Q}_{1,2,1}=\frac{\left(2+\kappa\right)\im}{2\pi p^2}\int_{-\infty}^{\infty}\ee^{-\im ps} \Phi'\left(\im s\right)s^{1+\kappa}\ee^{-\Phi\left(\im s\right)}\int_{s}^{\sgn\left(s\right)\infty}\frac{\ee^{\im \Y\xi}}{\Y^{1+\kappa}\left(1+\im \Y\right)}\ee^{\Phi\left(\im \Y\right)}\left(1-\coup\left(\Y\right)\right)\dd{\Y}\ds.
 \end{split}
\end{equation*}
  Changing variables, $\Y=s+t$, and rearranging we can integrate by parts in $s$ to get
 \begin{equation*}
  \begin{split}
  &\quad \lim_{R\to\infty}\widehat{Q}_{1,2,1}\\
  &=\frac{\left(2+\kappa\right)\im}{2\pi p^2}\int_{-\infty}^{\infty}\Phi'\left(\im s\right)s^{1+\kappa}\ee^{\im s\left(\xi-p\right)}\int_{0}^{\sgn\left(s\right)\infty}\frac{\left(1-\coup\left(s+t\right)\right)}{\left(s+t\right)^{1+\kappa}\left(1+\im\left(s+t\right)\right)}\ee^{\im t\xi}\ee^{\Phi\left(\im\left(s+t\right)\right)-\Phi\left(\im s\right)}\dt\ds\\
   &=\frac{2+\kappa}{2\pi p^2\left(\xi-p\right)}\int_{-\infty}^{\infty}\del_{s}\left(\ee^{\im\left(\xi-p\right)s}-1\right)\Phi'\left(\im s\right)\im s^{1+\kappa}\int_{0}^{\sgn\left(s\right)\infty}\left(\cdots\right)\dt\ds\\
   &=-\frac{2+\kappa}{2\pi p^2\left(\xi-p\right)}\int_{-\infty}^{\infty}\left(\ee^{\im s\left(\xi-p\right)}-1\right)\del_{s}\left(\Phi'\left(\im s\right)\im s^{1+\kappa}\int_{0}^{\sgn\left(s\right)\infty}\left(\cdots\right)\dt\right)\ds.
  \end{split}
 \end{equation*}
Expanding the derivative we get
\begin{equation*}
 \begin{split}
  &\quad \del_{s}\left(\Phi'\left(\im s\right) s^{1+\kappa}\int_{0}^{\sgn\left(s\right)\infty}\frac{\left(1-\coup\left(s+t\right)\right)}{\left(s+t\right)^{1+\kappa}\left(1+\im\left(s+t\right)\right)}\ee^{\im t\xi}\ee^{\Phi\left(\im\left(s+t\right)\right)-\Phi\left(\im s\right)}\dt\right)\\
  &=\del_{s}\left(\Phi'\left(\im s\right) s^{1+\kappa}\right)\int_{0}^{\sgn\left(s\right)\infty}\frac{\left(1-\coup\left(s+t\right)\right)}{\left(s+t\right)^{1+\kappa}\left(1+\im\left(s+t\right)\right)}\ee^{\im t\xi}\ee^{\Phi\left(\im\left(s+t\right)\right)-\Phi\left(\im s\right)}\dt\\
  &\quad +\left(\Phi'\left(\im s\right) s^{1+\kappa}\right)\int_{0}^{\sgn\left(s\right)\infty}\del_{s}\left(\frac{\left(1-\coup\left(s+t\right)\right)}{\left(s+t\right)^{1+\kappa}\left(1+\im\left(s+t\right)\right)}\ee^{\im t\xi}\right)\ee^{\Phi\left(\im\left(s+t\right)\right)-\Phi\left(\im s\right)}\dt\\
  &\quad+\im\left(\Phi'\left(\im s\right) s^{1+\kappa}\right)\int_{0}^{\sgn\left(s\right)\infty}\frac{\left(1-\coup\left(s+t\right)\right)}{\left(s+t\right)^{1+\kappa}\left(1+\im\left(s+t\right)\right)}\ee^{\im t\xi}\left(\Phi'\left(\im\left(s+t\right)\right)-\Phi'\left(\im s\right)\right)\ee^{\Phi\left(\im\left(s+t\right)\right)-\Phi\left(\im s\right)}\dt
 \end{split}
\end{equation*}
and together with Lemmas~\ref{Lem:properties:Phi} and~\ref{Lem:asymptotics:p:leq:xi} as well as $\eps\abs{s}^{-\alpha}\ee^{-\frac{D\eps}{\abs{s}^{\alpha}}}\leq C$ we find
\begin{equation*}
 \abs{\del_{s}\left(\Phi'\left(\im s\right) s^{1+\kappa}\int_{0}^{\sgn\left(s\right)\infty}\frac{\left(1-\coup\left(s+t\right)\right)}{\left(s+t\right)^{1+\kappa}\left(1+\im\left(s+t\right)\right)}\ee^{\im t\xi}\ee^{\Phi\left(\im\left(s+t\right)\right)-\Phi\left(\im s\right)}\dt\right)}\leq C\frac{\chi_{\left\{\abs{s}\leq 1\right\}}}{\abs{s}^{1+2\delta}}.
\end{equation*}
Using then that for $a\in \left(2\delta,\theta\right)$ we have $\abs{\ee^{\im s\left(\xi-p\right)}-1}\leq C\abs{s}^{a}\abs{\xi-p}^a$ it follows 
\begin{equation*}
 \begin{split}
    &\quad \lim_{R\to\infty}\widehat{Q}_{1,2,1}\leq \frac{C}{p^2\abs{\xi-p}^{1-a}}\int_{-2}^{2}\abs{s}^{a-1-2\delta}\ds\leq \frac{C\left(a\right)}{p^2\abs{\xi-p}^{1-a}}.
 \end{split}
\end{equation*}
The contribution to the integral in~\eqref{eq:int:Q0} can then be controlled by
\begin{equation}\label{eq:contr:Q1211}
 \begin{split}
  \frac{C}{p^2}\int_{p}^{\infty}\frac{1}{\xi^{2\theta-\nu}\abs{\xi-p}^{1-a}}\dxi\leq \frac{C}{p^{2+2\theta-\nu-a}}\int_{1}^{\infty}\frac{1}{\xi^{2\theta-\nu}\abs{1-\xi}^{1-a}}\dxi\leq \frac{C}{p^{2+2\theta-\nu-a}}\leq C\frac{\left(1+p\right)^{1-\theta}}{p^3}
 \end{split}
\end{equation}
 if we choose $\nu>0$ sufficiently small (for fixed $a$). This then finishes the considerations of the term $W_{1,2}$.

 \begin{remark}
  For $\alpha=0$ an argument similar to that one in Remark~\ref{Rem:adapt:11.4} applies.
 \end{remark}

\subsection{Contribution of $H_{0,2}$}

In the following we consider the contribution to the integral in~\eqref{eq:int:Q0} coming from $H_{0,2}\left(\Y,p\right)$. To do this we proceed in the same way as for the term $H_{0,1}\left(\Y,p\right)$, i.e.\ we first integrate by parts and use also $\Phi'\left(x\right)=-\eps\beta_W\left(x\right)x^{-1}\ee^{-x}$ to obtain
\begin{equation*}
 \begin{split}
  H_{0,2}\left(\Y,p\right)
  &=\frac{\eps}{p\Y\left(1+\Y\right)}\int_{0}^{\Y}\ee^{-\left(1+p\right)x}\beta_W\left(x\right)x\left(\frac{x}{\Y}\right)^{\kappa}\frac{\ee^{\Phi\left(\Y\right)}}{\ee^{\Phi\left(x\right)}}\dx\\
  &=-\frac{\eps}{p\left(1+p\right)\Y\left(1+\Y\right)}\int_{0}^{\Y}\del_{x}\left(\ee^{-\left(1+p\right)x}\right)\beta_W\left(x\right)x\left(\frac{x}{\Y}\right)^{\kappa}\frac{\ee^{\Phi\left(\Y\right)}}{\ee^{\Phi\left(x\right)}}\dx\\
  &=-\eps\frac{\ee^{-\left(p+1\right)\Y}\beta_W\left(\Y\right)}{p\left(1+p\right)\left(1+\Y\right)}+\frac{\eps}{p\left(1+p\right)}\frac{1}{\Y\left(1+\Y\right)}\int_{0}^{\Y}\ee^{-\left(p+1\right)x}\del_{x}\left(x\beta_W\left(x\right)\left(\frac{x}{\Y}\right)^{\kappa}\right)\frac{\ee^{\Phi\left(\Y\right)}}{\ee^{\Phi\left(x\right)}}\dx\\
  &\quad +\frac{1}{p\left(1+p\right)}\frac{1}{\Y\left(1+\Y\right)}\int_{0}^{\Y}\ee^{-p x}x^2\left(\Phi'\left(x\right)\right)^2\left(\frac{x}{\Y}\right)^{\kappa}\frac{\ee^{\Phi\left(\Y\right)}}{\ee^{\Phi\left(x\right)}}\dx\\
  &=\vcc W_{2,0}\left(\Y,p\right)+W_{2,1}\left(\Y,p\right)+W_{2,2}\left(\Y,p\right).
 \end{split}
\end{equation*}
As before we consider the contribution of the three terms separately. 

\subsection{Contribution of $W_{2,0}\left(\Y,p\right)$}

We have to consider the following integral
\begin{equation*}
 \begin{split}
  -\frac{1}{2\pi\im}\frac{\eps}{p\left(1+p\right)}\int_{-\im R}^{\im R}\frac{\ee^{\Y\left(\xi-\left(p+1\right)\right)}}{1+\Y}\beta_W\left(\Y\right)\dd{\Y},
 \end{split}
\end{equation*}
while we are interested in estimates with respect to $\xi$ and $p$ in the limit $R\to\infty$. First we note that we can deform the path of integration to some small semi-circle of radius $r$ in the right half-plane around the origin to avoid the singularity there (see Figure~\ref{fig:W20}). Precisely, the integral over this semi-circle can be estimated by 
\begin{equation*}
 \frac{C}{p\left(1+p\right)}r^{1-\alpha}\longrightarrow 0 \quad \text{as } r\longrightarrow 0.
\end{equation*}

  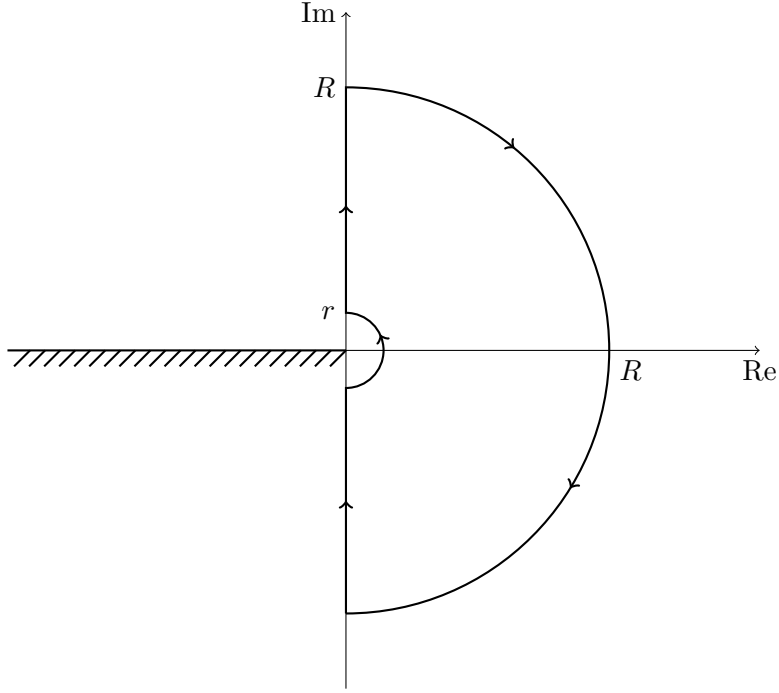
\begin{figure}
\begin{center}  
 \begin{tikzpicture}[contour/.style={postaction={decorate, decoration={markings,
mark=at position 1.5cm with {\arrow[line width=1pt]{>}},
mark=at position 4cm with {\arrow[line width=1pt]{>}},
mark=at position 6cm with {\arrow[line width=1pt]{>}},
mark=at position 10cm with {\arrow[line width=1pt]{>}},
mark=at position 15cm with {\arrow[line width=1pt]{>}}}}},
    interface/.style={postaction={draw,decorate,decoration={border,angle=45,
                    amplitude=0.3cm,segment length=2mm}}}]
                    
\draw[->] (0,0) -- (5.5,0) node[below] {$\Re$};
\draw[->] (0,-4.5) -- (0,4.5) node[left]{$\Im$};

\draw[line width=.8pt,interface](0,0)--(-4.5,0);

\path[draw,line width=0.8pt,contour] (0,-3.5)  -- (0,-0.5)  arc (-90:90:0.5) node[left] {$r$} -- (0, 3.5)node[left] {$R$} arc (90:-90:3.5);

\node[below right] at (3.5,0) {$R$};
\end{tikzpicture}
\caption{Contour for $W_{2,0}$}
\label{fig:W20}
 \end{center}
\end{figure}

We now assume first that $\xi<p+1$ and show that in this region there is no contribution from $W_{2,0}$. To do this we first integrate by parts to get
\begin{equation}\label{eq:W20}
 \begin{split}
  &\quad-\frac{\eps}{2\pi\im p\left(1+p\right)}\int_{-\im R}^{\im R}\frac{\ee^{\Y\left(\xi-\left(p+1\right)\right)}}{1+\Y}\beta_W\left(\Y\right)\dd{\Y}\\
  &=-\frac{\eps}{2\pi\im p\left(1+p\right)\left(\xi-\left(p+1\right)\right)}\int_{-\im R}^{\im R}\frac{\del_{\Y}\left(\ee^{\Y\left(\xi-\left(1+p\right)\right)}-1\right)}{1+\Y}\beta_W\left(\Y\right)\dd{\Y}\\
  &=-\frac{\eps}{2\pi\im p\left(1+p\right)\left(\xi-\left(p+1\right)\right)}\left(\frac{\ee^{\im R\left(\xi-\left(1+p\right)\right)}-1}{1+\im R}\beta_W\left(\im R\right)-\frac{\ee^{-\im R\left(\xi-\left(1+p\right)\right)}-1}{1-\im R}\beta_W\left(-\im R\right)\right)\\
  &\quad +\frac{\eps}{2\pi\im p\left(1+p\right)\left(\xi-\left(p+1\right)\right)}\int_{-\im R}^{\im R}\left(\ee^{\Y\left(\xi-\left(1+p\right)\right)}-1\right)\left(-\frac{\beta_W\left(\Y\right)}{\left(1+\Y\right)^2}+\frac{\beta_W'\left(\Y\right)}{1+\Y}\right)\dd{\Y}=\vcc(I)+(II).
 \end{split}
\end{equation}
For $(I)$ we obtain by similar estimates as above
\begin{equation*}
 \begin{split}
  \abs{(I)}\leq \frac{C}{p\left(1+p\right)}\frac{1}{\abs{\xi-\left(p+1\right)}}\frac{R^{\alpha}+R^{-\alpha}}{R}.
 \end{split}
\end{equation*}
The right-hand side converges to zero for $R\to\infty$. To treat the second term we use Cauchy's Theorem, using the analyticity of the integrand in the right half plane. Note that we already showed that we can avoid the singularity at zero by some arbitrary small semi-circle. Precisely we show that the integral over the semi-circle of radius $R$ converges to zero for $R\to\infty$ and thus by Cauchy's Theorem this also shows that $(II)\to 0$. Parametrising the semi-circle by $\varphi\left(s\right)=R\ee^{\im s}$ with $s\in\left[-\pi/2,\pi/2\right]$ we get
\begin{equation*}
 \begin{split}
  &\quad \abs{\frac{\eps}{2\pi p\left(1+p\right)}\frac{R}{\xi-\left(1+p\right)}\int_{-\pi/2}^{\pi/2}\left(\ee^{R\ee^{\im s}\left(\xi-\left(p+1\right)\right)}-1\right)\left(-\frac{\beta_W\left(R\ee^{\im s}\right)}{\left(1+R\ee^{\im s}\right)^2}+\frac{\beta_W'\left(R\ee^{\im s}\right)}{1+R\ee^{\im s}}\right)\ds}\\
  &\leq \frac{C}{p\left(1+p\right)}\frac{R^{1+a}}{\abs{\xi-\left(p+1\right)}^{1-a}}\int_{-\pi/2}^{\pi/2}\frac{R^{-\alpha}+R^{\alpha}}{R^2}\ds\leq \frac{C}{p\left(1+p\right)}\frac{R^{\alpha+a-1}}{\abs{\xi-\left(p+1\right)}^{1-a}}.
 \end{split}
\end{equation*}
The right hand side converges to zero for $R\to\infty$ and thus there is no contribution from the region $\xi<p+1$ from $W_{2,0}$. 

For $\xi>p+1$ we proceed similarly and use the expression obtained in~\eqref{eq:W20}, which holds independently of the relation of $\xi$ and $p$. For $(I)$ we have already shown that $(I)\to 0$ as $R\to\infty$. It thus remains only to consider $(II)$. We then get in the same way as above for $a\in\left(\alpha,\theta\right)$
\begin{equation*}
 \begin{split}
  \abs{(II)}&\leq \frac{C}{p\left(1+p\right)}\frac{1}{\abs{\xi-\left(1+p\right)}^{1-a}}\int_{-R}^{R}\frac{\abs{\Y}^a}{\left(1+\abs{\Y}\right)^2}\left(\abs{\Y}^{-\alpha}+\abs{\Y}^{\alpha}\right)+\frac{\abs{\Y}^a}{1+\abs{\Y}}\left(\abs{\Y}^{-\alpha-1}+\abs{\Y}^{\alpha-1}\right)\dd{\Y}\\
  &\leq \frac{C}{p\left(1+p\right)}\frac{1}{\abs{\xi-\left(1+p\right)}^{1-a}}
 \end{split}
\end{equation*}
 uniformly in $R$ due to the choice of $a$. For the contribution to the integral in~\eqref{eq:int:Q0} we then obtain
\begin{equation*}
 \frac{C}{p\left(1+p\right)}\int_{1+p}^{\infty}\frac{1}{\xi^{2\theta-\nu}\abs{\xi-\left(1+p\right)}^{1-a}}\leq \frac{C}{p\left(1+p\right)^{1+2\theta-\nu-a}}\int_{1}^{\infty}\frac{1}{\xi^{2\theta-\nu}\abs{\xi-1}^{1-a}}\dxi\leq C\frac{\left(1+p\right)^{1-\theta}}{p^3}
\end{equation*}
if we choose $\nu>0$ sufficiently small for fixed $a$.

\subsection{Contribution of $W_{2,1}\left(\Y,p\right)$}

We first recall
\begin{equation*}
 W_{2,1}\left(\Y,p\right)=-\frac{\eps}{p\left(1+p\right)\Y\left(1+\Y\right)}\int_{0}^{\Y}\ee^{-\left(1+p\right)x}\del_{x}\left(x\beta_W\left(x\right)\left(\frac{x}{\Y}\right)^{\kappa}\right)\frac{\ee^{\Phi\left(\Y\right)}}{\ee^{\Phi\left(x\right)}}\dx.
\end{equation*}
 The strategy is the same as for the term $W_{1,2}$, i.e.\ we consider the regions $\xi<p+1$ and $\xi>p+1$ separately, introduce a regularisation and a cut-off and integrate by parts.
  
\subsubsection{Contribution of $W_{2,1}\left(\Y,p\right)$ for $\xi<p+1$}
 
 We begin with $\xi<p+1$ and rewrite
 \begin{equation*}
  \begin{split}
   W_{2,1}\left(\Y,p\right)&=-\frac{\eps}{p\left(1+p\right)\Y\left(1+\Y\right)}\int_{0}^{\Y}\ee^{-\left(1+p\right)x}\del_{x}\left(x\beta_W\left(x\right)\left(\frac{x}{\Y}\right)^{\kappa}\right)\frac{\ee^{\Phi\left(\Y\right)}}{\ee^{\Phi\left(x\right)}}\coup\left(\im x\right)\dx\\
   &\quad -\frac{\eps}{p\left(1+p\right)\Y\left(1+\Y\right)}\int_{0}^{\Y}\ee^{-\left(1+p\right)x}\del_{x}\left(x\beta_W\left(x\right)\left(\frac{x}{\Y}\right)^{\kappa}\right)\frac{\ee^{\Phi\left(\Y\right)}}{\ee^{\Phi\left(x\right)}}\left(1-\coup\left(\im x\right)\right)\dx\\
   &=\vcc W_{2,1,1}\left(\Y,p\right)+W_{2,1,2}\left(\Y,p\right).
  \end{split}
 \end{equation*}
The term $W_{2,1,1}$ is again easier as we can integrate by parts and using 
\begin{equation*}
\del_{x}\left(x\beta_W\left(x\right)\left(\frac{x}{\Y}\right)^{\kappa}\right)=\left(\left(1+\kappa\right)\beta_W\left(x\right)+x\beta_W'\left(x\right)\right)\left(\frac{x}{\Y}\right)^{\kappa}
\end{equation*}
we obtain
\begin{equation*}
 \begin{split}
  W_{2,1,1}\left(\Y,p\right)&=\frac{\eps}{p\left(1+p\right)^2\Y\left(1+\Y\right)}\int_{0}^{\Y}\del_{x}\left(\ee^{-\left(1+p\right)x}\right)\del_{x}\left(x\beta_W\left(x\right)\left(\frac{x}{\Y}\right)^{\kappa}\right)\frac{\ee^{\Phi\left(\Y\right)}}{\ee^{\Phi\left(x\right)}}\coup\left(\im x\right)\dx\\
  &=\eps\frac{\ee^{-\left(1+p\right)\Y}\left(\left(1+\kappa\right)\beta_W\left(\Y\right)+\Y\beta_W'\left(\Y\right)\right)\coup\left(\im \Y\right)}{p\left(1+p\right)^2\Y\left(1+\Y\right)}\\
  &\quad -\frac{\eps}{p\left(1+p\right)^2\Y\left(1+\Y\right)}\int_{0}^{\Y}\ee^{-\left(1+p\right)x}\del_{x}\left(\del_{x}\left(x\beta_W\left(x\right)\left(\frac{x}{\Y}\right)^{\kappa}\right)\coup\left(\im x\right)\right)\frac{\ee^{\Phi\left(\Y\right)}}{\ee^{\Phi\left(x\right)}}\dx\\
  &\quad +\frac{\eps}{p\left(1+p\right)\Y\left(1+\Y\right)}\int_{0}^{\Y}\ee^{-\left(1+p\right)x}\del_{x}\left(x\beta_W\left(x\right)\left(\frac{x}{\Y}\right)^{\kappa}\right)\coup\left(\im x\right)\Phi'\left(x\right)\frac{\ee^{\Phi\left(\Y\right)}}{\ee^{\Phi\left(x\right)}}\dx.
 \end{split}
\end{equation*}
By similar arguments already used before we obtain 
\begin{align*}
 \abs{\eps\frac{\ee^{-\left(1+p\right)\Y}\left(\left(1+\kappa\right)\beta_W\left(\Y\right)+\Y\beta_W'\left(\Y\right)\right)\coup\left(\im \Y\right)}{p\left(1+p\right)^2\Y\left(1+\Y\right)}}
\leq C\eps\frac{\abs{\Y}^{\alpha}}{\abs{\Y}\left(1+\abs{\Y}\right)}\chi_{\left\{\abs{\Y}\geq \frac{1}{2}\right\}}.
 \end{align*}
Furthermore we get
\begin{align*}
 \abs{\del_{x}\left(\del_{x}\left(x\beta_W\left(x\right)\left(\frac{x}{\Y}\right)^{\kappa}\right)\coup\left(\im x\right)\right)}\leq \frac{C}{\abs{\Y}^{\kappa}}\abs{x}^{\alpha+\kappa-1}\chi_{\left\{\abs{x}\geq \frac{1}{2}\right\}}
\end{align*}
which yields
\begin{align*}
 \abs{\int_{0}^{\Y}\ee^{-\left(1+p\right)x}\del_{x}\left(\del_{x}\left(x\beta_W\left(x\right)\left(\frac{x}{\Y}\right)^{\kappa}\right)\coup\left(\im x\right)\right)\frac{\ee^{\Phi\left(\Y\right)}}{\ee^{\Phi\left(x\right)}}\dx}\leq \abs{\Y}^{\alpha}\chi_{\left\{\abs{\Y}\geq \frac{1}{2}\right\}}.                                                                                                                                                                                                                                                                                                                                                                                                                                                                    
\end{align*}
Finally we have
\begin{align*}
 \abs{\int_{0}^{\Y}\ee^{-\left(1+p\right)x}\del_{x}\left(x\beta_W\left(x\right)\left(\frac{x}{\Y}\right)^{\kappa}\right)\coup\left(\im x\right)\Phi'\left(x\right)\frac{\ee^{\Phi\left(\Y\right)}}{\ee^{\Phi\left(x\right)}}\dx}
 \leq C \abs{\Y}^{2\alpha}\chi_{\left\{\abs{\Y}\geq \frac{1}{2}\right\}}.
\end{align*}
Summarizing and estimating by the most singular contributions we obtain
\begin{equation*}
 \begin{split}
  \abs{W_{2,1,1}\left(\Y,p\right)}\leq \frac{C}{p\left(1+p\right)^2}\frac{\abs{\Y}^{2\alpha}\chi_{\left\{\abs{\Y}\geq \frac{1}{2}\right\}}}{\abs{\Y}\left(1+\abs{\Y}\right)}.
 \end{split}
\end{equation*}
The contribution to $Q_0$ can then be estimated by
\begin{equation*}
 \begin{split}
  \abs{\frac{1}{2\pi\im}\int_{-\im R}^{\im R}\ee^{\Y\xi}W_{2,1,1}\left(\Y,p\right)\dd{\Y}}\leq \frac{C}{p\left(1+p\right)^2}\int_{-\infty}^{\infty}\frac{\abs{\Y}^{2\alpha}\chi_{\left\{\abs{\Y}\geq \frac{1}{2}\right\}}}{\abs{\Y}\left(1+\abs{\Y}\right)}\dd{\Y}\leq \frac{C}{p\left(1+p\right)^2}.
 \end{split}
\end{equation*}
The contribution to the integral in~\eqref{eq:int:Q0} can be controlled similarly as in~\eqref{eq:contr:W121}.

We next come the the term $W_{2,1,2}$ and after taking the limit $R\to\infty$, rearranging, changing to real variables of integration and introducing the cut-off $\cut$ we obtain
\begin{equation*}
 \begin{split}
  \lim_{R\to\infty}\frac{1}{2\pi\im}&\int_{-\im R}^{\im R} \ee^{\Y\xi}W_{2,1,2}\left(\Y,p\right)\dd{\Y}\\
  &\begin{aligned}
    =-\frac{\eps}{2\pi p\left(1+p\right)}\int_{-\infty}^{\infty}&\ee^{-\left(p+1\right)\im s}\left(\left(1+\kappa\right)\beta_{W}\left(\im s\right)+\im s\beta_W'\left(\im s\right)\right)s^{\kappa}\ee^{-\Phi\left(\im s\right)}\left(1-\coup\left(s\right)\right)\cdot\\
  &\cdot\int_{s}^{\sgn\left(s\right)\infty}\frac{\ee^{\im \Y\xi}\left(\cut\left(\Y\right)+1-\cut\left(\Y\right)\right)}{\Y^{1+\kappa}\left(1+\im \Y\right)}\ee^{\Phi\left(\im \Y\right)}\dd{\Y}\ds
   \end{aligned}\\
  &=Q_{2,1,2,1}+Q_{2,1,2,2}
 \end{split}
\end{equation*}
with 
\begin{align*}
  Q_{2,1,2,1}\vcc=-\frac{\eps}{2\pi p\left(1+p\right)}\int_{-\infty}^{\infty}&\ee^{-\left(p+1\right)\im s}\left(\left(1+\kappa\right)\beta_{W}\left(\im s\right)+\im s\beta_W'\left(\im s\right)\right)s^{\kappa}\left(1-\coup\left(s\right)\right)\ee^{-\Phi\left(\im s\right)}\cdot\\
  &\cdot\int_{s}^{\sgn\left(s\right)\infty}\frac{\ee^{\im \Y\xi}\cut\left(\Y\right)}{\Y^{1+\kappa}\left(1+\im \Y\right)}\ee^{\Phi\left(\im \Y\right)}\dd{\Y}\ds,\\
  Q_{2,1,2,2}\vcc=-\frac{\eps}{2\pi p\left(1+p\right)}\int_{-\infty}^{\infty}&\ee^{-\left(p+1\right)\im s}\left(\left(1+\kappa\right)\beta_{W}\left(\im s\right)+\im s\beta_W'\left(\im s\right)\right)s^{\kappa}\left(1-\coup\left(s\right)\right)\ee^{-\Phi\left(\im s\right)}\cdot\\
  &\cdot\int_{s}^{\sgn\left(s\right)\infty}\frac{\ee^{\im \Y\xi}\left(1-\cut\left(\Y\right)\right)}{\Y^{1+\kappa}\left(1+\im \Y\right)}\ee^{\Phi\left(\im \Y\right)}\dd{\Y}\ds.
\end{align*}
As before the term $Q_{2,1,2,2}$ is the easier one as we can integrate by parts once more, i.e.\ we get
\begin{equation}\label{eq:Q2122}
 \begin{split}
  &\quad Q_{2,1,2,2}\\
  &=\frac{\eps}{2\pi\im p\left(1+p\right)^2}\int_{-\infty}^{\infty}\del_{s}\left(\ee^{-\left(p+1\right)\im s}-1\right)\left(\left(1+\kappa\right)\beta_{W}\left(\im s\right)+\im s\beta_W'\left(\im s\right)\right)s^{\kappa}\left(1-\coup\left(s\right)\right)\ee^{-\Phi\left(\im s\right)}\cdot\\
  &\qquad \qquad \qquad \qquad \cdot\int_{s}^{\sgn\left(s\right)\infty}\frac{\ee^{\im \Y\xi}\left(1-\cut\left(\Y\right)\right)}{\Y^{1+\kappa}\left(1+\im \Y\right)}\ee^{\Phi\left(\im \Y\right)}\dd{\Y}\ds\\
  &=-\frac{\eps}{2\pi\im p\left(1+p\right)^2}\int_{-\infty}^{\infty}\left(\ee^{-\left(p+1\right)\im s}-1\right)\del_{s}\left(\left(\left(1+\kappa\right)\beta_{W}\left(\im s\right)+\im s\beta_W'\left(\im s\right)\right)s^{\kappa}\left(1-\coup\left(s\right)\right)\right)\ee^{-\Phi\left(\im s\right)}\cdot\\
  &\qquad\qquad\qquad\qquad\cdot\int_{s}^{\sgn\left(s\right)\infty}\left(\cdots\right)\dd{\Y}\ds\\
  &\quad +\frac{\im\eps}{2\pi\im p\left(1+p\right)^2}\int_{-\infty}^{\infty}\left(\ee^{-\left(p+1\right)\im s}-1\right)\left(\left(1+\kappa\right)\beta_{W}\left(\im s\right)+\im s\beta_W'\left(\im s\right)\right)s^{\kappa}\left(1-\coup\left(s\right)\right)\Phi'\left(\im s\right)\ee^{-\Phi\left(\im s\right)}\cdot\\
  &\qquad\qquad\qquad\qquad\cdot\int_{s}^{\sgn\left(s\right)\infty}\left(\cdots\right)\dd{\Y}\ds\\
  &\quad +\frac{\eps}{2\pi\im p\left(1+p\right)^2}\int_{-\infty}^{\infty}\left(\ee^{-\left(p+1\right)\im s}-1\right)\left(\left(1+\kappa\right)\beta_{W}\left(\im s\right)+\im s\beta_W'\left(\im s\right)\right)\left(1-\coup\left(s\right)\right)\frac{\ee^{\im \xi s}\left(1-\cut\left(s\right)\right)}{ s\left(1+\im s\right)}\ds.
 \end{split}
\end{equation}
We first note that due to the cut-off and the decay in $\Y$ we obtain
\begin{equation*}
 \abs{\int_{s}^{\sgn\left(s\right)\infty}\frac{\ee^{\im \Y\xi}\left(1-\cut\left(\Y\right)\right)}{\Y^{1+\kappa}\left(1+\im \Y\right)}\ee^{\Phi\left(\im \Y\right)}\dd{\Y}}\leq C \quad \text{for all } s.
\end{equation*}
Furthermore in view of Lemma~\ref{Lem:asymptotics:exp:Phi} and using that $1-\coup$ is supported close to the origin, we have 
\begin{align*}
 \eps\abs{\del_{s}\left(\left(\left(1+\kappa\right)\beta_{W}\left(\im s\right)+\im s\beta_W'\left(\im s\right)\right)s^{\kappa}\left(1-\coup\left(s\right)\right)\right)\ee^{-\Phi\left(\im s\right)}}&\leq C\abs{s}^{\kappa-1}\chi_{\left\{\abs{s}\leq 1\right\}},\\
 \eps\abs{\left(\left(1+\kappa\right)\beta_{W}\left(\im s\right)+\im s\beta_W'\left(\im s\right)\right)s^{\kappa}\left(1-\coup\left(s\right)\right)\Phi'\left(\im s\right)\ee^{-\Phi\left(\im s\right)}}&\leq C\abs{s}^{\kappa-1}\chi_{\left\{\abs{s}\leq 1\right\}}.
\end{align*}
Finally, due to the additional cut-off $\cut$ we have
\begin{align*}
 \abs{\left(\left(1+\kappa\right)\beta_{W}\left(\im s\right)+\im s\beta_W'\left(\im s\right)\right)\left(1-\coup\left(s\right)\right)\frac{\ee^{\im \xi s}\left(1-\cut\left(s\right)\right)}{\im s\left(1+\im s\right)}}=0.
\end{align*}
Thus we conclude from~\eqref{eq:Q2122} and $\abs{\ee^{-\left(p+1\right)\im s}-1}\leq C(p+1)^{a}\abs{s}^{a}$ with $a\in(\delta,\theta)$ that
\begin{equation*}
 \abs{Q_{2,1,2,2}}\leq \frac{C}{p\left(1+p\right)^{2-a}}
\end{equation*}
and the contribution to~\eqref{eq:Q0} can be controlled similarly as in~\eqref{eq:contr:W121}.

\begin{remark}
 In the case $\alpha=0$ we can argue similarly as in Remark~\ref{Rem:est:bd:pert:v3}.
\end{remark}

We turn now to $Q_{2,1,2,1}$ and after changing variables, $\Y=s+t$, and integrating by parts in $s$ we obtain
\begin{equation*}
 \begin{split}
  &\quad Q_{2,1,2,1}\\
  &=-\frac{\eps}{2\pi\im p\left(1+p\right)\left(\xi-\left(1+p\right)\right)}\int_{-\infty}^{\infty}\del_{s}\left(\ee^{\im s\left(\xi-\left(p+1\right)\right)}-1\right)\left(\left(1+\kappa\right)\beta_W\left(\im s\right)+\im s\beta_W'\left(\im s\right)\right)s^{\kappa}\left(1-\coup\left(s\right)\right)\cdot\\
  &\qquad \cdot\int_{0}^{\sgn\left(s\right)\infty}\frac{\cut\left(s+t\right)\ee^{\im t\xi}}{\left(s+t\right)^{1+\kappa}\left(1+\im\left(s+t\right)\right)}\ee^{\Phi\left(\im\left(s+t\right)\right)-\Phi\left(\im s\right)}\dt\ds\\
  &=\frac{\eps}{2\pi\im p\left(1+p\right)\left(\xi-\left(1+p\right)\right)}\int_{-\infty}^{\infty}\left(\ee^{\im s\left(\xi-\left(p+1\right)\right)}-1\right)\del_{s}\left(\left(\left(1+\kappa\right)\beta_W\left(\im s\right)+\im s\beta_W'\left(\im s\right)\right)s^{\kappa}\left(1-\coup\left(s\right)\right)\right)\cdot\\
  &\qquad \cdot \int_{0}^{\sgn\left(s\right)\infty}\frac{\cut\left(s+t\right)\ee^{\im t\xi}}{\left(s+t\right)^{1+\kappa}\left(1+\im\left(s+t\right)\right)}\ee^{\Phi\left(\im\left(s+t\right)\right)-\Phi\left(\im s\right)}\dt\ds\\
  &+\frac{\eps}{2\pi\im p\left(1+p\right)\left(\xi-\left(1+p\right)\right)}\int_{-\infty}^{\infty}\left(\ee^{\im s\left(\xi-\left(p+1\right)\right)}-1\right)\left(\left(1+\kappa\right)\beta_W\left(\im s\right)+\im s\beta_W'\left(\im s\right)\right)s^{\kappa}\left(1-\coup\left(s\right)\right)\cdot\\
  &\qquad \cdot \int_{0}^{\sgn\left(s\right)\infty}\del_{s}\left(\frac{\cut\left(s+t\right)\ee^{\im t\xi}}{\left(s+t\right)^{1+\kappa}\left(1+\im\left(s+t\right)\right)}\right)\ee^{\Phi\left(\im\left(s+t\right)\right)-\Phi\left(\im s\right)}\dt\ds\\
  &\quad +\frac{\eps}{2\pi p\left(1+p\right)\left(\xi-\left(1+p\right)\right)}\int_{-\infty}^{\infty}\left(\ee^{\im s\left(\xi-\left(p+1\right)\right)}-1\right)\left(\left(1+\kappa\right)\beta_W\left(\im s\right)+\im s\beta_W'\left(\im s\right)\right)s^{\kappa}\left(1-\coup\left(s\right)\right)\cdot\\
  &\qquad \cdot \int_{0}^{\sgn\left(s\right)\infty}\frac{\cut\left(s+t\right)\ee^{\im t\xi}}{\left(s+t\right)^{1+\kappa}\left(1+\im\left(s+t\right)\right)}\left(\Phi'\left(\im\left(s+t\right)\right)-\Phi'\left(\im s\right)\right)\ee^{\Phi\left(\im\left(s+t\right)\right)-\Phi\left(\im s\right)}\dt\ds.
 \end{split}
\end{equation*}
Estimating by the most singular contribution we obtain for $\abs{s}\leq 2$ that
\begin{align*}
 \abs{\del_{s}\left(\left(\left(1+\kappa\right)\beta_W\left(\im s\right)+\im s\beta_W'\left(\im s\right)\right)s^{\kappa}\left(1-\coup\left(s\right)\right)\right)}&\leq C\abs{s}^{-\alpha-1+\kappa}\\
 \abs{\left(\left(1+\kappa\right)\beta_W\left(\im s\right)+\im s\beta_W'\left(\im s\right)\right)s^{\kappa}\left(1-\coup\left(s\right)\right)}&\leq C\abs{s}^{-\alpha+\kappa}. 
\end{align*}
Thus together with Lemmas~\ref{Lem:asymptotics:xi:leq:p} and~\ref{Lem:properties:Phi} as well as $\eps\abs{s}^{-\alpha}\ee^{-\frac{D\eps}{\abs{s}^{\alpha}}}\leq C$ we obtain for some $a\in\left(2\delta,\theta\right)$ that
\begin{equation*}
 \begin{split}
  \abs{Q_{2,1,2,1}}&\leq \frac{C}{p\left(1+p\right)\abs{\xi-\left(p+1\right)}^{1-a}}\int_{-2}^{2}\biggl(\abs{s}^{-\alpha-1+\kappa}\Bigl(\abs{s}^{\alpha-\kappa}+\eps\abs{s}^{-\delta}\ee^{-\frac{B\eps}{\abs{s}^{\alpha}}}\Bigr)\\
  &\quad+\abs{s}^{-\alpha+\kappa}\Bigl(\abs{s}^{\alpha-1-\kappa}+\eps\abs{s}^{-1-\kappa}\ee^{-\frac{B\eps}{\abs{s}^{\alpha}}}\Bigr)+\abs{s}^{-\alpha+\kappa}\Bigl(\abs{s}^{\alpha-1-\kappa}+\eps^{2}\abs{s}^{-1-\alpha}\ee^{-\frac{B\eps}{\abs{s}^{\alpha}}}\Bigr)\biggr)\ds\\
  &\leq \frac{C}{p\left(1+p\right)\abs{\xi-\left(p+1\right)}^{1-a}}\int_{-2}^{2}\abs{s}^{a-1-2\delta}\ds\leq \frac{C}{p\left(1+p\right)\abs{\xi-\left(p+1\right)}^{1-a}},
 \end{split}
\end{equation*}
while the contribution to the integral in~\eqref{eq:int:Q0} can be controlled analogously to~\eqref{eq:contr:Q1221}.

\begin{remark}
 If $\alpha=0$ we can proceed similarly as in Remark~\ref{Rem:adapt:11.4}.
\end{remark}

\subsubsection{Contribution of $W_{2,1}\left(\Y,p\right)$ for $\xi>p+1$}

We next consider the region $\xi>p+1$ and introducing the cut-off $\coup$ we have to consider the integral
\begin{equation*}
 \begin{split}
  &\quad-\frac{1}{2\pi\im}\frac{\eps}{p\left(1+p\right)}\int_{-R\im}^{R\im}\frac{\ee^{\Y\xi}}{\Y\left(1+\Y\right)}\int_{0}^{\Y}\ee^{-\left(p+1\right)x}\del_{x}\left(x\beta_W\left(x\right)\left(\frac{x}{\Y}\right)^{\kappa}\right)\frac{\ee^{\Phi\left(\Y\right)}}{\ee^{\Phi\left(x\right)}}\dx\left(\coup\left(\im \Y\right)+1-\coup\left(\im \Y\right)\right)\dd{\Y}\\
  &=\widehat{Q}_{2,1,1}+\widehat{Q}_{2,1,2}
 \end{split}
\end{equation*}
with 
\begin{align*}
 \widehat{Q}_{2,1,1}&\vcc=-\frac{1}{2\pi\im}\frac{\eps}{p\left(1+p\right)}\int_{-R\im}^{R\im}\frac{\ee^{\Y\xi}}{\Y\left(1+\Y\right)}\int_{0}^{\Y}\ee^{-\left(p+1\right)x}\del_{x}\left(x\beta_W\left(x\right)\left(\frac{x}{\Y}\right)^{\kappa}\right)\frac{\ee^{\Phi\left(\Y\right)}}{\ee^{\Phi\left(x\right)}}\dx\coup\left(\im \Y\right)\dd{\Y},\\
 \widehat{Q}_{2,1,1}&\vcc=-\frac{1}{2\pi\im}\frac{\eps}{p\left(1+p\right)}\int_{-R\im}^{R\im}\frac{\ee^{\Y\xi}}{\Y\left(1+\Y\right)}\int_{0}^{\Y}\ee^{-\left(p+1\right)x}\del_{x}\left(x\beta_W\left(x\right)\left(\frac{x}{\Y}\right)^{\kappa}\right)\frac{\ee^{\Phi\left(\Y\right)}}{\ee^{\Phi\left(x\right)}}\dx\left(1-\coup\left(\im \Y\right)\right)\dd{\Y}.
\end{align*}
 
Proceeding as before we first consider $\widehat{Q}_{2,1,1}$, where we can integrate by parts to obtain
\begin{equation*}
 \begin{split}
  &\quad \widehat{Q}_{2,1,1}\\
  &=-\frac{\eps}{2\pi \im p\left(1+p\right)\xi}\int_{-R\im}^{R\im}\frac{\del_{\Y}\left(\ee^{\Y\xi}-1\right)}{\Y^{1+\kappa}\left(1+\Y\right)}\coup\left(\im \Y\right)\int_{0}^{\Y}\ee^{-\left(1+p\right)x}\del_{x}\left(x^{1+\kappa}\beta_W\left(x\right)\right)\ee^{-\Phi\left(x\right)}\dx\ee^{\Phi\left(\Y\right)}\dd{\Y}\\
  &=-\frac{\eps}{2\pi \im p\left(1+p\right)\xi}\left.\left(\frac{\ee^{\Y\xi}-1}{\Y^{1+\kappa}\left(1+\Y\right)}\coup\left(\im \Y\right)\int_{0}^{\Y}\ee^{-\left(1+p\right)x}\del_{x}\left(x^{1+\kappa}\beta_W\left(x\right)\right)\ee^{-\Phi\left(x\right)}\dx\ee^{\Phi\left(\Y\right)}\right)\right|_{\Y=-R\im}^{\Y=R\im}\\
  &\quad +\frac{\eps}{2\pi \im p\left(1+p\right)\xi}\int_{-R\im}^{R\im}\left(\ee^{\Y\xi}-1\right)\del_{\Y}\left(\frac{\coup\left(\im \Y\right)}{\Y^{1+\kappa}\left(1+\Y\right)}\right)\int_{0}^{\Y}\ee^{-\left(1+p\right)x}\del_{x}\left(x^{1+\kappa}\beta_W\left(x\right)\right)\frac{\ee^{\Phi\left(\Y\right)}}{\ee^{\Phi\left(x\right)}}\dx\dd{\Y}\\
  &\quad+\frac{\eps}{2\pi \im p\left(1+p\right)\xi}\int_{-R\im}^{R\im}\left(\ee^{\Y\xi}-1\right)\frac{\coup\left(\im \Y\right)}{\Y\left(1+\Y\right)}\left(\left(1+\kappa\right)\beta_W\left(\Y\right)+\Y\beta_W'\left(\Y\right)\right)\ee^{-\left(1+p\right)\Y}\dd{\Y}\\
  &\quad +\frac{\eps}{2\pi \im p\left(1+p\right)\xi}\int_{-R\im}^{R\im}\left(\ee^{\Y\xi}-1\right)\frac{\coup\left(\im \Y\right)}{\Y^{1+\kappa}\left(1+\Y\right)}\Phi'\left(\Y\right)\int_{0}^{\Y}\ee^{-\left(1+p\right)x}\del_{x}\left(x^{1+\kappa}\beta_W\left(x\right)\right)\frac{\ee^{\Phi\left(\Y\right)}}{\ee^{\Phi\left(x\right)}}\dx\dd{\Y}\\
  &=\vcc(I)+(II)+(III)+(IV).
 \end{split}
\end{equation*}
As before we first note
\begin{equation*}
 \begin{split}
  \abs{\int_{0}^{\Y}\ee^{-\left(1+p\right)x}\del_{x}\left(x^{1+\kappa}\beta_W\left(x\right)\right)\ee^{-\Phi\left(x\right)}\dx}\leq C\max\left\{\abs{\Y}^{1+\alpha+\kappa},\abs{\Y}^{1-\alpha+\kappa}\right\}.
 \end{split}
\end{equation*}
Using this we obtain
\begin{equation*}
 \begin{split}
  \abs{(I)}\leq \frac{C}{p\left(1+p\right)\xi}\frac{R^{1+\alpha+\kappa}}{R^{2+\kappa}}\leq\frac{C}{p\left(1+p\right)\xi}R^{\alpha-1}\longrightarrow 0 \quad \text{for }R\longrightarrow \infty.
 \end{split}
\end{equation*}
Furthermore using the regularising effect of $\coup$ we find
\begin{equation*}
 \begin{split}
  \abs{(II)}&\leq \frac{C}{p\left(1+p\right)\xi}\int_{-R}^{R}\chi_{\left\{\abs{\Y}\geq \frac{1}{2}\right\}}\abs{\Y}^{-3-\kappa}\max\left\{\abs{\Y}^{1+\alpha+\kappa},\abs{\Y}^{1-\alpha+\kappa}\right\}\dd{\Y}\leq \frac{C}{p\left(1+p\right)\xi}
 \end{split}
\end{equation*}
uniformly in $R$ as $\alpha<1/2$. Analogously
\begin{equation*}
 \begin{split}
  \abs{(III)}\leq \frac{C}{p\left(1+p\right)\xi}\leq \int_{-\infty}^{\infty}\chi_{\left\{\abs{\Y}\geq \frac{1}{2}\right\}}\abs{\Y}^{-2}\max\left\{\abs{\Y}^{\alpha},\abs{\Y}^{-\alpha}\right\}\dd{\Y}\leq \frac{C}{p\left(1+p\right)\xi}.
 \end{split}
\end{equation*}
Finally we also get
\begin{equation*}
 \begin{split}
  \abs{(IV)}&\leq \frac{C}{p\left(1+p\right)\xi}\int_{-\infty}^{\infty}\chi_{\left\{\abs{\Y}\geq \frac{1}{2}\right\}}\abs{\Y}^{-2-\kappa}\left(\abs{\Y}^{-\alpha-1}+\abs{\Y}^{\alpha-1}\right)\max\left\{\abs{\Y}^{1+\alpha+\kappa},\abs{\Y}^{1-\alpha+\kappa}\right\}\dd{\Y}\\
  &\leq \frac{C}{p\left(1+p\right)\xi}
 \end{split}
\end{equation*}
as $\alpha<1/2$. Thus together we have
\begin{equation*}
 \abs{\widehat{Q}_{2,1,1}}\leq \frac{C}{p\left(1+p\right)\xi} \quad \text{for } p+1\leq \xi
\end{equation*}
and thus similarly as in~\eqref{eq:contr:Q122} the contribution to~\eqref{eq:Q0} can be controlled as desired.

It remains to consider $\widehat{Q}_{2,1,2}$. Taking the limit $R\to\infty$, which is possible as $1-\coup$ is supported close to the origin, transforming to integrals over real variables and rearranging we obtain
\begin{multline*}
  \lim_{R\to\infty}\widehat{Q}_{2,1,2}=\frac{\im\eps}{2\pi p\left(1+p\right)}\int_{-\infty}^{\infty}\ee^{-\im\left(p+1\right)s}\left(\left(1+\kappa\right)\beta_W\left(\im s\right)+\im s\beta_W'\left(\im s\right)\right)s^{\kappa}\ee^{-\Phi\left(\im s\right)}\cdot\\
  \cdot\int_{s}^{\sgn\left(s\right)\infty}\frac{\ee^{\im \Y\xi}\ee^{\Phi\left(\im \Y\right)}\left(1-\coup\left(\Y\right)\right)}{\im \Y^{1+\kappa}\left(1+\im \Y\right)}\dd{\Y}\ds.
\end{multline*}
Changing variables, $\Y=s+t$, and integrating by parts in $s$ by means of $\ee^{-\im\left(p+1\right)s}=\del_{s}\left(\ee^{\im s\left(\xi-\left(p+1\right)\right)}-1\right)$ we arrive at
\begin{equation*}
 \begin{split}
  \lim_{R\to\infty}\widehat{Q}_{2,1,2}
  &=-\frac{\eps}{2\pi\im p\left(1+p\right)}\frac{1}{\xi-\left(1+p\right)}\int_{-\infty}^{\infty}\left(\ee^{\im s\left(\xi-\left(1+p\right)\right)}-1\right)\del_{s}\left(\left(\left(1+\kappa\right)\beta_W\left(\im s\right)+\im s\beta_W'\left(\im s\right)\right)s^{\kappa}\right)\cdot\\
  &\qquad \cdot\int_{0}^{\sgn\left(s\right)\infty}\frac{\ee^{\im \xi t}\left(1-\coup\left(s+t\right)\right)}{\left(s+t\right)^{1+\kappa}\left(1+\im\left(s+t\right)\right)}\ee^{\Phi\left(\im\left(s+t\right)\right)-\Phi\left(\im s\right)}\dt\ds\\
  &\quad -\frac{\eps}{2\pi\im p\left(1+p\right)}\frac{1}{\xi-\left(1+p\right)}\int_{-\infty}^{\infty}\left(\ee^{\im s\left(\xi-\left(1+p\right)\right)}-1\right)\left(\left(1+\kappa\right)\beta_W\left(\im s\right)+\im s\beta_W'\left(\im s\right)\right)s^{\kappa}\cdot\\
  &\qquad \cdot\int_{0}^{\sgn\left(s\right)\infty}\del_{s}\left(\frac{\ee^{\im \xi t}\left(1-\coup\left(s+t\right)\right)}{\left(s+t\right)^{1+\kappa}\left(1+\im\left(s+t\right)\right)}\right)\ee^{\Phi\left(\im\left(s+t\right)\right)-\Phi\left(\im s\right)}\dt\ds\\
  &\quad-\frac{\eps}{2\pi p\left(1+p\right)}\frac{1}{\xi-\left(1+p\right)}\int_{-\infty}^{\infty}\left(\ee^{\im s\left(\xi-\left(1+p\right)\right)}-1\right)\left(\left(1+\kappa\right)\beta_W\left(\im s\right)+\im s\beta_W'\left(\im s\right)\right)s^{\kappa}\cdot\\
  &\qquad \cdot\int_{0}^{\sgn\left(s\right)\infty}\frac{\ee^{\im \xi t}\left(1-\coup\left(s+t\right)\right)}{\left(s+t\right)^{1+\kappa}\left(1+\im\left(s+t\right)\right)}\left(\Phi'\left(\im\left(s+t\right)\right)-\Phi'\left(\im s\right)\right)\ee^{\Phi\left(\im\left(s+t\right)\right)-\Phi\left(\im s\right)}\dt\ds.
 \end{split}
\end{equation*}
 For $\abs{s}\leq 1$ we have
 \begin{align*}
  \abs{\del_{s}\left(\left(\left(1+\kappa\right)\beta_W\left(\im s\right)+\im s\beta_W'\left(\im s\right)\right)s^{\kappa}\right)}&\leq C\abs{s}^{-\alpha-1+\kappa}\\
  \abs{\left(\left(1+\kappa\right)\beta_W\left(\im s\right)+\im s\beta_W'\left(\im s\right)\right)s^{\kappa}}&\leq C\abs{s}^{-\alpha+\kappa}.
 \end{align*}
Thus we obtain for some $a\in\left(2\delta,\theta\right)$ with $\delta>0$ sufficiently small together with Lemmas~\ref{Lem:asymptotics:p:leq:xi} and~\ref{Lem:properties:Phi} similarly as before
\begin{equation*}
 \begin{split}
  \abs{\widehat{Q}_{2,1,2}}\leq \frac{C}{p\left(1+p\right)\abs{\xi-\left(1+p\right)}^{1-a}}\int_{-1}^{1}\abs{s}^{a-1-2\delta}\ds\leq \frac{C}{p\left(1+p\right)\abs{\xi-\left(1+p\right)}^{1-a}},
 \end{split}
\end{equation*}
while the resulting contribution to~\eqref{eq:Q0} can be shown to be fine similarly as in~\eqref{eq:contr:Q1211}.

\begin{remark}
 In the case $\alpha=0$ we have to adapt the argument similarly as in Remark~\ref{Rem:adapt:11.4}.
\end{remark}

\subsection{Contribution of $W_{2,2}\left(\Y,p\right)$}

Finally we consider the contribution of 
\begin{equation*}
 \begin{split}
  W_{2,2}\left(\Y,p\right)=\frac{\eps}{p\left(1+p\right)\Y\left(1+\Y\right)}\int_{0}^{\Y}\ee^{-\left(1+p\right)x}x\beta_W\left(x\right)\left(\frac{x}{\Y}\right)^{\kappa}\Phi'\left(x\right)\frac{\ee^{\Phi\left(\Y\right)}}{\ee^{\Phi\left(x\right)}}\dx.
 \end{split}
\end{equation*}

\subsubsection{Contribution of $W_{2,2}\left(\Y,p\right)$ for $\xi<p$}

We first consider $\xi<p$. Rewriting (using $\Phi'\left(x\right)=-\eps\beta_W\left(x\right)x^{-1}\ee^{-x}$) and introducing the cut-off $\coup$ we get
\begin{equation*}
 \begin{split}
  W_{2,2}\left(\Y,p\right)&=-\frac{\eps^2}{p\left(1+p\right)}\frac{1}{\Y\left(1+\Y\right)}\int_{0}^{\Y}\ee^{-\left(p+2\right)x}\left(\beta_W\left(x\right)\right)^2\left(\frac{x}{\Y}\right)^{\kappa}\frac{\ee^{\Phi\left(\Y\right)}}{\ee^{\Phi\left(x\right)}}\coup\left(\im x\right)\dx\\
  &\quad -\frac{1}{p\left(1+p\right)}\frac{1}{\Y\left(1+\Y\right)}\int_{0}^{\Y}\ee^{-px}x^2\left(\Phi'\left(x\right)\right)^2\left(\frac{x}{\Y}\right)^{\kappa}\frac{\ee^{\Phi\left(\Y\right)}}{\ee^{\Phi\left(x\right)}}\left(1-\coup\left(\im x\right)\right)\dx\\
  &=\vcc W_{2,2,1}\left(\Y,p\right)+W_{2,2,2}\left(\Y,p\right).
 \end{split}
\end{equation*}
We first consider $W_{2,2,1}\left(\Y,p\right)$ and we integrate by parts to obtain
\begin{equation*}
 \begin{split}
   W_{2,2,1}\left(\Y,p\right)&=\frac{\eps^2}{p\left(1+p\right)\left(p+2\right)}\frac{1}{\Y\left(1+\Y\right)}\int_{0}^{\Y}\del_{x}\left(\ee^{-\left(p+2\right)x}\right)\left(\beta_W\left(x\right)\right)^2\left(\frac{x}{\Y}\right)^{\kappa}\frac{\ee^{\Phi\left(\Y\right)}}{\ee^{\Phi\left(x\right)}}\coup\left(\im x\right)\dx\\
  &=\frac{\eps^2\ee^{-\left(p+2\right)\Y}}{p\left(1+p\right)\left(2+p\right)}\frac{\left(\beta_W\left(\Y\right)\right)^2\coup\left(\im \Y\right)}{\Y\left(1+\Y\right)}\\
  &\quad -\frac{\eps^2}{p\left(1+p\right)\left(p+2\right)}\frac{1}{\Y\left(1+\Y\right)}\int_{0}^{\Y}\ee^{-\left(p+2\right)x}\del_{x}\left(\left(\beta_W\left(x\right)\right)^2\left(\frac{x}{\Y}\right)^{\kappa}\coup\left(\im x\right)\right)\frac{\ee^{\Phi\left(\Y\right)}}{\ee^{\Phi\left(x\right)}}\dx\\
  &\quad -\frac{\eps^3}{p\left(1+p\right)\left(p+2\right)}\frac{1}{\Y\left(1+\Y\right)}\int_{0}^{\Y}\ee^{-\left(p+3\right)x}\frac{\left(\beta_W\left(x\right)\right)^3}{x}\left(\frac{x}{\Y}\right)^{\kappa}\coup\left(\im x\right)\frac{\ee^{\Phi\left(\Y\right)}}{\ee^{\Phi\left(x\right)}}\dx\\
  &=\vcc W_g\left(\Y,p\right)+W_{b}\left(\Y,p\right)
 \end{split}
\end{equation*}
with 
\begin{align*}
 W_g\left(\Y,p\right)&\vcc=\eps^2\frac{\ee^{-\left(p+2\right)\Y}}{p\left(1+p\right)\left(2+p\right)}\frac{\left(\beta_W\left(\Y\right)\right)^2\coup\left(\im \Y\right)}{\Y\left(1+\Y\right)}\\
  &\quad -\frac{\eps^2}{p\left(1+p\right)\left(p+2\right)}\frac{1}{\Y\left(1+\Y\right)}\int_{0}^{\Y}\ee^{-\left(p+2\right)x}\del_{x}\left(\left(\beta_W\left(x\right)\right)^2\left(\frac{x}{\Y}\right)^{\kappa}\coup\left(\im x\right)\right)\frac{\ee^{\Phi\left(\Y\right)}}{\ee^{\Phi\left(x\right)}}\dx,\\
 W_{b}\left(\Y,p\right)&\vcc=-\frac{\eps^3}{p\left(1+p\right)\left(p+2\right)}\frac{1}{\Y\left(1+\Y\right)}\int_{0}^{\Y}\ee^{-\left(p+3\right)x}\frac{\left(\beta_W\left(x\right)\right)^3}{x}\left(\frac{x}{\Y}\right)^{\kappa}\coup\left(\im x\right)\frac{\ee^{\Phi\left(\Y\right)}}{\ee^{\Phi\left(x\right)}}\dx.
\end{align*}
The term $W_{g}\left(\Y,p\right)$ can be estimated directly and we obtain by similar arguments as before that
\begin{equation*}
 \abs{W_{g}\left(\Y,p\right)}\leq \frac{C}{p\left(1+p\right)^2}\chi_{\left\{\abs{\Y}\geq \frac{1}{2}\right\}}\abs{\Y}^{2\alpha-2}.
\end{equation*}
Thus the contribution to $Q_0$ is controlled by
\begin{equation*}
 \begin{split}
  \abs{\frac{1}{2\pi\im}\int_{-\im R}^{\im R}\ee^{\Y\xi}W_{g}\left(\Y,p\right)\dd{\Y}}\leq \frac{C}{p\left(1+p\right)^2}\int_{-\infty}^{\infty}\chi_{\left\{\abs{\Y}\geq \frac{1}{2}\right\}}\abs{\Y}^{2\alpha-2}\dd{\Y}\leq \frac{C}{p\left(1+p\right)^2},
 \end{split}
\end{equation*}
which can be treated similarly as in~\eqref{eq:contr:W121}. To control the term $W_{b}\left(\Y,p\right)$ we have to integrate by parts once more to get enough decay in $\Y$ which is no problem due to the cut-off $\coup$, i.e.\ we have
\begin{equation*}
 \begin{split}
  W_{b}\left(\Y,p\right)&=\frac{\eps^3}{p\left(1+p\right)\left(2+p\right)\left(3+p\right)\Y\left(1+\Y\right)}\int_{0}^{\Y}\del_{x}\left(\ee^{-\left(p+3\right)x}\right)\frac{\left(\beta_W\left(x\right)\right)^3}{x}\left(\frac{x}{\Y}\right)^{\kappa}\coup\left(\im x\right)\frac{\ee^{\Phi\left(\Y\right)}}{\ee^{\Phi\left(x\right)}}\dx\\
  &=\frac{\eps^3\ee^{-\left(p+3\right)\Y}}{p\left(1+p\right)\left(2+p\right)\left(3+p\right)\Y^2\left(1+\Y\right)}\coup\left(\im \Y\right)\left(\beta_W\left(\Y\right)\right)^3\\
  &\quad -\frac{\eps^3}{p\left(1+p\right)\left(2+p\right)\left(3+p\right)\Y\left(1+\Y\right)}\int_{0}^{\Y}\ee^{-\left(p+3\right)x}\del_{x}\left(\frac{\left(\beta_W\left(x\right)\right)^3}{x}\left(\frac{x}{\Y}\right)^{\kappa}\coup\left(\im x\right)\right)\frac{\ee^{\Phi\left(\Y\right)}}{\ee^{\Phi\left(x\right)}}\dx\\
  &\quad +\frac{\eps^4}{p\left(1+p\right)\left(2+p\right)\left(3+p\right)\Y\left(1+\Y\right)}\int_{0}^{\Y}\ee^{-\left(p+4\right)x}\frac{\left(\beta_W\left(x\right)\right)^4}{x^2}\left(\frac{x}{\Y}\right)^{\kappa}\coup\left(\im x\right)\frac{\ee^{\Phi\left(\Y\right)}}{\ee^{\Phi\left(x\right)}}\dx.
 \end{split}
\end{equation*}
 Similarly as before we can estimate now
 \begin{equation*}
  \begin{split}
   \abs{W_{b}\left(\Y,p\right)}\leq \frac{C}{p\left(1+p\right)^2}\chi_{\left\{\abs{\Y}\geq \frac{1}{2}\right\}}\max\left\{\abs{\Y}^{4\alpha-3},\abs{\Y}^{-2}\right\}.
  \end{split}
 \end{equation*}
Thus the contribution to $Q_0$ can be controlled by
\begin{equation*}
 \begin{split}
  \abs{\frac{1}{2\pi\im}\int_{-\im R}^{\im R}\ee^{\Y\xi}W_{b}\left(\Y,p\right)\dd{\Y}}\leq \frac{C}{p\left(1+p\right)^2}\int_{-\infty}^{\infty}\chi_{\left\{\abs{\Y}\geq \frac{1}{2}\right\}}\max\left\{\abs{\Y}^{4\alpha-3},\abs{\Y}^{-2}\right\}\dd{\Y}\leq \frac{C}{p\left(1+p\right)^2},
 \end{split}
\end{equation*}
which is the same as for $W_g$ and thus also gives the right contribution similarly to~\eqref{eq:contr:W121}

Next we consider the contribution of $W_{2,2,2}\left(\Y,p\right)$. Taking the limit $R\to\infty$, changing to real variables and introducing the cut-off we then find
\begin{equation*}
 \begin{split}
  &\quad -\lim_{R\to\infty}\frac{1}{2\pi\im p\left(1+p\right)}\int_{-\im R}^{\im R}\ee^{\Y\xi}W_{2,2,2}\left(\Y,p\right)\dd{\Y}\\
  &=-\frac{1}{2\pi p\left(1+p\right)}\int_{-\infty}^{\infty}\ee^{-\im p s}\left(\im s\Phi'\left(\im s\right)\right)^2s^{\kappa}\ee^{-\Phi\left(\im s\right)}\int_{s}^{\sgn\left(s\right)\infty}\ee^{\im \Y\xi}\frac{\cut\left(\Y\right)\left(1-\coup\left(s\right)\right)}{\Y^{1+\kappa}\left(1+\im \Y\right)}\ee^{\Phi\left(\im \Y\right)}\dd{\Y}\ds\\
  &\quad -\frac{1}{2\pi p\left(1+p\right)}\int_{-\infty}^{\infty}\ee^{-\im p s}\left(\im s\Phi'\left(\im s\right)\right)^2s^{\kappa}\ee^{-\Phi\left(\im s\right)}\int_{s}^{\sgn\left(s\right)\infty}\ee^{\im \Y\xi}\frac{\left(1-\cut\left(\Y\right)\right)\left(1-\coup\left(s\right)\right)}{\Y^{1+\kappa}\left(1+\im \Y\right)}\ee^{\Phi\left(\im \Y\right)}\dd{\Y}\ds\\
  &=\vcc Q_{2,2,1}+Q_{2,2,2}.
 \end{split}
\end{equation*}
As previously we first consider $Q_{2,2,2}$ where we can integrate by parts again to obtain
\begin{equation*}
 \begin{split}
  Q_{2,2,2}&=\frac{\im}{2\pi p^2\left(1+p\right)}\int_{-\infty}^{\infty}\del_{s}\left(\ee^{-\im p s}-1\right)\left(\Phi'\left(\im s\right)\right)^2s^{2+\kappa}\left(1-\coup\left(s\right)\right)\ee^{-\Phi\left(\im s\right)}\cdot\\
  &\qquad\qquad\qquad\qquad\qquad\qquad\qquad\cdot\int_{s}^{\sgn\left(s\right)\infty}\ee^{\im \Y\xi}\frac{1-\cut\left(\Y\right)}{ \Y^{1+\kappa}\left(1+\im \Y\right)}\ee^{\Phi\left(\im \Y\right)}\dd{\Y}\ds\\
  &=-\frac{\im}{2\pi p^2\left(1+p\right)}\int_{-\infty}^{\infty}\left(\ee^{-\im p s}-1\right)\del_{s}\left(\left(\Phi'\left(\im s\right)\right)^2s^{2+\kappa}\left(1-\coup\left(s\right)\right)\ee^{-\Phi\left(\im s\right)}\right)\int_{s}^{\sgn\left(s\right)\infty}\left(\cdots\right)\dd{\Y}\ds\\
  &\quad +\frac{\im}{2\pi p^2\left(1+p\right)}\int_{-\infty}^{\infty}\left(\ee^{-\im p s}-1\right)\left(\Phi'\left(\im s\right)\right)^2 s^2\left(1-\coup\left(s\right)\right)\ee^{\im s\xi}\frac{1-\cut\left(s\right)}{ s^{1+\kappa}\left(1+\im s\right)}\ds.
 \end{split}
\end{equation*}
 Similarly as before we obtain from Lemma~\ref{Lem:asymptotics:exp:Phi} together with the cut-off functions $\cut$ and $1-\coup$ that
 \begin{align*}
  \abs{\del_{s}\left(\left(\Phi'\left(\im s\right)\right)^2s^{2+\kappa}\left(1-\coup\left(s\right)\right)\ee^{-\Phi\left(\im s\right)}\right)\int_{s}^{\sgn\left(s\right)\infty}\ee^{\im \Y\xi}\frac{1-\cut\left(\Y\right)}{ \Y^{1+\kappa}\left(1+\im \Y\right)}\ee^{\Phi\left(\im \Y\right)}\dd{\Y}}&\leq C\abs{s}^{\kappa-1}\chi_{\left\{\abs{s}\leq 1\right\}}\\
  \abs{\left(\Phi'\left(\im s\right)\right)^2 s^2\left(1-\coup\left(s\right)\right)\ee^{\im s\xi}\frac{1-\cut\left(s\right)}{s^{1+\kappa}\left(1+\im s\right)}}&=0.
 \end{align*}
With these estimates and $\abs{\ee^{-\im p s}-1}\leq Cp^{a}\abs{s}^{a}$ with $a\in(\delta,\theta)$ we immediately get
\begin{equation*}
 \begin{split}
  \abs{Q_{2,2,2}}\leq \frac{C}{p^{2-a}\left(1+p\right)},
 \end{split}
\end{equation*}
which gives a contribution to the integral in~\eqref{eq:Q0} that can be controlled similarly to~\eqref{eq:contr:W121}.

\begin{remark}
 The case $\alpha=0$ is treated analogously as in Remark~\ref{Rem:est:bd:pert:v3}. 
\end{remark}

To estimate the contribution of $Q_{2,2,1}$ we proceed in the same was as before, i.e.\ we change variables, $\Y=s+t$, rearrange and integrate by parts in $s$ to obtain
\begin{equation*}
 \begin{split}
  Q_{2,2,1}&=\frac{\im}{2\pi p\left(1+p\right)\left(\xi-p\right)}\int_{-\infty}^{\infty}\left(\ee^{\im s\left(\xi-p\right)}-1\right)\del_{s}\left(\left(\Phi'\left(\im s\right)\right)^2s^{2+\kappa}\left(1-\coup\left(s\right)\right)\right)\cdot\\
  &\qquad\qquad \cdot\int_{0}^{\sgn\left(s\right)\infty}\frac{\cut\left(s+t\right)\ee^{\im t\xi}}{\left(s+t\right)^{1+\kappa}\left(1+\im\left(s+t\right)\right)}\ee^{\Phi\left(\im\left(s+t\right)\right)-\Phi\left(\im s\right)}\dt\ds\\
  &\quad +\frac{\im}{2\pi p\left(1+p\right)\left(\xi-p\right)}\int_{-\infty}^{\infty}\left(\ee^{\im s\left(\xi-p\right)}-1\right)\left(\Phi'\left(\im s\right)\right)^2s^{2+\kappa}\left(1-\coup\left(s\right)\right)\cdot\\
  &\qquad \qquad \cdot\int_{0}^{\sgn\left(s\right)\infty}\del_{s}\left(\frac{\cut\left(s+t\right)\ee^{\im t\xi}}{\left(s+t\right)^{1+\kappa}\left(1+\im\left(s+t\right)\right)}\right)\ee^{\Phi\left(\im\left(s+t\right)\right)-\Phi\left(\im s\right)}\dt\ds\\
  &\quad -\frac{1}{2\pi p\left(1+p\right)\left(\xi-p\right)}\int_{-\infty}^{\infty}\left(\ee^{\im s\left(\xi-p\right)}-1\right)\left(\Phi'\left(\im s\right)\right)^2s^{2+\kappa}\left(1-\coup\left(s\right)\right)\cdot\\
  &\qquad \qquad \cdot\int_{0}^{\sgn\left(s\right)\infty}\frac{\cut\left(s+t\right)\ee^{\im t\xi}}{\left(s+t\right)^{1+\kappa}\left(1+\im\left(s+t\right)\right)}\left(\Phi'\left(\im\left(s+t\right)\right)-\Phi'\left(\im s\right)\right)\ee^{\Phi\left(\im\left(s+t\right)\right)-\Phi\left(\im s\right)}\dt\ds.
 \end{split}
\end{equation*}
Using the estimates
\begin{align*}
 \abs{\del_{s}\left(\left(\Phi'\left(\im s\right)\right)^2s^{2+\kappa}\left(1-\coup\left(s\right)\right)\right)}&\leq C\eps^{2}\abs{s}^{-2\alpha-1+\kappa},\\
 \abs{\left(\Phi'\left(\im s\right)\right)^2s^{2+\kappa}\left(1-\coup\left(s\right)\right)}&\leq C\eps^{2}\abs{s}^{-2\alpha+\kappa}
\end{align*}
for $\abs{s}\leq 2$ together with Lemmas~\ref{Lem:asymptotics:xi:leq:p} and~\ref{Lem:properties:Phi} as well as $\eps\abs{s}^{-\alpha}\ee^{-\frac{D\eps}{\abs{s}^{\alpha}}}\leq C$ we obtain for some $a\in\left(\alpha+2\delta,\theta\right)$ with $\delta>0$ sufficiently small that
\begin{equation*}
 \begin{split}
  \abs{Q_{2,2,1}}\leq \frac{C}{p\left(1+p\right)\abs{\xi-p}^{1-a}}\int_{-2}^{2}\abs{s}^{a-\alpha-1-2\delta}\ds\leq \frac{C}{p\left(1+p\right)\abs{\xi-p}^{1-a}}.
 \end{split}
\end{equation*}
Thus it follows for the contribution to the integral in~\eqref{eq:int:Q0} that
\begin{equation*}
 \begin{split}
  &\quad \frac{C}{p^2}\int_{0}^{p}\frac{1}{\left(1+\xi\right)^{2\theta-\nu}\abs{\xi-p}^{1-a}}\dxi\leq \frac{C}{p^2}\left(\frac{1}{p^{1-a}}\int_{0}^{p/2}\frac{1}{\xi^{2\theta-\nu}}\dxi+\frac{1}{p^{2\theta-\nu}}\int_{p/2}^{p}\frac{1}{\abs{\xi-p}^{1-a}}\dxi\right)\\
  &\leq \frac{C}{p^{2+2\theta-a-\nu}}\leq C\frac{\left(1+p\right)^{1-\theta}}{p^3}
 \end{split}
\end{equation*}
if we choose $\nu>0$ sufficiently small for fixed $a$.

\begin{remark}
 The case $\alpha=0$ can be treated similarly (see also Remark~\ref{Rem:adapt:11.4}).
\end{remark}

\subsubsection{Contribution of $W_{2,2}\left(\Y,p\right)$ for $\xi>p$}

It remains now to consider the region $\xi>p$ and we have to consider
\begin{equation*}
 \begin{split}
  &\quad \frac{1}{2\pi\im}\int_{-\im R}^{\im R}\ee^{\Y\xi}W_{2,2}\left(\Y,p\right)\dd{\Y}\\
  &=\frac{1}{2\pi\im}\int_{-\im R}^{\im R}\ee^{\Y\xi}W_{2,2}\left(\Y,p\right)\coup\left(\im \Y\right)\dd{\Y}+\frac{1}{2\pi\im}\int_{-\im R}^{\im R}\ee^{\Y\xi}W_{2,2}\left(\Y,p\right)\left(1-\coup\left(\im \Y\right)\right)\dd{\Y}=\vcc\widehat{Q}_{2,2,1}+\widehat{Q}_{2,2,2}.
 \end{split}
\end{equation*}

We first consider $\widehat{Q}_{2,2,1}$, change to real variables and integrate by parts, using $\ee^{\Y\xi}=\del_{\Y}\left(\xi^{-1}\ee^{\Y\xi}\right)$, to finally get
\begin{equation*}
 \begin{split}
  \widehat{Q}_{2,2,1}&=\frac{1}{2\pi p\left(1+p\right)\xi}\left.\left(\frac{\ee^{\im \Y\xi}\coup\left(\Y\right)}{ \Y^{1+\kappa}\left(1+\im \Y\right)}\int_{0}^{\Y}\ee^{-p\im s}\left(\im s\Phi'\left(\im s\right)\right)^2s^{\kappa}\frac{\ee^{\Phi\left(\im \Y\right)}}{\ee^{\Phi\left(\im s\right)}}\ds\right)\right|_{\Y=-R}^{\Y=R}\\
  &\quad -\frac{1}{2\pi p\left(1+p\right)\xi}\int_{-R}^{R}\ee^{\im \Y \xi}\del_{\Y}\left(\frac{\coup\left(\Y\right)}{ \Y^{1+\kappa}\left(1+\im \Y\right)}\right)\int_{0}^{\Y}\ee^{-p\im s}\left(\im s\Phi'\left(\im s\right)\right)^2s^{\kappa}\frac{\ee^{\Phi\left(\im \Y\right)}}{\ee^{\Phi\left(\im s\right)}}\ds\dd{\Y}\\
  &\quad -\frac{1}{2\pi p\left(1+p\right)\xi}\int_{-R}^{R}\ee^{\im \Y \xi}\frac{\coup\left(\Y\right)}{ \Y\left(1+\im \Y\right)}\ee^{-p\im \Y}\left(\im \Y\Phi'\left(\im \Y\right)\right)^2\dd{\Y}\\
  &\quad -\frac{1}{2\pi p\left(1+p\right)\xi}\int_{-R}^{R}\ee^{\im \Y \xi}\frac{\coup\left(\Y\right)\Phi'\left(\im \Y\right)}{ \Y^{1+\kappa}\left(1+\im \Y\right)}\int_{0}^{\Y}\ee^{-p\im s}\left(\im s\Phi'\left(\im s\right)\right)^2s^{\kappa}\frac{\ee^{\Phi\left(\im \Y\right)}}{\ee^{\Phi\left(\im s\right)}}\ds\dd{\Y}\\
  &=\vcc(I)+(II)+(III)+(IV).
 \end{split}
\end{equation*}
Using the estimate
\begin{equation}\label{eq:large:xi:Y:int}
 \begin{split}
  \abs{\int_{0}^{\Y}\ee^{-p\im s}\left(\im s\Phi'\left(\im s\right)\right)^2s^{\kappa}\frac{\ee^{\Phi\left(\im \Y\right)}}{\ee^{\Phi\left(\im s\right)}}\ds}\leq C\eps^{2}\max\left\{\abs{\Y}^{2\alpha+1+\kappa},\abs{\Y}^{-2\alpha+1+\kappa}\right\}
 \end{split}
\end{equation}
we immediately get
\begin{equation*}
 \begin{split}
  \abs{(I)}\leq\frac{C}{p\left(1+p\right)\xi}R^{2\alpha-1}\longrightarrow 0\quad \text{for } R\longrightarrow\infty.
 \end{split}
\end{equation*}
Furthermore using that $\coup$ is supported away from the origin we get
\begin{equation*}
 \begin{split}
  \abs{(II)}\leq \frac{C}{p\left(1+p\right)\xi}\int_{-\infty}^{\infty}\chi_{\left\{\abs{\Y}\geq \frac{1}{2}\right\}}\abs{\Y}^{-3-\kappa}\max\left\{\abs{\Y}^{2\alpha+1+\kappa},\abs{\Y}^{-2\alpha+1+\kappa}\right\}\dd{\Y}\leq\frac{C}{p\left(1+p\right)\xi} 
 \end{split}
\end{equation*}
as $\alpha<1/2$. Similarly we find
\begin{equation*}
 \begin{split}
  \abs{(III)}\leq \frac{C}{p\left(1+p\right)\xi}\int_{-\infty}^{\infty}\chi_{\left\{\abs{\Y}\geq \frac{1}{2}\right\}}\abs{\Y}^{-2-\kappa}\max\left\{\abs{\Y}^{2\alpha+\kappa},\abs{\Y}^{-2\alpha+\kappa}\right\}\dd{\Y}\leq \frac{C}{p\left(1+p\right)\xi}.
 \end{split}
\end{equation*}
Analogously to the case $\xi<p$, the decay in $\Y$ is insufficient in the term $(IV)$ to get an estimate uniform in $R$. Thus we have to integrate by parts again there, which is possible due to the cut-off $\coup$ and we get, by rewriting also $\Phi'\left(\im \Y\right)$, as before
\begin{equation*}
 \begin{split}
  (IV)&=\frac{\eps}{2\pi p\left(1+p\right)\xi\left(1+\xi\right)}\int_{-R}^{R}\del_{\Y}\left(\ee^{\im \Y\left(\xi+1\right)}\right)\frac{\coup\left(\Y\right)\beta_W\left(\im \Y\right)}{\Y^{2+\kappa}\left(1+\im \Y\right)}\int_{0}^{\Y}\ee^{-p\im s}\left(\Phi'\left(\im s\right)\right)^2s^{2+\kappa}\frac{\ee^{\Phi\left(\im \Y\right)}}{\ee^{\Phi\left(x\right)}}\ds\dd{\Y}\\
  &=\frac{\eps}{2\pi p\left(1+p\right)\xi\left(1+\xi\right)}\left.\left(\ee^{\im \Y\left(\xi+1\right)}\frac{\coup\left(\Y\right)\beta_W\left(\im \Y\right)}{\Y^{2+\kappa}\left(1+\im \Y\right)}\int_{0}^{\Y}\ee^{-p\im s}\left(\Phi'\left(\im s\right)\right)^2s^{2+\kappa}\frac{\ee^{\Phi\left(\im \Y\right)}}{\ee^{\Phi\left(x\right)}}\ds\right)\right|_{\Y=-R}^{\Y=R}\\
  &\quad -\frac{\eps}{2\pi p\left(1+p\right)\xi\left(1+\xi\right)}\int_{-R}^{R}\ee^{\im \Y\left(\xi+1\right)}\del_{\Y}\left(\frac{\coup\left(\Y\right)\beta_W\left(\im \Y\right)}{\Y^{2+\kappa}\left(1+\im \Y\right)}\right)\int_{0}^{\Y}\ee^{-p\im s}\left(\Phi'\left(\im s\right)\right)^2s^{2+\kappa}\frac{\ee^{\Phi\left(\im \Y\right)}}{\ee^{\Phi\left(x\right)}}\ds\dd{\Y}\\
  &\quad -\frac{\eps}{2\pi p\left(1+p\right)\xi\left(1+\xi\right)}\int_{-R}^{R}\ee^{\im \Y\left(\xi+1\right)}\frac{\coup\left(\Y\right)\beta_W\left(\im \Y\right)}{\left(1+\im \Y\right)}\ee^{-p\im \Y}\left(\Phi'\left(\im \Y\right)\right)^2\dd{\Y}\\
  &\quad +\frac{\eps^2}{2\pi\im p\left(1+p\right)\xi\left(1+\xi\right)}\int_{-R}^{R}\ee^{\im \Y\left(\xi+1\right)}\frac{\coup\left(\Y\right)\left(\beta_W\left(\im \Y\right)\right)^2}{ \Y^{3+\kappa}\left(1+\im \Y\right)}\int_{0}^{\Y}\ee^{-p\im s}\left(\Phi'\left(\im s\right)\right)^2s^{2+\kappa}\frac{\ee^{\Phi\left(\im \Y\right)}}{\ee^{\Phi\left(x\right)}}\ds\dd{\Y}\\
  &=\vcc(IV)_1+(IV)_2+(IV)_3+(IV)_4.
 \end{split}
\end{equation*}
 In addition to the estimate obtained in~\eqref{eq:large:xi:Y:int} we have
  \begin{align*}
  \abs{\frac{\coup\left(\Y\right)\beta_W\left(\im \Y\right)}{\Y^{2+\kappa}\left(1+\im \Y\right)}}\leq C\abs{\Y}^{\alpha-3-\kappa}\chi_{\left\{\abs{\Y}\geq \frac{1}{2}\right\}},\quad
  \abs{\del_{\Y}\left(\frac{\coup\left(\Y\right)\beta_W\left(\im \Y\right)}{\Y^{2+\kappa}\left(1+\im \Y\right)}\right)}\leq C\abs{\Y}^{\alpha-4-\kappa}\chi_{\left\{\abs{\Y}\geq \frac{1}{2}\right\}},\\
  \abs{\frac{\coup\left(\Y\right)\beta_W\left(\im \Y\right)}{\left(1+\im \Y\right)}\ee^{-p\im \Y}\left(\Phi'\left(\im \Y\right)\right)^2}\leq C\abs{\Y}^{3\alpha-3}\chi_{\left\{\abs{\Y}\geq \frac{1}{2}\right\}},\quad
  \abs{\frac{\coup\left(\Y\right)\left(\beta_W\left(\im \Y\right)\right)^2}{ \Y^{3+\kappa}\left(1+\im \Y\right)}}\leq C\abs{\Y}^{2\alpha-4-\kappa}\chi_{\left\{\abs{\Y}\geq \frac{1}{2}\right\}}.
 \end{align*}
 Thus on the one hand we get
 \begin{equation*}
  \begin{split}
   \abs{(IV)_1}\leq \frac{C}{p\left(1+p\right)\xi\left(1+\xi\right)}R^{3\alpha-2}\longrightarrow 0 \quad \text{for } R\longrightarrow\infty.
  \end{split}
 \end{equation*}
On the other hand we obtain
\begin{align*}
 \abs{(IV)_2}&\leq \frac{C}{p\left(1+p\right)\xi\left(1+\xi\right)}\int_{-\infty}^{\infty}\chi_{\left\{\abs{\Y}\geq \frac{1}{2}\right\}}\abs{\Y}^{\alpha-4-\kappa}\max\left\{\abs{\Y}^{2\alpha+1+\kappa},\abs{\Y}^{-2\alpha+1+\kappa}\right\}\dd{\Y}\\
 &\leq \frac{C}{p\left(1+p\right)\xi\left(1+\xi\right)},\\
 \abs{(IV)_3}&\leq \frac{C}{p\left(1+p\right)\xi\left(1+\xi\right)}\int_{-\infty}^{\infty}\chi_{\left\{\abs{\Y}\geq \frac{1}{2}\right\}}\abs{\Y}^{3\alpha-3}\dd{\Y}\\
 &\leq \frac{C}{p\left(1+p\right)\xi\left(1+\xi\right)},\\
 \abs{(IV)_4}&\leq \frac{C}{p\left(1+p\right)\xi\left(1+\xi\right)}\int_{-\infty}^{\infty}\chi_{\left\{\abs{\Y}\geq \frac{1}{2}\right\}}\abs{\Y}^{2\alpha-4-\kappa}\max\left\{\abs{\Y}^{2\alpha+1+\kappa},\abs{\Y}^{-2\alpha+1+\kappa}\right\}\dd{\Y}\\
 &\leq \frac{C}{p\left(1+p\right)\xi\left(1+\xi\right)}.
\end{align*}
Thus in summary we have in the limit $R\to\infty$ that
\begin{equation*}
 \begin{split}
  \limsup_{R\to\infty}\abs{(IV)}\leq \frac{C}{p\left(1+p\right)\xi\left(1+\xi\right)}
 \end{split}
\end{equation*}
and it easily follows by similar estimates as before that this gives the desired contribution to the integral in~\eqref{eq:Q0}.

Thus it finally remains only to consider $\widehat{Q}_{2,2,2}$, where we proceed as before, i.e.\ taking the limit $R\to\infty$, changing to real variables and rearranging, to get
\begin{equation*}
 \begin{split}
  \lim_{R\to\infty}\widehat{Q}_{2,2,2}&=-\frac{1}{2\pi p\left(1+p\right)}\int_{-\infty}^{\infty}\ee^{-\im ps}\left(\im s\Phi'\left(\im s\right)\right)^2s^{\kappa}\ee^{-\Phi\left(\im s\right)}\int_{s}^{\sgn\left(s\right)\infty}\ee^{\im \Y \xi}\frac{\left(1-\coup\left(\Y\right)\right)}{ \Y^{1+\kappa}\left(1+\im \Y\right)}\ee^{\Phi\left(\im \Y\right)}\dd{\Y}\ds.
 \end{split}
\end{equation*}
Changing variables, $\Y=s+t$, as before and integrating by parts we get
\begin{multline*}
  \quad \lim_{R\to\infty}\widehat{Q}_{2,2,2}\\
  \shoveleft{=-\frac{1}{2\pi p\left(1+p\right)}\int_{-\infty}^{\infty}\ee^{\im s\left(\xi-p\right)}\left(\im s\Phi'\left(\im s\right)\right)^2s^{\kappa}\int_{0}^{\sgn\left(s\right)\infty}\frac{\left(1-\coup\left(s+t\right)\right)}{ \left(s+t\right)^{1+\kappa}\left(1+\im\left(s+t\right)\right)}\ee^{\im t\xi}\ee^{\Phi\left(\im\left(s+t\right)\right)-\Phi\left(\im s\right)}\dt\ds}\\
  \shoveleft{=\frac{\im}{2\pi p\left(1+p\right)\left(\xi-p\right)}\int_{-\infty}^{\infty}\left(\ee^{\im s\left(\xi-p\right)}-1\right)\del_{s}\left(\left(\Phi'\left(\im s\right)\right)^2s^{2+\kappa}\right)\cdot}\\
  \shoveright{\cdot\int_{0}^{\sgn\left(s\right)\infty}\frac{\left(1-\coup\left(s+t\right)\right)}{\left(s+t\right)^{1+\kappa}\left(1+\im\left(s+t\right)\right)}\ee^{\im t\xi}\ee^{\Phi\left(\im\left(s+t\right)\right)-\Phi\left(\im s\right)}\dt\ds}\\
  \shoveleft{\quad +\frac{\im}{2\pi p\left(1+p\right)\left(\xi-p\right)}\int_{-\infty}^{\infty}\left(\ee^{\im s\left(\xi-p\right)}-1\right)\left(\Phi'\left(\im s\right)\right)^2s^{2+\kappa}\cdot}\\
  \shoveright{\cdot\int_{0}^{\sgn\left(s\right)\infty}\del_{s}\left(\frac{\left(1-\coup\left(s+t\right)\right)}{ \left(s+t\right)^{1+\kappa}\left(1+\im\left(s+t\right)\right)}\right)\ee^{\im t\xi}\ee^{\Phi\left(\im\left(s+t\right)\right)-\Phi\left(\im s\right)}\dt\ds}\\
  \shoveleft{\quad -\frac{1}{2\pi p\left(1+p\right)\left(\xi-p\right)}\int_{-\infty}^{\infty}\left(\ee^{\im s\left(\xi-p\right)}-1\right)\left(\Phi'\left(\im s\right)\right)^2s^{2+\kappa}\cdot}\\
  \cdot\int_{0}^{\sgn\left(s\right)\infty}\frac{\left(1-\coup\left(s+t\right)\right)}{ \left(s+t\right)^{1+\kappa}\left(1+\im\left(s+t\right)\right)}\left(\Phi'\left(\im\left(s+t\right)\right)-\Phi'\left(\im s\right)\right)\ee^{\im t\xi}\ee^{\Phi\left(\im\left(s+t\right)\right)-\Phi\left(\im s\right)}\dt\ds.
\end{multline*}
We have similarly as before for $\abs{s}\leq 1$ that
\begin{equation*}
 \abs{\del_{s}\left(\left(\Phi'\left(\im s\right)\right)^2s^{2+\kappa}\right)}\leq C\eps^{2}\abs{s}^{-2\alpha-1+\kappa}\quad \text{and}\quad  \abs{\left(\Phi'\left(\im s\right)\right)^2s^{2+\kappa}}\leq C\eps^{2}\abs{s}^{-2\alpha+\kappa}.
\end{equation*}
Thus together with Lemmas~\ref{Lem:asymptotics:p:leq:xi} and~\ref{Lem:properties:Phi} as well as $\eps\abs{s}^{-\alpha}\ee^{-\frac{D\eps}{\abs{s}^{\alpha}}}\leq C$ we obtain for $a\in\left(\alpha+2\delta,\theta\right)$ with $\delta>0$ sufficiently small that
\begin{equation*}
 \begin{split}
  \abs{\widehat{Q}_{2,2,2}}\leq \frac{C}{p\left(1+p\right)\abs{\xi-p}^{1-a}}\int_{-1}^{1}\abs{s}^{a-\alpha-2\delta-1}\ds\leq \frac{C}{p\left(p+1\right)\abs{\xi-p}^{1-a}}.
 \end{split}
\end{equation*}
As already shown in the cases before (see for example~\eqref{eq:contr:Q1211}), the contribution is fine if we choose $\nu>0$ sufficiently small for $a$ fixed. 

\begin{remark}
 The case $\alpha=0$ can again be treated by an argument similar to that one in Remark~\ref{Rem:adapt:11.4}.
\end{remark}

This then finishes the proof.
\end{proof}

\section{Asymptotics for $\beta_W$, $\Phi$ and several integral estimates}\label{Sec:asymptotics}

In this section we show several estimates that are necessary to control the integrals occurring in the Proof of Proposition~\ref{Prop:Q0:estimate}.

\begin{lemma}\label{Lem:asymptotics:beta}
 Denoting by $m_{\alpha}\vcc=\int_{0}^{\infty}z^{\alpha}\mu\left(z\right)\ee^{-z}\dz$ it holds for $x\in\C\setminus\left(-\infty,0\right]$ that
 \begin{equation*}
  \beta_W\left(x,\mu\right)\sim C_W m_{\alpha}x^{-\alpha}\quad \text{as } x\to 0
 \end{equation*}
 in the sense of~\eqref{eq:def:asymp}.
\end{lemma}

\begin{proof}
 By assumption~\eqref{eq:asymp:W}, for any $r>0$ there exists $\delta_r>0$ such that 
 \begin{equation*}
  \abs{W\left(\xi,1\right)-C_W\xi^{-\alpha}}\leq r \abs{\xi}^{-\alpha} \quad \text{for } \abs{\xi}\leq \delta_r.
 \end{equation*}
 Furthermore by definition we have
 \begin{equation*}
  \begin{split}
   \beta_W\left(x\right)&=\int_{0}^{\infty}W\left(\frac{x}{z},1\right)\mu\left(z\right)\ee^{-z}\dz\\
   &=\int_{0}^{\frac{\abs{x}}{\delta_r}}W\left(\frac{x}{z},1\right)\mu\left(z\right)\ee^{-z}\dz+\int_{\frac{\abs{x}}{\delta_r}}^{\infty}\left(W\left(\frac{x}{z},1\right)-C_W\left(\frac{x}{z}\right)^{-\alpha}+C_W\left(\frac{x}{z}\right)^{-\alpha}\right)\mu\left(z\right)\ee^{-z}\dz.
  \end{split}
 \end{equation*}
We can then estimate
\begin{equation*}
 \begin{split}
  &\quad \abs{x^{\alpha}\int_{0}^{\frac{\abs{x}}{\delta_r}}W\left(\frac{x}{z},1\right)\mu\left(z\right)\ee^{-z}\dz}\leq \abs{x}^{\alpha}\int_{0}^{\frac{\abs{x}}{\delta_r}}\left(\abs{\frac{x}{z}}^{-\alpha}+\abs{\frac{x}{z}}^{\alpha}\right)\mu\left(z\right)\ee^{-z}\dz\\
  &\leq \int_{0}^{\frac{\abs{x}}{\delta_r}}z^{-\alpha}\mu\left(z\right)\ee^{-z}\dz\left(1+\abs{x}^{2\alpha}\right)\longrightarrow 0\quad \text{as } \abs{x}\longrightarrow 0.
 \end{split}
 \end{equation*}
 Furthermore using~\eqref{eq:asymp:W} we get
 \begin{equation*}
  \begin{split}
   \abs{x^{\alpha}\int_{\frac{\abs{x}}{\delta_r}}^{\infty}\left(W\left(\frac{x}{z},1\right)-C_W\left(\frac{x}{z}\right)^{-\alpha}\right)\mu\left(z\right)\ee^{-z}\dz}\leq r\int_{\frac{\abs{x}}{\delta_r}}^{\infty}z^{\alpha}\mu\left(z\right)\ee^{-z}\dz\leq Cr
  \end{split}
 \end{equation*}
uniformly in $x$. Finally we have
\begin{equation*}
 \begin{split}
  x^{\alpha}\int_{\frac{\abs{x}}{\delta_r}}^{\infty}C_W\left(\frac{x}{z}\right)^{-\alpha}\mu\left(z\right)\ee^{-z}\dz&=C_W\int_{\frac{\abs{x}}{\delta_r}}^{\infty}z^{\alpha}\mu\left(z\right)\ee^{-z}\dz\longrightarrow C_W m_{\alpha} \quad \text{for } \abs{x}\longrightarrow 0.
 \end{split}
\end{equation*}
This shows the claim.
\end{proof}

\begin{lemma}\label{Lem:Asymp:Phi}
 With the same notation as in Lemma~\ref{Lem:asymptotics:beta} it holds for $\Y\in\C\setminus\left(-\infty,0\right]$ that
 \begin{equation*}
  \Phi\left(\Y\right)\sim \frac{C_W m_{\alpha}\eps}{\alpha}\Y^{-\alpha} \quad \text{as } \Y\to 0.
 \end{equation*}
 in the sense of~\eqref{eq:def:asymp}.
\end{lemma}

\begin{proof}
 From the definition we get for some $0<\delta<1$ and $\abs{\Y}\leq \delta$ that
 \begin{equation*}
  \begin{split}
   \Phi\left(\Y\right)=\int_{\Y}^{\infty}\frac{\eps\beta_W\left(t\right)}{t}\ee^{-t}\dt=\int_{\Y}^{\delta}\frac{\eps\beta_W\left(t\right)}{t}\dt+\int_{\Y}^{\delta}\frac{\eps\beta_W\left(t\right)}{t}\left(\ee^{-t}-1\right)\dt+\int_{\delta}^{\infty}\frac{\eps\beta_W\left(t\right)}{t}\ee^{-t}\dt,
  \end{split}
 \end{equation*}
 where the path of integration connecting $\Y$ and $\delta$ is chosen to be contained in in the set  $\C\setminus\left(-\infty,0\right]\cap\left(\left\{\xi\in\C\;|\; \abs{\xi}=\abs{\Y}\right\}\cup\R_{+}\right)$ as shown in Figure~\ref{fig:asymp}.
 
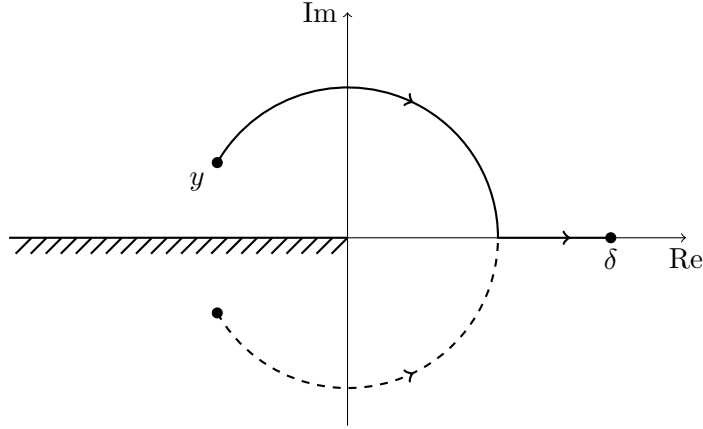
\begin{figure}
\begin{center}
   \begin{tikzpicture}[contour/.style={postaction={decorate, decoration={markings,
mark=at position 3cm with {\arrow[line width=1pt]{>}},mark=at position 6.2cm with {\arrow[line width=1pt]{>}}
}}},
    interface/.style={postaction={draw,decorate,decoration={border,angle=45,
                    amplitude=0.3cm,segment length=2mm}}}]

\draw[->] (0,0) -- (4.5,0) node[below] {$\Re$};
\draw[->] (0,-2.5) -- (0,3) node[left] {$\Im$};
\draw[line width=.8pt,interface](0,0)--(-4.5,0);

\path[draw,line width=0.8pt,contour] (-1.732,1) arc(150:0:2) -- (3.5,0)node[below]{$\delta$};
\path[draw, dashed, line width=0.8pt, contour] (-1.732,-1) arc(-150:0:2);

\draw[black,fill=black] (-1.732,1) circle (.4ex);
\draw[black,fill=black] (-1.732,-1) circle (.4ex);
\draw[black,fill=black] (3.5,0) circle (.4ex);
\node[below] at (-2,1) {$\Y$};
\end{tikzpicture}

\caption{Contour connecting $\Y$ and $\delta$}
\label{fig:asymp}
 \end{center}
\end{figure}
 
 Using $\abs{\ee^{-t}-t}\leq 2t$ and the estimates obtained in Lemma~\ref{Lem:prop:betaW} we immediately get
 \begin{equation*}
  \abs{\int_{\delta}^{\infty}\frac{\eps\beta_W\left(t\right)}{t}\ee^{-t}\dt}\leq C\eps\delta^{-\alpha}\quad \text{and} \quad \abs{\int_{\Y}^{\delta}\frac{\eps\beta_W\left(t\right)}{t}\left(\ee^{-t}-1\right)\dt}\leq C\eps.
 \end{equation*}
 Furthermore we have 
 \begin{equation*}
  \begin{split}
   \int_{\Y}^{\delta}\frac{\eps\beta_W\left(t\right)}{t}\dt=\int_{\Y}^{\delta}\frac{C_W m_{\alpha}\eps t^{-\alpha}}{t}\dt+\int_{\Y}^{\delta}\frac{\eps\beta_W\left(t\right)-C_W m_{\alpha}\eps t^{-\alpha}}{t}\dt.
  \end{split}
 \end{equation*}
 Using that from Lemma~\ref{Lem:asymptotics:beta} we have $\beta_W\left(x,\mu\right)\sim C_W m_{\alpha}x^{-\alpha}$ as $x\to 0$ it follows for some $r>0$ that
 \begin{equation*}
  \begin{split}
   \abs{\int_{\Y}^{\delta}\frac{\eps\beta_W\left(t\right)-C_W m_{\alpha}\eps t^{-\alpha}}{t}\dt}\leq \frac{\eps r}{\alpha}\abs{\Y}^{-\alpha}
  \end{split}
 \end{equation*}
 if we choose $\delta<\delta_r$. On the other hand we get 
 \begin{equation*}
  \int_{\Y}^{\delta}\frac{C_W m_{\alpha}\eps t^{-\alpha}}{t}\dt=\frac{C_W m_{\alpha}\eps }{\alpha}\Y^{-\alpha}-\frac{C_W m_{\alpha}\eps\delta^{-\alpha}}{\alpha}.
 \end{equation*}
 Summarizing we have
 \begin{equation*}
  \abs{\Y^{\alpha}\Phi\left(\Y\right)-\frac{C_W m_{\alpha}\eps}{\alpha}}\leq C_{\delta_{r}}\eps\abs{\Y}^{\alpha}+\frac{\eps}{\alpha}r
 \end{equation*}
 and the claim follows by choosing first $r$ and then $\abs{\Y}$ small. 
\end{proof}

\begin{lemma}\label{Lem:asymptotics:exp:Phi}
 There exist constants $C,d>0$ such that 
 \begin{equation*}
  \abs{\ee^{-\Phi\left(\im s\right)}}\leq C\ee^{-\frac{d\eps}{\abs{s}^{\alpha}}} \quad \text{for } \abs{s}\leq 1.
 \end{equation*}
\end{lemma}

\begin{proof}
 We have 
 \begin{equation*}
  \Re\left(\left(\im s\right)^{-\alpha}\right)=\abs{s}^{-\alpha}\cos\left(\frac{\alpha \pi}{2}\right).
 \end{equation*}
Then from Lemma~\ref{Lem:Asymp:Phi} there exists $\delta_{*}>0$ such that
\begin{equation*}
 \abs{\Phi\left(\im s\right)-\frac{C_W m_{\alpha}\eps }{\alpha}\left(\im s\right)^{-\alpha}}\leq \frac{C_W m_{\alpha}\eps }{\alpha}\frac{\abs{s}^{-\alpha}}{2}\cos\left(\frac{\alpha\pi}{2}\right)\quad \text{for }\abs{s}\leq \delta_{*}.
\end{equation*}
On the other hand we have from~\eqref{eq:bound:Phi}
\begin{equation*}
 \abs{\Phi\left(\im s\right)-\frac{C_W m_{\alpha}\eps}{\alpha}\left(\im s\right)^{-\alpha}}\leq C_{\delta_{*}}\eps \quad \text{for }\abs{s}\geq \delta_{*}.
\end{equation*}
Thus we get
\begin{equation*}
 \begin{split}
  \abs{\ee^{-\Phi\left(\im s\right)}}\leq \ee^{\abs{\Phi\left(\im s\right)-\frac{-C_W m_{\alpha}\eps}{\alpha}\left(\im s\right)^{-\alpha}}}\ee^{-\Re\left(\frac{C_W m_{\alpha}\eps}{\alpha}\left(\im s\right)^{-\alpha}\right)}\leq C_{\delta_{*}}\ee^{\frac{C_W m_{\alpha}\eps}{2\alpha}\cos\left(\frac{\alpha \pi}{2}\right)\abs{s}^{-\alpha}},
 \end{split}
\end{equation*}
showing the claim.
\end{proof}

\begin{lemma}\label{Lem:asymptotics:xi:leq:p}
 Let $\cut$ be as in the Proof of Proposition~\ref{Prop:Q0:estimate} (see~\eqref{eq:cutoff}). Then for each $\delta>0$ there exists some constant $C>0$ such that it holds
 \begin{align*}
  \abs{\int_{0}^{\sgn\left(s\right)\infty}\frac{\cut\left(s+t\right)\ee^{\im t\xi}}{\left(s+t\right)^{1+\kappa}\left(1+\im\left(s+t\right)\right)}\ee^{\Phi\left(\im\left(s+t\right)\right)-\Phi\left(\im s\right)}\dt}&\leq C\left(\frac{\abs{s}^{\alpha-\kappa}}{\eps}+\abs{s}^{-\delta}\ee^{-\frac{B\eps}{\abs{s}^{\alpha}}}\right)\chi_{\left\{\abs{s}\leq 2\right\}},\\
  \abs{\int_{0}^{\sgn\left(s\right)\infty}\del_{s}\left(\frac{\cut\left(s+t\right)\ee^{\im t\xi}}{\left(s+t\right)^{1+\kappa}\left(1+\im\left(s+t\right)\right)}\right)\ee^{\Phi\left(\im\left(s+t\right)\right)-\Phi\left(\im s\right)}\dt}&\leq C\left(\frac{\abs{s}^{\alpha-1-\kappa}}{\eps}+\abs{s}^{-1-\kappa}\ee^{-\frac{B\eps}{\abs{s}^{\alpha}}}\right)\chi_{\left\{\abs{s}\leq 2\right\}},\\
  \abs{\int_{0}^{\sgn\left(s\right)\infty}\frac{\cut\left(s+t\right)\ee^{\im t\xi}\left(\Phi'\left(\im\left(s+t\right)\right)-\Phi'\left(\im s\right)\right)}{\left(s+t\right)^{1+\kappa}\left(1+\im\left(s+t\right)\right)}\ee^{\Phi\left(\im\left(s+t\right)\right)-\Phi\left(\im s\right)}\dt}&\leq C\left(\frac{\abs{s}^{\alpha-1-\kappa}}{\eps}+\frac{\eps\ee^{-\frac{B\eps}{\abs{s}^{\alpha}}}}{\abs{s}^{1+\alpha+\delta}}\right)\chi_{\left\{\abs{s}\leq 2\right\}}
 \end{align*}
 for $\eps$ sufficiently small.
\end{lemma}
\begin{proof}
Note first that due to the cut-off $\cut$ it suffices to consider $\abs{s},\abs{t}\leq 2$. We will assume in the following always $s,t>0$ while the case $s,t<0$ can be treated analogously. We first rewrite and change to real variables, i.e.\ we have
 \begin{equation*}
  \Re\left(\Phi\left(\im \left(s+t\right)\right)-\Phi\left(\im s\right)\right)=\Re\int_{\im s}^{\im\left(s+t\right)}\Phi'\left(\xi\right)\dxi=\int_{s}^{s+t}\Re\left(\psi'\left(z\right)\right)\dz,
 \end{equation*}
where we define $\psi\left(z\right)\vcc=\Phi\left(\im z\right)$. Defining furthermore $\lambda\left(z\right)\vcc=\Re\left(\psi'\left(z\right)\right)$ we have due to Lemma~\ref{Lem:Asymp:Phi} that
\begin{equation*}
 \lambda\left(z\right)\sim -\frac{a\eps}{z^{1+\alpha}}\quad \text{as } z\to 0 
\end{equation*}
for some constant $a>0$. We furthermore set
\begin{equation*}
 \Lambda\left(s\right)\vcc=\int_{s}^{1}\lambda\left(z\right)\dz\sim -\frac{A\eps}{s^{\alpha}}\quad \text{for } s\to 0,
\end{equation*}
with $A>0$. Using the asymptotic behaviour of $\Lambda'=-\lambda$ close to zero, we then obtain 
\begin{equation*}
  \Re\left(\Phi\left(\im\left(s+t\right)\right)-\Phi\left(\im s\right)\right)=\Lambda\left(s\right)-\Lambda\left(s+t\right)\leq \begin{cases}
                                                    -\frac{B\eps}{s^{\alpha}}+C & s\leq t\\
                                                    -\frac{\sigma\eps}{s^{\alpha+1}}t +C & s\geq t
                                                   \end{cases}
\end{equation*}
where $A,C,\sigma>0$ are constants.

With this we obtain, using also that due to $\cut$ the integrand is supported in $\left[0,2\right]$, that
\begin{equation}\label{eq:integral:asymp}
 \begin{split}
  &\quad \abs{\int_{0}^{\sgn\left(s\right)\infty}\frac{\cut\left(s+t\right)\ee^{\im t\xi}}{\left(s+t\right)^{1+\kappa}\left(1+\im\left(s+t\right)\right)}\ee^{\Phi\left(\im\left(s+t\right)\right)-\Phi\left(\im s\right)}\dt}\leq C\int_{0}^{2}\frac{\ee^{\Lambda\left(s\right)-\Lambda\left(s+t\right)}}{\left(s+t\right)^{1+\kappa}}\dt\\
  &\leq C\int_{0}^{s}\frac{\ee^{-\frac{\sigma\eps}{s^{\alpha+1}}t}}{\left(s+t\right)^{1+\kappa}}\dt+C\int_{s}^{2}\frac{\ee^{-\frac{B\eps}{s^{\alpha}}}}{\left(s+t\right)^{1+\kappa}}\dt.
 \end{split}
\end{equation}
Changing variables we obtain for the first integral on the right-hand side
\begin{equation*}
 \begin{split}
  C\int_{0}^{s}\frac{\ee^{-\frac{\sigma\eps}{s^{\alpha+1}}t}}{\left(s+t\right)^{1+\kappa}}\dt\leq C\frac{s^{1+\alpha}}{\sigma\eps}\int_{0}^{\frac{\sigma\eps}{s^{\alpha}}}\ee^{-t}\left(s+\frac{ts^{1+\alpha}}{\sigma\eps}\right)^{-1-\kappa}\dt\leq \frac{s^{\alpha-\kappa}}{\sigma\eps}\int_{0}^{\infty}\ee^{-t}\dt\leq Cs^{\alpha-\kappa}.
 \end{split}
\end{equation*}
On the other hand we get for the second integral on the right-hand side of~\eqref{eq:integral:asymp} that
\begin{equation*}
 \begin{split}
  \int_{s}^{2}\frac{\ee^{-\frac{B\eps}{s^{\alpha}}}}{\left(s+t\right)^{1+\kappa}}\dt\leq \ee^{-\frac{B\eps}{s^{\alpha}}}\int_{s}^{2}\frac{1}{\left(s+t\right)^{1+\kappa}}\dt\leq \ee^{-\frac{B\eps}{s^{\alpha}}}\int_{s}^{2}\frac{1}{\left(s+t\right)^{1+\delta}}\dt\leq \frac{C}{s^{\delta}}\ee^{-\frac{B\eps}{s^{\alpha}}},
 \end{split}
\end{equation*}
where we used $\abs{\kappa}\leq \delta$ for $\eps$ sufficiently small. Noting that the support in $s$ of the integral on the left-hand side in~\eqref{eq:integral:asymp} is clearly contained in $\left(-2,2\right)$ by the choice of $\cut$, the first estimate of the Lemma follows.

The second estimate can be shown in the same way noting that we have
\begin{equation*}
 \abs{\del_{s}\left(\frac{\cut\left(s+t\right)\ee^{\im t\xi}}{\left(s+t\right)^{1+\kappa}\left(1+\im\left(s+t\right)\right)}\right)}\leq \frac{C}{\left(s+t\right)^{2+\kappa}},
\end{equation*}
while for the third one we have to use that for $s+t\leq 2$ it holds
\begin{equation*}
 \begin{split}
  \abs{\Phi'\left(\im\left(s+t\right)\right)-\Phi'\left(\im s\right)}=\abs{\int_{\im s}^{\im\left(s+t\right)}\Phi''\left(\xi\right)\dxi}\leq C\eps\int_{s}^{s+t}\xi^{-2-\alpha}\dxi\leq C\eps\frac{t}{s^{2+\alpha}}.
 \end{split}
\end{equation*}
\end{proof}

\begin{lemma}\label{Lem:asymptotics:p:leq:xi}
 For $\coup$ as given in the Proof of Proposition~\ref{Prop:Q0:estimate} (see~\eqref{eq:cutoff}) and each $\delta>0$ there exists some constant $C>0$ such that it holds
 \begin{align*}
  \abs{\int_{0}^{\sgn\left(s\right)\infty}\frac{\left(1-\coup\left(s+t\right)\right)}{\left(s+t\right)^{1+\kappa}\left(1+\im\left(s+t\right)\right)}\ee^{\im t\xi}\ee^{\Phi\left(\im\left(s+t\right)\right)-\Phi\left(\im s\right)}\dt}&\leq C\left(\frac{\abs{s}^{\alpha-\kappa}}{\eps}+\frac{\ee^{-\frac{B\eps}{\abs{s}^{\alpha}}}}{\abs{s}^{\delta}}\right)\chi_{\left\{\abs{s}\leq 2\right\}},\\
  \abs{\int_{0}^{\sgn\left(s\right)\infty}\del_{s}\left(\frac{\left(1-\coup\left(s+t\right)\right)}{\left(s+t\right)^{1+\kappa}\left(1+\im\left(s+t\right)\right)}\ee^{\im t\xi}\right)\ee^{\Phi\left(\im\left(s+t\right)\right)-\Phi\left(\im s\right)}\dt}&\leq C\left(\frac{\abs{s}^{\alpha-1-\kappa}}{\eps}+\frac{\ee^{-\frac{B\eps}{\abs{s}^{\alpha}}}}{\abs{s}^{1+\kappa}}\right)\chi_{\left\{\abs{s}\leq 2\right\}}
  \end{align*}
  and
  \begin{multline*}
  \abs{\int_{0}^{\sgn\left(s\right)\infty}\frac{\left(1-\coup\left(s+t\right)\right)\ee^{\im t\xi}\left(\Phi'\left(\im\left(s+t\right)\right)-\Phi'\left(\im s\right)\right)}{\left(s+t\right)^{1+\kappa}\left(1+\im\left(s+t\right)\right)}\ee^{\Phi\left(\im\left(s+t\right)\right)-\Phi\left(\im s\right)}\dt}\\*
  \leq C\left(\frac{\abs{s}^{\alpha-1-\kappa}}{\eps}+\frac{\eps\ee^{-\frac{B\eps}{\abs{s}^{\alpha}}}}{\abs{s}^{1+\alpha+\delta}}\right)\chi_{\left\{\abs{s}\leq 2\right\}},
 \end{multline*}
 for $\eps$ sufficiently small.
\end{lemma}

\begin{proof}
 This follows in the same way as Lemma~\ref{Lem:asymptotics:xi:leq:p} noting that $1-\coup$ is supported in $\left[-1,1\right]$.
\end{proof}

\begin{remark}\label{Rem:generalisation}
Note that for any kernel $W$ that allows to obtain the estimates given in Lemmas~\ref{Lem:asymptotics:xi:leq:p} and~\ref{Lem:asymptotics:p:leq:xi} we could also show uniqueness of self-similar profiles to~\eqref{eq:selfsim}. 
\end{remark}

\section{Appendix}

\subsection{Elementary properties of the norm}\label{Sec:prop:norm}
In this section we collect some elementary properties of the (semi-) norms $\snorm{k,\chi}{\cdot}$ while the first Lemma shows that the semi-norms $\snorm{k,\chi}{\cdot}$ are in fact norms and they are equivalent to the norms $\fnorm{k,\chi}{\cdot}$ on the space $\esp_{k,\chi}$ for $k=0,1,2$.

\begin{lemma}\label{Lem:norm:equ}
 Let $\omega\in \esp_{k,\chi}$ for $\chi>0$ and $k\in\left\{0,1,2\right\}$. Then there exists some constant $C>0$ such that $\fnorm{k,\chi}{\omega}\leq C\snorm{k,\chi}{\omega}<\infty$ for $k=0,1,2$.
\end{lemma}

\begin{proof}
 Assuming $k=2$ and $\snorm{2,\chi}{\omega}<\infty$ and denoting by $\Omega$ the Laplace transform of $\omega$ we have
 \begin{equation*}
  \begin{split}
   \abs{\Omega'\left(p\right)}&=\abs{\int_{p}^{\infty}\Omega''\left(s\right)\ds}\leq C\snorm{2,\chi}{\omega}\int_{p}^{\infty}\frac{\left(s+1\right)^{1-\chi}}{s^3}\ds\leq C\snorm{2,\chi}{\omega}\frac{\Lambda_{\chi}\left(p\right)}{p}.
  \end{split}
 \end{equation*}
 Together with~\eqref{eq:Lambda} this then shows $\snorm{1,\chi}{\omega}\leq C\snorm{2,\chi}{\omega}$. In the same way we also obtain $\fnorm{0,\chi}{\omega}\leq C\snorm{1,\chi}{\omega}$ and this finishes the proof.
\end{proof}

We also have the following interpolation inequality which is a consequence of the Landau-inequality.

\begin{lemma}\label{Lem:interpolation}
 Let $\delta>0$. Then there exists some constant $C>0$ such that for every $\omega\in\esp$ it holds
 \begin{equation*}
  \fnorm{1}{\omega}\leq C\fnorm{0}{\omega}+\delta\snorm{2}{\omega}.
 \end{equation*}
\end{lemma}

\begin{proof}
 It suffices to show that $\snorm{1}{\omega}\leq C\fnorm{0}{\omega}+\delta\snorm{2}{\omega}$. Denoting by $\Omega$ the Laplace transform of $\omega$ it holds from the Landau-inequality for every $R>0$ that
\begin{equation}\label{eq:Landau}
 \sup_{R\leq p\leq 2R}\abs{\Omega'\left(p\right)}^2\leq 4\sup_{R\leq p\leq 2R}\abs{\Omega\left(p\right)}\sup_{R\leq p\leq 2R}\abs{\Omega''\left(p\right)}.
\end{equation}
On the other hand there exists some constant $C>0$ independent of $R$ such that
\begin{equation*}
 \sup_{R\leq p\leq 2R}\left(\frac{p^2}{\left(1+p\right)^{1-\theta}}\right)^2\leq C\inf_{R\leq p\leq 2R}\left(\frac{p}{\left(1+p\right)^{1-\theta}}\right)\inf_{R\leq p\leq 2R}\left(\frac{p^3}{\left(1+p\right)^{1-\theta}}\right).
\end{equation*}
Using this in~\eqref{eq:Landau} it follows
\begin{equation*}
 \sup_{R\leq p\leq 2R}\left(\frac{p^2}{\left(1+p\right)^{1-\theta}}\abs{\Omega'\left(p\right)}\right)^2\leq C\sup_{R\leq p\leq 2R}\left(\frac{p}{\left(1+p\right)^{1-\theta}}\abs{\Omega\left(p\right)}\right)\sup_{R\leq p\leq 2R}\left(\frac{p^3}{\left(1+p\right)^{1-\theta}}\abs{\Omega''\left(p\right)}\right)
\end{equation*}
and thus by taking the supremum in $R$ on both sides we get $\snorm{1}{\omega}\leq 2\fnorm{0}{\omega}^{1/2}\snorm{2}{\omega}^{1/2}$. The claim now follows from Young's inequality.
\end{proof}

We also mention the following elementary estimate.
\begin{lemma}\label{Lem:elem:est:norm}
 Let $k\in\left\{0,1,2\right\}$, $n\in\N$ and $\omega\in\esp_k$. Then it holds
 \begin{equation*}
  \snorm{k}{\ee^{-n\cdot} \omega}\leq \snorm{k}{\omega} \quad \text{and} \quad \snorm{k}{\left(1-\ee^{-n\cdot}\right)\omega}\leq 2\snorm{k}{\omega}.
 \end{equation*}
\end{lemma}

\begin{proof}
 Using that the Laplace transform of $\ee^{-n\cdot} \omega$ is given by $\Omega\left(\cdot+n\right)$ and that $\frac{p^{k+1}}{\left(1+p\right)^{1-\theta}}$ is non-decreasing, the first estimate follows from the definition of the norm. The second estimate then follows directly from the first one.
\end{proof}

\begin{remark}\label{Rem:elem:norm:est}
 Although we need this estimate for general $n\in\N$, we emphasize here in particular the case $n=1$ which reads as
\begin{equation*}
 \fnorm{k}{\zeta \omega}\leq \fnorm{k}{\omega} \quad \text{and} \quad \fnorm{k}{\left(1-\zeta \right)\omega}\leq 2\fnorm{k}{\omega} \quad \text{for } k=0,1,2
\end{equation*}
and correspondingly for the semi-norms. We also note that formulated in terms of the $\lnorm{k}{\cdot}$-norm this reads as
\begin{equation*}
 \lnorm{k}{\left(\T\omega\right)\left(\cdot+1\right)}\leq \fnorm{k}{\omega}\quad \text{and} \quad \lnorm{k}{\left(\T\omega\right)\left(\cdot\right)-\left(\T\omega\right)\left(\cdot+1\right)}\leq 2\fnorm{k}{\omega}\quad \text{for } k=0,1,2.
\end{equation*}
\end{remark}

Actually we can estimate the quantity $\snorm{\ell}{\left(1-\ee^{-n\cdot}\right)\omega}$ in terms of $\snorm{\ell}{\left(1-\zeta\right)\omega}$.

\begin{lemma}\label{Lem:norm:est:split}
 For any $n\in\N$ and any $\omega\in\esp_{\ell}$ for $\ell\leq 2$ it holds
 \begin{equation*}
  \snorm{\ell}{\left(1-\ee^{-n\cdot}\right)\omega}\leq n^{2-\theta}\snorm{\ell}{\left(1-\zeta\right)\omega}.
 \end{equation*}
\end{lemma}

\begin{proof}
 For $n\in\N$ we can write $1-\ee^{-zn}=\left(1-\ee^{-z}\right)\sum_{k=0}^{n-1}\ee^{-kz}$. Thus we obtain
\begin{equation*}
 \begin{split}
  &\quad\int_{0}^{\infty}z^{\ell}\omega\left(z\right)\left(1-\ee^{-nz}\right)\ee^{-pz}\dz=\sum_{k=0}^{n-1}\int_{0}^{\infty}\left(1-\zeta\left(z\right)\right)z^{\ell}\omega\left(z\right)\ee^{-\left(k+p\right)z}\dz.
 \end{split}
\end{equation*}
Then we find
\begin{equation*}
 \begin{split}
  &\quad\abs{\del_p^{\ell}\int_{0}^{\infty}\left(1-\ee^{-nz}\right)\omega\left(z\right)\ee^{-pz}\dz}\leq \sum_{k=0}^{n-1}\snorm{\ell}{\left(1-\zeta\right)\omega}\frac{\left(k+p+1\right)^{1-\theta}}{\left(k+p\right)^{\ell+1}}\\
  &\leq\snorm{\ell}{\left(1-\zeta\right)\omega}\frac{\left(1+p\right)^{1-\theta}}{p^{\ell+1}}\sum_{k=0}^{n-1}\left(\frac{k+p+1}{1+p}\right)^{1-\theta}\frac{p}{k+p}\leq \snorm{\ell}{\left(1-\zeta\right)\omega}\frac{\left(1+p\right)^{1-\theta}}{p^{\ell+1}}\sum_{k=0}^{n-1}\left(1+k\right)^{1-\theta}\\
  &\leq n^{2-\theta}\snorm{\ell}{\left(1-\zeta\right)\omega}\frac{\left(1+p\right)^{1-\theta}}{p^{\ell+1}}.
 \end{split}
\end{equation*}
This then shows the claim.
\end{proof}

\begin{lemma}\label{Lem:reg:weight}
 For all $\chi>0$ there exists some constant $C>0$ such that
 \begin{equation*}
  \fnorm{k,\chi+1}{\left(1-\zeta\right)\omega}\leq C\fnorm{k+1,\chi}{\omega} 
 \end{equation*}
 holds for $k=0,1$ and all $\omega\in\esp_{k+1,\chi}$.
\end{lemma}

\begin{proof}
 Denoting as before by $\Omega$ the Laplace transform of $\omega$ we have
 \begin{equation*}
  \begin{split}
   \fnorm{0,\chi+1}{\left(1-\zeta\right)\omega}=\sup_{p>0}p\left(1+p\right)^{\chi}\abs{\Omega\left(p\right)-\Omega\left(p+1\right)}.
  \end{split}
 \end{equation*}
We first consider the region $p\leq 1$ and get
\begin{equation*}
 \begin{split}
  &\quad\sup_{0<p\leq 1}p\left(1+p\right)^{\chi}\abs{\Omega\left(p\right)-\Omega\left(p+1\right)}\leq C\sup_{0<p\leq 1}p\left(\abs{\Omega\left(p\right)}+\abs{\Omega\left(p+1\right)}\right)\\
  &\leq C\fnorm{0,\chi}{\omega}\sup_{0<p\leq 1}\left(\left(1+p\right)^{1-\chi}+\frac{p\left(2+p\right)^{1-\chi}}{1+p}\right)\leq C\fnorm{0,\chi}{\omega}.
 \end{split}
\end{equation*}
On the other hand, for $p>1$ we first notice that
\begin{equation*}
 \begin{split}
  &\quad\abs{\Omega\left(p+1\right)-\Omega\left(p\right)}\leq \int_{p}^{p+1}\abs{\Omega'\left(s\right)}\ds\leq \fnorm{1,\chi}{\omega}\int_{p}^{p+1}\frac{\left(1+s\right)^{1-\chi}}{s^2}\ds\leq C\frac{\fnorm{1,\chi}{\omega}}{p^{1+\chi}}.
 \end{split}
\end{equation*}
Using this we obtain for $p> 1$
\begin{equation*}
 \begin{split}
  \sup_{p>1}p\left(1+p\right)^{\chi}\abs{\Omega\left(p\right)-\Omega\left(p+1\right)}\leq C\fnorm{1,\chi}{\omega}\sup_{p>1}\frac{p\left(1+p\right)^{\chi}}{p^{1+\chi}}\leq C\fnorm{1,\chi}{\omega},
 \end{split}
\end{equation*}
showing the claim for $k=0$. The case $k=1$ is similar.
\end{proof}

\begin{remark}\label{Rem:reg:weight}
 Reformulated for the space $X_{k,\chi}$, Lemma~\ref{Lem:reg:weight} reads as
 \begin{equation*}
  \lnorm{k,\chi+1}{\Omega\left(\cdot\right)-\Omega\left(\cdot+1\right)}\leq C\lnorm{k+1,\chi}{\Omega}
 \end{equation*}
 for all $\Omega\in X_{k+1,\chi}$ and $\chi>0$, where the constant may depend on $\chi$.
\end{remark}

 We furthermore have the following Lemma which holds for non-negative functions $\omega$.
\begin{lemma}\label{Lem:norm:est:pos}
 Assume that $\omega\in\mathcal{M}_{+}$ and $\fnorm{0}{\omega}<\infty$. Then it holds $\fnorm{2}{\omega}\leq C\fnorm{0}{\omega}$.
\end{lemma}

\begin{proof}
 Using the non-negativity of $\omega$ we can write
 \begin{equation*}
  \abs{\Omega'\left(p\right)}=\int_{0}^{\infty}x\omega\left(x\right)\ee^{-px}\dx=\frac{1}{p}\int_{0}^{\infty}px\ee^{-\frac{px}{2}}\omega\left(x\right)\ee^{-\frac{p}{2}x}\dx.
 \end{equation*}
Furthermore as $y\ee^{-y/2}\leq 1$ for all $y\geq 0$ we obtain
\begin{equation*}
 \abs{\Omega'\left(p\right)}\leq \frac{1}{p}\int_{0}^{\infty}\omega\left(x\right)\ee^{-\frac{p}{2}x}\dx=\frac{\Omega\left(\frac{p}{2}\right)}{p}\leq 2\fnorm{0}{\omega}\frac{\left(1+\frac{p}{2}\right)^{1-\theta}}{p^2}\leq 2\fnorm{0}{\omega}\frac{\left(1+p\right)^{1-\theta}}{p^2}.
\end{equation*}
Whence $\snorm{1}{\omega}\leq C\fnorm{0}{\omega}$. In the same way we can show $\snorm{2}{\omega}\leq C\fnorm{0}{\omega}$.
\end{proof}

\subsection{Proof of Proposition~\ref{P.repkernels} and some useful estimates on $\Gamma$}\label{Sec:appendix:rep:ker}

\begin{proof}[Proof of Proposition~\ref{P.repkernels}]
Denoting $G(z)\vcc=W(z,1)$, the general approach will be, to construct a (continuous) function $\phi\colon(0,\infty)\to\R$ satisfying
\begin{equation}\label{eq:prob:G}
 \frac{G(z)-G(-1)}{1+z}=\int_{0}^{\infty}\frac{\phi(\eta)}{\eta+z}\deta\quad \text{for all }z\in\C.
\end{equation}
Note that $G(-1)$ is a well-defined constant due to Remark~\ref{Rem:W:minus1} and the reason for subtracting it on the right-hand side of~\eqref{eq:prob:G} is to remove the pole at $z=-1$. Furthermore, it will turn out that this will lead to the occurrence of the singular part of the measure $\Gamma$ (in the statement of Proposition~\ref{P.repkernels}) as the solution to the problem
\begin{equation}\label{eq:prob:Dirac}
 \frac{1}{1+z}=\int_{0}^{\infty}\frac{\widehat{\phi}(\eta)}{\eta+z}\deta
\end{equation}
is explicitly given by $\widehat{\phi}=\delta(\cdot-1)$. 

Using $G(z)=W(z,1)$ we have from~\eqref{eq:W:decay} that 
\begin{equation}\label{eq:growth:G}
 \abs{G(z)}\leq C\max\left\{\abs{z}^{-\alpha},\abs{z}^{\alpha}\right\}.
\end{equation}
Additionally, the function $G$ is analytic in $\C\setminus\left(-\infty,0\right]$ and we denote by $G_{\pm}(z)\vcc= W_{\pm}(z)$ the restrictions to the half-spaces $\left\{\Im(\xi)\geq 0\right\}\setminus\{0\}$ and $\left\{\Im(\xi)\leq 0\right\}\setminus\{0\}$ that exist due to~\eqref{kernel0b}.

Assuming that $\phi$ is continuous and solves~\eqref{eq:prob:G} we obtain by the Sokhotski-Plemelj formula of complex analysis that
\begin{equation*}
 \begin{split}
  \lim_{\nu\to 0^{+}}\frac{G\left(z_0+\nu\im\right)-G\left(-1\right)}{z_0+\nu\im +1}&=-\pi\im\phi\left(\abs{z_0}\right)+PV\int_{0}^{\infty}\frac{\phi\left(\eta\right)}{\eta-\abs{z_0}}\deta\\
  \lim_{\nu\to 0^{+}}\frac{G\left(z_0-\nu\im\right)-G\left(-1\right)}{z_0-\nu\im +1}&=\pi\im\phi\left(\abs{z_0}\right)+PV\int_{0}^{\infty}\frac{\phi\left(\eta\right)}{\eta-\abs{z_0}}\deta
 \end{split}
\end{equation*}
for $z_{0}\in\left(-\infty,0\right)$. Here we use that $\frac{1}{\eta-\abs{z_0}+\nu\im}\to -\pi\im \delta_{\abs{z_0}}+PV\left(\frac{1}{\eta-\abs{z_0}}\right)$.

This motivates to define for $s\in\left(0,\infty\right)$ the function $\phi$ by
\begin{equation*}
 \phi\left(s\right)\vcc=\frac{1}{2\pi\im\left(1-s\right)}\left[\lim_{\nu\to 0^{+}}\left(G\left(-s-\nu\im\right)-G\left(-1\right)\right)-\lim_{\nu\to 0^{+}}\left(G\left(-s+\nu\im\right)-G\left(-1\right)\right)\right].
\end{equation*}
Together with $G_{+}$ and $G_{-}$ and using that it holds $G_{+}(-1)=G_{-}(-1)=G(-1)$ due to Remark~\ref{Rem:W:minus1}, we can rewrite $\phi$ as
\begin{equation*}
 \phi\left(s\right)=\frac{G_{-}\left(-s\right)-G_{-}\left(-1\right)-\left(G_{+}\left(-s\right)-G_{+}\left(-1\right)\right)}{2\pi\im \left(1-s\right)}=\frac{G_{-}\left(-s\right)-G_{+}\left(-s\right)}{2\pi\im \left(1-s\right)}.
\end{equation*}
From this it immediately follows $\phi(1)=\frac{1}{2\pi\im}\left(G'_{-}(-1)-G'_{+}(-1)\right)$. We further note that due to~\eqref{eq:W:decay} and~\eqref{eq:growth:G} we get immediately that
\begin{equation}\label{eq:phi:decay}
 \abs{\phi(s)}\leq C\begin{cases}
                     s^{-\alpha} & s\leq 1\\
                     s^{\alpha-1} & s\geq 1.
                    \end{cases}
\end{equation}
Hence for $z\in\C\setminus\left(-\infty,0\right]$ the integral $\int_{0}^{\infty}\frac{\phi\left(\eta\right)}{z+\eta}\deta$ is well-defined and defines an analytic function.

Next we will show that $\phi$ is locally uniformly Hölder-continuous of exponent $\gamma$, while this precisely means that
\begin{equation}\label{eq:phi:hoel}
 \frac{\abs{\phi(s)-\phi(t)}}{\abs{s-t}^{\gamma}}\leq C\begin{cases}
                                                        \min\left\{s,t\right\}^{-\alpha-\gamma} & \min\left\{s,t\right\}\leq 1\\
                                                        \min\left\{s,t\right\}^{\alpha-1-\gamma} & \min\left\{s,t\right\}\geq 1
                                                       \end{cases}
\end{equation}
for $\abs{s-t}\leq \frac{1}{2}\min\left\{s,t\right\}$ and $s,t>0$. Using this it then follows that we can also extend the integral $\int_{0}^{\infty}\frac{\phi\left(\eta\right)}{z+\eta}\deta$ to $z\in\left(-\infty,0\right)$ either in $\left\{\Im(\xi)\geq 0\right\}\setminus\{0\}$ or $\left\{\Im(\xi)\leq 0\right\}\setminus\{0\}$ by means of the Sokhotski-Plemelj formula.

To show~\eqref{eq:phi:hoel} we first notice that it suffices to consider $G_{+}$ and $G_{-}$ separately. Let $0<t<s$ be such that $s-t\leq t/2$ and $s,t\neq 1$. The two special cases $t=1$ and $s=1$ will be considered separately. We have to distinguish several cases for the different values of $s$ and $t$ and also mention that we can rewrite 
\begin{equation*}
 \frac{G_{+}\left(-s\right)-G_{+}\left(-1\right)}{1-s}-\frac{G_{+}\left(-t\right)-G_{+}\left(-1\right)}{1-t}=\int_{t}^{s}\frac{\int_{-1}^{-\tau}G_{+}'\left(\xi\right)-G_{+}'\left(-\tau\right)\dxi}{\left(1-\tau\right)^2}\dtau.
\end{equation*}
\begin{enumerate}[label=\arabic*.]
 \item $0<t\leq 4/5$:
 \begin{enumerate}[label*=\alph*.]
  \item \label{It:far:one} Consider first $s-t<1-s$. Then it holds $1-s\geq 1/10$ and thus both $s$ and $t$ are separated from $1$. Using~\eqref{eq:W:decay} we can then estimate
  \begin{equation*}
   \begin{split}
    &\quad\abs{\int_{t}^{s}\frac{\int_{-1}^{-\tau}G_{+}'\left(\xi\right)-G_{+}'\left(-\tau\right)\dxi}{\left(1-\tau\right)^2}\dtau}\leq C\abs{\int_{t}^{s}G_{+}\left(-\tau\right)-G_{+}\left(-1\right)-G_{+}'\left(-\tau\right)\left(1-\tau\right)\dtau}\\
    &\leq C\left(t^{-\alpha}\left(s-t\right)+\int_{t}^{s}\tau^{-\alpha-1}\dtau\right)\leq Ct^{-\alpha-\gamma}\left(\left(s-t\right)+s^{\gamma}-t^{\gamma}\right)\leq Ct^{-\alpha-\gamma}\abs{s-t}^{\gamma}.
   \end{split}
  \end{equation*}
  \item \label{It:large:dist} Assume now $1-s\leq s-t$. Then we have $s-t\geq 1/10$. Furthermore, due to~\eqref{eq:W:decay} it easily follows
  \begin{equation}\label{eq:G:pm:est}
   \abs{\frac{G_{\pm}\left(\xi\right)-G_{\pm}\left(-1\right)}{1-\xi}}\leq C\left(\abs{\xi}^{-\alpha}+\abs{\xi}^{\alpha-1}\right).
  \end{equation}
  Using this we immediately get
  \begin{equation*}
   \abs{\frac{G_{+}\left(-s\right)-G_{+}\left(-1\right)}{1-s}-\frac{G_{+}\left(-t\right)-G_{+}\left(-1\right)}{1-t}}\leq Ct^{-\alpha}\leq Ct^{-\alpha-\gamma}\abs{s-t}^{\gamma}.
  \end{equation*}
 \end{enumerate}
 \item $t\in\left[4/5,6/5\right]$:
 \begin{enumerate}[label*=\alph*.]
  \item Consider first $s\in\left(4/5,\min\left\{3/2t,3/2\right\}\right]$. Using~\eqref{eq:hoelder} we then get
  \begin{equation*}
   \begin{split}
    \abs{\int_{t}^{s}\frac{\int_{-1}^{-\tau}G_{+}'\left(\xi\right)-G_{+}\left(-\tau\right)\dxi}{\left(1-\tau\right)^2}\dtau}&\leq C\int_{t}^{s}\frac{\abs{\int_{-1}^{-\tau}\abs{\tau-\xi}^{\gamma}\dxi}}{\left(1-\tau\right)^2}\dtau\\
    &\leq C\int_{t}^{s}\abs{\tau-1}^{\gamma-1}\dtau\leq C\abs{s-t}^{\gamma}.
   \end{split}
  \end{equation*}
  \item For $s\geq \min\left\{3/2t,3/2\right\}$ we get $s-t\geq 2/5$ and thus this case can be treated similarly to~\ref{It:large:dist}. 
 \end{enumerate}
 \item For $t>5/6$ it follows similarly to~\ref{It:far:one} and using also that we always have $s-t\leq t/2$ that
 \begin{equation*}
  \begin{split}
   &\quad\abs{\int_{t}^{s}\frac{\int_{-1}^{-\tau}G_{+}'\left(\xi\right)-G_{+}\left(-\tau\right)\dxi}{\left(1-\tau\right)^2}\dtau}\leq \frac{C}{\left(1-t\right)^2}\abs{\int_{t}^{s}G_{+}\left(-\tau\right)-G_{+}\left(-1\right)-G_{+}'\left(-\tau\right)\left(1-\tau\right)\dtau}\\
   &\leq \frac{C t^{\alpha}}{t^2}\left(s-t\right)+Ct^{\alpha-1}\left(s-t\right)\frac{1}{t}\leq Ct^{\alpha-1}\left(s-t\right)^{\gamma}\frac{\left(s-t\right)^{1-\gamma}}{t}\leq Ct^{\alpha-1-\gamma}\left(s-t\right)^{\gamma}
  \end{split}
 \end{equation*}
 \item \label{It:t:one} For $t=1$ we get similarly as before
 \begin{equation*}
  \begin{split}
   &\quad \abs{\frac{G_{+}\left(-s\right)-G_{+}\left(-1\right)}{1-s}-G_{+}\left(-1\right)}=\abs{\frac{\int_{-1}^{-s}G_{+}'\left(\xi\right)-G_{+}'\left(-1\right)\dxi}{1-s}}\leq C\abs{\frac{\int_{-1}^{-s}\left(-\xi-1\right)^{\gamma}\dxi}{s-1}}\\
   &\leq C\abs{s-1}^{\gamma}.
  \end{split}
 \end{equation*}
 \item The case $s=1$ can be treated similarly as~\ref{It:t:one}
 \end{enumerate}
 The function $G_{-}$ can be treated in exactly the same way.

Next, we will show that $\phi$ in fact satisfies~\eqref{eq:prob:G} and we therefore consider the function
\begin{equation*}
 \Phi\left(z\right)\vcc=\frac{G\left(z\right)-G\left(-1\right)}{1+z}-\int_{0}^{\infty}\frac{\phi\left(\eta\right)}{z+\eta}\deta.
\end{equation*}
Precisely, we will show that $\Phi$ is analytic on $\C$ and satisfies $\abs{\Phi\left(z\right)}\leq C\max\left\{\abs{z}^{-\alpha},\abs{z}^{\alpha-1}\right\}$. Then due to Liouville's Theorem it follows that $\frac{G\left(z\right)}{1+z}=\int_{0}^{\infty}\frac{\phi\left(\eta\right)}{z+\eta}\deta+\frac{G\left(-1\right)}{z+1}$.

It is readily seen that $\Phi$ is analytic in $\C\setminus\left(-\infty,0\right]$ and that $\frac{G(z)-G(-1)}{1+z}$ is analytic in $\C\setminus\left\{0\right\}$. Furthermore, denoting in analogy to $G$ and $W$ by $\Phi_{\pm}$ the restriction of $\Phi$ to $\left\{\Im(z)> 0\right\}\setminus\{0\}$ and $\left\{\Im(z)< 0\right\}\setminus\{0\}$ we have that $\Phi_{\pm}$ can be extended to $\left\{\Im(z)\geq 0\right\}\setminus\{0\}$ and $\left\{\Im(z)\leq 0\right\}\setminus\{0\}$ as for both, $\frac{G(z\pm \nu\im)-G(-1)}{1+z\pm\im\nu}$ and $\int_{0}^{\infty}\frac{\phi(\eta)}{z\pm\im\nu+\eta}\deta$ the limit $\nu\to 0^{+}$ exists for $z\in\left(-\infty,0\right)$. From the construction of $\phi$ it follows additionally that $\Phi_{+}(z)=\Phi_{-}(z)$ for all $z\in\left(-\infty,0\right)$. Thus, by Morera's Theorem we have that $\Phi$ is analytic in $\C\setminus\left\{0\right\}$. It thus remains to show that $\Phi(z)\leq C\max\left\{\abs{z}^{-\alpha},\abs{z}^{\alpha-1}\right\}$ because this together with Riemann's Theorem and Liouville's Theorem implies that $\Phi\equiv 0$. To show the latter estimate we first observe that
\begin{equation*}
 \abs{\frac{G\left(z\right)-G\left(-1\right)}{1+z}}\leq C\max\left\{\abs{z}^{-\alpha},\abs{z}^{\alpha-1}\right\}
\end{equation*}
as a direct consequence of~\eqref{eq:G:pm:est}. It thus remains to show
\begin{equation}\label{eq:int:phi:decay}
 \abs{\int_{0}^{\infty}\frac{\phi\left(\eta\right)}{z+\eta}\deta}\leq C\max\left\{\abs{z}^{-\alpha},\abs{z}^{\alpha-1}\right\}.
\end{equation}
To see this we first consider $\Re(z)\geq 0$. Using~\eqref{eq:phi:decay} we find
\begin{equation*}
 \begin{split}
  \abs{\int_{0}^{\infty}\frac{\phi\left(\eta\right)}{z+\eta}\deta}\leq \frac{1}{\abs{z}}\int_{0}^{\abs{z}}\abs{\phi\left(\eta\right)}\deta+\int_{\abs{z}}^{\infty}\frac{\phi\left(\eta\right)}{\eta}\deta\leq C\max\left\{\abs{z}^{-\alpha},\abs{z}^{\alpha-1}\right\}.
 \end{split}
\end{equation*}
The case $\Re(z)<0$ is slightly more complicated. We can split $z=-\xi+\kappa\im$ with $\xi\in\left(0,\infty\right)$ and $\kappa\in\R$. We first consider the case $\xi<\abs{\kappa}$. For this we get similarly as for $\Re(z)\geq 0$ that
\begin{equation*}
 \begin{split}
  &\quad \abs{\int_{0}^{\infty}\frac{\phi(\eta)}{z+\eta}\deta}\leq \int_{0}^{\infty}\frac{\abs{\phi(\eta)}}{\abs{(\eta-\xi)+\kappa\im}}\deta\leq \frac{1}{\abs{\kappa}}\int_{0}^{2\abs{\kappa}}\abs{\phi(\eta)}\deta+\int_{2\abs{\kappa}}^{\infty}\frac{\abs{\phi(\eta)}}{\eta-\xi}\deta\\
  &\leq C\max\left\{\abs{\kappa}^{-\alpha},\abs{\kappa}^{\alpha-1}\right\}\leq C\max\left\{\abs{z}^{-\alpha},\abs{z}^{\alpha-1}\right\}.
 \end{split}
\end{equation*}
It remains to consider $\xi>\abs{\kappa}$. For this we will use~\eqref{eq:phi:hoel} and also recall that for $z\in\left(-\infty,0\right)$ the integral in~\eqref{eq:int:phi:decay} has to be understood as a limit using $z=\lim_{\nu\to 0^{+}}\left(z\pm \im\nu\right)$. We then get
\begin{equation*}
 \begin{split}
  \int_{0}^{\infty}\frac{\phi\left(\eta\right)}{\eta+z}\deta=\int_{0}^{\infty}\frac{\eta-\xi}{\left(\eta-\xi\right)^2+\kappa^2}\phi\left(\eta\right)\deta-\kappa\im\int_{0}^{\infty}\frac{\phi\left(\eta\right)}{\left(\eta-\xi\right)^2+\kappa^2}\deta=\vcc (I)+(II)
 \end{split}
\end{equation*}
We first consider $(II)$ and obtain by splitting the integral and using~\eqref{eq:phi:decay} that
\begin{equation*}
 \begin{split}
  \abs{(II)}\leq C\abs{\kappa}\int_{0}^{1}\frac{\eta^{-\alpha}}{\left(\eta-\xi\right)^2+\kappa^2}\deta+C\abs{\kappa}\int_{1}^{\infty}\frac{\eta^{\alpha-1}}{\left(\eta-\xi\right)^2+\kappa^2}\deta=\vcc (II)_a+(II)_b.
 \end{split}
\end{equation*}
For $\xi\leq 2$ we obtain by changing variables $\eta\to\xi\eta$ that
\begin{equation*}
 \begin{split}
  \abs{(II)_a}\leq C\xi^{-\alpha}\frac{\abs{\kappa}}{\xi}\int_{0}^{\infty}\frac{1}{\eta^{\alpha}\left(\left(\eta-1\right)^2+\left(\frac{\abs{\kappa}}{\xi}\right)^2\right)}\deta\leq C\xi^{-\alpha}.
 \end{split}
\end{equation*}
On the other hand we obtain for $\xi\geq 2$ that
\begin{equation*}
 \begin{split}
  \abs{(II)_a}\leq C\frac{\abs{\kappa}}{\left(\xi-1\right)^2+\kappa^2}\leq C\xi^{\alpha-1}.
 \end{split}
\end{equation*}
The term $(II)_b$ can be estimated similarly, i.e. we obtain for $\xi\geq 1/2$ that
\begin{equation*}
 \begin{split}
  \abs{(II)_b}\leq \frac{C}{\xi^{1-\alpha}}\frac{\abs{\kappa}}{\xi}\int_{0}^{\infty}\frac{\eta^{\alpha-1}}{\left(\eta-1\right)^2+\left(\frac{\kappa}{\xi}\right)^2}\deta\leq C\xi^{\alpha-1}.
 \end{split}
\end{equation*}
For $\xi\leq 1/2$ we find
\begin{equation*}
 \begin{split}
  \abs{(II)_b}\leq \abs{\kappa}\int_{1}^{\infty}\frac{\eta^{\alpha-1}}{\left(\eta-\frac{1}{2}\right)^2+\kappa^2}\deta\leq C.
 \end{split}
\end{equation*}
It thus remains to estimate $(I)$. For this we introduce a cut-off function $\zeta\colon\R\to [0,1]$ such that $\zeta=1$ on $(-\infty,2)$ and $\zeta=0$ on $[4,\infty]$. We can then rewrite $(I)$ as
\begin{equation*}
 \begin{split}
  (I)=\int_{0}^{\infty}\frac{\eta-\xi}{\left(\eta-\xi\right)^2+\kappa^2}\left(\phi\left(\eta\right)-\zeta\left(\frac{\eta}{\xi}\right)\phi\left(\xi\right)\right)\deta+\int_{0}^{\infty}\frac{\eta-\xi}{\left(\eta-\xi\right)^2+\kappa^2}\zeta\left(\frac{\eta}{\xi}\right)\phi\left(\xi\right)\deta=\vcc (I)_a+(I)_b.
 \end{split}
\end{equation*}
Then on the one hand we obtain
\begin{equation*}
 \begin{split}
  (I)_a&=\int_{0}^{2\xi}\frac{\eta-\xi}{\left(\eta-\xi\right)^2+\kappa^2}\left(\phi(\eta)-\phi(\xi)\right)\deta+\int_{2\xi}^{\infty}\frac{\eta-\xi}{(\eta-\xi)^2+\kappa^2}\Biggl(\phi(\eta)-\zeta\left(\frac{\eta}{\xi}\right)\phi(\xi)\Biggr)\deta\\
  &=\vcc(I)_{a,1}+(I)_{a,2}.
 \end{split}
\end{equation*}
Considering the two terms separately we first find
\begin{equation*}
 \begin{split}
  \abs{(I)_{a,2}}\leq C\int_{2\xi}^{\infty}\frac{\abs{\phi(\eta)}}{\eta}\deta+C\abs{\phi(\xi)}\leq C\max\left\{\xi^{-\alpha},\xi^{\alpha-1}\right\}.
 \end{split}
\end{equation*}
To estimate $(I)_{a,1}$ we split the integral and use~\eqref{eq:phi:hoel} to obtain
\begin{equation*}
 \begin{split}
  &\quad \abs{(I)_{a,1}}=\abs{\int_{0}^{2\xi}\frac{\eta-\xi}{\left(\eta-\xi\right)^2+\kappa^2}\left(\phi\left(\eta\right)-\phi\left(\xi\right)\right)\deta}\leq \abs{\int_{0}^{\frac{2}{3}\xi}\left(\cdots\right)\deta}+\abs{\int_{\frac{2}{3}\xi}^{\frac{3}{2}\xi}\left(\cdots\right)\deta}+\abs{\int_{\frac{3}{2}\xi}^{2\xi}\left(\cdots\right)\deta}\\
  &\leq \frac{C}{\xi}\int_{0}^{\frac{2}{3}\xi}\max\left\{\eta^{-\alpha},\eta^{\alpha-1}\right\}\deta+C\max\left\{\xi^{-\alpha-\gamma},\xi^{\alpha-\gamma-1}\right\}\int_{\frac{2}{3}\xi}^{\frac{3}{2}\xi}\abs{\eta-\xi}^{\gamma-1}\deta\leq C\max\left\{\xi^{-\alpha},\xi^{\alpha-1}\right\}.
 \end{split}
\end{equation*}
Finally we get by changing variables
\begin{equation*}
 \begin{split}
  \abs{(I)_b}&=\abs{\phi(\xi)\int_{0}^{4\xi}\frac{\eta-\xi}{\left(\eta-\xi\right)^2+\kappa^2}\zeta\left(\frac{\eta}{\xi}\right)\deta}=\abs{\phi(\xi)\int_{0}^{4}\frac{\eta-1}{\left(\eta-1\right)^2+\left(\frac{\kappa}{\xi}\right)^2}\zeta\left(\eta\right)\deta}\\
  &\leq \abs{\phi(\xi)}\abs{\left.\log\left(\left(\eta-1\right)^2+\left(\frac{\eta}{\xi}\right)^2\right)\right|_{0}^2}+\abs{\phi(\xi)}\int_{2}^{4}\frac{\eta-1}{\left(\eta-1\right)^2+\left(\frac{\kappa}{\xi}\right)^2}\deta\leq C\max\left\{\xi^{-\alpha},\xi^{\alpha-1}\right\}.
 \end{split}
\end{equation*}
Summarizing, this then shows $\abs{\int_{0}^{\infty}\frac{\phi(\eta)}{z+\eta}\deta}\leq C\max\left\{\abs{\Re(z)}^{-\alpha},\abs{\Re(z)}^{\alpha-1}\right\}$ for $\xi=-\Re(z)>\abs{\Im(z)}=\abs{\kappa}$. From this it follows immediately that we also have
\begin{equation*}
 \abs{\int_{0}^{\infty}\frac{\phi(\eta)}{\eta+z}\deta}\leq C\max\left\{\abs{z}^{-\alpha},\abs{z}^{\alpha-1}\right\}.
\end{equation*}
 Then as already explained it follows by Liouville's Theorem together with Riemann's Theorem that $\Phi\equiv 0$, i.e.
 \begin{equation}\label{eq:rep:G}
  \frac{G\left(z\right)}{1+z}=\int_{0}^{\infty}\frac{\phi\left(\eta\right)}{\eta+z}\deta+\frac{G\left(-1\right)}{1+z}
 \end{equation}
 We define now
 \begin{equation*}
  \widetilde{\Gamma}\left(\xi,\eta\right)\vcc=\frac{\phi\left(\frac{\xi}{\eta}\right)}{\eta}\qquad \text{and}\qquad \Gamma\left(\xi,\eta\right)\vcc=\widetilde{\Gamma}\left(\xi,\eta\right)+W_{\pm}\left(-1\right)\delta(\xi-\eta).
 \end{equation*}
 From~\eqref{eq:phi:decay} one can then easily deduce that
 \begin{equation*}
  \abs{\widetilde{\Gamma}\left(\xi,\eta\right)}\leq C\frac{1}{\left(\eta+\xi\right)^{1-\alpha}}\left(\frac{1}{\xi^{\alpha}}+\frac{1}{\eta^{\alpha}}\right).
 \end{equation*}
 We conclude now the proof by computing the integral
 \begin{equation*}
  \int_{0}^{\infty}\int_{0}^{\infty}\Gamma\left(\xi,\eta\right)\ee^{-\xi y-\eta z}\deta\dxi=\int_{0}^{\infty}\int_{0}^{\infty}\frac{\phi\left(\frac{\xi}{\eta}\right)}{\eta}\ee^{-y\xi-z\eta}\deta\dxi+W_{\pm}\left(-1\right)\int_{0}^{\infty}\ee^{-\xi\left(y+z\right)}\dxi.
 \end{equation*}
 Changing variables in the first integral on the right-hand side and using that $W_{\pm}(-1)=G(-1)$ we further obtain
 \begin{equation*}
  \begin{split}
   \int_{0}^{\infty}\int_{0}^{\infty}\Gamma\left(\xi,\eta\right)\ee^{-\xi y-\eta z}\deta\dxi=\int_{0}^{\infty}\int_{0}^{\infty}\frac{\phi\left(\frac{1}{\eta}\right)}{\eta}\ee^{-\xi \left(y+\eta z\right)}\dxi\deta+\frac{G\left(-1\right)}{y+z}=\int_{0}^{\infty}\frac{\phi\left(\frac{1}{\eta}\right)}{\eta\left(y+\eta z\right)}\deta+\frac{G\left(-1\right)}{y+z}
  \end{split}
 \end{equation*}
Changing variables once more, rearranging and using the relation~\eqref{eq:rep:G} we get
\begin{equation*}
 \begin{split}
  \int_{0}^{\infty}\int_{0}^{\infty}\Gamma\left(\xi,\eta\right)\ee^{-\xi y-\eta z}\deta\dxi=\int_{0}^{\infty}\frac{\phi\left(\eta\right)}{z+\eta y}\deta+\frac{G\left(-1\right)}{y+z}=\frac{1}{y}\int_{0}^{\infty}\frac{\phi\left(\eta\right)}{\frac{z}{y}+\eta}\deta+\frac{G\left(-1\right)}{y+z}=\frac{G\left(\frac{z}{y}\right)}{y+z}
 \end{split}
\end{equation*}
Noting that due to homogeneity it holds $G\bigl(\frac{z}{y}\bigr)=W\left(y,z\right)$ the claim follows.
 \end{proof}

Next we prove some estimates on integrals involving $\Gamma$ that occur in several proofs.

\begin{lemma}\label{Lem:est:Gamma:int}
 For $k,\ell\in\N_0$ such that $1\leq k+\ell\leq 2$ it holds
 \begin{equation*}
  \int_{0}^{\infty}\int_{0}^{\infty}\abs{\Gamma\left(\xi,\eta\right)}\frac{1}{\left(1+\xi\right)^{k+\theta}}\frac{1}{\left(1+\eta\right)^{\ell+\theta}}\deta\dxi\leq C.
 \end{equation*}
\end{lemma}

\begin{proof}
 By symmetry it suffices to consider only the case $k\leq \ell$. Using the estimate from Proposition~\ref{P.repkernels} we obtain
 \begin{equation*}
  \begin{split}
   &\quad \int_{0}^{\infty}\int_{0}^{\infty}\abs{\Gamma\left(\xi,\eta\right)}\frac{1}{\left(1+\xi\right)^{k+\theta}}\frac{1}{\left(1+\eta\right)^{\ell+\theta}}\deta\dxi\\
   &\leq C\int_{0}^{\infty}\int_{0}^{\infty}\frac{1}{\left(\xi+\eta\right)^{1-\alpha}}\left(\frac{1}{\xi^{\alpha}}+\frac{1}{\eta^{\alpha}}\right)\frac{1}{\left(1+\xi\right)^{k+\theta}}\frac{1}{\left(1+\eta\right)^{\ell+\theta}}\deta\dxi+C\int_{0}^{\infty}\frac{1}{\left(\xi+1\right)^{k+\ell+2\theta}}\dxi\\\
   &=C\int_{0}^{\infty}\int_{0}^{\infty}\frac{1}{\xi^{\alpha}}\frac{1}{\left(\xi+\eta\right)^{1-\frac{\theta+\alpha}{2}}}\frac{1}{\left(\xi+\eta\right)^{\frac{\theta-\alpha}{2}}}\frac{1}{\left(1+\xi\right)^{k+\theta}}\frac{1}{\left(1+\eta\right)^{\ell+\theta}}\deta\dxi\\
   &\qquad+C\int_{0}^{\infty}\int_{0}^{\infty}\frac{1}{\eta^{\alpha}}\frac{1}{\left(1+\eta\right)^{\ell+\theta}}\frac{1}{\left(\xi+\eta\right)^{\frac{\theta-\alpha}{2}}}\frac{1}{\left(\xi+\eta\right)^{1-\frac{\theta+\alpha}{2}}}\frac{1}{\left(1+\xi\right)^{k+\theta}}\dxi\deta+C.
  \end{split}
 \end{equation*}
 Estimating $\left(\xi+\eta\right)^{-a}\leq \xi^{-a}$ and $\left(\xi+\eta\right)^{-a}\leq \eta^{-a}$ for $a>0$ it further follows
 \begin{equation*}
  \begin{split}
   &\quad \int_{0}^{\infty}\int_{0}^{\infty}\frac{\abs{\Gamma\left(\xi,\eta\right)}}{\left(1+\xi\right)^{k+\theta}\left(1+\eta\right)^{\ell+\theta}}\deta\dxi\leq C\int_{0}^{\infty}\frac{1}{\xi^{\ell-\frac{\theta-\alpha}{2}}}\frac{1}{\left(1+\xi\right)^{k+\theta}}\dxi\int_{0}^{\infty}\frac{1}{\eta^{\frac{\theta-\alpha}{2}}}\frac{1}{\left(1+\eta\right)^{\ell+\theta}}\deta\\
   &\quad+C\int_{0}^{\infty}\frac{1}{\eta^{\frac{\theta+\alpha}{2}}}\frac{1}{\left(1+\eta\right)^{\ell+\theta}}\deta\int_{0}^{\infty}\frac{1}{\xi^{1-\frac{\theta+\alpha}{2}}}\frac{1}{\left(1+\xi\right)^{k+\theta}}\dxi+C\leq C.
  \end{split}
 \end{equation*}
 \end{proof}

\begin{lemma}\label{Lem:est:Gamma:int:2}
 For $k\in\N$ it holds
 \begin{equation*}
  \int_{0}^{\infty}\int_{0}^{\infty}\abs{\Gamma\left(\xi,\eta\right)}\frac{1}{\eta^{\theta}}\frac{1}{\left(1+\xi\right)^{k+\theta}}\dxi\deta\leq C.
 \end{equation*}
By symmetry the same is true with $\xi$ and $\eta$ interchanged.
\end{lemma}

\begin{proof}
 From the estimate of Proposition~\ref{P.repkernels} we obtain
 \begin{equation*}
  \begin{split}
   &\quad\int_{0}^{\infty}\int_{0}^{\infty}\frac{\abs{\Gamma\left(\xi,\eta\right)}}{\eta^{\theta}\left(1+\xi\right)^{k+\theta}}\deta\dxi\\
   &\leq C\int_{0}^{\infty}\int_{0}^{\infty}\frac{1}{\left(\xi+\eta\right)^{1-\alpha}}\left(\frac{1}{\xi^{\alpha}}+\frac{1}{\eta^{\alpha}}\right)\frac{1}{\eta^{\theta}}\frac{1}{\left(1+\xi\right)^{k+\theta}}\deta\dxi+C\int_{0}^{\infty}\frac{1}{\xi^{\theta}\left(1+\xi\right)^{k+\theta}}\dxi\\
   &\leq C\int_{0}^{\infty}\frac{1}{\left(1+\xi\right)^{k+\theta}}\frac{1}{\xi^{\alpha}}\int_{0}^{\infty}\frac{1}{\left(\eta+\xi\right)^{1-\alpha}}\frac{1}{\eta^{\theta}}\deta\dxi+C\int_{0}^{\infty}\frac{1}{\left(1+\xi\right)^{k+\theta}}\int_{0}^{\infty}\frac{1}{\eta^{\alpha+\theta}}\frac{1}{\left(\xi+\eta\right)^{1-\alpha}}\deta\dxi+C\\
   &\leq C\int_{0}^{\infty}\frac{1}{\left(1+\xi\right)^{k+\theta}}\frac{1}{\xi^{\theta}}\dxi+C\leq C,
  \end{split}
 \end{equation*}
where we changed variables in the $\eta$-integrals in the next to last step.
\end{proof}

\begin{lemma}\label{Lem:Gamma:eta:int}
 There exists a constant $C>0$ such that
 \begin{align*}
    \int_{0}^{\infty}\abs{\Gamma\left(\xi,\eta\right)}\frac{1}{\left(\eta+1\right)^{\theta}}\deta&\leq C \left(\xi^{-\alpha}+\log\left(\xi^{-1}\right)\right)\chi_{(0,1]}(\xi)+\xi^{-\theta}\chi_{[1,\infty)}(\xi)\\
    \int_{0}^{\infty}\abs{\Gamma\left(\xi,\eta\right)}\frac{1}{\left(\eta+1\right)^{k+\theta}}\deta&\leq C \left(\xi^{-\alpha}+\log\left(\xi^{-1}\right)\right)\chi_{(0,1]}(\xi)+\xi^{\alpha-1}\chi_{[1,\infty)}(\xi)\quad \text{for all } k\in\N.
 \end{align*}
\end{lemma}

\begin{proof}
 We first consider $\xi\leq 1$ and observe together with Proposition~\ref{P.repkernels} for all $k\in\N_0$
 \begin{equation*}
  \begin{split}
   &\quad\int_{0}^{\infty}\abs{\Gamma\left(\xi,\eta\right)}\frac{1}{\left(\eta+1\right)^{k+\theta}}\deta\leq C\int_{0}^{\infty}\frac{1}{\left(\eta+1\right)^{k+\theta}\left(\xi+\eta\right)^{1-\alpha}}\left(\frac{1}{\xi^{\alpha}}+\frac{1}{\eta^{\alpha}}\right)\deta+\frac{C}{\left(1+\xi\right)^{k+\theta}}\\
   &\leq \frac{C}{\xi^{\alpha}}\left(\int_{0}^{1}\frac{1}{\left(\xi+\eta\right)^{1-\alpha}}\deta+\int_{1}^{\infty}\frac{1}{\eta^{1+k+\theta-\alpha}}\deta\right)\\
   &\quad+C\left(\int_{0}^{\xi}\frac{1}{\eta^{\alpha}\left(\xi+\eta\right)^{1-\alpha}}\deta+\int_{\xi}^{1}\frac{1}{\eta}\deta+\int_{1}^{\infty}\frac{1}{\eta^{1+k+\theta}}\deta\right)+C\\
   &\leq \frac{C}{\xi^{\alpha}}\left(\frac{\left(\xi+1\right)^{\alpha}-\xi^{\alpha}}{\alpha}+C\right)+C+C\log\left(\frac{1}{\xi}\right)\leq C\left(\xi^{-\alpha}+\log\left(\xi^{-1}\right)\right).
  \end{split}
 \end{equation*}
 For $\xi\geq 1$ we have to treat the two cases $k=0$ and $k\in\N$ separately. For $k=0$ we find
 \begin{equation*}
  \begin{split}
   &\quad\int_{0}^{\infty}\abs{\Gamma\left(\xi,\eta\right)}\frac{1}{\left(\eta+1\right)^{\theta}}\deta\leq C\int_{0}^{\infty}\frac{1}{\left(\eta+1\right)^{\theta}\left(\xi+\eta\right)^{1-\alpha}}\left(\frac{1}{\xi^{\alpha}}+\frac{1}{\eta^{\alpha}}\right)\deta+\frac{C}{\left(1+\xi\right)^{\theta}}\\
   &\leq \frac{C}{\xi^{\alpha}}\int_{0}^{\infty}\frac{1}{\left(\eta+\xi\right)^{1-\alpha}\eta^{\theta}}\deta+\int_{0}^{\infty}\frac{1}{\left(\eta+\xi\right)^{1-\alpha}\eta^{\theta+\alpha}}\deta+\frac{C}{\xi^{\theta}}\\
   &=\frac{C}{\xi^{\theta}}\left(\int_{0}^{\infty}\frac{1}{\left(1+\eta\right)^{1-\alpha}\eta^{\theta}}\deta+\int_{0}^{\infty}\frac{1}{\left(1+\eta\right)^{1-\alpha}\eta^{\theta+\alpha}}\deta+1\right)\leq \frac{C}{\xi^{\theta}}.
  \end{split}
 \end{equation*}
On the other hand, for $k\in\N$ we obtain similarly
\begin{equation*}
 \begin{split}
  &\quad\int_{0}^{\infty}\abs{\Gamma\left(\xi,\eta\right)}\frac{1}{\left(\eta+1\right)^{k+\theta}}\deta\leq C\int_{0}^{\infty}\frac{1}{\left(\eta+1\right)^{k+\theta}\left(\xi+\eta\right)^{1-\alpha}}\left(\frac{1}{\xi^{\alpha}}+\frac{1}{\eta^{\alpha}}\right)\deta+\frac{C}{\left(\xi+1\right)^{k+\theta}}\\
   &\leq \frac{C}{\xi}\int_{0}^{\infty}\frac{1}{\left(\eta+1\right)^{k+\theta}}\deta+\frac{C}{\xi^{1-\alpha}}\int_{0}^{\infty}\frac{1}{\left(\eta+1\right)^{k+\theta}\eta^{\alpha}}\deta+\frac{C}{\xi^{k+\theta}}\leq \frac{C}{\xi^{1-\alpha}}.
 \end{split}
\end{equation*}
\end{proof}
  
\subsection{Proof or Lemma~\ref{L.uniqueprefactor}}\label{Sec:proof:prefactor}

We first recall that $K( y,z) =2+\varepsilon W (y,z)$, as well as the representation formula for $W$ in terms of $\Gamma$ and the corresponding estimate given in Proposition~\ref{P.repkernels}. We define the bilinear functional
\begin{equation*}
 \B_{K}\left(\omega_1,\omega_2\right)\left(x\right)\vcc=\frac{1}{2x}\int_{0}^{x}K\left(y,x-y\right)\omega_1\left(y\right)\omega_2\left(x-y\right)\dy.
\end{equation*}

There exists a non-trivial constant solution of \eqref{S1E4} given by $\mu=c_{\varepsilon }$ with
\[
1=\frac{c_{\varepsilon }}{2}\int_{0}^{1}K( s,1{-}s)
\ds=c_{\varepsilon }\left[ 1+\frac{\varepsilon }{2}\int_{0}^{1}W(s,1{-}s) \ds\right].
\]

Notice that $c_{\varepsilon }\rightarrow 1$ if $\varepsilon \rightarrow 0$. We have proved in \cite{NV13a} that the non-trivial
solutions of (\ref{S1E4}) are close to $c_{\varepsilon }$ if $\varepsilon $ is small enough
in the sense that 
\begin{equation}\label{eq:pclose}
 \pnorm{0}{\mu -c_{\varepsilon }}\leq \delta  \quad \text{where } \delta \text{ is small if } \varepsilon \text{ is small}. 
\end{equation}
Note the definition of $\pnorm{0}{\cdot}$ in \eqref{norm1def}.

We recall that \eqref{S1E4} is equivalent to $\mu =\B_{K}( \mu ,\mu )$ and write $\mu =c_{\varepsilon }+m.$ Then
\begin{equation}\label{S8E1}
m-2c_{\varepsilon }\B_{2}\left( 1,m\right) =2\varepsilon c_{\varepsilon }\B_{W}\left( 1,m\right) +\B_{K}\left(m,m\right).
\end{equation}
We multiply \eqref{S8E1} by $x$ and take
the Laplace transform to obtain $-\frac{\dd{M}}{\dd{p}}-J_{1}=J_{2}+J_{3}$
with
\begin{align*}
J_{1} &=\int_{0}^{\infty }2xc_{\varepsilon }\B_{2}(1,m) \ee^{-xp}\dx\,, \qquad 
J_{2} =\int_{0}^{\infty }2x\varepsilon c_{\varepsilon }\B_{W}( 1,m) \ee^{-xp}\dx \\
J_{3} &=\int_{0}^{\infty }x\B_{K}( m,m) \ee^{-xp}\dx.
\end{align*}

We now compute these quantities. We have
\begin{align*}
J_{1} &=\int_{0}^{\infty }2xc_{\varepsilon }\B_{2}(1,m) \ee^{-xp}\dx =2c_{\varepsilon }\int_{0}^{\infty }\ee^{-xp}\int_{0}^{x}m\left(x-y\right)\dy\dx  \\
&=2c_{\varepsilon }\int_{0}^{\infty }\int_{y}^{\infty }m(x{-}y)
\ee^{-xp}\dx\dy = 2c_{\varepsilon }\int_{0}^{\infty }\ee^{-py}\int_{0}^{\infty }m(x) \ee^{-xp}\dx\dy =\frac{2c_{\varepsilon }}{p}M(p) \,.
\end{align*}

Concerning $J_{2}$ we find similarly $J_{2}=\varepsilon c_{\varepsilon }\int_{0}^{\infty }\int_{0}^{\infty}W( x,y) m( x) \ee^{-( x+y) p}\dx\dy$.
We then use the representation formula from Proposition~\ref{P.repkernels} for $\frac{W( x,y) }{ y+x }$ to obtain
\begin{align*}
J_{2} &=\varepsilon c_{\varepsilon }\int_{0}^{\infty }\int_{0}^{\infty }\frac{W( x,y) }{x+y }m( x) (x+y) \ee^{-( x+y) p}\dx\dy \\
&=-\varepsilon c_{\varepsilon }\frac{\dd}{\dd p}\int_{0}^{\infty }\int_{0}^{\infty } \Gamma ( \xi ,\eta) \int_{0}^{\infty}\int_{0}^{\infty }m(x) \ee^{-( x+y) p}\ee^{-\xi x-\eta y}\dx\dy\deta\dxi \\
&=-\varepsilon c_{\varepsilon }\frac{\dd}{\dd p}\left( \int_{0}^{\infty }M(\xi +p) \int_{0}^{\infty } \frac{\Gamma( \xi ,\eta) }{\eta +p}\right)\deta\dxi \\
&=-\varepsilon c_{\varepsilon }\int_{0}^{\infty }\frac{\dd M}{\dd p}( \xi+p) \int_{0}^{\infty } \frac{\Gamma ( \xi ,\eta ) 
}{\eta +p}\deta\dxi+\varepsilon c_{\varepsilon }\int_{0}^{\infty }M( \xi+p) \int_{0}^{\infty } \frac{\Gamma ( \xi ,\eta ) 
}{( \eta +p) ^{2}}\deta\dxi.
\end{align*}

Therefore, using the definition of $\psnorm{\ell}{m}$ we obtain
\begin{equation*}
 \begin{split}
  \abs{J_{2}}\leq C \varepsilon\int_{0}^{\infty }\int_{0}^{\infty }\abs{\Gamma ( \xi,\eta )} \left[ \frac{\psnorm{1}{m}}{( \xi +p)^{2}(\eta +p) }+\frac{\psnorm{0}{m}}{( \xi +p) ( \eta +p) ^{2}}\right]\dxi\deta.
 \end{split}
\end{equation*}
Changing variables we find together with the homogeneity of $\Gamma$
\begin{equation*}
 \begin{split}
  \abs{J_2}&\leq C\pnorm{1}{m}\eps\int_{0}^{\infty}\int_{0}^{\infty}\abs{\Gamma\left(\xi,\eta\right)}\left(\frac{1}{( \xi +p)^{2}(\eta +p) }+\frac{1}{( \xi +p) ( \eta +p) ^{2}}\right)\dxi\deta\\
  &=C\eps\frac{\pnorm{1}{m}}{p^2}\int_{0}^{\infty}\int_{0}^{\infty}\abs{\Gamma\left(\xi,\eta\right)}\left(\frac{1}{( \xi +1)^{2}(\eta +1) }+\frac{1}{( \xi +1) ( \eta +1) ^{2}}\right)\dxi\deta\leq C\eps\frac{\pnorm{1}{m}}{p^2},
 \end{split}
\end{equation*}
where we used that the last integral is bounded due to Lemma~\ref{Lem:est:Gamma:int}. We finally estimate $J_3$. We have
\begin{equation*}
 \begin{split}
  J_{3}&=\frac{1}{2}\int_{0}^{\infty}\ee^{-px}\int_{0}^{x}K\left(y,x-y\right)m\left(y\right)m\left(x-y\right)\dy\dx\\
  &=\frac{1}{2}\int_{0}^{\infty}\int_{0}^{\infty}K\left(y,z\right)m\left(y\right)m\left(z\right)\ee^{-\left(y+z\right)p}\dy\dz\\
  &=\int_{0}^{\infty}\int_{0}^{\infty}m\left(y\right)m\left(z\right)\ee^{-\left(y+z\right)p}\dz\dy+\frac{\eps}{2}\int_{0}^{\infty}\int_{0}^{\infty}W\left(y,z\right)m\left(y\right)m\left(z\right)\ee^{-\left(y+z\right)p}\dz\dy\\
  &=\left(M\left(p\right)\right)^2+\frac{\eps}{2}\int_{0}^{\infty}\int_{0}^{\infty}W\left(y,z\right)m\left(y\right)m\left(z\right)\ee^{-\left(y+z\right)p}\dz\dy=\vcc J_{3,1}+J_{3,2}.
 \end{split}
\end{equation*}
Obviously we have $\abs{J_{3,1}}\leq \frac{\pnorm{0}{m}^2}{p^2}$. Furthermore
\begin{equation*}
 \begin{split}
  J_{3,2}&=\frac{\eps}{2}\int_{0}^{\infty}\int_{0}^{\infty}\Gamma\left(\xi,\eta\right)\int_{0}^{\infty}\int_{0}^{\infty}m\left(y\right)m\left(z\right)\left(y+z\right)\ee^{-\left(y+z\right)p}\ee^{-\xi y-\eta z}\dz\dy\deta\dxi\\
  &=-\frac{\eps}{2}\frac{\dd}{\dd p}\int_{0}^{\infty}\int_{0}^{\infty}\Gamma\left(\xi,\eta\right)\int_{0}^{\infty}\int_{0}^{\infty}m\left(y\right)m\left(z\right)\ee^{-\left(y+z\right)p}\ee^{-\xi y-\eta z}\dz\dy \deta\dxi\\
  &=-\frac{\eps}{2}\frac{\dd}{\dd p}\int_{0}^{\infty}\int_{0}^{\infty}\Gamma\left(\xi,\eta\right)M\left(\xi+p\right)M\left(\eta+p\right)\deta\dxi\\
  &=-\eps \int_{0}^{\infty}\int_{0}^{\infty}\Gamma\left(\xi,\eta\right)M'\left(\xi+p\right)M\left(\eta+p\right)\deta\dxi,
 \end{split}
\end{equation*}
where we used the symmetry of $\Gamma$ in the last step. Then, using the definition of $\psnorm{k}{\cdot}$, changing variables, using the homogeneity of $\Gamma$ and Lemma~\ref{Lem:est:Gamma:int} we obtain
\begin{equation*}
 \begin{split}
  \abs{J_{3,2}}\leq C\eps \pnorm{0}{m}\psnorm{1}{m}\int_{0}^{\infty}\int_{0}^{\infty}\frac{\abs{\Gamma\left(\xi,\eta\right)}}{\left(\xi+p\right)^2\left(\eta+p\right)}\deta\dxi\leq C\eps\frac{\pnorm{0}{m}\psnorm{1}{m}}{p^2}.
 \end{split}
\end{equation*}
Consequently we have
\begin{equation*}
 \abs{J_3}\leq \frac{C}{p^2}\left(\pnorm{0}{m}^2+\eps\pnorm{0}{m}\psnorm{1}{m}\right).
\end{equation*}
Thus we obtain the equation
\begin{equation}\label{eq:pM}
 M'\left(p\right)+\frac{2c_{\eps}}{p}M\left(p\right)=R\left(p\right) 
\end{equation}
with 
\begin{equation*}
 \abs{R\left(p\right)}\leq C\frac{\pnorm{0}{m}^2+\eps\left(\pnorm{1}{m}+\pnorm{0}{m}\psnorm{1}{m}\right)}{p^2}.
\end{equation*}
Using $m=\mu-c_{\eps}$ as well as~\eqref{eq:pclose} we also obtain for $\delta$ and $\eps$ sufficiently small that
\begin{equation}\label{eq:pR}
  \abs{R\left(p\right)}\leq \frac{C}{p}\left(\delta\pnorm{0}{m}+\eps\psnorm{1}{m}\right).
\end{equation}
From~\eqref{eq:pM} we get $\del_{p}\left(p^{2c_{\eps}}M\left(p\right)\right)=p^{2c_{\eps}}R\left(p\right)$ and since $2c_{\eps}$ is close to two and $R\left(p\right)$ is bounded by $p^{-2}$ we obtain that the right-hand side is integrable. Moreover $p^{2c_{\eps}}M\left(p\right)\to 0$ as $p\to 0$, whence $M\left(p\right)=p^{-2c_{\eps}}\int_{0}^{p}\xi^{2 c_{\eps}}R\left(\xi\right)\dxi$. Thus, we obtain together with~\eqref{eq:pR} that
\begin{equation}\label{eq:pM:1}
 \abs{M\left(p\right)}\leq \frac{C}{p}\left(\delta\pnorm{0}{m}+\eps\psnorm{1}{m}\right).
\end{equation}
With this it follows from~\eqref{eq:pM} and~\eqref{eq:pR} that
\begin{equation*}
 \abs{M'\left(p\right)}\leq \frac{C}{p}\abs{M\left(p\right)}+\abs{R\left(p\right)}\leq \frac{C}{p^2}\left(\delta\pnorm{0}{m}+\eps\psnorm{1}{m}\right),
\end{equation*}
yielding $\psnorm{1}{m}\leq C\delta \pnorm{0}{m}$ for sufficiently small $\eps$. Using this in~\eqref{eq:pM:1} then gives
\begin{equation*}
 \pnorm{0}{m}\leq C\delta \pnorm{0}{m} \quad \text{with } \delta\to 0\text{ for } \eps\to 0. 
\end{equation*}
Thus for $\eps$ sufficiently small it follows $m=0$, i.e.\ $\mu=c_{\eps}$.

\subsection*{Acknowledgements} 
The authors acknowledge support through the CRC 1060 \textit{The mathematics of emergent effects} at the University of Bonn that is funded through the German Science Foundation (DFG).

 \providecommand{\bysame}{\leavevmode\hbox to3em{\hrulefill}\thinspace}
\providecommand{\MR}{\relax\ifhmode\unskip\space\fi MR }
% \MRhref is called by the amsart/book/proc definition of \MR.
\providecommand{\MRhref}[2]{%
  \href{http://www.ams.org/mathscinet-getitem?mr=#1}{#2}
}
\providecommand{\href}[2]{#2}

\end{document}